\documentclass[1p]{elsarticle}

\usepackage{enumitem}

\usepackage{graphicx}
\usepackage{caption}
\usepackage{subcaption}
\usepackage{pdfpages,float}

\usepackage[T1]{fontenc}

\usepackage{amsfonts,latexsym,amssymb,amsmath,amsthm,mathrsfs}
\usepackage{color}
\usepackage{thmtools}
\usepackage{thm-restate}
\usepackage{caption}
\usepackage[colorlinks, citecolor=blue, hypertexnames=false]{hyperref}

\declaretheorem[name=Theorem,numberwithin=section]{thm}
\newtheorem{lem}[thm]{Lemma}
\newtheorem{prop}[thm]{Proposition}

\theoremstyle{definition}
\newtheorem{defn}[thm]{Definition}
\newtheorem{rem}[thm]{Remark}
\newtheorem{assume}[thm]{Assumption}
\newtheorem{claim}[thm]{Claim}

\numberwithin{equation}{section}
\numberwithin{table}{section}
\numberwithin{figure}{section}

\newcommand{\e}{\varepsilon}
\newcommand{\p}{\varphi}

\newcommand{\al}{\alpha}

\newcommand{\Om}{\Omega}

\newcommand{\R}{\mathbb{R}}
\newcommand{\Z}{\mathbb{Z}}
\newcommand{\N}{\mathbb{N}}

\newcommand{\C}{\mathbb{C}}
\newcommand{\Q}{\mathbb{Q}}
\newcommand{\E}{\mathbb{E}}
\newcommand{\Pb}{\mathbb{P}}

\newcommand{\D}{\mathscr{D}}
\newcommand{\Dd}{\mathcal{D}}

\newcommand{\bmo}{{\rm BMO}}

\newcommand{\hardy}{H^1_{{\rm prod}}(\mu) }

\newcommand{\intav}{-\!\!\!\!\!\!\int}

\DeclareMathOperator{\supp}{supp}

\def\ls{\lesssim}
\def\gs{\gtrsim}

\def\laz{\langle}
\def\raz{\rangle}
\def\noz{\nonumber}

\allowdisplaybreaks

\begin{document}
\setlength{\parskip}{5pt}
\begin{frontmatter}

\title{Nonhomogeneous $T(1)$ Theorem  on Product Quasimetric Spaces\tnoteref{mytitlenote}}
\tnotetext[mytitlenote]{ The first and second authors are supported by DP170101060. The second author is supported by a Lift-off Fellowship of the Australian Mathematical Society. The third author is supported by DP160100153.  The fourth author is partially supported by National Science Foundation DMS grant \# 1800057 and Australian Research Council DP 190100970.}

\author[mymainaddress]{Ji Li}
\ead{ji.li@mq.edu.au}

\author[mysecondaryaddress]{Trang T.T. Nguyen}
\ead{trang.t.nguyen1@mymail.unisa.edu.au}

\author[mysecondaryaddress]{Lesley A. Ward}
\ead{lesley.ward@unisa.edu.au}

\author[mythirdaddress]{Brett D. Wick
}
\ead{bwick@wustl.edu}

\address[mymainaddress]{Department of Mathematics\\ Macquarie University\\ NSW 2109, Australia}

\address[mysecondaryaddress]{School of Information Technology and Mathematical Sciences\\ University of South Australia\\ Mawson Lakes SA 5095, Australia}

\address[mythirdaddress]{Department of Mathematics \& Statistics\\ Washington University -- St. Louis\\ One Brookings Drive, St. Louis, MO USA 63130-4899}

\begin{abstract}
In this paper, we provide a non-homogeneous  $T(1)$ theorem on product spaces~$(X_1 \times X_2, \rho_1 \times \rho_2, \mu_1 \times \mu_2)$ equipped with a quasimetric~$\rho_1 \times \rho_2$ and a Borel measure~$\mu_1 \times \mu_2$, which, need not be doubling but satisfies an upper control on the size of quasiballs.
\end{abstract}


\end{frontmatter}


\tableofcontents

\section{Introduction And Statement Of Main Results}
\setcounter{equation}{0}

The \emph{$T(1)$ Theorem} of David and Journ\'{e} was a ground-breaking result that gave easily-checked criteria for a SIO~$T$ to be bounded from~$L^2(\mathbb{R}^n)$ to~$L^2(\mathbb{R}^n)$ \cite{DJ84}.
The name is derived from the fact that these criteria are expressed in terms of the image~$T(1)$ of the constant function~$1$.
Specifically, the~$T(1)$ theorem says that a Calder\'on--Zygmund operator is bounded on~$L^2(\R^n)$ if and only if $T(1) \in \bmo(\R^n)$; $T^*(1) \in \bmo(\R^n)$, and~$T$ satisfies the so-called weak boundedness property.  Here~$T^*$ is the adjoint operator of~$T$.  While incredibly useful, sometimes checking the behavior on $T(1)$ is not obvious nor direct and instead one wants to test on a function $b$ that has properties similar to the function $1$ (a so called para-accretive function).  This extension was also investigated and established by David, Journ\'{e} and Semmes \cite{DJS85} providing at $T(b)$ theorem.

The setting of this classical $T(1)$ theorem of David and Journ\'{e} is Euclidean space~$\mathbb{R}^n$, equipped with the usual Euclidean metric and Lebesgue measure. In particular, the functions~$f$ that the operator~$T$ acts on are defined on~$\mathbb{R}^n$.
Since then, several remarkable results have been obtained along this line of research, further unveiling the strong relationship between the function space~$\bmo$ and the $L^2$-boundedness of singular integral operators. In the last twenty years or so, substantial progresses have been made by Tolsa \cite{T01a,T01b,T01c, T03}, Nazarov--Treil--Volberg \cite{NTV03}, in showing that a suitable version of  $T(1)$ theorem also hold without the doubling property of the measure.
Specifically, in $\mathbb R^n$, the standard Lebesgue measure can be replaced by a more general measure $\mu$ satisfying only the \emph{growth condition}:\ \  for some fixed positive constants $C$ and $d\in (0, n]$,
$$ \mu(B(x,r))\leq Cr^d \quad {\rm for\ all\ } x\in \mathbb R^n,\ \  r>0.$$

When it comes to the \emph{multi-parameter setting} (or \emph{product setting}), meaning one considers  singular integral operators that are invariant under non-isotropic dilations
$(x_1, \ldots, x_n) \rightarrow (\delta_1x_1, \ldots, \delta_n x_n )$, where $x_i \in \mathbb R^{n_i}$, $\delta_i >0$, the theory of Calder\'{o}n and Zygmund for the one-parameter setting does not carry over in a straightforward way. Multiparameter harmonic analysis was introduced in the '70s and studied extensively in the '80s, led by S.-Y. A. Chang, R. Fefferman, R. Gundy, J. Journ\'e, J. Pipher, E. Stein and others (see for example \cite{Cha,GS,CF1,Fef,CF2,FS82,CF3,KM,Fef86,Fef87,P}).
 The first $T(1)$ type theorem for product spaces was proved by Journ\'e \cite{Jou85}, where he formulated the assumptions in the language of vector-valued Calder\'on--Zygmund theory. Later, this has been studied extensively by many authors via Haar expansions, or via product representation theorem, see for example (but not limited to)  Pott--Villarroya \cite{PV13}, Martikainen \cite{Mar12}, Herr\'an \cite{He}, Ou \cite{Ou15},  Li--Martikainen--Vuorinen \cite{LMV}.  Han, Lin and the first author \cite{HLL15} also established the $T(1)$ theorem on the tensor product of spaces of homogeneous type in the sense of Coifman and Weiss, by using the product version of discrete reproducing formula.

A recent breakthrough was due to Hyt\"{o}nen and Martikainen \cite{HM14}, where they proved a $T(1)$ theorem on bi-parameter Euclidean spaces~$\R^{n_1} \times \R^{n_2}$, equipped with an upper doubling measure.

Ten year ago, Hyt\"onen \cite{H10} set up a more general setting for non-homogeneous analysis on metric spaces $(X, d, \mu)$, where $(X, d)$ is geometrically doubling, meaning that every ball of radius $r$ can be covered by at most $N$ balls of radius $r/2$, and the measure $\mu$ satisfies the so-called upper doubling condition as follows:  there exists a dominating function $\lambda: X\times\mathbb R_+\to\mathbb R_+$ and a positive constant $C_\lambda$ such that for all $x\in X$ and $r>0$,
$$ {\rm(i)}:\ r\to\lambda(x,r)\ {\rm is \ increasing\ },\quad {\rm(ii)}:\ \lambda(x,2r)\leq C_\lambda \lambda(x,r),\quad {\rm(iii)}:\  \mu(B(x,r))\leq \lambda(x,r). $$
Later, many results on the $T(b)$ and local $T(b)$ theorems  were extended to this general setting. See, for example, \cite{HK12,HM12a,HM12b,HYY12} and the references therein.

The goal of this paper is to establish a full version of the product $T(1)$ theorem for Calder\'on--Zygmund type singular integrals $T$ on non-homogeneous metric measure spaces setting.  This result will marry the multiparameter theory with that of non-homogeneous harmonic analysis.

To be more precise,
let $(\mathcal{X},\rho,\mu) := (X_1 \times X_2, \rho_1 \times \rho_2, \mu_1 \times \mu_2)$ be a non-homogeneous bi-parameter quasimetric space such that for $i = 1,2$ we have $X_i$ is a set,
 $\rho_i$ is a quasimetric on~$X_i$ satisfying the regularity property, $\mu_i$ is an upper doubling measure on~$X_i$.
On $(\mathcal{X},\rho,\mu)$,
let $C^{\eta}_{c}(\mathcal{X}) = C_c^{\eta}(X_1) \times C_c^{\eta}(X_2)$, $\eta >0$ denote the space of continuous functions~$f$ with compact support such that
\begin{equation}\label{eq:cts_func}
\|f\|_{\eta} := \sup_{x_1 \neq y_1, x_2 \neq y_2}
 \frac{|f(x_1,x_2) - f(y_1,x_2) - f(x_1,y_2) + f(y_1,y_2)|}
 {\rho_1(x_1,y_1)^{\eta} \, \rho_2(x_2,y_2)^{\eta}} < \infty.
\end{equation}
We study the operators which satisfy the four Assumptions~\ref{assum_1}--\ref{assum_4} as in Section~\ref{subsec:assume}.
\begin{defn}\label{defn:bipara_SIO}
  A linear operator $T: C^{\eta}_{c}(\mathcal{X}) \rightarrow C^{\eta}_{c}(\mathcal{X})'$ is called a \emph{bi-parameter singular integral operator (bi-parameter SIO)} if it satisfies Assumptions~\ref{assum_1}--\ref{assum_4} on size and regularities.
\end{defn}

The main result of this paper is as follows.
\begin{thm}\label{thm:T1thm_ver3}
Let $(\mathcal{X},\rho,\mu)$ be a non-homogeneous bi-parameter quasimetric space.
Let $T$ be a bi-parameter SIO as defined above.
Assume also that~$T$ satisfies Assumptions~\ref{assum:5b}--\ref{assum:7b} on the weak boundedness properties, and $S(1) \in \bmo_{\textup{prod}}(\mu)$ for all
$S \in \{T,T^*,T_1,T^*_1\}$.
Then~$T$ extends to a bounded operator on~$L^2(\mu)$ and $\|T\| \ls 1$, with a bound depending only on the assumptions and the $\bmo_{\textup{prod}}(\mu)$ norms of the four $S(1)$.
\end{thm}

\noindent \emph{Outline of the proof of Theorem~\ref{thm:T1thm_ver3}.}
Inspired by the idea in~\cite{NTV03},
our goal is divided into three main steps:
\begin{enumerate}
  \item[Step 1] Establish Theorem~\ref{thm:T1thm_ver3} under the \emph{a priori} assumption that the operator norm~$\|T\|$ of~$T$ is finite, that is, $\|T\| < \infty$. See Theorem~\ref{thm:T1thm_ver1}.
  \item[Step 2] Establish Theorem~\ref{thm:T1thm_ver3} under the weaker \emph{a priori} assumption that for all functions $f, g \in C^{\eta}_{\textup{c}}(\mathcal{X}) $ which are supported on a rectangle~$R \subset \mathcal{X}$, we have
      $|\langle Tf,g \rangle| \leq C(R) \|g\|_{L^2(\mu)} \|g\|_{L^2(\mu)}.$ See Theorem~\ref{thm:T1thm_ver2}.
  \item[Step 3] Establish Theorem~\ref{thm:T1thm_ver3}, which the full version of a $T(1)$ theorem on $(\mathcal{X},\rho,\mu)$ in which the   assumption of \emph{a priori} boundedness of~$T$ is relaxed.
\end{enumerate}

Initially, it may seem surprising that in Step~1, we assume an \emph{a priori} bound on the operator norm~$\|T\|$, and then prove that in fact there is then a bound depends (only) on the assumptions and the~$\bmo$ norms of~$S(1)$.
The big picture is that in Step~2, this assumption of an \emph{a priori} bound on~$\|T\|$ is weakened, and in Step~3, it is removed entirely.

In fact, the major result of this paper is Step~1. This is because Steps~2 and~3 can always be reduced to Step~1.

We first state the result in Step~1.
\begin{restatable}{thm}{Tonethmverone}\label{thm:T1thm_ver1}
Let $(\mathcal{X},\rho,\mu)$ be a non-homogeneous bi-parameter quasimetric space.
Let $T$ be a bi-parameter singular integral operator as in Definition \ref{defn:bipara_SIO}, such that $\|T\| < \infty$.
Assume also that~$T$ satisfies Assumptions~\ref{assum_5}--\ref{assum_6} on the weak boundedness properties, and $S(1) \in \bmo_{\textup{prod}}(\mu)$ for all
$S \in \{T,T^*,T_1,T^*_1\}$.

Then there holds that $\|T\| \ls 1$, with a bound depending only on the assumptions and the $\bmo_{\textup{prod}}(\mu)$ norms of the four functions $S(1)$, but independent of the \emph{a priori} bound on~$\|T\|$.
\end{restatable}

There is a clear path toward proving Theorem~\ref{thm:T1thm_ver1}.
We adapt and combine methods of proofs that are developed in two recent articles: (1) a non-homogeneous $T(1)$ theorem on product Euclidean spaces $\R^{n_1} \times \R^{n_1}$, which deals with an extension from the one-parameter setting to the product setting~\cite{HM14}, and (2) a non-homogeneous $T(b)$ theorem on one-parameter metric measure spaces~$(X,d,\mu)$, which deals with an extension from the setting of Euclidean spaces to the setting of metric measure spaces~\cite{HM12a}.

In terms of the proof of Theorem~\ref{thm:T1thm_ver1}, we would like to highlight the powerful tools of representing arbitrary singular integral operators as averages of simpler components, in particular, Haar functions; and of eliminating bad cases by averaging over random dyadic lattices.
These techniques have already been used by other authors. See for example~\cite{HM14,HM12a,NTV03}.  However, when we use them in our setting, it requires an additional layer of difficulty, due to the complexity of working both in nonhomogeneous setting whose underlying spaces are of homogeneous type while incorporating the multiparameter theory.

\noindent \emph{Outline of the proof of Theorem~\ref{thm:T1thm_ver1}.}
Fix functions $f$, $g \in C^{\eta}_{\textup{c}}(\mathcal{X}) $ so that
$\|f\|_{L^2(\mu)} = \|g\|_{L^2(\mu)} = 1$,
$\supp f \subset R$, $\supp g \subset R$ where $R$ is a rectangle in~$\mathcal{X}$, and
\[0.7 \|T\| \leq |\langle Tf,g \rangle|.\]
We want to show that $|\langle Tf,g \rangle| \leq C$, where the constant~$C$ is independent of~$f$ and~$g$.
To do so, we will decompose
\begin{equation}\label{eq4:T1part}
|\langle Tf,g \rangle|
\leq |\langle Tf_{\text{good}}, g_{\text{good}} \rangle|
+ |\langle Tf_{\text{good}}, g_{\text{bad}} \rangle|
+ |\langle Tf_{\text{good}}, g \rangle|,
\end{equation}
where $f_{\text{good}}$ and $f_{\text{bad}}$ are called the \emph{good part} and \emph{bad part} of~$f$, respectively; a similar statement holds for the function $g$.
Then we will show that
\begin{eqnarray}
  |\langle Tf_{\text{good}}, g \rangle| &\leq& 0.05 \|T\|, \label{eq1:T1part}\\
  |\langle Tf_{\text{good}}, g_{\text{bad}} \rangle| &\leq& 0.05 \|T\| \quad
  \text{and} \label{eq2:T1part}\\
  |\langle Tf_{\text{good}}, g_{\text{good}} \rangle| &\leq& C + 0.1 \|T\|. \label{eq3:T1part}
\end{eqnarray}
This implies $0.7 \|T\| \leq |\langle Tf,g \rangle| \leq C + 0.2 \|T\|$, and so $\|T\| \leq 2C$. This establishes Theorem~\ref{thm:T1thm_ver1}.

In fact, inequalities~\eqref{eq1:T1part} and~\eqref{eq2:T1part} are straightforward. Therefore, most of the proof is devoted to showing~\eqref{eq3:T1part}. Please refer to Section~\ref{subsec:initial_reduce} for the details of this reduction.
Also refer to Section~\ref{subsec:outline} for further outline of the proof of~\eqref{eq3:T1part}.
The proof of Theorem~\ref{thm:T1thm_ver1} is given in Sections~\ref{sec:strategy}--\ref{sec:mix_para}.

The main result in Step~2 is stated formally in Theorem~\ref{thm:T1thm_ver2} below.
\begin{thm}\label{thm:T1thm_ver2}
 Let $(\mathcal{X},\rho,\mu)$ be a non-homogeneous bi-parameter quasimetric space.
Let $T$ be a bi-parameter singular integral operator as in  Definition \ref{defn:bipara_SIO}, such that  for all functions $f, g \in C^{\eta}_{\textup{c}}(\mathcal{X}) $ with $\supp f \subset R$ and $\supp g \subset R$ for some rectangle~$R \subset \mathcal{X}$ we have
      \[|\langle Tf,g \rangle| \leq C(R) \|f\|_{L^2(\mu)} \|g\|_{L^2(\mu)}.\]
Assume also that~$T$ satisfies Assumptions~\ref{assum_5}--\ref{assum_6} on the weak boundedness properties, and $S(1) \in \bmo_{\textup{prod}}(\mu)$ for all
$S \in \{T,T^*,T_1,T^*_1\}$.

Then there holds that $\|T\| \ls 1$, with a bound depending only on the assumptions and the $\bmo_{\textup{prod}}(\mu)$ norms of the four functions $S(1)$, but independent of the \emph{a priori} bound~$C(R)$.
\end{thm}

\noindent \emph{Outline of the proof of Theorem~\ref{thm:T1thm_ver2}.}
Let~$E$ be the set of functions in $C^{\eta}_{\textup{c}}(\mathcal{X}) $ such that $\supp f \subset R$, $\supp g \subset R$ where $R$ is a rectangle in~$\mathcal{X}$, and  $\|f\|_{L^2(\mu)} = \|g\|_{L^2(\mu)} = 1$.
 Define
\[M(R) := \sup \{|\langle Tf,g \rangle|: f, g \in E\}.\]
Fix $f, g \in E$ such that
\[0.7 M(R) \leq |\langle Tf,g \rangle|.\]
We want to show that $M(R) \leq C$, where the constant~$C$ is independent of~$f$ and~$g$.
As in~\eqref{eq4:T1part}, we decompose $|\langle Tf, g \rangle|$ into three parts, and aim to prove that
\begin{eqnarray*}
  |\langle Tf_{\text{good}}, g \rangle| &\leq& 0.05 M(R), \\
  |\langle Tf_{\text{good}}, g_{\text{bad}} \rangle| &\leq& 0.05 M(R) \quad
  \text{and} \\
  |\langle Tf_{\text{good}}, g_{\text{good}} \rangle| &\leq& C + 0.1 M(R).
\end{eqnarray*}
This implies  $\|T\| = M(R) \leq 2C$.

Notice that the structure of the proof of Theorem~\ref{thm:T1thm_ver2} in Step~2 is similar to that of Theorem~\ref{thm:T1thm_ver1} in Step~1, except replacing $\|T\|$ by~$M(R)$. In fact, the proof of Theorem~\ref{thm:T1thm_ver2} is almost the same as that of Theorem~\ref{thm:T1thm_ver1}. See Section~\ref{sec:T1 ver 2} for further details of the proof of Theorem~\ref{thm:T1thm_ver2}, and how it can be reduced to the proof of Theorem~\ref{thm:T1thm_ver1}.

In Step~3, we consider the most general case, where the operator~$T$ does not necessarily satisfy the \emph{a priori }bounds.
The idea is that we will approximate the operator~$T$ by truncated operators~$T_{\tau}$, where~$\tau>0$.
As shown in~\cite[p.165]{NTV03}, it is reasonable to think of boundedness of~$T$ as the uniform boundedness of~$T_{\tau}$.
In Section~\ref{sec:T1 full} we will show that hypotheses of Theorem~\ref{thm:T1thm_ver2} hold for~$T_{\tau}$.
This implies that the sequence~$T_{\tau}$ is uniformly bounded on~$L^2(\mu)$ and hence, $T$ is bounded on~$L^2(\mu)$. Thus, Theorem~\ref{thm:T1thm_ver3} is established.

Throughout the paper,
we use the usual notation
$\langle f\rangle _E=\intav_{E} f\,d\mu = \frac{1}{\mu(E)}\int_{E} f\,d\mu$, where $E \subset X$.
We denote by~$C$ a positive constant that is independent of the main parameters but may vary from line to line.
If $f\le Cg$, we write $f\ls g$ or $g\gs f$; and if $f \ls g\ls
f$, we  write $f\sim g$, or~$f \sim_C g$ when we want to emphasise the implied constant.

This paper is organised as follows.
In Section \ref{sec:defns}, we present the mathematical concepts needed later in the paper.
In Section~\ref{subsec:product_BMO}, we introduce the Hardy and~$\bmo$ spaces on non-homogeneous setting and discuss their duality.
The  proof of Theorem~\ref{thm:T1thm_ver1} is carried out in Sections~\ref{sec:strategy}--\ref{sec:mix_para}.
In Section~\ref{sec:T1 ver 2}, we establish Theorem~\ref{thm:T1thm_ver2}.
In Section~\ref{sec:T1 full}, we prove the main result of this paper which is Theorem~\ref{thm:T1thm_ver3}.
Finally, in Section~\ref{sec:appln}, we describe an application of our result to Carleson measures on reproducing kernel Hilbert spaces over the bidisc.  These results are applicable to the Hardy spaces and Bergman spaces, and more generally to the Besov-Sobolev spaces of analytic functions on the bidisc.

\section{Preliminaries}\label{sec:defns}
This section is organised as follows.
In Section \ref{subsec:spaceX}, we introduce metric measure spaces $(X,d,\mu)$ and spaces of homogeneous type~$(X,\rho,\mu)$.
In Section~\ref{subsec:upper_dbl}, we define upper doubling measures.
In Section~\ref{subsec:product_space}, we introduce non-homogeneous bi-parameter quasimetric spaces~$(\mathcal{X},\rho,\mu)$.
In Section~\ref{subsec:assume}, we introduce the class of Calder\'on-Zygmund operators~$T$ that we work with.
In Section~\ref{sec:WBP}, we discuss the weak boundedness properties of the operator~$T$.
In Section \ref{sec:dyadiccubes}, we explain systems of dyadic cubes.
In Section~\ref{subsec:good_bad_cube}, we discuss good and bad cubes.
In Section~\ref{subsec:Haar_func}, we recall Haar functions.
In Section~\ref{subsec:intepretation}, we give the interpretation of~$T(1)$.

\subsection{Spaces of homogeneous type~$(X,\rho,\mu)$}\label{subsec:spaceX}
In this section, we define metrics, quasimetrics and doubling measures, which will let us define metric measure spaces and spaces of homogeneous type.

A \emph{metric} on a set $X$ is a function $d: X\times X \rightarrow [0, \infty)$ satisfying the following conditions for all $x, y, z\in X$:
  $d(x,y) = 0$ if and only if $x=y$,
  $d(x,y) = d(y,x)$, and
  $d(x,z) \leq d(x,y) + d(y,z)$.
The pair~$(X,d)$ is called a \emph{metric space}.

A \emph{quasimetric} on a set~$X$ is a function~$\rho: X\times X \rightarrow [0, \infty)$ satisfying the same conditions as a metric, except that the triangle inequality is replaced by a \emph{quasitriangle inequality}:
\begin{equation}\label{eq:quasi_triangle_inequality}
\rho(x,z) \leq A_0\rho(x,y) + A_0d\rho(y,z),
\end{equation}
where the \emph{quasitriangle constant} $A_0 \geq 1$ does not depend on $x, y$ or $z$.
In this paper, we assume that the quasimetric~$\rho$ also satisfies the following \emph{regularity} property: for all~$\e>0$ there exists $A(\e) < \infty$ so that
\begin{equation}\label{eq:regularity}
\rho(x,y) \leq (1+\e)\rho_1(x,z)+ A(\e)\rho_1(z,y).
\end{equation}
The pair~$(X,\rho)$ is called a \emph{quasimetric space}.
A \emph{doubling measure} on the space $(X, d)$ is a measure $\mu$  on $X$ such that the balls in $(X, d)$ are $\mu$-measurable sets, and the following condition holds for all $x \in X$ and all  $r > 0$:
\begin{equation}\label{eq:dbl_measure_1}
0 < \mu(B(x, 2r)) \leq A_1 \mu(B(x, r)) < \infty,
\end{equation}
where $B(x,r) := \{y \in X: d(x,y)<r\}$ and the \emph{doubling constant} $A_1 \geq 1$ does not depend on $x$ and $r$.
In fact, inequality~\eqref{eq:dbl_measure_1} implies a more general property of the doubling measure~$\mu$. Namely, for all~$x \in X$, $r > 0$ and~$\lambda > 1$ we have
\begin{equation}\label{eq:dbl_measure_2}
    \mu(B(x, \lambda r)) \leq A_1^{1 + \log_2 \lambda} \mu(B(x, r)).
\end{equation}

A doubling measure on a quasimetric space~$(X,\rho)$ is defined in the same way, except the metric~$d$ is replaced by the quasimetric~$\rho$ in the definition of~$B(x,r)$.
When a metric space~$(X,d)$ is equipped with a doubling measure~$\mu$, the triple~$(X,d,\mu)$ is called a \emph{metric measure space}.
When a quasimetric space~$(X,\rho)$ is equipped with a doubling measure~$\mu$, the triple~$(X,\rho,\mu)$ is called a \emph{space of homogeneous type}.

We note that every quasimetric space is \emph{geometrically doubling} \cite{CW71}, meaning that there exists $N$ such that every ball $B(x,r)$ in~$(X,\rho)$ can be covered by at most $N$ balls of radius~$r/2$.


\subsection{Upper doubling measures}\label{subsec:upper_dbl}
Let~$(X,\rho)$ be a quasimetric space as defined in Section~\ref{subsec:spaceX}.
We equip~$X$ with a so-called \emph{upper doubling} measure~$\mu$, which is a generalisation of both doubling measures and those with the \emph{upper power bound }property $\mu(B(x,r)) \leq Cr^d$.
\begin{defn}[\textbf{Upper doubling measure}]\label{defn:upper_dbl}
  A Borel measure~$\mu$ on a quasimetric space~$(X,\rho)$ is called \emph{upper doubling} if there exist a dominating function $\lambda: X \times (0,\infty) \rightarrow (0,\infty)$ and a constant~$C_{\lambda}$ so that $r \mapsto \lambda(x,r)$ is non-decreasing for each $x \in X$, $\lambda(x,2r) \leq C_{\lambda}\lambda(x,r)$ and $\mu(B(x,r)) \leq \lambda(x,r)$ for all~$x \in X$ and~$r>0$. We set $t_{\lambda}:= \log_2 C_{\lambda}$. We call $C_{\lambda}$ the \emph{upper doubling constant}.
  The space~$(X,\rho,\mu)$ is called a \emph{non-homogeneous quasimetric space}.
\end{defn}
We also assume that the dominating function~$\lambda$ satisfies the additional symmetry property that there exists a constant~$C$ such that for all $x,y \in X$ with~$\rho(x,y) \leq r$,
\begin{equation}\label{eq2:up_dbl}
  \lambda(x,r) \leq C\lambda(y,r).
\end{equation}
Then it follows from~\eqref{eq2:up_dbl} that
\begin{equation}\label{eq1:up_dbl}
  \lambda(x,\rho(x,y)) \sim \lambda(y,\rho(x,y)).
\end{equation}
Assumption~\eqref{eq2:up_dbl} is reasonable because as shown in~\cite[Proposition~1.1]{HYY12}, on a non-homogeneous quasimetric space~$(X,\rho,\mu)$ where~$\mu$ is an upper doubling measure associated with the dominating function~$\lambda$, the function $\Lambda(x,r) := \inf_{z \in X} \lambda(z, r + \rho(x,z))$ satisfies that $r \mapsto \Lambda(x,r)$ is non-decreasing for each $x \in X$, $\Lambda(x,2r) \leq C_{\lambda}\Lambda(x,r)$, $\mu(B(x,r)) \leq \Lambda(x,r)$, $\Lambda(x,r) \leq \lambda(x,r)$ and $\Lambda(x,r) \leq C_{\lambda} \Lambda(y,r)$ if $\rho(x,y) \leq r$.

The following result in~\cite{HM12a} will be useful in later computation.
\begin{lem}[\text{\cite[Lemma~2.2]{HM12a}}]
\label{upper_dbl_lem1}
  Let~$(X,\rho,\mu)$ be a non-homogeneous quasimetric space.
  Then for every ball $B = B(x_B, r_B) \subset X$ and for every $\e > 0$ we have that
  \[\int_{X \backslash B} \frac{\rho(x,x_B)^{-\e}}{\lambda(x_B,\rho(x,x_B))}\, d\mu(x) \leq C_{\lambda} A_{\e} r_B^{-\e},\]
  where $A_{\e} = 2^{\e}/(2^{\e}-1)$.
\end{lem}
\subsection{Non-homogeneous bi-parameter quasimetric spaces $(\mathcal{X},\rho,\mu)$}\label{subsec:product_space}

We denote the \emph{product} of such spaces by $\mathcal{X} := X_1 \times \cdots \times X_n$ equipped with a product quasimetric $\rho := \rho_1 \times \cdots \times \rho_n$ and a product upper doubling measure $\mu := \mu_1 \times \cdots \times \mu_n$, where for each $i =1,\ldots,n$, $\rho_i$ is a quasimetric and $\mu_i$ is an upper doubling measure. The space~$(\mathcal{X},\rho,\mu)$ is called \emph{non-homogeneous product quasimetric space}.

In this paper, we only work on the case~$n =2$.
Specifically, we consider \emph{non-homogeneous bi-parameter quasimetric spaces}
$$(\mathcal{X},\rho,\mu) := (X_1 \times X_2, \rho_1 \times \rho_2, \mu_1 \times \mu_2).$$
The quasimetrics~$\rho_1$ and~$\rho_2$ satisfy the quasitriangle inequality~\eqref{eq:quasi_triangle_inequality} with quasitriangle constant~$A_0^{(1)}$ and~$A_0^{(2)}$, respectively. Without loss of generality, we assume that $A_0^{(1)} = A_0^{(2)} = A_0$.

\subsection{Singular integral operators on $(\mathcal{X},\rho,\mu)$}\label{subsec:assume}
Recall that in the original $T(1)$ theorem of David and Journ\'{e}, the hypotheses involve assumptions on the size and smoothness of the kernel, a weak boundedness property, and~$\bmo$ conditions.
Therefore, when it comes to the bi-parameter setting, it is natural to formulate assumptions about `mixed type' conditions, which is actually what Hyt\"{o}nen and Martikainen did in~\cite{HM14}. Following their idea, we define a class of singular integral operators on non-homogeneous bi-parameter quasimetric spaces.

On a non-homogeneous bi-parameter quasimetric space $(\mathcal{X},\rho,\mu)$,
let $C^{\eta}_{c}(\mathcal{X}) = C_c^{\eta}(X_1) \times C_c^{\eta}(X_2)$, $\eta >0$ denote the space of continuous functions~$f$ with compact support such that
\[
\|f\|_{\eta} := \sup_{x_1 \neq y_1, x_2 \neq y_2}
 \frac{|f(x_1,x_2) - f(y_1,x_2) - f(x_1,y_2) + f(y_1,y_2)|}
 {\rho_1(x_1,y_1)^{\eta} \, \rho_2(x_2,y_2)^{\eta}} < \infty.
\]

Let~$T$ be a linear operator continuously mapping $C_c^{\eta}(\mathcal{X})$ to its dual~$C_c^{\eta}(\mathcal{X})'$.
Besides the usual adjoint~$T^*$, we also consider the \emph{partial adjoint}~$T_1$ defined by
\[\langle T_1(f_1 \otimes f_2), g_1 \otimes g_2\rangle
= \langle T(g_1 \otimes f_2), f_1 \otimes g_2\rangle.\]
$T$ is assumed to satisfy the following four assumptions.
\begin{assume}\label{assum_1}\textbf{(Full kernel representation)}
If~$f=f_1 \otimes f_2 \in C^{\eta}_{c}(\mathcal{X}) $ and $g=g_1\otimes g_2 \in C^{\eta}_{c}(\mathcal{X}) $ with
$f_1,g_1: (X_1,\rho_1,\mu_1) \rightarrow \C$, $f_2,g_2: (X_2,\rho_2,\mu_2) \rightarrow \C$, $\supp f_1 \cap \supp g_1 = \emptyset$ and~$\supp f_2 \cap \supp g_2 = \emptyset$, we have the \emph{full kernel representation}
\begin{eqnarray*}
  \langle Tf,g\rangle &=& \int_{\mathcal{X}} \int_{\mathcal{X}} K(x,y)f(y)g(x) \,d\mu(x) \,d\mu(y)\\
   &=& \iint_{X_1 \times X_2}\iint_{X_1 \times X_2} K(x_1,x_2;y_1,y_2)f(y_1,y_2)g(x_1,x_2)\,d\mu(x_1,x_2)\,d\mu(y_1,y_2)
\end{eqnarray*}
where the kernel~$K$ is a function
\[K: (\mathcal{X} \times \mathcal{X})\backslash \{(x,y) \in \mathcal{X} \times \mathcal{X}: x_1 = y_1 \text{ or } x_2 = y_2\} \rightarrow \C.\]
For the kernel representation for~$T^*, T_1$ and~$T^*_1$, we use
\begin{equation*}
  K^*(x,y) := K(y_1,y_2;x_1,x_2), \hspace{0.5cm}
  K_1(x,y) := K(y_1,x_2;x_1,y_2), \text{ and} \ \
  K_1^*(x,y) := K(x_1,y_2;y_1,x_2).
\end{equation*}
\end{assume}

\begin{assume}\label{assum_2}\textbf{(Full standard size and regularity estimates)} There exist $\al_1$, $\al_2 > 0$ and~$ C_K >~1$ such that
the kernel~$K$ satisfies the \emph{size condition}
\[|K(x,y)| \leq C \frac{1}{\lambda_1(x_1,\rho_1(x_1,y_1))}\frac{1}
{\lambda_2(x_2,\rho_2(x_2,y_2))},\]
the \emph{H\"{o}lder condition}
\begin{eqnarray*}
  \lefteqn{|K(x,y) - K(x,(y_1,y_2')) - K(x,(y_1',y_2)) + K(x,y')|} \\ &\leq&C\frac{\rho_1(y_1,y_1')^{\al_1}}{\rho_1(x_1,y_1)^{\al_1}\lambda_1(x_1,\rho_1(x_1,y_1))}
  \frac{\rho_2(y_2,y_2')^{\al_2}}{\rho_2(x_2,y_2)^{\al_2}\lambda_2(x_2,\rho_2(x_2,y_2))}
\end{eqnarray*}
whenever~$\rho_1(x_1,y_1) \geq C_K \rho_1(y_1,y_1')$ and~$\rho_2(x_2,y_2) \geq C_K \rho_2(y_2,y_2')$,
and the \emph{mixed H\"{o}lder and size condition}
\[|K(x,y) - K(x,(y_1',y_2))|
\leq C \frac{\rho_1(y_1,y_1')^{\al_1}}{\rho_1(x_1,y_1)^{\al_1}\lambda_1(x_1,\rho_1(x_1,y_1))}
\frac{1}{\lambda_2(x_2,\rho_2(x_2,y_2))}\]
whenever~$\rho_1(x_1,y_1) \geq C_K \rho_1(y_1,y_1')$.
 The same conditions are imposed on~$K^*$, $K_1$, and~$K^*_1$.
\end{assume}
\begin{assume}\label{assum_3}\textbf{(Partial kernel representation)}
  If~$f=f_1 \otimes f_2 \in C^{\eta}_{c}(\mathcal{X}) $ and $g=g_1\otimes g_2 \in C^{\eta}_{c}(\mathcal{X}) $ with
  $\supp f_1 \cap \supp g_1 = \emptyset$, then we assume the \emph{partial kernel representation}
  \[ \langle Tf,g\rangle = \int_{X_1} \int_{X_1} K_{f_2,g_2}(x_1,y_1)f_1(y_1)g_1(x_1) \,d\mu_1(x_1) \,d\mu_1(y_1).\]
\end{assume}
\begin{assume}\label{assum_4}\textbf{(Partial boundedness $\times$ standard estimates)}
The kernel from Assumption~\ref{assum_3}
\[K_{f_2,g_2}:(X_1 \times X_1)\backslash \left\{(x_1,y_1)\in X_1 \times X_1: x_1=y_1\right\} \rightarrow \C\]
is assumed to satisfy the \emph{size condition}
\[|K_{f_2,g_2}(x_1,y_1)| \leq C(f_2,g_2)\frac{1}{\lambda_1(x_1,\rho_1(x_1,y_1))},\]
for some constant $C(f_2,g_2)$,
and the \emph{H\"{o}lder conditions}
\[|K_{f_2,g_2}(x_1,y_1) - K_{f_2,g_2}(x_1',y_1)| \leq C(f_2,g_2) \frac{\rho_1(x_1,x_1')^{\al_1}}{\rho_1(x_1,y_1)^{\al_1}\lambda_1(x_1,\rho_1(x_1,y_1))}\]
whenever~$\rho_1(x_1,y_1) \geq C_K \rho_1(x_1,x_1')$, and
\[|K_{f_2,g_2}(x_1,y_1) - K_{f_2,g_2}(x_1,y_1')| \leq C(f_2,g_2) \frac{\rho_1(y_1,y_1')^{\al_1}}{\rho_1(x_1,y_1)^{\al_1}\lambda_1(x_1,\rho_1(x_1,y_1))}\]
whenever~$\rho_1(x_1,y_1) \geq C_K \rho_1(y_1,y_1')$.

We assume that~$C(f_2,g_2) \ls \|f_2\|_{L^2(\mu_2)}\|g_2\|_{L^2(\mu_2)}$.
We also assume the analogous representation and properties with a kernel~$K_{f_1,g_1}$ in the case $\supp f_2 \cap \supp g_2 = \emptyset.$
\end{assume}

Here we do assume some type of $L^2$-boundedness separately on~$(X_1,\mu_1)$ and $(X_2,\mu_2)$, namely the bounds for~$C(f_1,g_1)$ and~$C(f_2,g_2)$. However, as noted in~\cite{HM14}, this seems to be a standard assumption in the classical homogeneous work, and we also require it in our non-homogeneous setting.
\begin{rem}\label{rem:prod kernel}
  Let~$K$ be a kernel on~$\mathcal{X} = X_1 \times X_2$ which can be written in form
  \[
  K(x,y) = K(x_1,x_2;y_1,y_2) := K_1(x_1,y_1)K_2(x_2,y_2),
  \]
  where $K_1(x_1,y_1)$, $K_2(x_2,y_2)$ are standard kernels on~$X_1$ and~$X_2$, respectively.
  See~\cite[Section~2.3]{HM12a} for the definition of (one-parameter) standard kernels.
  Then the kernel~$K$ satisfies Assumptions~\ref{assum_1}--\ref{assum_4}.
\end{rem}

The following lemmas will be useful for later calculations, they will give estimates on the quantity $|\langle T(\phi), \theta \rangle|$, where $\phi$ and~$\theta \in L^2(\mu)$.
\begin{lem}\label{lem1:property_kernel}
  Let $T: L^2(\mu) \rightarrow L^2(\mu)$ be a bi-parameter SIO as defined in Definition~\ref{defn:bipara_SIO}.
   Take $\phi = \phi_1 \otimes \phi_2 \in L^2(\mu)$ and $\theta = \theta_1 \otimes \theta_2 \in L^2(\mu)$.
   Suppose $\supp \phi_1 \cap \supp \theta_1 = \emptyset$ and
   $\supp \phi_2 \cap \supp \theta_2 = \emptyset$.
    Moreover, suppose that
    for all $x_1 \in \supp \theta_1 =: E_{\theta_1}$ and $y_1, z \in \supp \phi_1 =: E_{\phi_1}$ we have $\rho_1(y_1,z) \leq \rho_1(x_1,z)/C_K$ for some constant~$C_K$ depending on the kernel of~$T$,
    and for all $x_2 \in \supp \theta_2 =: E_{\theta_2}$ and $y_2, w \in \supp \phi_2 =: E_{\phi_2}$ we have $\rho_2(y_2,w) \leq \rho_2(x_2,w)/C_K$.
    Also assume that $\int \phi_1 \,d\mu_1 = \int \phi_2 \,d\mu_2 = 0$.
  Then
  \begin{eqnarray*}
|\langle T(\phi_1 \otimes \phi_2), \theta_1 \otimes \theta_2 \rangle|
     &\ls&  \|\phi_1\|_{L^2(\mu_1)} \|\phi_2\|_{L^2(\mu_2)}
     \mu_1(E_{\phi_1})^{1/2} \mu_2(E_{\phi_2})^{1/2}
      \\
     &&\times \int_{E_{\theta_1}} \frac{\rho_1(y_1,z)^{\al_1}}{\rho_1(x_1,z)^{\al_1}\lambda_1(z,\rho_1(x_1,z))}
     |\theta_1(x_1)|\, d\mu_1(x_1) \\
     &&\times  \int_{E_{\theta_2}} \frac{\rho_2(y_2,w)^{\al_2}}{\rho_2(x_2,w)^{\al_2}\lambda_2(w,\rho_2(x_2,w))}
     |\theta_2(x_2)|\, d\mu_2(x_2).
  \end{eqnarray*}
\end{lem}
\begin{proof}
  Since $\supp \phi_1 \cap \supp \theta_1 = \emptyset$ and $\supp \phi_2 \cap \supp \theta_2 = \emptyset$,
using Assumption~\ref{assum_1}, we can write $\langle T(\phi_1 \otimes \phi_2), \theta_1 \otimes \theta_2 \rangle$ as
\[\int_{E_{\phi_1}} \int_{E_{\phi_2}} \int_{E_{\theta_1}} \int_{E_{\theta_1}} K(x,y)
\phi_1(y_1) \phi_2(y_2) \theta_1(x_1) \theta_2(x_2)
\,d\mu_2(x_2)\,d\mu_1(x_1)\,d\mu_2(y_2)\,d\mu_1(y_1)\]
  then replace~$K(x,y)$ by
  \[K(x,y) - K(x,(y_1,w)) - K(x,(z,y_2))+ K(x,(z,w)),\]
  where~$z \in E_{\phi_1}$ and~$w \in E_{\phi_2}$ are arbitrary. The replacement can be done since $\int \phi_1\,d\mu_1 = \int \phi_2\,d\mu_2 = 0$.

  As $\rho_1(y_1,z) \leq \rho_1(x_1,z)/C_K$ for all $x_1 \in  E_{\theta_1}$ and $y_1, z \in E_{\phi_1}$ ,
    and $\rho_2(y_2,w) \leq \rho_2(x_2,w)/C_K$ for all $x_2 \in E_{\theta_2}$ and $y_2, w \in E_{\phi_2}$,
   using the H\"{o}lder condition for~$K$ in Assumption~\ref{assum_2}, this yields
  \begin{eqnarray*}
    \lefteqn{|K(x,y) - K(x,(y_1,w)) - K(x,(z,y_2)+ K(x,(z,w))|} \\
     &\ls&  \frac{\rho_1(y_1,z)^{\al_1}}{\rho_1(x_1,z)^{\al_1}\lambda_1(z,\rho_1(x_1,z))}
  \frac{\rho_2(y_2,w)^{\al_2}}{\rho_2(x_2,w)^{\al_2}\lambda_2(w,\rho_2(x_2,w))}. \\
  \end{eqnarray*}
    Hence, Lemma~\ref{lem1:property_kernel} follows.
\end{proof}

\begin{lem}\label{lem:property_kernel}
  Let $T: L^2(\mu) \rightarrow L^2(\mu)$ be a bi-parameter SIO as defined in Definition~\ref{defn:bipara_SIO}.
   Take $\phi = \phi_1 \otimes \phi_2 \in L^2(\mu)$ and $\theta = \theta_1 \otimes \theta_2 \in L^2(\mu)$.
  If $\supp \phi_1 \cap \supp \theta_1 = \emptyset$; for all $x_1 \in \supp \theta_1 =: E_{\theta_1}$ and $y_1, z \in \supp \phi_1 =: E_{\phi_1}$ we have $\rho_1(y_1,z) \leq \rho_1(x_1,z)/C_K$ for some constant~$C_K$ depending on the kernel of~$T$; and $\int \phi_1 \,d\mu_1 = 0$
  then
  \begin{eqnarray*}
    |\langle T(\phi_1 \otimes \phi_2), \theta_1 \otimes \theta_2 \rangle|
     &\ls& \|\phi_1\|_{L^2(\mu_1)} \|\phi_2\|_{L^2(\mu_2)}  \mu_1(E_{\phi_1})^{1/2} \|\theta_2\|_{L^2(\mu_2)}
      \\
     &&\times \int_{E_{\theta_1}} \frac{\rho_1(y_1,z)^{\al_1}}{\rho_1(x_1,z)^{\al_1}\lambda_1(z,\rho_1(x_1,z))}
     |\theta_1(x_1)|\, d\mu_1(x_1).
  \end{eqnarray*}

  Similarly, if $\supp \phi_2 \cap \supp \theta_2 = \emptyset$; for all $x_2 \in \supp \theta_2 =: E_{\theta_2}$ and $y_2, w \in \supp \phi_2 =: E_{\phi_2}$ we have $\rho_2(y_2,w) \leq \rho_2(x_2,w)/C_K$; and $\int \phi_2 \,d\mu_2 = 0$
  then
  \begin{eqnarray*}
    |\langle T(\phi_1 \otimes \phi_2), \theta_1 \otimes \theta_2 \rangle|
     &\ls& \|\phi_1\|_{L^2(\mu_1)} \|\phi_2\|_{L^2(\mu_2)} \|\theta_1\|_{L^2(\mu_1)}  \mu_2(E_{\phi_2})^{1/2}
      \\
     &&\times \int_{E_{\theta_2}} \frac{\rho_2(y_2,w)^{\al_2}}{\rho_2(x_2,w)^{\al_2}\lambda_2(w,\rho_2(x_2,w))}
     |\theta_2(x_2)|\, d\mu_2(x_2).
  \end{eqnarray*}
\end{lem}
\begin{proof}
We will prove the first conclusion. The second conclusion can be shown analogously.
  Since $\supp \phi_1 \cap \supp \theta_1 = \emptyset$,
using Assumption~\ref{assum_3}, we can write $\langle T(\phi_1 \otimes \phi_2), \theta_1 \otimes \theta_2 \rangle$ as
\[\int_{E_{\phi_1}} \int_{E_{\theta_1}}K_{\phi_2,\theta_2}(x_1,y_1) \phi_1(y_1) \theta_1(x_1) \,d\mu_1(x_1) \, d\mu_1(y_1),\]
and replace $K_{\phi_1,\theta_1}(x_1,y_1)$ by
$K_{\phi_2,\theta_2}(x_1,y_1) - K_{\phi_2,\theta_2}(x_1,z),$
where $z \in E_{\phi_1}$ is arbitrary.
 The replacement can be done since~$\int \phi_1\,d\mu_1 =0$.
Since $\rho_1(y_1,z) \leq \rho_1(x_1,z)/C_K$ for all $x_1 \in E_{\theta_1}$ and $y_1, z \in  E_{\phi_1}$, we can use Assumption~\ref{assum_4} to estimate
\begin{eqnarray*}
  |K_{\phi_2,\theta_2}(x_1,y_1) - K_{\phi_2,\theta_2}(x_1,z)|
   &\leq& C(\phi_2,\theta_2)  \frac{\rho_1(y_1,z)^{\al_1}}{\rho_1(x_1,z)^{\al_1}
      \lambda_1(z,\rho_1(x_1,z))} \\
   &\ls& \|\phi_2\|_{L^2(\mu_2)}  \|\theta_2\|_{L^2(\mu_2)}
    \frac{\rho_1(y_1,z)^{\al_1}}{\rho_1(x_1,z)^{\al_1}
      \lambda_1(z,\rho_1(x_1,z))}.
\end{eqnarray*}
Thus
\begin{eqnarray*}
  \lefteqn{|\langle T(\phi_1 \otimes \phi_2), \theta_1 \otimes \theta_2 \rangle|} \\
   &\ls& \|\phi_2\|_{L^2(\mu_2)}  \|\theta_2\|_{L^2(\mu_2)}
   \int_{E_{\phi_1}} |\phi_1(y_1)| \,d\mu_1(y_1)  \int_{E_{\theta_2}}  \frac{\rho_1(y_1,z)^{\al_1}}{\rho_1(x_1,z)^{\al_1}
      \lambda_1(z,\rho_1(x_1,z))}
    |\theta_1(x_1)| \, d\mu_1(x_1) \\
   &\leq&  \|\phi_2\|_{L^2(\mu_2)}  \|\theta_2\|_{L^2(\mu_2)}
   \|\phi_1\|_{L^2(\mu_1)} \mu_1(E_{\phi_1})^{1/2}  \int_{E_{\theta_2}}
   \frac{\rho_1(y_1,z)^{\al_1}}{\rho_1(x_1,z)^{\al_1}
      \lambda_1(z,\rho_1(x_1,z))}
    |\theta_1(x_1)| \, d\mu_1(x_1),
\end{eqnarray*}
as required.
\end{proof}

\subsection{Weak Boundedness Properties}\label{sec:WBP}
Depending on the way we interpret the bi-parameter SIO~$T$, we need to use suitable assumptions on the weak boundedness property of~$T$.
In this paper, we give two ways to interpret the operator~$T$. First, we assume that~$T$ has some \emph{a priori} boundedness. Second, we assume that the bilinear form~$\langle Tf,g\rangle$ is initially well-defined on functions~$f$ and~$g \in C^{\eta}_c(\mathcal{X})$.

When we say~$T$ has  \emph{a priori} boundedness, we either mean that the operator norm~$\|T\|$ of~$T$ is finite, that is, $\|T\| < \infty$, or that
for all functions $f, g \in C^{\eta}_{\textup{c}}(\mathcal{X}) $ with $\supp f \subset B(x,R)$ and $\supp g \subset B(x, R)$ for some $x \in \mathcal{X}$ and $R > 0$, we have
      \[|\langle Tf,g \rangle| \leq C(R) \|g\|_{L^2(\mu)} \|g\|_{L^2(\mu)}.\]
We will make it clear in the paper which definition is used whenever we mention the \emph{a priori} boundedness.
With this interpretation, the operator~$T$ is assumed to satisfy additional Assumptions~\ref{assum_5}--\ref{assum_6}.
\begin{assume}\label{assum_5}\textbf{(Full weak boundedness property)}
There exist~$C>0$ and~$\Lambda >1$ such that
for every ball~$B_1 \subset X_1$ and~$B_2 \subset X_2$,
\[|\langle T(\chi_{B_1} \otimes \chi_{B_2}), \chi_{B_1} \otimes \chi_{B_2}\rangle| \leq C\mu_1(\Lambda B_1)\mu_2(\Lambda B_2).\]
\end{assume}

\begin{assume}\label{assum_6}\textbf{(Partial weak boundedness $\times$ $\bmo$ condition)}
There exist~$C>0$ and~$\Lambda >1$ such that
for all~$B_1 \subset X_1$, $B_2 \subset X_2$  and for all functions~$a_{B_1}$, $a_{B_2}$ with~$\int a_{B_1}\,d\mu_1 = \int a_{B_2}\,d\mu_2 = 0, \supp a_{B_1} \subset B_1$ and $\supp a_{B_2} \subset B_2$ we have
\begin{align*}
&|\langle T(\chi_{B_1} \otimes \chi_{B_2}), a_{B_1} \otimes \chi_{B_2}\rangle|
+ |\langle T^*(\chi_{B_1} \otimes \chi_{B_2}), a_{B_1} \otimes \chi_{B_2}\rangle| \leq C\|a_{B_1}\|_{L^2(\mu_1)}\mu_1(\Lambda B_1)^{1/2}\mu_2(\Lambda B_2)
\end{align*}
and
\begin{align*}
&|\langle T(\chi_{B_1} \otimes \chi_{B_2}), \chi_{B_1} \otimes a_{B_2}\rangle|
+|\langle T^*(\chi_{B_1} \otimes \chi_{B_2}), \chi_{B_1} \otimes a_{B_2}\rangle|  \leq C\|a_{B_2}\|_{L^2(\mu_2)}\mu_1(\Lambda B_1)\mu_2(\Lambda B_2)^{1/2}.
 \end{align*}
\end{assume}
As noted in~\cite{HM12a}, it is sometimes of interest to replace the rough test functions~$\chi_B$ by some more regular ones. Specifically, in our setting we can replace~$\chi_B$ by $\widetilde{\chi}_{B,\epsilon} \in C^{\eta}(\mathcal{X})$, where $\chi_B(x) \leq \widetilde{\chi}_{B,\epsilon}(x) \leq \chi_{(1+\epsilon)B}(x)$ for every~$x \in \mathcal{X}$, every ball~$B$ and every~$\epsilon \in (0,1]$. Therefore, we can obtain weaker conditions than those in Assumptions~\ref{assum_5} and~\ref{assum_6} by replacing~$\chi_{B_1}$ and~$\chi_{B_2}$ by~$\widetilde{\chi}_{B_1,\epsilon}$
and~$\widetilde{\chi}_{B_2,\epsilon}$ on the left-hand side of the inequalities. Also, the constant~$C = C(\epsilon)< \infty$ depends on~$\epsilon$, but is independent of the other quantities.

If we interpret the operator~$T$ in the sense that the bilinear form~$\langle Tf,g \rangle$ is defined on functions~$f$ and~$g \in C^{\eta}_c(\mathcal{X})$, then the weak boundedness of~$T$ means
Assumptions~\ref{assum:5b}--\ref{assum:7b}.
Again, for every ball~$B$ and every~$\epsilon \in (0,1]$ we have function $\widetilde{\chi}_{B,\epsilon} \in C^{\eta}(\mathcal{X})$ and
$\chi_B \leq \widetilde{\chi}_{B,\epsilon} \leq \chi_{(1+\epsilon)B}$ .

\begin{assume}\label{assum:5b}\textbf{(Full weak boundedness property)}
There exist~$C>0$ and~$\Lambda >1$ such that
for every ball~$B_1, B'_1 \subset X_1$ with $\supp B_1 \cap \supp B'_1 \neq \emptyset$ and~$B_2 \subset B'_2 \subset X_2$ with $\supp B_2 \cap \supp B'_2 \neq \emptyset$ we have
\[|\langle T(\widetilde{\chi}_{B'_1,\epsilon} \otimes
\widetilde{\chi}_{B'_2,\epsilon}),
\widetilde{\chi}_{B_1,\epsilon} \otimes
\widetilde{\chi}_{B_2,\epsilon}\rangle| \leq C\mu_1(\Lambda B_1)\mu_2(\Lambda B_2).\]
\end{assume}

\begin{assume}\label{assum:6b}\textbf{(Partial weak boundedness property)}
There exist~$C>0$ and~$\Lambda >1$ such that
for every ball~$B_1, B'_1 \subset X_1$ with $\supp B_1 \cap \supp B'_1 \neq \emptyset$ and~$B_2 \subset B'_2 \subset X_2$ with $\supp B_2 \cap \supp B'_2 \neq \emptyset$  we have
\begin{align*}
|\langle T(\widetilde{\chi}_{B'_1,\epsilon} \otimes
\widetilde{\chi}_{B'_2,\epsilon}),
\widetilde{\chi}_{B_1,\epsilon} \otimes
\delta_{x_{B_2}}\rangle|
+ |\langle T^*(\widetilde{\chi}_{B'_1,\epsilon} \otimes
\widetilde{\chi}_{B'_2,\epsilon}),
\widetilde{\chi}_{B_1,\epsilon} \otimes
\delta_{x_{B_2}}\rangle|
\leq C\mu_1(\Lambda B_1),
\end{align*}
and
\[|\langle T(\widetilde{\chi}_{B'_1,\epsilon} \otimes
\widetilde{\chi}_{B'_2,\epsilon}),
\delta_{x_{B_1}} \otimes
\widetilde{\chi}_{B_2,\epsilon}\rangle|
+ |\langle T^*(\widetilde{\chi}_{B'_1,\epsilon} \otimes
\widetilde{\chi}_{B'_2,\epsilon}),
\delta_{x_{B_1}} \otimes
\widetilde{\chi}_{B_2,\epsilon}\rangle|
\leq C\mu_2(\Lambda B_2),\]
where for $i=1,2$,
$x_{B_i}$ is the centre of the ball~$B_i$ and
$\delta_{x_{B_i}}(x) = 1$ if $x=x_{B_i}$ and $0$ otherwise.
\end{assume}

\begin{assume}\label{assum:7b}\textbf{(Partial weak boundedness $\times$ $\bmo$ condition)}
There exist~$C>0$ and~$\Lambda >1$ such that
for all~$B_1 \subset X_1$, $B_2 \subset X_2$  and for all functions~$a_{B_1}$, $a_{B_2}$ with~$\int a_{B_1}\,d\mu_1 = \int a_{B_2}\,d\mu_2 = 0, \supp a_{B_1} \subset B_1$ and $\supp a_{B_2} \subset B_2$ we have
\begin{align*}
&|\langle T(\widetilde{\chi}_{B_1,\epsilon} \otimes
\widetilde{\chi}_{B_2,\epsilon}),
a_{B_1} \otimes
\widetilde{\chi}_{B_2,\epsilon}\rangle|
+ |\langle T^*(\widetilde{\chi}_{B_1,\epsilon} \otimes
 \widetilde{\chi}_{B_2,\epsilon}),
 a_{B_1} \otimes \widetilde{\chi}_{B_2,\epsilon}\rangle|\\
 & \leq C\|a_{B_1}\|_{L^2(\mu_1)}\mu_1(\Lambda B_1)^{1/2}\mu_2(\Lambda B_2)
\end{align*}
and
\begin{align*}
&|\langle T(\widetilde{\chi}_{B_1,\epsilon} \otimes
\widetilde{\chi}_{B_2,\epsilon}),
\widetilde{\chi}_{B_1,\epsilon} \otimes a_{B_2}\rangle|
+|\langle T^*(\widetilde{\chi}_{B_1,\epsilon} \otimes
\widetilde{\chi}_{B_2,\epsilon}),
\widetilde{\chi}_{B_1,\epsilon} \otimes a_{B_2}\rangle| \\
& \leq C\|a_{B_2}\|_{L^2(\mu_2)}\mu_1(\Lambda B_1)\mu_2(\Lambda B_2)^{1/2}.
 \end{align*}
\end{assume}

Note that the same conditions hold when replacing~$T$ by~$T_1$ in Assumptions~\ref{assum_5}--\ref{assum:7b}.

\subsection{Dyadic cubes in~$(X,\rho)$}\label{sec:dyadiccubes}
In this section, we recall the constructions of dyadic cubes, adjacent systems of dyadic cubes and random systems of dyadic cubes.
We use the construction from~\cite{HK12}. We present here the (slightly reworded) version that appears in~\cite[Section~2]{KLPW16}.
For the history of the development of systems of dyadic cubes, and collection of such systems which generalise the ``one-third trick'', see \cite{HK12} and the references therein, especially \cite{Chr55} and \cite {SaWh}.

\begin{defn}[\cite{KLPW16}]\label{dyadiccubes} (\textbf{A system of dyadic cubes}) In a geometrically doubling quasimetric space $(X, \rho)$, a countable family
\[\mathscr{D} = \bigcup_{k \in \mathbb{Z}} \mathscr{D}_k, \hspace{0.6cm} \mathscr{D}_k = \{Q_{\alpha}^k: \alpha \in \mathscr{A}_k\},\]
of Borel sets $Q_{\alpha}^k \subset X$ together with a fixed collection of countably many (reference) points $x_{\alpha}^k$ in $X$, with $x_{\alpha}^k \in Q_{\alpha}^k$ for each $k \in \mathbb{Z}$ and each $\alpha \in \mathscr{A}_k$,
is called a \emph{system of dyadic cubes with parameters $\delta \in (0,1)$, $c_Q$ and $C_Q$} such that $0 < c_Q < C_Q < \infty$ if it has the following properties.
\begin{align}
  &\hspace{.5cm} 1. \text{ } X = \bigcup_{\alpha \in \mathscr{A}_k} Q_{\alpha}^k \text{ (disjoint union) for all } k \in \mathbb{Z}.\label{a1}\\
  &\hspace{.5cm} 2. \text{ If }l \geq k, \text{ then either } Q_{\beta}^l \subset Q_{\alpha}^k \text{ or } Q_{\alpha}^k \cap Q_{\beta}^l = \emptyset. \label{a2}\\
  &\hspace{.5cm} 3. \text{ } B(x_{\alpha}^k, c_Q\delta ^k) \subset Q_{\alpha}^k \subset B(x_{\alpha}^k, C_Q\delta ^k) =: B(Q_{\alpha}^k).\label{a5}\\
 &\hspace{.5cm} 4. \text{ If } l \geq k \text{ and } Q_{\beta}^l \subset Q_{\alpha}^k, \text{ then } B(Q_{\beta}^l) \subset B(Q_{\alpha}^k).\label{a6}\\
&\hspace{.5cm} 5. \text{ For each } (k, \alpha) \text{ and  each } l \leq k, \text{ there exists a unique } \beta \text{ such that } Q_{\alpha}^k \subset  Q_{\beta}^l. \label{a3}\\
& \hspace{.5cm}6. \text{ For each }  (k, \alpha) {\rm there\ is\ at\ most\ } M {\rm (a\ fixed\ geometric\ constant)}\  cubes \ Q_{\beta}^{k+1} \nonumber
 \end{align}
 \begin{equation}\label{a4}
\hskip-4cm
{\rm\ such\ that\ }  Q_{\beta}^{k+1} \subset Q_{\alpha}^k, \text{ and } Q_{\alpha}^k = \bigcup_{Q \in \mathscr{D}_{k+1}, Q\subset Q_{\alpha}^k} Q.
  \end{equation}
The set $Q_{\alpha}^k$ is called a \emph{dyadic cube of generation $k$} with \emph{centre point} $x_{\alpha}^k \in Q_{\alpha}^k$ and \emph{side length}~$\delta ^k.$ We will use $\ell(Q_{\alpha}^k)$ to denote the side length of $Q_{\alpha}^k$.
\end{defn}

\begin{thm}[\text{\cite[Theorem 2.1]{KLPW16}}] \label{klp}
Let $(X, \rho)$ be a geometrically doubling quasimetric space. Then there exists a system~$\D$ of dyadic cubes with parameters $0 < \delta \leq (12A_0^3)^{-1}$, $c_Q = (3A_0^2)^{-1}$ and $C_Q = 2A_0.$ The construction only depends on some fixed set of countably many centre points $x_{\alpha}^k$, satisfying the two inequalities
\begin{equation*}
  \rho (x_{\alpha}^k, x_{\beta}^k) \geq \delta ^k \hspace{0.5 cm} (\alpha \neq \beta),
  \hspace{1.5cm} \min_{\alpha} \rho (x, x_{\alpha}^k) < \delta ^k \text{ for all } x \in X,
\end{equation*}
and a certain partial order $\leq$ among their index pairs $(k, \alpha)$.
\end{thm}

\begin{defn}[\cite{KLPW16}] \label{cubesinX}  (\textbf{Adjacent Systems of Dyadic Cubes})
In a geometrically doubling quasimetric space $(X, \rho)$, a finite collection $\{\mathscr{D}^t: t = 1,2,\ldots,T\}$ of families $\mathscr{D}^t$ is called a \emph{collection of adjacent systems of dyadic cubes with parameters $\delta \in (0,1)$, $c_Q$ and $C_Q$} such that $0 < c_Q < C_Q < \infty$ and $C \in [1, \infty)$ if it has the following properties: individually, each $\mathscr{D}^t$ is a system of dyadic cubes with parameters $\delta \in (0,1)$ and $0 < c_Q < C_Q < \infty$; collectively, for each ball $B(x,r)\subset X$ with $\delta^{k+3} <r< \delta^{k+2}, k \in \mathbb{Z}$, there exist $t \in \{1,2,\ldots,T\}$ and $Q \in \mathscr{D}^t$ of generation $k$ and with centre point $x_{\alpha}^{k,t}$ such that $\rho(x, x_{\alpha}^{k,t}) < 2A_0\delta ^k$ and
\begin{equation}\label{a8}
  B(x,r) \subset Q \subset B(x, Cr).
\end{equation}
\end{defn}

\begin{thm}[\text{\cite[Theorem 2.7]{KLPW16}}]\label{klp1}
Let $(X, \rho)$ be a geometrically doubling quasimetric space. Then there exists a collection $\{\mathscr{D}^t = 1,2,\ldots,T\}$ of adjacent systems of dyadic cubes with parameters $0 < \delta \leq (96A_0^6)^{-1}$, $c_Q = (12A_0^4)^{-1}$, $C_Q = 4A_0^2$ and $C = 8A_0^3\delta^{-3}$. The centre points $x_{\alpha}^{k,t}$ of the cubes $Q \in \mathscr{D}^t_k $ satisfy, for each $t \in \{1,2,\ldots,T\}$, the  two inequalities
\begin{equation*}
 \rho (x_{\alpha}^{k,t}, x_{\beta}^{k,t}) \geq (4A_0^2)^{-1}\delta ^k \hspace{0.5 cm} (\alpha \neq \beta),
  \hspace{1.5cm} \min_{\alpha} \rho (x, x_{\alpha}^{k,t}) < 2A_0\delta ^k \text{ for all } x \in X,
\end{equation*}
and a certain partial order $\leq$ among their index pairs $(k, \alpha)$.
\end{thm}

We want to construct systems of dyadic cubes randomly and independently.
This is done by defining a maximal collection of reference points~$x^k_{\al}$ with fixed separation, as introduced in Definition~\ref{dyadiccubes}. From these reference points, we can build a single system of dyadic cubes. However, we need more than one system. Therefore, we generate new systems of dyadic cubes by varying the reference points. This way we can eventually randomise the grids. Please refer to \cite{HK12} for more details of the construction.
\begin{thm}[\text{\cite{HK12}}]\textnormal{(\textbf{Random systems of dyadic cubes})}
\label{thm:random_sys_dyadic_cube}
Given a set of reference points $\{x_{\al}^k\}, k \in \Z, \al \in \mathscr{A}_k$, suppose the constant~$\delta \in (0,1)$ satisfies $96A_0^6\delta \leq 1$. Then there exists a probability space~$(\Omega, \mathbb{P})$ such that every~$\omega \in \Omega$ defines a dyadic system~$\mathscr{D}(\omega) = \{Q_{\al}^k(\omega)\}_{k,\al}$, related to new dyadic points~$ \{z_{\al}^k(\omega)\}_{k,\al}$, with the properties~\eqref{a1}--\eqref{a4} in Definition~\ref{dyadiccubes}. Further, the probability space~$(\Omega, \mathbb{P})$ has the following properties:\\
\textup{(i)} $\Omega = \prod_{k \in \Z} \Omega_k, \quad \omega = (\omega_k)_{k \in \Z}$ with  $\omega_k \in \Omega_k$ being independent;  \\
\textup{(ii)} $ z_{\al}^k(\omega) = z_{\al}^k(\omega_k)$; and \\
\textup{(iii)} if $(k +1,\beta) \leq (k,\al)$,  then  $\mathbb{P}(\{\omega \in \Omega: z_{\al}^k(\omega) = x_{\beta}^{k+1} \geq \tau_0 > 0\})$.
\end{thm}

In~\cite{HM12a}, they also present a construction of dyadic cubes, with the same properties described above, except the choice of the parameters. In particular, $\delta \leq 1/1000$, $c_Q = 1/100$ and $C_Q = 10$.
As noted in~\cite{HK12}, the construction of dyadic cubes in~\cite{HK12} is a simplified and streamlined version of that in~\cite{HM12a}.

All the results presented in~\cite{HM12a} are in terms of the construction of dyadic cubes presented there. In fact, these results still hold if we use the construction of dyadic cubes in~\cite{HK12}. This is because their proofs rely on the properties of the dyadic cubes, not on how the cube are constructed, and both constructions in~\cite{HM12a} and~\cite{HK12} yield the same properties of dyadic cubes.

\subsection{Good and bad cubes}\label{subsec:good_bad_cube}
Let $(\mathcal{X},\rho,\mu) := (X_1 \times X_2, \rho_1 \times \rho_2, \mu_1 \times \mu_2)$ be a non-homogeneous bi-parameter quasimetric space.
 Let~$\Dd_1, \Dd_1' \in X_1$ and~$\Dd_2, \Dd_2' \in X_2$ be four independent random systems of dyadic cubes, which are defined from a probability space~$(\Omega, \mathbb{P})$ (see Theorem~\ref{thm:random_sys_dyadic_cube}, Section~\ref{sec:dyadiccubes}).

We recall property~\eqref{a5} of Definition~\ref{dyadiccubes} that
there exist universal constants~$c_Q$ and~$C_Q$ such that for all dyadic cubes~$Q$ in a fixed dyadic grid~$\Dd$ with centre point~$x_Q$ and side length~$\ell(Q)$ we have
\begin{equation}\label{eq1:Part3}
B(x_Q,c_Q\ell(Q)) \subset Q\subset B(x_Q,C_Q\ell(Q)) =: B(Q).
\end{equation}
In general, the constants~$c_Q$ and~$C_Q$ can be different for different dyadic grids~$\Dd$. However, for simplicity, without loss of generality, we will use the constants~$c_Q$ and~$C_Q$ for both grids~$\Dd_1$ and~$\Dd_2$.

We separate the cubes of~$\Dd_1$ into \emph{good} and \emph{bad} cubes depending on their side length and how close they are to the boundary of cubes in~$\Dd'_1$. In simple terms, a cube~$Q \in \Dd_1$ is bad if there is a much bigger cube~$Q' \in \Dd'_1$ such that $Q$ is close to the boundary of~$Q'$.
\begin{defn}\label{def:bad cube}
   A cube~$Q$ in~$\Dd_1$ is called \emph{bad} if there exists a cube~$Q'\in \Dd_1'$ for which~$\ell(Q) \leq \delta^r \ell(Q')$ and
\begin{equation}\label{grideq.1}
  \rho_1(Q,\partial Q') \leq \delta\mathcal{C}\ell(Q)^{\gamma_1}\ell(Q')^{1-\gamma_1},
\end{equation}
 where
  $\mathcal{C} := 2A_0C_QC_K$, $A_0$ is the quasitriangle constant of~$\rho_1$, $C_Q$ is from equation~\eqref{eq1:Part3} and $C_K$ is from the kernel estimates in Assumptions~\ref{assum_2} and~\ref{assum_4}.
 Here~$\ell(Q) =  \delta^k$ denotes the \emph{side length} of~$Q$ where $\delta \in (0,1)$ and $k = \text{gen}(Q)$, $\gamma_1 :=\frac{\al_1}{2(\al_1 +t_1)}$ where~$\al_1 >0$ appears in the kernel estimates, $t_1:= \log_2 C_{\lambda_1}$, $C_{\lambda_1}$ is the upper doubling constant of~$\mu_1$, and~$r$ is a sufficiently large parameter that will be chosen later.

The badness of cubes in~$\Dd_1$ is relative to the grid~$\Dd'_1$. Thus, here bad really means~$\Dd_1'$-bad, or even~$(\Dd_1',r,\gamma_1)$-bad.
\end{defn}
Notice that condition~\eqref{grideq.1} is equivalent to that of both
\[\rho_1(Q,Q') \leq \delta\mathcal{C}\ell(Q)^{\gamma_1}\ell(Q')^{1-\gamma_1}
\text{ and }  \rho_1(Q,X_1 \backslash Q') \leq \delta\mathcal{C}\ell(Q)^{\gamma_1}\ell(Q')^{1-\gamma_1}.\]
Moreover,  the cubes~$Q$ and~$Q'$ can be disjoint, nested, next to each other or overlapping.

A cube~$Q$ is \emph{good} if it is not bad. Denote the collection of all good cubes in~$\Dd_1$ by~$\Dd_1'$-good, and the rest by~$\Dd_1'$-bad $:= \Dd_1 \backslash \Dd_1$-good.

We find that for fixed $\delta, \gamma_1 \in (0,1)$, the probability of being bad is vanishing, when $r \rightarrow \infty$.
\begin{thm}[\text{\cite[Theorem 10.2]{HM12a}}]\label{thm.10.2.HM12}
There exists $\kappa \in (0,1]$ such that for all cube~$Q \in \Dd_1$,
\[\mathbb{P}(Q \in \Dd'_1\textup{-bad}) \ls \delta^{r\gamma_1\kappa}.\]
\end{thm}
The badness of cubes in~$X_2$ is defined analogously.

\subsection{Haar functions on $(\mathcal{X},\rho,\mu)$}\label{subsec:Haar_func}
As mentioned in the Introduction, we are going to use the technique of representing functions in terms of Haar functions. We are going to use the construction of Haar functions, that is presented in~\cite[Section~4]{KLPW16}; see also the references therein.

Let~$(X,\rho,\mu)$ be a (one-parameter) quasimetric space equipped with a positive Borel measure.
Note that as in~\cite{KLPW16} they also follow the same construction of dyadic cubes of~\cite{HK12}.
Let~$\Dd$ be one of the dyadic grids on~$X$.
Let~$Q$ be a dyadic cube in~$\Dd^k$.
For~$f \in L^2_{\textup{loc}}(\mu)$ set
\[E_Qf := \langle f\rangle_Q \chi_Q, \qquad \Delta_Qf := \sum_{Q' \in \text{ch}(Q)}(\langle f\rangle_{Q'} - \langle f\rangle_Q)\chi_{Q'},\]
where ch$(Q) = \{Q_j: j = 1,\ldots, M_Q\}$ is the collection of dyadic children of~$Q$, and~$M_Q := \# \text{ch}(Q) = \#\{Q_j \in \Dd^{k+1}: Q_j \subset Q\}$, and~$\langle f\rangle_Q = \mu_1(Q)^{-1}\int_{Q} f \,d\mu$.

For all~$m \in \Z$ we have the orthogonal decomposition
\[f = \sum_{Q \in \Dd^m}E_Qf + \sum_{k \geq m} \sum_{Q \in \Dd^k} \Delta_Qf.\]
We index the children of~$Q$ by~$u = 1,2,\ldots, M_Q$ in such a way that
\[\mu_1(\widehat{Q}_u) \geq [1-(u-1)M_Q^{-1}]\mu_1(Q), \text{ where } \quad \widehat{Q}_u = \bigcup_{j=u}^{M_Q}Q_j.\]
We then continue to decompose
\[E_Qf = \langle f, h_0^Q\rangle h_0^Q \quad \text{and}
\quad \Delta_Qf = \sum_{u=1}^{M_Q-1} \langle f, h_u^Q\rangle h_u^Q.\]
Here~$h_0^Q := \mu_1(Q)^{-1/2}\chi_Q$ and~$h_u^Q := a_u\chi_{Q_u} - b_u\chi_{\widehat{Q}_{u+1}}$, where~$u = 1,\ldots,M_Q-1,$
\[a_u := \frac{\mu_1(\widehat{Q}_{u+1})^{1/2}}{\mu_1(Q_u)^{1/2}\mu_1(\widehat{Q}_u)^{1/2}}
\quad \text{and} \quad b_u := \frac{\mu_1(Q_u)^{1/2}}{\mu_1(\widehat{Q}_u)^{1/2}\mu_1(\widehat{Q}_{u+1})^{1/2}}.\]
\cite[Theorem 4.2]{KLPW16} lists the properties of~$h_u^Q$.

\begin{thm}[\text{\cite[Theorem 4.2]{KLPW16}}]
\label{thm.4.2_KLPW}
The Haar functions~$h_u^Q$, $Q \in \Dd$, $u \in \{1,\ldots, M_Q -1\}$, have the following properties:\\
\indent \textup{(i)} $h_u^Q$ is a simple Borel-measurable real function on~$X$;\\
\indent \textup{(ii)} $h_u^Q$ is supported on~$Q$;\\
\indent \textup{(iii)} $h_u^Q$ is constant of each $Q' \in \textup{ch}(Q)$;\\
\indent \textup{(iv)} $\int h_u^Q \,d\mu = 0$ (cancellation);\\
\indent \textup{(v)} $\langle h_u^Q, h_{u'}^Q \rangle= 0$ for $u \neq u'$, $u, u' \in \{1,\ldots, M_Q -1\}$;\\
\indent \textup{(vi)} the collection
\[\{\mu(Q)^{-1/2}\chi_Q\} \cup \{h_u^Q: u = 1,\ldots, M_Q -1\}\]
is an orthogonal basis for the vector space $V(Q)$ of all functions on~$Q$ that are constant on each sub-cube $Q' \in \textup{ch}(Q)$;\\
\indent \textup{(vii)} for $u \in \{1,\ldots, M_Q -1\}$, if~$h_u^Q \neq 0$ then for each $Q_u \in \textup{ch}(Q)$
\[\|h_u^Q\|_{L^p(\mu)} \simeq \mu(Q)^{\frac{1}{p} - \frac{1}{2}} \quad \text{for } 1 \leq p \leq \infty;\]
and \\
\indent \textup{(viii)} \qquad \qquad \qquad $\|h_u^Q\|_{L^1(\mu)} \cdot \|h_u^Q\|_{L^{\infty}(\mu)} \sim 1$.
\end{thm}
In addition to the properties listed above, using property~(vii) and H\"{o}lder's inequality, for all Haar functions $h_{u_1}^{Q_1}, h_{u_2}^{Q_2}$, where $Q_1, Q_2 \in \Dd^k$, $u_1 \in \{1,\ldots, M_{Q_1} -1\}$, $u_2 \in \{1,\ldots, M_{Q_2} -1\}$ we have
\begin{equation}\label{eq1:pro_Haar}
  \|h_{u_1}^{Q_1} \otimes h_{u_2}^{Q_2}\|_{L^2(\mu)}
  = \|h_{u_1}^{Q_1}\|_{L^2(\mu)}
  \|h_{u_2}^{Q_2}\|_{L^2(\mu)}
  \sim 1.
\end{equation}
Moreover, for all $ Q' \in \textup{ch}(Q)$, using H\"{o}lder's inequality we have
\begin{equation}\label{eq2:pro_Haar}
  \langle h_u^Q \rangle_{Q'}
  = \frac{1}{\mu(Q')}\int_{Q'} h_u^Q \,d\mu
  \leq \frac{1}{\mu(Q')} \|h_u^Q\|_{L^2(\mu)} \mu(Q')^{1/2} \ls \mu(Q')^{-1/2}.
\end{equation}

Let $(\mathcal{X},\rho,\mu)$ be a non-homogeneous bi-parameter quasimetric space.
Denote~$\Dd_1, \Dd_1' \in X_1$ and~$\Dd_2, \Dd_2' \in X_2$ be four independent random systems of dyadic cubes (see Theorem~\ref{thm:random_sys_dyadic_cube}, Section~\ref{sec:dyadiccubes}).
For each~$i = 1,2$, let~$Q_i \in \Dd_i$, $Q'_i \in \Dd'_i$, $u_i \in \{0,\ldots, M_{Q_i} -1\}$ and $u'_i \in \{0,\ldots, M_{Q'_i} -1\}$.
We denote the Haar functions with respect to the dyadic grids $\Dd_i,\Dd'_i$ by $h_{u_i}^{Q_i}, h_{u'_i}^{Q'_i}$, respectively.


\subsection{Interpretation of~$T(1)$ on $(\mathcal{X},\rho,\mu)$}\label{subsec:intepretation}
Among the criteria mentioned in Theorem~\ref{thm:T1thm_ver1}, four of them are establishing certain bounds on pairings involving~$T$ acting on the constant function~$1$. Therefore, it is necessary to understand how these objects are defined.

Let $(\mathcal{X},\rho,\mu) := (X_1 \times X_2, \rho_1 \times \rho_2, \mu_1 \times \mu_2)$ be a non-homogeneous bi-parameter quasimetric space.
Let $\mathcal{D}_1$ and $\mathcal{D}_2$ be fixed dyadic systems in~$X_1$ and~$X_2$, respectively.
Take~$Q_1 \in \Dd_1$ and~$Q_2 \in \Dd_2$.
Let~$h_{Q_1} := h_{u_1}^{Q_1}$ and~$h_{Q_2} := h_{u_2}^{Q_2}$ be Haar functions as defined in Section~\ref{subsec:Haar_func}.
To interpret the condition~$T(1) \in \bmo_{\textup{prod}}(\mu)$, it is enough to define the pairing~$\langle T(1) , h_{Q_1}\otimes h_{Q_2}\rangle$.
Then by symmetry of our assumptions, such pairings will also be defined when we replace~$T$ by~$T^*$, $T_1$ or~$T_1^*$.

Let~$U \subset X_1$ and~$V \subset X_2$ be arbitrary balls with the property that
$$C_KB(Q_1) := B(x_{Q_1}, C_KC_Q\ell(Q_1)) \subset U  \text{ \,and\, } C_KB(Q_2) := B(x_{Q_2}, C_KC_Q\ell(Q_2)) \subset V,$$
where~$C_K>1$ is the constant depending on the kernel of~$T$ which appears in Assumptions~\ref{assum_2} and~$\ref{assum_4}$, and~$C_Q>1$ is the universal constant which appear in property~\ref{eq1:Part3}.
Note that for $x_1 \in Q_1$ and $y_1 \in U^c$ we have
$$\rho_1(x_{Q_1},y_1) \geq C_KC_Q \ell(Q_1) \geq C_K \rho_1(x_{Q_1},x_1).$$
Similarly, for $x_2 \in Q_2$ and $y_2 \in V^c$ we also have
$\rho_2(x_{Q_2},y_2) \geq C_K \rho_2(x_{Q_2},x_2)$.

Set~$\langle T1, h_{Q_1} \otimes h_{Q_2}\rangle := \sum_{i=1}^{4}A_i(U,V)$ where
\begin{eqnarray*}
  A_1(U,V) &:=& \langle T(\chi_U \otimes \chi_V), h_{Q_1} \otimes h_{Q_2}\rangle, \\
  A_2(U,V) &:=& \langle T(\chi_{U^c} \otimes \chi_V), h_{Q_1} \otimes h_{Q_2}\rangle, \\
  A_3(U,V) &:=& \langle T(\chi_{U} \otimes \chi_{V^c}), h_{Q_1} \otimes h_{Q_2}\rangle, \quad \text{and} \\
  A_4(U,V) &:=& \langle T(\chi_{U^c} \otimes \chi_{V^c}), h_{Q_1} \otimes h_{Q_2}\rangle.
\end{eqnarray*}
Now, we will estimate~$A_i(U,V)$ for each $i = 1,\ldots,4$.
We start with $A_1(U,V)$:
\begin{eqnarray*}
  |A_1(U,V)| &\leq& \|T\| \|\chi_U\|_{L^2(\mu_1)} \|\chi_V\|_{L^2(\mu_2)} \|h_{Q_1}\|_{L^2(\mu_1)} \|h_{Q_2}\|_{L^2(\mu_2)}
  \ls \|T\| \mu_1(U)^{1/2} \mu_2(V)^{1/2}.
\end{eqnarray*}

Next, we consider $A_2(U,V)$. Since $U^c \cap Q_1 = \emptyset$ and $\int h_{Q_1} \,d\mu_1 = 0$, we can use Assumption~\ref{assum_3} to write
\begin{equation*}
  |A_2(U,V)| \leq \int_{U^c} \int_{Q_1}\left| K_{\chi_V,h_{Q_2}}(x_1, y_1) - K_{\chi_V, h_{Q_2}}(x_{Q_1}, y_1)\right| |h_{Q_1}(x_1)|\,d\mu_1(x_1) \,d\mu_1(y_1).
\end{equation*}
Furthermore, by Assumption~\ref{assum_4} we get
\begin{eqnarray*}
 \left| K_{\chi_V,h_{Q_2}}(x_1, y_1) - K_{\chi_V, h_{Q_2}}(x_{Q_1}, y_1)\right|
  \ls \|\chi_V\|_{L^2(\mu_2)}  \|h_{Q_2}\|_{L^2(\mu_2)}
   \frac{\rho_1(x_1,x_{Q_1})^{\al_1}}{\rho_1(x_{Q_1},y_1)^{\al_1}
   \lambda_1(x_{Q_1},\rho_1(x_{Q_1},y_1))}.
\end{eqnarray*}
Thus,
\begin{eqnarray*}
  |A_2(U,V)| \ls \mu_2(V)^{1/2} \int_{Q_1} |h_{Q_1}(x_1)| \,d\mu_1(x_1)  \int_{U^c}  \frac{C_Q^{\al_1}\ell(Q_1)^{\al_1}}
   {\rho_1(x_{Q_1},y_1)^{\al_1} \lambda_1(x_{Q_1},\rho_1(x_{Q_1},y_1))} \,d\mu_1(y_1).
\end{eqnarray*}
By Lemma~\ref{upper_dbl_lem1}, the integral above is bounded by
\begin{align}\label{eq1_inter}
  \ell(Q_1)^{\al_1} \int_{X_1 \backslash B(c_{Q_1}, C_KC_Q\ell(Q_1))}  \frac{\rho_1(x_{Q_1},y_1)^{-\al_1}}
   {\lambda_1(x_{Q_1},\rho_1(x_{Q_1},y_1))}\,d\mu_1(y_1)
   \ls \ell(Q_1)^{\al_1} (C_KC_Q)\ell(Q_1)^{-\al_1} \ls 1.
\end{align}
Therefore,
$|A_2(U,V)| \ls \mu_1(Q_1)^{1/2} \mu_2(V)^{1/2} \ls \mu_1(U)^{1/2} \mu_2(V)^{1/2}$.  Following the same argument, we obtain $|A_3(U,V)| \ls \mu_1(U)^{1/2} \mu_2(V)^{1/2}$.

Now we estimate~$A_4(U,V)$. Since $U^c \cap Q_1 = \emptyset$ and $V^c \cap Q_2 = \emptyset$, we can use Assumption~\ref{assum_1} to write
\begin{eqnarray*}
  |A_4(U,V)|
   \leq  \int_{U^c} \int_{V^c} \int_{Q_1} \int_{Q_2} |K(x,y)|| h_{Q_1}(x_1)|| h_{Q_2}(x_2)| \,d\mu_2(x_2) \,d\mu_1(x_1) \,d\mu_2(y_2) \,d\mu_1(y_1).
\end{eqnarray*}
Using the fact that $\int h_{Q_1} \,d\mu_1 = \int h_{Q_2} \,d\mu_2 = 0$ together with Assumption~\ref{assum_2}, we can replace and estimate $|K(x,y)|$ by
\begin{eqnarray*}
\lefteqn{|K(x,y) - K((x_{Q_1},x_2),y) - K((x_1,x_{Q_2}),y) + K((x_{Q_1},x_{Q_2})| } \\
   &\ls&  \frac{\rho_1(x_1,x_{Q_1})^{\al_1}}{\rho_1(x_{Q_1},y_1)^{\al_1}
   \lambda_1(x_{Q_1},\rho_1(x_{Q_1},y_1))}  \frac{\rho_2(x_2,x_{Q_2})^{\al_2}}{\rho_2(x_{Q_2},y_2)^{\al_2}
   \lambda_2(x_{Q_2},\rho_2(x_{Q_2},y_2))}.
\end{eqnarray*}
Now we have
\begin{eqnarray*}
  |A_4(U,V)|
  &\leq& \|h_{Q_1}\|_{L^2(\mu_1)} \|h_{Q_2}\|_{L^2(\mu_2)}
  \mu_1(Q_1)^{1/2} \mu_2(Q_2)^{1/2}\\
  &&  \times \int_{U^c}  \frac{C_Q^{\al_1}\ell(Q_1)^{\al_1}}
   {\rho_1(x_{Q_1},y_1)^{\al_1} \lambda_1(x_{Q_1},\rho_1(x_{Q_1},y_1))} \,d\mu_1(y_1) \\
   && \times \int_{V^c}  \frac{C_Q^{\al_2}\ell(Q_2)^{\al_2}}
   {\rho_2(x_{Q_2},y_2)^{\al_2} \lambda_2(x_{Q_2},\rho_2(x_{Q_2},y_2))} \,d\mu_1(y_2).
\end{eqnarray*}
As shown in~\eqref{eq1_inter}, the first integral above is controlled by~$1$. Similarly, the second integral above is also controlled~$1$.  This implies that $|A_4(U,V)| \ls \mu_1(Q_1)^{1/2} \mu_2(Q_2)^{1/2} \leq \mu_1(U)^{1/2} \mu_2(V)^{1/2}$.

Hence,~$A_i(U,V)$ converges for all $i = 1, \ldots, 4$, and so
$\langle T1,h_{Q_1} \otimes h_{Q_2} \rangle$ is well defined.

In Lemma~\ref{interlem.1} below, we will show that the definition of~$A_i(U,V)$ is in fact independent of the choice of the balls~$U$ and~$V$.
\begin{lem}\label{interlem.1}
Suppose $(\mathcal{X},\rho,\mu) := (X_1 \times X_2, \rho_1 \times \rho_2, \mu_1 \times \mu_2)$ is a non-homogeneous bi-parameter quasimetric space.
Let $\mathcal{D}_1$ and $\mathcal{D}_2$ be fixed dyadic systems in~$X_1$ and~$X_2$, respectively.
Take~$Q_1 \in \Dd_1$ and~$Q_2 \in \Dd_2$.
Let~$U \subset X_1$ and~$V \subset X_2$ be arbitrary balls with the property that~$C_KB(Q_1) \subset U$ and~$C_KB(Q_2) \subset V$. There holds that
\[\sum_{i=1}^{4} A_i(U,V) = \sum_{i=1}^{4} A_i\left(C_KB(Q_1),C_KB(Q_2)\right).\]
\end{lem}
\begin{proof}
We will consider $A_i\left(C_KB(Q_1),C_KB(Q_2)\right) - A_i(U,V)$, $i = 1,\ldots, 4$.
First, we have
\begin{eqnarray*}
  \lefteqn{A_1\left(C_KB(Q_1),C_KB(Q_2)\right) - A_1(U,V)}\\
&=& \langle T(\chi_{C_KB(Q_1)} \otimes \chi_{C_KB(Q_2)}), h_{Q_1} \otimes h_{Q_2}\rangle
-  \langle T(\chi_U \otimes \chi_V), h_{Q_1} \otimes h_{Q_2}\rangle.
\end{eqnarray*}
  Next, consider
\begin{eqnarray*}
  \lefteqn{A_2\left(C_KB(Q_1),C_KB(Q_2)\right) - A_2(U,V)}\\
   &=& \langle T(\chi_{C_KB(Q_1)^c} \otimes \chi_{C_KB(Q_2)}), h_{Q_1} \otimes h_{Q_2}\rangle
   - \langle T(\chi_{U^c} \otimes \chi_V), h_{Q_1} \otimes h_{Q_2}\rangle  \\
  &=& \langle T(\chi_{U \backslash C_KB(Q_1)} \otimes \chi_{C_KB(Q_2)}), h_{Q_1} \otimes h_{Q_2}\rangle
   + \langle T(\chi_{U^c} \otimes\chi_{C_KB(Q_2)}), h_{Q_1} \otimes h_{Q_2}\rangle \\
  && - \langle T(\chi_{U^c} \otimes \chi_{V \backslash C_KB(Q_2)}), h_{Q_1} \otimes h_{Q_2}\rangle
  -  \langle T(\chi_{U^c} \otimes \chi_{C_KB(Q_2)}), h_{Q_1} \otimes h_{Q_2}\rangle \\
  &=&  \langle T(\chi_{U \backslash C_KB(Q_1)} \otimes \chi_{C_KB(Q_2)}), h_{Q_1} \otimes h_{Q_2}\rangle
  - \langle T(\chi_{U^c} \otimes \chi_{V \backslash C_KB(Q_2)}), h_{Q_1} \otimes h_{Q_2}\rangle.
\end{eqnarray*}
Similarly, we see that
\begin{eqnarray*}
  \lefteqn{A_3\left(C_KB(Q_1),C_KB(Q_2)\right) - A_3(U,V)}\\
  &=&  \langle T(\chi_{C_KB(Q_1)} \otimes \chi_{V \backslash C_KB(Q_2)}), h_{Q_1} \otimes h_{Q_2}\rangle
  - \langle T(\chi_{U \backslash C_KB(Q_1)} \otimes \chi_{V^c}), h_{Q_1} \otimes h_{Q_2}\rangle.
\end{eqnarray*}
Now, using the decomposition
\[\int\limits_{C_KB(Q_1)^c} \int\limits_{C_KB(Q_2)^c}
= \int\limits_{U \backslash C_KB(Q_1)} \int\limits_{V \backslash C_KB(Q_2)}
+ \int\limits_{U \backslash C_KB(Q_1)} \int\limits_{V^c}
+ \int\limits_{U^c} \int\limits_{V \backslash C_KB(Q_2)}
+ \int\limits_{U^c} \int\limits_{V^c}\]
we have
\begin{eqnarray*}
  \lefteqn{A_4\left(C_KB(Q_1),C_KB(Q_2)\right) - A_4(U,V)}\\
  &=& \langle T(\chi_{C_KB(Q_1)^c} \otimes \chi_{C_KB(Q_2)^c}), h_{Q_1} \otimes h_{Q_2}\rangle
  - \langle T(\chi_{U^c} \otimes \chi_{V^c}), h_{Q_1} \otimes h_{Q_2}\rangle \\
  &=& \langle T(\chi_{U \backslash C_KB(Q_1)} \otimes \chi_{V \backslash C_KB(Q_2)}), h_{Q_1} \otimes h_{Q_2}\rangle \\
  &&+ \,\langle T(\chi_{U \backslash C_KB(Q_1)} \otimes \chi_{V^c}), h_{Q_1} \otimes h_{Q_2}\rangle
   + \langle T(\chi_{U^c} \otimes \chi_{V \backslash C_KB(Q_2)}), h_{Q_1} \otimes h_{Q_2}\rangle.
\end{eqnarray*}
Combining all these terms together we obtain
\begin{eqnarray*}
  \lefteqn{\sum_{i=2}^{4} \left[A_i\left(C_KB(Q_1),C_KB(Q_2)\right) - A_i(U,V)\right]} \\
  &=& \langle T(\chi_{U \backslash C_KB(Q_1)} \otimes \chi_{C_KB(Q_2)}), h_{Q_1} \otimes h_{Q_2}\rangle \\
  &&+ \, \langle T(\chi_{C_KB(Q_1)} \otimes \chi_{V \backslash C_KB(Q_2)}), h_{Q_1} \otimes h_{Q_2}\rangle \\
  &&+ \, \langle T(\chi_{U \backslash C_KB(Q_1)} \otimes \chi_{V \backslash C_KB(Q_2)}), h_{Q_1} \otimes h_{Q_2}\rangle \\
  &=& \langle T(\chi_{U} \otimes \chi_{C_KB(Q_2)}), h_{Q_1} \otimes h_{Q_2}\rangle
  - \langle T(\chi_{C_KB(Q_1)} \otimes \chi_{C_KB(Q_2)}), h_{Q_1} \otimes h_{Q_2}\rangle \\
  && +\, \langle T(\chi_{U} \otimes \chi_{V}), h_{Q_1} \otimes h_{Q_2}\rangle
  - \langle T(\chi_{U} \otimes \chi_{C_KB(Q_2)}), h_{Q_1} \otimes h_{Q_2}\rangle \\
  &=& -A_1\left(C_KB(Q_1),C_KB(Q_2)\right) + A_1(U,V).
\end{eqnarray*}
This completes the proof of Lemma~\ref{interlem.1}.
\end{proof}

\section{The Hardy, $\bmo$ spaces on $(\mathcal{X},\rho,\mu)$ and Duality}\label{subsec:product_BMO}

In this section, we will introduce the $\bmo$ and Hardy spaces, and discuss their duality in non-homogeneous settings. First, when talking about the one-parameter setting ($X_1$ or $X_2$), we use the families of spaces~$\bmo_{\kappa}^p$ and $H^{1,p}$, where $1 \leq p <\infty$, $\kappa >1$ as defined in Definitions~\ref{defn: BMO_k_p} and~\ref{defn:hardy} below.
Second, when talking about the bi-parameter setting, we use the $\bmo_{\textup{prod}}$ and $\hardy$ spaces defined in Definitions~\ref{defn:BMO_prod} and~\ref{defn:Hardy_prod} below.
We formulate our main theorem with respect to the $\bmo_{\textup{prod}}$ space. The $\bmo_{\kappa}^p$ space is used as a tool in the proof of our main theorem.

We define the $\bmo_{\kappa}^p$ and $H^{1,p}$ spaces on the first factor~$X_1$. The definition on the second factor~$X_2$ is analogous.
Please refer to~\cite[Section~1.1]{NTV03} and~\cite[Chapter~7]{YYH13} for further details related to these spaces.
\begin{defn}\label{defn: BMO_k_p}
  Given $p \geq 1$ and $\kappa >1$,
  we say that~$b \in L^1_{\text{loc}}(\mu_1)$ belongs to $\bmo_{\kappa}^{p}(\mu_1)$, if for each ball $B \subset X_1$ there exists a constant~$b_B$ such that
  \[\bigg(\int_B |b - b_B|^p \,d\mu_1\bigg)^{1/p} \leq L \mu_1(\kappa B)^{1/p},\]
where the constant~$L$ does not depend on~$B$.
The best constant~$C$ is called the $\bmo_{\kappa}^{p}(\mu_1)$-norm of~$b$, and is denoted by~$\|b\|_{\bmo_{\kappa}^{p}(\mu_1)}$.
\end{defn}

Given $p \geq 1$, a function~$a$ is called \emph{$p$-atom} associated with a ball~$B$ if $\supp a \subset B$, $\int_{B} a \,d\mu_1 = 0$ and $\|a\|_{L^p(\mu_1)} \leq \mu_1(B)^{1/p-1}$.

\begin{defn}\label{defn:hardy}
A function~$f \in L^1(\mu_1)$ is said to belong to the \emph{Hardy space} $H^1(\mu_1) = H^{1,p}(\mu_1)$, $p \geq 1$ if there exist $p$-atoms~$\{a_i\}_{i \in \N}$ such that
\[f = \sum_{i=1}^{\infty} \lambda_i a_i \quad \text{and}\quad
\sum_{i=1}^{\infty}|\lambda_i| < \infty.\]
\end{defn}
By duality, to show that a function~$b \in \bmo_{\kappa}^{p}(\mu_1)$, it is sufficient to show that there exists a constant such that $|\langle b,a\rangle| \leq C$ for all $p$-atoms of the Hardy space $H^{1,p}(\mu_1)$ \cite[Chapter~7]{YYH13}.

Let $(\mathcal{X},\rho,\mu)$ be a non-homogeneous bi-parameter quasimetric space.
 Let~$\Dd_1, \Dd_1' \in X_1$ and~$\Dd_2, \Dd_2' \in X_2$ be four independent random systems of dyadic cubes.
 Denote $\Dd = \Dd_1 \times \Dd_2$ and $\Dd' = \Dd_1' \times \Dd_2'$.
 A \emph{rectangle}~$R = Q_1 \times Q_2$ is in~$\Dd$ if $Q_1 \in \Dd_1$ and~$Q_2 \in \Dd_2$.
 Similarly, $R' = Q'_1 \times Q'_2$ is in~$\Dd'$ if $Q'_1 \in \Dd'_1$ and~$Q'_2 \in \Dd'_2$.
 For each $R'= Q'_1 \times Q'_2 \in \Dd'$, define
 \[\mathcal{C}_{R'} := \{R \in \Dd'\text{-good}: R \subset R', \text{gen}(R)=\text{gen}(R')+(r,r)\}.\]
  Here~gen$(R):= (\text{gen}(Q_1), \text{gen}(Q_2))$. The rectangle~$R$ is \emph{good ($\Dd'$-good)} means~$Q_1$ is~$\Dd_1'$-good and~$Q_2$ is~$\Dd_2'$-good.
  For each $R = Q_1 \times Q_2 \in \mathcal{C}_{R'}$, denote $S(R) := R'$. That is, $S(Q_1) = Q'_1$, $S(Q_2) = Q'_2$ and $S(Q_1) \times S(Q_2) = S(R) = R'$.
  Below, we define the \emph{dyadic product~$\bmo$} space~$\bmo_{\textup{prod}}(\mu)$ and the \emph{dyadic product Hardy} space~$\hardy$.
\begin{defn}\label{defn:BMO_prod}
Given $b \in L^1_{\text{loc}}(\mu)$,
we say~$b \in \bmo_{\textup{prod}}(\mu)$ if for every~$\Dd = \Dd_1 \times \Dd_2$, $\Dd' = \Dd_1' \times \Dd_2'$ there holds that
\[\left(
\sum_{\substack{R=Q_1 \times Q_2 \\R \text{ good}, R \subset \Om }}
\sum_{u_1}\sum_{u_2}
|\langle b, h_{u_1}^{Q_1} \otimes h_{u_2}^{Q_2}\rangle|^2\right)^{1/2}
 =\left(
\sum_{\substack{R' \in \Dd' \\ R' \subset \Omega}}
\sum_{R \in \mathcal{C}_{R'}}
|\langle b, h_R\rangle|^2\right)^{1/2}
\leq L\mu(\Omega)^{1/2}\]
for all sets~$\Omega \subset X_1 \times X_2$ such that~$\mu(\Omega) < \infty$ and such that for all~$x \in \Omega$, there exists~$R' \in \Dd'$ such that~$x \in R' \subset \Omega$. The best constant~$L$ is denoted by~$\|b\|_{\bmo_{\textup{prod}}(\mu)}$.
\end{defn}
Definition~\ref{defn:BMO_prod} is motivated by \cite[Definition 2.11]{HM14} of the $\bmo$ space $\bmo_{\textup{prod}}$ on product Euclidean spaces~$\R^n \times \R^m$, and by \cite[Definition~5.7]{KLPW16} of the $\bmo$ space $\bmo_{d,d}$ on product spaces of homogeneous type.
Definition~\ref{defn:BMO_prod} is a natural generalisation of \cite[Definition~2.11]{HM14}. In comparing to \cite[Definition~5.7]{KLPW16}, the only difference is that in \cite[Definition~5.7]{KLPW16}, they sum over all $R \in \Dd$, not just $R \in \Dd'$-good.

\begin{defn}\label{defn:Hardy_prod}
Given $\varphi\in L^1(\mu)$, we say $\varphi \in \hardy$ if for every $\Dd = \Dd_1 \times \Dd_2$ we have
$  S_{\Dd_1 \Dd_2}\varphi \in L^1(\mu)$,
where $S_{\Dd_1 \Dd_2}\varphi$ is
 the \emph{square function} defined by
\begin{align*}
  S_{\Dd_1 \Dd_2}\varphi(x_1,x_2) & := \bigg(
  \sum_{\substack{Q_1 \in \Dd_1\\ Q_1{\text{ good}}}}
  \sum_{\substack{Q_2 \in \Dd_2 \\ Q_2{\text{ good}}}}
  \sum_{u_1} \sum_{u_2}
  |\langle \varphi, h_{u_1}^{Q_1} \otimes h_{u_2}^{Q_2}\rangle|^2
  \frac{\chi_{S(Q_1)}(x_1) \otimes \chi_{S(Q_2)}(x_2)}{\mu_1(S(Q_1))\mu_2(S(Q_2))}
  \bigg)^{1/2} \\
  & = \bigg( \sum_{R' \in \Dd'} \sum_{R \in \mathcal{C}_{R'}}
  |\langle \varphi, h_R\rangle|^2 \frac{\chi_{R'}}{\mu(R')}
  \bigg)^{1/2}.
\end{align*}
Moreover, we define
$$\|\varphi\|_{\hardy} := \|S_{\Dd_1 \Dd_2}\varphi\|_{L^1(\mu)}.$$
\end{defn}
Definition~\ref{defn:Hardy_prod} is motivated by \cite[Definition~7.1]{HM14} of the square functions $S_{\Dd_n \Dd_m}$ on product Euclidean spaces~$\R^n \times \R^m$, and by \cite[Definition~5.6]{KLPW16} of the dyadic Hardy space $H^1_{d,d}$ on product spaces of homogeneous type.
Definition~\ref{defn:Hardy_prod} is a natural generalisation of \cite[Definition~7.1]{HM14}. In comparing to \cite[Definition~5.6]{KLPW16}, the only difference is that in \cite[Definition~5.6]{KLPW16}, the cubes $S(Q_1)$ and~$S(Q_2)$ are replaced by $Q_1$ and~$Q_2$.

Then we have the following argument.
\begin{thm}\label{thm:product_duality}
The dual of $\hardy$ is $\bmo_{\textup{prod}}(\mu)$.
\end{thm}

\begin{proof}
To prove Theorem~\ref{thm:product_duality}, we first introduce the product sequence spaces~$s^1$ and~$t^1$, and show their duality: $(s^1)' = t^1$. These two spaces are models for~$\hardy$ and $\bmo_{\textup{prod}}(\mu)$, respectively. Then, we use the lifting - projection argument to show that $\hardy$ can be lifted to~$s^1$ and~$s^1$ can be projected to~$\hardy$, and that the combination of the lifting and projection operators equals the identity operator on~$\hardy$. We also obtain similar results for~$\bmo_{\textup{prod}}(\mu)$ and~$t^1$. Hence, Theorem~\ref{thm:product_duality} follows. The idea of using product sequence spaces and the lifting - projection argument has been used previously in~\cite{HLL15, HLLu13} to prove duality results in other settings.

Let us define~$s^1$ and~$t^1$ and their properties.

The \emph{product sequence space~$s^1$} is defined as the collection of all sequences $s = \{s_{Q_1,u_1,Q_2,u_2}\} = \{s_R\}$ of complex numbers such that
\begin{align}\label{eq:s1}
  \|s\|_{s^1} & := \bigg\|
  \bigg(
  \sum_{\substack{Q_1 \in \Dd_1\\ Q_1{\text{ good}}}}
  \sum_{\substack{Q_2 \in \Dd_2 \\ Q_2{\text{ good}}}}
  \sum_{u_1} \sum_{u_2}
  |s_{Q_1,u_1,Q_2,u_2}|^2
  \frac{\chi_{S(Q_1)}(x_1) \otimes \chi_{S(Q_2)}(x_2)}{\mu_1(S(Q_1))\mu_2(S(Q_2))}
  \bigg)^{1/2}\bigg\|_{L^1(\mu)} \noz\\
  & = \bigg\|\bigg( \sum_{R' \in \Dd'} \sum_{R \in \mathcal{C}_{R'}}
  |s_R|^2 \frac{\chi_{R'}}{\mu(R')}
  \bigg)^{1/2}\bigg\|_{L^1(\mu)} < \infty.
\end{align}

The \emph{product sequence space~$t^1$ }is defined as the collection of all sequences $t = \{t_{Q_1,u_1,Q_2,u_2}\} = \{t_R\}$ of complex numbers such that
\begin{align}\label{eq:t1}
  \|t\|_{t^1}  := \sup_{\Om}
  \bigg(\frac{1}{\mu(\Om)}
  \sum_{\substack{R=Q_1 \times Q_2 \\R \text{ good}, R \subset \Om }}
  \sum_{u_1} \sum_{u_2}
  |t_{Q_1,u_1,Q_2,u_2}|^2\bigg)^{1/2} =  \sup_{\Om}\bigg(\frac{1}{\mu(\Om)}
\sum_{\substack{R' \subset \Dd' \\ R' \in \Om}} \sum_{R \in \mathcal{C}_{R'}}
  |t_R|^2 \bigg)^{1/2} <\infty,
\end{align}
where the supremum is taken over all open sets~$\Om \in \mathcal{X}$ with finite measure.

The duality result of these two sequences is stated in Theorem~\ref{thm:s1_t1_dual}.
\begin{thm}\label{thm:s1_t1_dual}
The dual of~$s^1$ is $t^1$: $(s^1)' = t^1$.
\end{thm}
\begin{proof}
  First, we prove that for each $t \in t^1$, if for all $s \in s^1$ we have
  \begin{equation}\label{eq1:s1t1dual}
    L(s) = \langle s,t\rangle
    := \sum_{R' \in \Dd'} \sum_{R \in \mathcal{C}_{R'}} s_R \cdot \bar{t}_R,
  \end{equation}
  then there exists a constant~$C >0$ such that
  \[|L(s)| \leq C \|s\|_{s^1} \|t\|_{t^1}.\]
  To see this, let us recall
  the strong maximal function $M_{\Dd'}$ with respect~$\Dd'$, being defined as
  \[M_{\Dd'}f(x) := \sup_{\substack{R' \in \Dd' \\ R' \ni x}} \frac{1}{\mu(5R')} \int_{R'} |f(y)| \, d\mu(y).\]
  We claim that this $M_{\Dd'}$ is bounded on $L^2(\mu)$. In fact, it is easy to see that
  $$ M_{\Dd'}(f)(x_1,x_2)\leq M_1 M_2(f)(x_1,x_2),$$
  where $X_i$ is the \emph{Hardy--Littlewood maximal function} on $(X_i,\rho_i,\mu_i)$ for $i=1,2$, given by
  $$ M_i(h)(x) =   \sup_{\substack{Q \in \Dd_i \\ Q \ni x}} \frac{1}{\mu_i(5Q)} \int_Q |f(y)| \, d\mu_i(y),\quad x\in X_i. $$
  Since each  $ M_i$ is bounded on $L^2(X_i,\mu_i)$, we get that $M_{\Dd'}$ is bounded on $L^2(\mu)$.

  Now, for each $k \in \Z$, we define the following sets
  \begin{align*}
    \Omega_k &  := \left\{x \in \mathcal{X}: \bigg( \sum_{R' \in \Dd'} \sum_{R \in \mathcal{C}_{R'}}
  |s_R|^2 \frac{\chi_{R'}}{\mu(R')}
  \bigg)^{1/2} > 2^k\right\}, \\
    \mathcal{R}'_k &:= \left\{ R' \in \Dd': \mu(R' \cap \Omega_k) > \frac{1}{2} \mu(R') \text{ and } \mu(R' \cap \Omega_{k+1}) \leq \frac{1}{2} \mu(R')\right\}, \quad \text{and} \\
    \widetilde{\Om}_{k} & := \left\{ x \in \mathcal{X}: M_{\Dd'}(\chi_{\Om_k}) > \frac{1}{2}\right\}.
  \end{align*}
We have the following properties\\
  \indent (i) each rectangle~$R' \in \Dd'$ belongs to exactly one set~$\mathcal{R}'_k$; \\
  \indent (ii) if~$R' \in \mathcal{R}'_k$, then $R' \subset \widetilde{\Om}_{k}$, and so $\cup_{R' \in \mathcal{R}'_k} R' \subset \widetilde{\Om}_{k}$ and $\mathcal{R}'_k \subset \{R: R \subset \widetilde{\Om}_{k}\}$;\\
  \indent (iii) if $R' \in \mathcal{R}'_k$, then $\mu((\widetilde{\Om}_{k} \backslash \Om_{k+1})\cap R') \geq \frac{\mu(R')}{2}$; and\\
  \indent (iv) there exists a positive constant $C$ such that $\mu(\widetilde{\Om}_{k}) \leq C\mu(\Om_k)$ for all $k$, since $M_{\Dd'}$ is bounded on $L^2(\mu)$.

  Now using property~(i) we can write
  \begin{align}\label{eq4a:dual_prod}
    |L(s)| & = |\sum_{R' \in \Dd'} \sum_{R \in \mathcal{C}_{R'}}
    s_R \cdot \bar{t}_R|
     \leq \sum_{R' \in \Dd'} \sum_{R \in \mathcal{C}_{R'}}  |s_R \cdot \bar{t}_R|
     = \sum_{k \in \Z} \sum_{R' \in \mathcal{R}'_k} \sum_{R \in \mathcal{C}_{R'}}  |s_R \cdot \bar{t}_R|.
  \end{align}
  By H\"{o}lder's inequality we have
  \begin{equation}\label{eq1a:dual_prod}
    \sum_{R' \in \mathcal{R}'_k} \sum_{R \in \mathcal{C}_{R'}}
    |s_R \cdot \bar{t}_R|
    \leq \Big(  \sum_{R' \in \mathcal{R}'_k} \sum_{R \in \mathcal{C}_{R'}}  |s_R|^2\Big)^{1/2}
    \Big(  \sum_{R' \in \mathcal{R}'_k} \sum_{R \in \mathcal{C}_{R'}}  |t_R|^2\Big)^{1/2}.
  \end{equation}
  Using property~(ii), the definition of~$\|t\|_{t^1}$ in~\eqref{eq:t1} and property~(i), the second term in the right-hand side of~\eqref{eq1a:dual_prod} is bounded by
  \begin{align} \label{eq2a:dual_prod}
   \Big(  \sum_{R' \in \mathcal{R}'_k} \sum_{R \in \mathcal{C}_{R'}}  |t_R|^2\Big)^{1/2}
   & \leq \mu(\widetilde{\Om}_k)^{1/2} \|t\|_{t^1}
    \ls  \mu(\Om_k)^{1/2} \|t\|_{t^1}.
  \end{align}
  Using property~(iii) and the definition of~$\Om_k$, we can estimate the first term in the right-hand side of~\eqref{eq1a:dual_prod} as
  \begin{align}\label{eq3a:dual_prod}
      &\sum_{R' \in \mathcal{R}'_k} \sum_{R \in \mathcal{C}_{R'}}  |s_R |^2\nonumber\\
      & \leq 2  \sum_{R' \in \mathcal{R}'_k} \sum_{R \in \mathcal{C}_{R'}}  |s_R|^2
      \frac{\mu((\widetilde{\Om}_{k} \backslash \Om_{k+1})\cap R')}{\mu(R')}
       \leq 2 \int_{\widetilde{\Om}_{k} \backslash \Om_{k+1}}
       \sum_{R' \in \mathcal{R}'_k} \sum_{R \in \mathcal{C}_{R'}}
       |s_R |^2
       \frac{\chi_{R'}}{\mu(R')} \,d\mu \noz\\
      & \leq  2 \int_{\widetilde{\Om}_{k} \backslash \Om_{k+1}}
      2^{2(k+1)} \,d\mu \leq 2 \cdot 2^{2(k+1)} \mu(\widetilde{\Om}_k)\nonumber
      \\
      &\ls 2^{2k} \mu(\Om_k),
  \end{align}
  where the last inequality follows from property~(iv).

Combining~\eqref{eq4a:dual_prod}--\eqref{eq3a:dual_prod} we obtain
  \begin{align*}
       |L(s)| & \leq C\sum_{k \in \Z}
       2^k \mu(\Om_k)^{1/2}
        \mu(\Om_k)^{1/2} \|t\|_{t^1}
        = C\sum_{k \in \Z} 2^k \mu(\Om_k) \|t\|_{t^1} \leq  C\|s\|_{s^1} \|t\|_{t^1}, \text{ as required}.
  \end{align*}

Second, we prove that for all $L \in (s^1)'$, there exists $t \in t^1$ with $\|t\|_{t^1} \leq \|L\|$ such that for all $s \in s^1$ we have
  \[
  L(s) = \sum_{R' \in \Dd'} \sum_{R \in \mathcal{C}_{R'}} s_R \cdot \bar{t}_R.
  \]
Let $s^i_{Q_1,u_1,Q_2,u_2} = s^i_R$ be the sequence that equals 1 when $R = R_i \in \mathcal{C}_{R'_i}$, $R'_i \in \Dd'$ and equals 0 everywhere else.
In other words, $s^i_R$ equals 1 at the $(Q_1,u_1,Q_2,u_2)$ entry and equals 0 everywhere else.
Notice that $\|s^i_R\|_{s^1} = 1$.
For all~$s \in s^1$ we have
\[
s = \sum_{R' \in \Dd'} \sum_{R \in \mathcal{C}_{R'}} s_R \cdot s^i_R,
\]
where the limit holds in the norm of $s^1$.
For each $L \in (s^1)'$, let $\bar{t}_R := L(s^i_R)$.
Since~$L$ is a bounded linear functional on~$s^1$, we have
\begin{align*}
  L(s) = L(\sum_{R' \in \Dd'} \sum_{R \in \mathcal{C}_{R'}} s_R \cdot s^i_R) & =\sum_{R' \in \Dd'} \sum_{R \in \mathcal{C}_{R'}} s_R \cdot L(s^i_R)
  = \sum_{R' \in \Dd'} \sum_{R \in \mathcal{C}_{R'}} s_R \cdot \bar{t}_R.
\end{align*}
Let $t := \{t_R\}$. Then we only have to check that $\|t\|_{t^1} \leq \|L\|$.

For each fixed open set~$\Om \subset \mathcal{X}$ with finite measure, we define a new measure $\bar{\mu}$ such that $\bar{\mu}(R) = \frac{\mu(R)}{\mu(\Om)}$ when $R \subset \Om$ and $\bar{\mu}(R) = 0$ when $R \nsubseteq \Om$.
Let $\ell^2(\bar{\mu})$ be a sequence space such that when $s \in \ell^2(\bar{\mu})$ we have
\[
\|s\|_{\ell^2(\bar{\mu})}
:= \bigg(\sum_{\substack{R' \in \Dd'\\ R' \subset \Om}} \sum_{R \in \mathcal{C}_{R'}} |s_R|^2 \frac{\mu(R')}{\mu(\Om)} \bigg)^{1/2} < \infty.
\]
Notice that $(\ell^2(\bar{\mu}))' = \ell^2(\bar{\mu})$. Then
\begin{align*}
  &\bigg(\frac{1}{\mu(\Om)} \sum_{\substack{R' \in \Dd'\\ R' \subset \Om}} \sum_{R \in \mathcal{C}_{R'}} |t_R|^2 \bigg)^{1/2}\\
  & = \|\{\mu(R')^{-1/2} t_R\}\|_{\ell^2(\bar{\mu})}
   = \sup_{s: \|s\|_{\ell^2(\bar{\mu})} \leq 1 }
  \bigg|\sum_{\substack{R' \in \Dd'\\ R' \subset \Om}} \sum_{R \in \mathcal{C}_{R'}}  \mu(R')^{-1/2} t_R s_R \frac{\mu(R')}{\mu(\Om)} \bigg| \\
  & = \sup_{s: \|s\|_{\ell^2(\bar{\mu})} \leq 1 }
  \bigg| L\bigg(\Big\{s_R \frac{\mu(R')^{1/2}}{\mu(\Om)}\Big\}_{R: R \in \mathcal{C}_{R'},R' \in \Dd, R' \subset \Om}
   \bigg)\bigg|\\
  & \leq \sup_{s: \|s\|_{\ell^2(\bar{\mu})} \leq 1 } \|L\|
  \bigg\| \Big\{s_R \frac{\mu(R')^{1/2}}{\mu(\Om)}\Big\}_{R: R \in \mathcal{C}_{R'},R' \in \Dd, R' \subset \Om} \bigg\|_{s^1}.
\end{align*}
By the definition of $\|s\|_{s^1}$ and H\"{o}lder's inequality we have
\begin{align*}
 &\hspace{-0.5cm}  \bigg\| \Big\{s_R \frac{\mu(R')^{1/2}}{\mu(\Om)}\Big\}_{R: R \in \mathcal{C}_{R'},R' \in \Dd, R' \subset \Om} \bigg\|_{s^1}
  = \bigg\|\sum_{\substack{R' \in \Dd'\\ R' \subset \Om}} \sum_{R \in \mathcal{C}_{R'}}
 \Big|s_R \frac{\mu(R')^{1/2}}{\mu(\Om)} \Big|^2 \frac{\chi_{R'}}{\mu(R')}
  \bigg\|_{L^1(\mu)} \\
 & \leq \bigg(\sum_{\substack{R' \in \Dd'\\ R' \subset \Om}} \sum_{R \in \mathcal{C}_{R'}}
 |s_R|^2 \frac{\mu(R')}{\mu(\Om)}\bigg)^{1/2} = \|s\|_{\ell^2(\bar{\mu})}.
\end{align*}
Hence
\[
\bigg(\frac{1}{\mu(\Om)} \sum_{\substack{R' \in \Dd'\\ R' \in \Om}} \sum_{R \in \mathcal{C}_{R'}} |t_R|^2 \bigg)^{1/2}
\leq  \sup_{s: \|s\|_{\ell^2(\bar{\mu})} \leq 1 } \|L\| \, \|s\|_{\ell^2(\bar{\mu})}
\leq \|L\|.
\]
Since this is true for all sets $\Om \subset \mathcal{X}$ with finite measure, $\|t\|_{t^1} \leq \|L\|$.
This completes the proof of Theorem~\ref{thm:s1_t1_dual}.
\end{proof}

Before discussing the lifting - projection argument, let us introduce the sequence space~$s^2$, which is a model for function space~$L^2(\mu)$.
We define~$s^2$ to be the space consisting of the sequences $s:= \{s_{Q_1,u_1,Q_2,u_2}\}=\{s_R\}$ of complex numbers with
\begin{align*}
    \|s\|_{s^2} & := \bigg\|
  \bigg(
  \sum_{\substack{Q_1 \in \Dd_1\\ Q_1{\text{ good}}}}
  \sum_{\substack{Q_2 \in \Dd_2 \\ Q_2{\text{ good}}}}
  \sum_{u_1} \sum_{u_2}
  |s_{Q_1,u_1,Q_2,u_2}|^2
  \frac{\chi_{S(Q_1)}(x_1) \otimes \chi_{S(Q_2)}(x_2)}{\mu_1(S(Q_1))\mu_2(S(Q_2))}
  \bigg)^{1/2}\bigg\|_{L^2(\mu)} \noz\\
  & = \bigg\|\bigg( \sum_{R' \in \Dd'} \sum_{R \in \mathcal{C}_{R'}}
  |s_R|^2 \frac{\chi_{R'}}{\mu(R')}
  \bigg)^{1/2}\bigg\|_{L^2(\mu)} < \infty.
\end{align*}
We note that the sets $s^1 \cap s^2$ is dense in~$s^1$ in terms of the~$s^1$ norm~\cite{KLPW16}.

Now, we introduce the \emph{lifting operator}~$T_L$ and \emph{projection operator}~$T_P$ and their properties.
\begin{defn}
  Given functions $\varphi \in L^2(\mu)$, we define the \emph{lifting operator}~$T_L$ by
  \[
  T_L(\varphi) := \{\langle \varphi, h_{u_1}^{Q_1} \otimes h_{u_2}^{Q_2}\rangle\}_{Q_1,u_1,Q_2,u_2}
  = \{\langle \varphi, h_R\rangle\}_{R \in \mathcal{C}_{R'}, R' \in \Dd'}.
  \]
\end{defn}

\begin{defn}
  Given sequences $s = \{s_{Q_1,u_1,Q_2,u_2}\} = \{s_R\} $, we define the \emph{projection operator}~$T_P$ by
  \[
  T_P(s) :=  \sum_{\substack{Q_1 \in \Dd_1\\ Q_1{\text{ good}}}}
  \sum_{\substack{Q_2 \in \Dd_2 \\ Q_2{\text{ good}}}}
  \sum_{u_1} \sum_{u_2}
  s_{Q_1,u_1,Q_2,u_2}
  h_{u_1}^{Q_1} \otimes h_{u_2}^{Q_2}
  = \sum_{R' \in \Dd'} \sum_{R \in \mathcal{C}_{R'}} s_R h_R.
  \]
\end{defn}
Notice that for all $\varphi \in L^2(\mu)$ we have $\varphi = T_P \circ T_L(\varphi)$ due to by the Haar expansion of~$\varphi$.

The following four propositions are obtained from the definitions of Hardy  spaces~$\hardy$ and BMO space $\bmo_{\textup{prod}}(\mu)$ and the sequence spaces~$s^1$, $s^2$ and~$t^1$.
\begin{prop}\label{prop2: dual prod}
  For all $\varphi \in L^2(\mu) \cap \hardy$, we have
  $
  \|T_L(\varphi)\|_{s^1} = \|\varphi\|_{\hardy}.
  $
\end{prop}
\begin{proof}
  \begin{align*}
    \|T_L(\varphi)\|_{s^1}
    & = \bigg\| \bigg( \sum_{R' \in \Dd'} \sum_{R \in \mathcal{C}_{R'}}
  |\langle \varphi, h_R\rangle|^2 \frac{\chi_{R'}}{\mu(R')}
  \bigg)^{1/2}\bigg\|_{L^1(\mu)}
   = \|S_{\Dd} (\p)\|_{L^1(\mu)} = \|\p\|_{\hardy}.
  \end{align*}
\end{proof}

\begin{prop}\label{prop3: dual prod}
  For all $b \in \bmo_{\textup{prod}}(\mu)$, we have
  $
  \|T_L(b)\|_{t^1} = \|b\|_{\bmo_{\textup{prod}}(\mu)}.
  $
\end{prop}
\begin{proof}
  \begin{align*}
    \|T_L(b)\|_{t^1}
     & = \sup_{\Om}\bigg(\frac{1}{\mu(\Om)}
\sum_{\substack{R' \in \Dd'  \\ R' \subset \Om}} \sum_{R \in \mathcal{C}_{R'}}
  |\langle b, h_R\rangle|^2 \bigg)^{1/2}
  =  \|b\|_{\bmo_{\textup{prod}}(\mu)}.
  \end{align*}
\end{proof}

\begin{prop}\label{prop4: dual prod}
  For all $s \in s^1 \cap s^2$, we have
  $
  \|T_P(s)\|_{\hardy} \ls \|s\|_{s^1}.
  $
\end{prop}
\begin{proof}
  \begin{align*}
    \|T_P(s)\|_{\hardy}
    & = \|S_{\Dd}(T_P(s))\|_{L^1(\mu)}
     = \bigg\| \bigg( \sum_{R' \in \Dd'} \sum_{R \in \mathcal{C}_{R'}}
   \Big| \Big \langle \sum_{\widetilde{R}' \in \Dd'}
   \sum_{\widetilde{R} \in \mathcal{C}_{\widetilde{R}'}}
   s_{\widetilde{R}} h_{\widetilde{R}}, h_R \Big\rangle \Big|^2 \frac{\chi_{R'}}{\mu(R')}
  \bigg)^{1/2}\bigg\|_{L^1(\mu)} \\
  & = \bigg\| \bigg( \sum_{R' \in \Dd'} \sum_{R \in \mathcal{C}_{R'}}
   \Big| \sum_{\widetilde{R}' \in \Dd'}
   \sum_{\widetilde{R} \in \mathcal{C}_{\widetilde{R}'}}
   \langle
   s_{\widetilde{R}} h_{\widetilde{R}}, h_R \rangle \Big|^2 \frac{\chi_{R'}}{\mu(R')}
  \bigg)^{1/2}\bigg\|_{L^1(\mu)} \\
  & = \bigg\| \bigg( \sum_{R' \in \Dd'} \sum_{R \in \mathcal{C}_{R'}}
  |\langle s_{R}h_R, h_R\rangle|^2 \frac{\chi_{R'}}{\mu(R')}
  \bigg)^{1/2}\bigg\|_{L^1(\mu)} \\
  & \leq \bigg\| \bigg( \sum_{R' \in \Dd'} \sum_{R \in \mathcal{C}_{R'}}
  \|s_{R}h_R\|_{L^2(\mu)}^2 \|h_R\|_{L^2}^2 \frac{\chi_{R'}}{\mu(R')}
  \bigg)^{1/2}\bigg\|_{L^1(\mu)} \\
  & \ls \bigg\| \bigg( \sum_{R' \in \Dd'} \sum_{R \in \mathcal{C}_{R'}}
  |s_{R}|^2  \frac{\chi_{R'}}{\mu(R')}
  \bigg)^{1/2}\bigg\|_{L^1(\mu)} = \|s\|_{s^1}.
  \end{align*}
Note that $\langle h_{\widetilde{R}}, h_R\rangle = 1$ when $R = \widetilde{R}$ and equals 0 otherwise, $\|h_R\|_{L^2(\mu)} \ls 1$. Also, since $\{s_R\}$ is a sequence of numbers, $\|s_{R}\|_{L^2(\mu)} = |s_R|$.
\end{proof}

\begin{prop}\label{prop5: dual prod}
  For all $t \in t^1 \cap s^2$, we have
  $
  \|T_P(t)\|_{\bmo_{\textup{prod}}(\mu)} \ls \|t\|_{t^1}.
  $
\end{prop}
\begin{proof}By definition we have
  \begin{align*}
    \|T_P(t)\|_{\bmo_{\textup{prod}}(\mu)}
     & = \sup_{\Om}\bigg(\frac{1}{\mu(\Om)}
\sum_{\substack{R' \in \Dd' \\ R' \subset \Om}} \sum_{R \in \mathcal{C}_{R'}}
  |\langle t_R, h_R\rangle|^2 \bigg)^{1/2} \\
 &\leq \sup_{\Om}\bigg(\frac{1}{\mu(\Om)}
\sum_{\substack{R' \in \Dd' \\R' \subset \Om}} \sum_{R \in \mathcal{C}_{R'}}
  |t_R|^2 \bigg)^{1/2}
  = \|t\|_{t^1}.  \qedhere
  \end{align*}
\end{proof}

The following theorem completes the proof of Theorem~\ref{thm:product_duality}.
\begin{thm}\label{thm:product duality 2}
For each $b \in \bmo_{\textup{prod}}(\mu)$, the linear functional given by
\begin{equation}\label{eq5:dual_prod}
\mathcal{L}(\p) = \langle \p, b\rangle := \int_{\mathcal{X}} \p(x)b(x) \,d\mu(x),
\end{equation}
initially defined on $\hardy \cap L^2(\mu)$, has a unique bounded extension to $\hardy$ with
\[
\|\mathcal{L}\| \leq C \|b\|_{\bmo_{\textup{prod}(\mu)}}.
\]
Conversely, every bounded linear functional~$\mathcal{L}$ on $\hardy \cap L^2(\mu)$ can be realised in the form of~\eqref{eq5:dual_prod}. That is, there exists $b \in \bmo_{\textup{prod}(\mu)}$ such that~\eqref{eq5:dual_prod} holds for all $\p \in \hardy \cap L^2(\mu)$ and
\[
\|b\|_{\bmo_{\textup{prod}(\mu)}} \leq C\|\mathcal{L}\|.
\]
\end{thm}
\begin{proof}
  Suppose $b \in \bmo_{\textup{prod}}(\mu)$ and we define the linear functional as in~\eqref{eq5:dual_prod} for $\p \in \hardy \cap L^2(\mu)$. Then by Haar expansion of~$\p$, Propositions~\ref{prop2: dual prod} and~\ref{prop3: dual prod} we have
  \begin{align*}
    |\mathcal{L}(\p)|
    & = |\int_{\mathcal{X}} b(x) \sum_{R' \in \Dd'} \sum_{R \in \mathcal{C}_{R'}} \langle \p, h_R\rangle h_R(x) \,d\mu(x)|
     = |\sum_{R' \in \Dd'} \sum_{R \in \mathcal{C}_{R'}}
    \langle \p, h_R\rangle
    \langle b, h_R\rangle| \\
    & \leq \|\{\langle \p, h_R\rangle\}\|_{s^1}
    \|\{\langle b, h_R\rangle\}\|_{t^1}
     = \|T_L(\p)\|_{s^1}  \|T_L(b)\|_{t^1} \\
    & \leq C \|\p\|_{\hardy}   \|b\|_{\bmo_{\textup{prod}}(\mu)}.
  \end{align*}
This implies that~$\mathcal{L}$ is a bounded linear functional on $\hardy \cap L^2(\mu)$ and hence has a unique bounded extension to~$\hardy$ with $\|\mathcal{L}\| \leq C \|b\|_{\bmo_{\textup{prod}}(\mu)}$ since $\hardy \cap L^2(\mu)$ is dense in~$\hardy$.

Conversely, let $\mathcal{L}$ be a bounded linear operator on $\hardy \cap L^2(\mu)$, and define $\mathcal{L}_1 := \mathcal{L} \circ T_P$, thus $\mathcal{L}_1$ is a bounded linear functional on $s^1 \cap s^2$.
Recall that $\p = T_P \circ T_L(\p)$ for all $\p \in L^2(\mu)$.
Hence we have
\begin{equation}\label{eq6:dual_prod}
  \mathcal{L}(\p) = \mathcal{L}(T_P \circ T_L(\p)) = \mathcal{L}_1(T_L(\p)).
\end{equation}
Since $(s^1)' = t^1$, there exists $t \in t^1$ with $\|t\|_{t^1} \leq \|\mathcal{L}\|$ such that $\mathcal{L}_1(s) = \langle t,s\rangle$ for all $s \in s^1$.
Hence we have
\begin{align}\label{eq7:dual_prod}
  \mathcal{L}_1(T_L(\p))
  & = \laz t, T_L(\p) \raz
   = \sum_{R' \in \Dd'} \sum_{R \in \mathcal{C}_{R'}} t_R \laz \p , h_R \raz
   = \laz \p, \sum_{R' \in \Dd'} \sum_{R \in \mathcal{C}_{R'}} t_R h_R \raz
   = \laz \p, T_P(t)\raz.
\end{align}
From~\eqref{eq6:dual_prod} and~\eqref{eq7:dual_prod} we have
\[
\mathcal{L}(\p) = \laz \p, T_P(t)\raz.
\]
Moreover, by Proposition~\ref{prop5: dual prod} we have
$
\|T_P(t)\|_{\bmo_{\textup{prod}}(\mu)} \ls \|t\|_{t^1} \leq \|\mathcal{L}\|,
$
as required.
\end{proof}
Hence, Theorem~\ref{thm:product_duality} is established.
\end{proof}

The following proposition is immediate from Theorem~\ref{thm:product duality 2}. It will be used in the proof of Proposition~\ref{prop1:mixed_para}.
\begin{prop}\label{prop1:dual_prod}
  Given~$b \in \bmo_{\textup{prod}}(\mu)$ and dyadic grids $\Dd_1$, $\Dd'_1$, $\Dd_2$, $\Dd'_2$, let~$\varphi$ be a function which is of the form
  \begin{align*}
  \varphi &=   \sum_{Q_1 \in \Dd_1{\text{good}}}
  \sum_{Q_2 \in \Dd_2{\text{good}}}
  \sum_{u_1} \sum_{u_2}
  \lambda_{u_1 u_2}^{Q_1 Q_2}
   h_{u_1}^{Q_1} \otimes h_{u_2}^{Q_2}
   := \sum_{R' \in \Dd'} \sum_{R \in \mathcal{C}_R}
  \lambda_R h_R,
  \end{align*}
  where~$\{\lambda_{u_1 u_2}^{Q_1 Q_2}\}$ is a sequence of real numbers.
  Then there holds that
  \[|\langle b,\varphi\rangle| \ls \|b\|_{\bmo_{\textup{prod}(\mu)}} \|S_{\Dd_1 \Dd_2} \varphi\|_{L^1(\mu)}.\]
\end{prop}

\section{Strategy of the Proof of Theorem~\ref{thm:T1thm_ver1}}\label{sec:strategy}
In this section, we present the strategy of proof of Theorem~\ref{thm:T1thm_ver1}, namely a $T(1)$ type theorem on non-homogeneous bi-parameter quasimetric spaces~$(\mathcal{X},\rho,\mu)$, under the \emph{a priori} boundedness assumption that~$\|T\| < \infty$.
The proof of Theorem~\ref{thm:T1thm_ver1} follows the same lines as that in~\cite{HM12a} and~\cite{HM14}, which introduce non-homogeneous $T(1)$ Theorems on metric spaces~$(X,d,\mu)$ and on bi-parameter Euclidean spaces~$\R^n \times \R^m$.
For the reader's convenience, we point out the differences between our notation and that in \cite{HM14} in Table~\ref{tab:notation} below.
\begin{table}[H]
\caption{ Differences between our notation and that in~\cite{HM14}.}\label{tab:notation}
\begin{center}
\begin{tabular}{|l| l| l |}
  \hline
   & Our notation & Notation in~\cite{HM14} \\ \hline
  Bi-parameter spaces & $X_1 \times X_2$ & $\R^n \times \R^m$  \\ \hline
  Dyadic grids & $\Dd_1$, $\Dd'_1$ in $X_1$,  & $\Dd_n$, $\Dd'_n$ in $\R^n$,\\
                & $\Dd_2$, $\Dd'_2$ in $X_2$  & $\Dd_m$, $\Dd'_m$ in $\R^m$\\ \hline
  Dyadic cubes &$Q_1 \in \Dd_1$, $Q'_1 \in \Dd'_1$,  & $I_1 \in \Dd_n$, $I_2, \in \Dd'_n$   \\
                &$Q_2 \in \Dd_2$, $Q'_2 \in \Dd'_2$  & $J_1 \in \Dd_n$, $J_2 \in \Dd'_m$   \\
  \hline
\end{tabular}
\end{center}
\end{table}


In Section~\ref{subsec:initial_reduce}, we explain the initial reductions of the proof.
In Section~\ref{subsec:outline}, the outline of the main part of the proof is given.

%

\subsection{Initial reductions}\label{subsec:initial_reduce}
Let $(\mathcal{X},\rho,\mu) := (X_1 \times X_2, \rho_1 \times \rho_2, \mu_1 \times \mu_2)$ be a non-homogeneous bi-parameter quasimetric space.
Let~$\Dd_1, \Dd_1'$ on $X_1$ and~$\Dd_2, \Dd_2'$ on $X_2$ be four independent random systems of dyadic cubes, which are defined from a probability space~$(\Omega, \mathbb{P})$ (see Theorem~\ref{thm:random_sys_dyadic_cube}, Section~\ref{sec:dyadiccubes}).
For each~$i = 1,2$, let~$Q_i \in \Dd_i$, $Q'_i \in \Dd'_i$, $u_i \in \{1,\ldots, M_{Q_i} -1\}$ and $u'_i \in \{1,\ldots, M_{Q'_i} -1\}$,
where $M_{Q_i}$ and $M_{Q'_i}$ are the number of children of the cubes~$Q_i$ and~$Q'_i$, respectively.
We denote the Haar functions with respect to the dyadic grids $\Dd_i,\Dd'_i$ by $h_{u_i}^{Q_i}, h_{u'_i}^{Q'_i}$, respectively.

Let $T: L^2(\mu) \rightarrow L^2(\mu)$ be a bi-parameter singular integral operator on~$\mathcal{X}$ defined as in Section~\ref{subsec:assume}, such that $\|T\| := \|T\|_{L^2(\mu) \rightarrow L^2(\mu)} < \infty$.
For each $x = x_1 \times x_2 \in \mathcal{X}$, $m \in \Z$ and~$\delta \in (0,1)$ satisfying $96A_0^6 \delta \leq 1$, define
\[
\mathcal{E} = \mathcal{E}(x,m,\delta)
:= \{f \in \mathcal{X}: \|f\|_{L^2(\mu)} = 1, \supp f \subset R_{x,m,\delta}\},
\]
where $R_{x,m,\delta} := Q_{1,x_1,m,\delta} \times Q_{2,x_2,m,\delta}$,
with~$Q_{i,x,m,\delta}$ being dyadic cubes in~$X_i$, centred at~$x_i$ and at level~$\delta^m$, $i = 1,2$.
Choose~$f$, $g \in \mathcal{E}$ so that~$0.7\|T\| \leq |\langle Tf,g\rangle|$.

For each integer $k >m$ we define
\begin{eqnarray*}
  E_kf &:=& \sum_{\substack{Q_1 \in \Dd_1 \\ k>\text{gen}(Q_1) \geq m}} \sum_{u_1=1}^{M_{Q_1}-1}
   h_{u_1}^{Q_1} \otimes \langle f,h_{u_1}^{Q_1} \rangle_1
   + \sum_{\substack{Q_1 \in \Dd_1 \\ \text{gen}(Q_1) = m}}
   h_{0}^{Q_1} \otimes \langle f,h_{0}^{Q_1} \rangle_1\\
   &=:&  E_{k,u} f +E_{k,0} f, \\
  E'_kg &:=& \sum_{\substack{Q'_1 \in \Dd'_1 \\ k>\text{gen}(Q'_1) \geq m}}
  \sum_{u'_1=1}^{M_{Q'_1}-1}
  h_{u'_1}^{Q'_1} \otimes \langle g,h_{u'_1}^{Q'_1} \rangle_1
  + \sum_{\substack{Q'_1 \in \Dd'_1 \\ \text{gen}(Q'_1) = m}}
   h_{0}^{Q'_1} \otimes \langle g,h_{0}^{Q'_1} \rangle_1, \\
   &=:&  E'_{k,u} g +E'_{k,0} g, \\
  \widetilde{E}_kf &:=& \sum_{\substack{Q_2 \in \Dd_2 \\ k>\text{gen}(Q_2) \geq m}}
  \sum_{u_2=1}^{M_{Q_2}-1}
  h_{u_2}^{Q_2} \otimes \langle f,h_{u_2}^{Q_2} \rangle_2
   + \sum_{\substack{Q_2 \in \Dd_2 \\ \text{gen}(Q_2) = m}}
   h_{0}^{Q_2} \otimes \langle f,h_{0}^{Q_2} \rangle_2, \quad \text{and} \\
   &=:& \widetilde{E}_{k,u} f +\widetilde{E}_{k,0} f, \\
  \widetilde{E}'_kg &:=& \sum_{\substack{Q'_2 \in \Dd'_2 \\ k>\text{gen}(Q'_2) \geq m}}
  \sum_{u'_2=1}^{M_{Q'_2}-1}
   h_{u'_2}^{Q'_2} \otimes \langle g,h_{u'_2}^{Q'_2} \rangle_2
   + \sum_{\substack{Q'_2 \in \Dd'_2 \\ \text{gen}(Q'_2) = m}}
   h_{0}^{Q'_2} \otimes \langle g,h_{0}^{Q'_2} \rangle_2\\
   &=:& \widetilde{E}'_{k,u} g+\widetilde{E}'_{k,0} g.
\end{eqnarray*}
Here $\langle f,h_{u_1}^{Q_1} \rangle_1(y_2) = \langle f(\cdot,y_2),h_{u_1}^{Q_1} \rangle = \int f(y_1,y_2)h_{u_1}^{Q_1}(y_1)\,d\mu_1(y_1)$, and $\text{gen}(Q)$ means the generation of the cube~$Q$.

We recall the definitions of good and bad cubes in Section~\ref{subsec:good_bad_cube}. We fix the parameters $\gamma_i$, $\al_i$ and $t_i$, where $i = 1,2$, which are used to defined good and bad cubes.

We also denote $E_{k,u} = E_{k,u,\text{good}}+ E_{k,u,\text{bad}}$, where
\begin{align*}
E_{k,u,\text{good}} f&:=\sum_{\substack{Q_1 \in \Dd_1, \text{good} \\ k>\text{gen}(Q_1) \geq m}} \sum_{u_1=1}^{M_{Q_1}-1}
   h_{u_1}^{Q_1} \otimes \langle f,h_{u_1}^{Q_1} \rangle_1;\\
E_{k,u,\text{bad}} f&:=\sum_{\substack{Q_1 \in \Dd_1, \text{bad} \\ k>\text{gen}(Q_1) \geq m}} \sum_{u_1=1}^{M_{Q_1}-1}
   h_{u_1}^{Q_1} \otimes \langle f,h_{u_1}^{Q_1} \rangle_1.
\end{align*}
We perform a similar split for other terms~$E_{k,0}f$, $\ldots$, $\widetilde{E}'_{k,0}g$ above, and the terms $E_{k,0,\text{good}} f$, $E_{k,0,\text{bad}} f$,$\ldots$, $\widetilde{E}'_{k,0,\text{good}} g$, $\widetilde{E}'_{k,0,\text{bad}} g$ are defined analogously to~$E_{k,u,\text{good}}f$ and~$E_{k,u,\text{bad}}f$ above.

Note that the operator~$E_k$ and $E'_k$ in fact are the inhomogeneous Haar expansions with respect to generation~$k$ in the first variable of functions~$f$ and~$g$, respectively.
Similarly,
 $\widetilde{E}_k$ and $\widetilde{E}'_k$ are the inhomogeneous Haar expansions with respect to generation~$k$ in the second variable of functions~$f$ and~$g$, respectively.
Thus, they satisfy the following properties.
\begin{enumerate}
  \item[(1)] $E_k$, $E'_k$, $\widetilde{E}_k$ and $\widetilde{E}'_k$ are linear operators,
  \item[(2)] $f = \lim_{k \rightarrow \infty}E_kf = \lim_{k \rightarrow \infty}\widetilde{E}_kf$ on~$L^2(\mu)$,
  \item[(3)] $\|E_kf\|_{L^2(\mu)}, \|\widetilde{E}_kf\|_{L^2(\mu)} \leq \|f\|_{L^2(\mu)} = 1$, and
  \item[(4)] $E_k\widetilde{E}_kf = \widetilde{E}_kE_kf$.
\end{enumerate}

We write
\begin{eqnarray*}
  \langle Tf,g\rangle &=& \langle T(f-E_kf),g\rangle + \langle T(E_kf), g-E'_kg\rangle + \langle T(E_kf - \widetilde{E}_kE_kf), E'_kg\rangle \\
  &&+ \langle T(\widetilde{E}_kE_kf) , E'_kg - \widetilde{E}'_kE'_kg\rangle
+\langle T(\widetilde{E}_kE_kf),\widetilde{E}'_kE'_kg\rangle.
\end{eqnarray*}
More briefly,
\[\langle Tf,g\rangle = \langle T(\widetilde{E}_kE_kf),\widetilde{E}'_kE'_kg\rangle + \e_k(\Dd), \quad
\text{where~$\Dd = (\Dd_1, \Dd'_1, \Dd_2, \Dd'_2)$}.\]
Using the fact that~$\|T\| < \infty$ and property~(2) listed above we obtain
$\lim_{k \rightarrow \infty} \e_k(\Dd) = 0.$
Then, using the Dominated Convergence Theorem, we deduce that
\begin{equation}\label{eq2:Part3}
\langle Tf,g\rangle = \lim_{k \rightarrow \infty} \E \langle T(\widetilde{E}_kE_kf),\widetilde{E}'_kE'_kg\rangle.
\end{equation}
Here $ \E := \E_{\Dd} = \E_{(\Dd_1, \Dd'_1, \Dd_2, \Dd'_2)} = \E_{\Dd_1} \E_{\Dd'_1} \E_{\Dd_2} \E_{\Dd'_2}$ is the mathematical expectation taken over all the random dyadic lattices $\Dd_1$, $\Dd'_1$, $\Dd_2$ and $\Dd'_2$ constructed above.
Note that
\begin{eqnarray}\label{initial_eq5}
\widetilde{E}_kE_kf = \widetilde{E}_{k,u} E_{k,u} f +\widetilde{E}_{k,0} E_{k,u} f+\widetilde{E}_{k,u} E_{k,0} f+
\widetilde{E}_{k,0} E_{k,0} f.
\end{eqnarray}
We continue to decompose $\widetilde{E}_kE_kf$ into smaller sums as follows.

Denote~$f_k := \widetilde{E}_kE_kf$.
Let $f_k = f_{k,\text{good}} + f_{k,\text{bad}}$ where~$f_{k,\text{good}}$ contain the sums over~$Q_1$ and~$Q_2$ where both cubes are good, and~$f_{k,\text{bad}}$ contain the sums over~$Q_1$ and~$Q_2$ where at least one cube is bad:
\begin{eqnarray*}
  f_{k,\text{good}} := \widetilde E_{k,u,\text{good}} E_{k,u,\text{good}}f +\widetilde E_{k,0,\text{good}} E_{k,u,\text{good}}f + \widetilde E_{k,u,\text{good}} E_{k,0,\text{good}}f +\widetilde E_{k,0,\text{good}} E_{k,0,\text{good}}f.
\end{eqnarray*}
and
\begin{align*}
  f_{k,\text{bad}} :=\,&
   \widetilde E_{k,u,\text{good}} E_{k,u,\text{bad}}f
   +\widetilde E_{k,0,\text{good}} E_{k,u,\text{bad}}f
   + \widetilde E_{k,u,\text{good}} E_{k,0,\text{bad}}f
   +\widetilde E_{k,0,\text{good}} E_{k,0,\text{bad}}f\\
   &+ \, \widetilde E_{k,u,\text{bad}} E_{k,u}f
   +\widetilde E_{k,0,\text{bad}} E_{k,u}f
   + \widetilde E_{k,u,\text{bad}} E_{k,0}f
   +\widetilde E_{k,0,\text{bad}} E_{k,0}f.
\end{align*}

Notice that
\begin{align*}
  &\|f_k\|_{L^2(\mu)} = \|\widetilde{E}_kE_kf\|_{L^2(\mu)}
  \leq \|E_kf\|_{L^2(\mu)} \leq \|f\|_{L^2(\mu)}, \quad \text{and} \\
  &\|f_k\|_{L^2(\mu)} \leq \|f_{k,\text{good}}\|_{L^2(\mu)}  + \|f_{k,\text{bad}}\|_{L^2(\mu)}.
\end{align*}
Since $\|f\|_{L^2(\mu)} = \|g\|_{L^2(\mu)} = 1$, we have
\begin{eqnarray}
 && \|f_{k,\text{good}}\|_{L^2(\mu)}  \leq \|f\|_{L^2(\mu)} = 1, \quad \text{and}  \label{initial_eq2} \\
 && \|f_{k,\text{bad}}\|_{L^2(\mu)} \leq \|f\|_{L^2(\mu)} = 1. \label{initial_eq3}
\end{eqnarray}

Similarly, denote
$g_k := \widetilde{E}'_kE'_kg = g_{k,\text{good}} + g_{k,\text{bad}}.$
We now write
\begin{align}\label{eq3:Part3}
\langle T(\widetilde{E}_kE_kf),\widetilde{E}'_kE'_kg\rangle
&= \langle Tf_k, g_k\rangle
= \langle Tf_{k,\text{good}}, g_{k,\text{good}}\rangle
+ \langle Tf_{k,\text{good}}, g_{k,\text{bad}}\rangle
+ \langle Tf_{k,\text{bad}}, g_{k}\rangle.
\end{align}
Using~\eqref{eq2:Part3}, \eqref{initial_eq2}, \eqref{initial_eq3} and~\eqref{eq3:Part3} we have
\begin{eqnarray*}
  |\langle Tf,g\rangle| &=& \lim_{k \rightarrow \infty} |\E \langle T(\widetilde{E}_kE_kf),\widetilde{E}'_kE'_kg\rangle|
   = \lim_{k \rightarrow \infty} |\E \langle Tf_k, g_k\rangle| \\
   &\leq&  \lim_{k \rightarrow \infty} \left( \E |\langle Tf_{k,\text{good}}, g_{k,\text{good}}\rangle|
   + \E |\langle Tf_{k,\text{good}}, g_{k,\text{bad}}\rangle|
   + \E |\langle Tf_{k,\text{bad}}, g_{k}\rangle|\right) \\
   &\leq&  \lim_{k \rightarrow \infty}  \left(\E |\langle Tf_{k,\text{good}}, g_{k,\text{good}}\rangle|
   +  \|T\| \E \|g_{k,\text{bad}}\|_{L^2(\mu)}
   +  \|T\| \E \|f_{k,\text{bad}}\|_{L^2(\mu)}\right).
\end{eqnarray*}
Now, we will estimate $\E \|f_{k,\text{bad}}\|_{L^2(\mu)}$.
We start with the first term $\widetilde E_{k,u,\text{good}} E_{k,u,\text{bad}}f$ in the definition of~$f_{k,\text{bad}}$.
Using the property~\eqref{eq1:pro_Haar} of Haar functions, we have
\begin{eqnarray*}
 &&\hspace{-1.5cm} \|\widetilde E_{k,u,\text{good}} E_{k,u,\text{bad}}f\|_{L^2(\mu)}\\
  &\ls&  \bigg( \sum_{\substack{Q_1 \in \Dd_1,\text{good} \\k> \text{gen}(Q_1) \geq m}} \,
\sum_{\substack{Q_2 \in \Dd_2,\text{bad} \\ k>\text{gen}(Q_2) \geq m}}
\sum_{u_1=1}^{M_{Q_1}-1}
\sum_{u_2=1}^{M_{Q_2}-1}
|\langle f, h_{u_1}^{Q_1} \otimes h_{u_2}^{Q_2}\rangle|^2\bigg)^{1/2} \\
  &\leq& \bigg( \sum_{\substack{Q_1 \in \Dd_1 \\k> \text{gen}(Q_1) \geq m}} \,
\sum_{\substack{Q_2 \in \Dd_2 \\ k>\text{gen}(Q_2) \geq m}}
\sum_{u_1=1}^{M_{Q_1}-1}
\sum_{u_2=1}^{M_{Q_2}-1}
\chi_{\text{bad}}(Q_2)
|\langle f, h_{u_1}^{Q_1} \otimes h_{u_2}^{Q_2}\rangle|^2\bigg)^{1/2},
\end{eqnarray*}
where  the function $\chi_{\text{bad}}(Q_2)$ is defined by
\[ \chi_{\text{bad}}(Q_2) := \left\{
      \begin{array}{l l}
        1, & \quad \text{if $Q_2$ is bad;}\\
        0, & \quad \text{if $Q_2$ is good.}
      \end{array} \right.\]
Hence,
\begin{eqnarray*}
  \lefteqn{\E \|\widetilde E_{k,u,\text{good}} E_{k,u,\text{bad}}f\|_{L^2(\mu)}}\\
  &=& \E_{(\Dd_1,\Dd'_1,\Dd_2,\Dd'_2)} \|\widetilde E_{k,u,\text{good}} E_{k,u,\text{bad}}f\|_{L^2(\mu)} \\
  &=&
 \E_{(\Dd_1,\Dd'_1,\Dd_2)} \bigg( \sum_{\substack{Q_1 \in \Dd_1 \\k> \text{gen}(Q_1) \geq m}} \,
\sum_{\substack{Q_2 \in \Dd_2 \\ k>\text{gen}(Q_2) \geq m}}
\sum_{u_1=1}^{M_{Q_1}-1}
\sum_{u_2=1}^{M_{Q_2}-1}
\Pb(Q_2 \in \Dd'_2\text{-bad})
|\langle f, h_{u_1}^{Q_1} \otimes h_{u_2}^{Q_2}\rangle|^2\bigg)^{1/2} \\
&\ls& c(r),
\end{eqnarray*}
where the last inequality follows from Theorem~\ref{thm.10.2.HM12} with $c(r) \sim \delta^{r\gamma_2\kappa}$, where $\delta \in (0,1)$, $\gamma_2 \in (0,1)$ and $\kappa \in (0,1]$.
It is clear that~$c(r) \rightarrow 0$, when~$r \rightarrow \infty$.
Recall that $ \Dd'_1\text{-bad}$ and $\Dd'_2\text{-bad}$ are collections of all bad cubes $Q_1 \in \Dd_1$ and $Q_2 \in \Dd_2$, respectively (see Section~\ref{subsec:good_bad_cube}).
Here~$\Pb$ is the probability measure on the collection~$\Omega$ of random dyadic grids.

Following similar calculations, we obtain the same result for other terms in the definition of~$f_{k,\text{bad}}$. Then we obtain the bound
\begin{eqnarray*}
\E  \|f_{k,\text{bad}}\|_{L^2(\mu)} \ls c(r).
\end{eqnarray*}

The same result holds for $\E  \|g_{k,\text{bad}}\|_{L^2(\mu)}$.
Fixing~$r$ to be sufficiently large, we have shown that
\begin{equation}\label{Initialeq.1}
0.7\|T\| \leq |\langle Tf,g\rangle| \leq \lim_{k \rightarrow \infty} |\E\langle Tf_{k,\text{good}}, g_{k,\text{good}}\rangle| + 0.1\|T\|.
\end{equation}
Now our goal is to show that $ \lim_{k \rightarrow \infty} |\E\langle Tf_{k,\text{good}}, g_{k,\text{good}}\rangle| \leq C + 0.1\|T\|$. In fact, most of the rest of this paper is to show this.
\subsection{Outline of the core of the proof}\label{subsec:outline}
Notice that $|\E\langle Tf_{k,\text{good}}, g_{k,\text{good}}\rangle|$
can be split into 16 terms, and
our proof is reduced to showing that
\begin{eqnarray}\label{coreeq.1}
  |\E\langle Tf_{k,\text{good}}, g_{k,\text{good}}\rangle|
 &=&
 |\E \langle T( \widetilde E_{k,u,\text{good}} E_{k,u,\text{good}}f),  \widetilde E'_{k,u,\text{good}} E'_{k,u,\text{good}}g\rangle \noz \\
 &&\,+ \cdots +
 \E \langle T( \widetilde E_{k,0,\text{good}} E_{k,0,\text{good}}f),  \widetilde E'_{k,0,\text{good}} E'_{k,0,\text{good}}g\rangle| \noz \\
  &=&
  \Bigg|\E \sum_{\substack{Q_1,Q_2 \text{ good } \\ k>\text{gen}(Q_1), \text{gen}(Q_2) \geq m }} \,
  \sum_{\substack{Q'_1,Q'_2 \text{ good} \\ k>\text{gen}(Q'_1), \text{gen}(Q'_2) \geq m }}
  \, \sum_{\substack{u_1,u_2 \\u'_1,u'_2}}
  \langle f, h_{u_1}^{Q_1} \otimes h_{u_2}^{Q_2}\rangle \noz\\
  &&\quad\times \langle g, h_{u'_1}^{Q'_1} \otimes h_{u'_2}^{Q'_2}\rangle
  \langle T(h_{u_1}^{Q_1} \otimes h_{u_2}^{Q_2}), h_{u'_1}^{Q'_1} \otimes h_{u'_2}^{Q'_2}\rangle  \noz\\
  && \,+ \cdots +   \E \sum_{\substack{Q_1,Q_2 \text{ good } \\ \text{gen}(Q_1)= \text{gen}(Q_2) = m }} \,
  \sum_{\substack{Q'_1,Q'_2 \text{ good} \\ \text{gen}(Q'_1)=\text{gen}(Q'_2) = m }}
  \langle f, h_{0}^{Q_1} \otimes h_{0}^{Q_2}\rangle \noz\\
  &&\quad\times \langle g, h_{0}^{Q'_1} \otimes h_{0}^{Q'_2}\rangle
  \langle T(h_{0}^{Q_1} \otimes h_{0}^{Q_2}), h_{0}^{Q'_1} \otimes h_{0}^{Q'_2}\rangle \Bigg| \noz\\
    &\leq& (C + 0.1\|T\|) \|f\|_{L^2(\mu)} \|g\|_{L^2(\mu)},
\end{eqnarray}
where~$C$ is a constant depending only on the assumptions and the $\bmo_{\textup{prod}}(\mu)$ norms of the four $S(1)$, $S \in \{T,T^*,T_1,T^*_1\}$, but independent of the \emph{a priori} bound.

Combining with~\eqref{Initialeq.1}, we get
$0.6\|T\| \leq C+ 0.1\|T\|,$
which implies
$\|T\| \leq 2C.$

Recall that we have already fixed $m$. Now we also fix~$k$. Note that the restriction $k > \text{ gen}(Q_1) \geq m$ is equivalent to $\delta^k < \ell(Q_1) \leq \delta^m$. The same holds for $Q'_1$, $Q_2$ and~$Q'_2$. We will suppress this restriction from the sum.

We perform the splitting
\begin{eqnarray}\label{coreeq.3}
  \sum_{Q_1,Q'_1} \sum_{Q_2,Q'_2}
  &=& \sum_{\ell(Q_1) \leq \ell(Q_1')} \sum_{\ell(Q_2) \leq \ell(Q_2')}
  + \sum_{\ell(Q_1) \leq \ell(Q_1')} \sum_{\ell(Q_2) > \ell(Q_2')} \noz \\
  &&+ \sum_{\ell(Q_1) > \ell(Q_1')} \sum_{\ell(Q_2) \leq \ell(Q_2')}
  + \sum_{\ell(Q_1) > \ell(Q_1')} \sum_{\ell(Q_2) > \ell(Q_2')}.
\end{eqnarray}
We have to deal with $|\E\langle Tf_{k,\text{good}}, g_{k,\text{good}}\rangle|$ in~\eqref{coreeq.1} corresponding to each of the four subseries on the right-hand side of~\eqref{coreeq.3}. This is where the conditions on $T(1)$, $\widetilde{T}(1)$, $\widetilde{T}^*(1)$ and $T^*(1)$, respectively, are needed.

We mainly concentrate on the first subseries when $\ell(Q_1) \leq \ell(Q_1')$ and~$\ell(Q_2) \leq \ell(Q_2')$, and the second subseries $\ell(Q_1) \leq \ell(Q_1')$ and~$\ell(Q_2) > \ell(Q_2')$. The remaining two subseries are completely symmetric to one of these first two.

The calculation involving the first subseries is also similar to that of the second subseries. The only difference is that the paraproduct $\Pi_{a}^{u_2} \omega$ has different forms in two cases.
When $\ell(Q_1) \leq \ell(Q_1')$ and~$\ell(Q_2) \leq \ell(Q_2')$, we call  $\Pi_{a}^{u_2} \omega$ \emph{full paraproduct}.
When $\ell(Q_1) \leq \ell(Q_1')$ and~$\ell(Q_2) > \ell(Q_2')$, we call  $\Pi_{a}^{u_2} \omega$ \emph{mixed paraproduct}.

From now on, we focus on the summation when $\ell(Q_1) \leq \ell(Q_1')$ and~$\ell(Q_2) \leq \ell(Q_2')$.
The exception related to the mixed paraproduct will be handled explicitly in Section~\ref{sec:mix_para}.

We perform the splitting
\[\sum_{\ell(Q_1) \leq \ell(Q'_1)}
=  \sum_{\substack{Q_1, Q'_1 \\ \text{ separated } }}
+ \sum_{\substack{Q_1, Q'_1 \\ \text{ nested } }}
+ \sum_{\substack{Q_1, Q'_1 \\ \text{ adjacent } }} \]
where
\begin{eqnarray*}
 \sum_{\substack{Q_1, Q'_1 \\ \text{ separated } }}
 &:=& \sum_{\substack{\ell(Q_1) \leq \ell(Q'_1) \\ \rho_1(Q_1,Q'_1) > \mathcal{C}\ell(Q_1)^{\gamma_1}\ell(Q'_1)^{1-\gamma_1} }}, \quad
\sum_{\substack{Q_1, Q'_1 \\ \text{ nested } }}
:= \sum_{\substack{\ell(Q_1) < \delta^r \ell(Q'_1) \\ \rho_1(Q_1,Q'_1) \leq \mathcal{C}\ell(Q_1)^{\gamma_1}\ell(Q'_1)^{1-\gamma_1} }}  \\
&&\quad \text{and} \sum_{\substack{Q_1, Q'_1 \\ \text{ adjacent } }}
:= \sum_{\substack{\delta^r \ell(Q'_1) \leq \ell(Q_1) \leq \ell(Q'_1)  \\ \rho_1(Q_1,Q'_1)\leq \mathcal{C}\ell(Q_1)^{\gamma_1}\ell(Q'_1)^{1-\gamma_1} }}.
\end{eqnarray*}
Here $\mathcal{C} = 2A_0C_QC_K$, $A_0$ is the quasitriangle constant of~$\rho_1$, $C_Q$ is from equation~\eqref{a5} and $C_K$ is from the kernel estimates in Assumptions~\ref{assum_2} and~\ref{assum_4}; $\delta \in (0,1)$ satisfies $96A_0^6 \delta \leq 1$; and $\gamma_1 \in (0,1)$ is from Definition~\ref{def:bad cube} of bad cubes.
Figure~\ref{fig:split} presents the relative positions of separated cubes, nested cubes and adjacent cubes on the real line for a representative idea of how this would work in a more general metric space.

It is clear from the definition of \emph{separated} cubes that if $Q_1$ and $Q'_1$ are separated, then $Q_1 \cap Q'_1 = \emptyset$.

The term \emph{nested} makes sense due to the fact that~$Q_1$ is good.
In the nested term, the first condition $\ell(Q_1) < \delta^r \ell(Q'_1)$ means that~$Q_1$ is quite small compared to~$Q'_1$, and the second condition $\rho_1(Q_1,Q'_1) \leq \mathcal{C}\ell(Q_1)^{\gamma_1}\ell(Q'_1)^{1-\gamma_1}$
means that $Q_1$ and~$Q'_1$ are not too far from each other, but also cannot be too close to each other since~$Q_1$ is good.
Moreover, as~$Q_1$ is good and $\ell(Q_1) < \delta^r \ell(Q'_1)$, we have
\begin{equation}\label{coreeq.2}
  \rho_1(Q_1,\partial Q'_1) > \mathcal{C}\ell(Q_1)^{\gamma_1}\ell(Q'_1)^{1-\gamma_1}.
\end{equation}
If $Q_1 \cap Q'_1 \neq \emptyset$, then $\rho_1(Q_1,\partial Q'_1) = 0$, so~\eqref{coreeq.2} does not hold.
If $Q_1 \cap Q'_1 =\emptyset$, then~\eqref{coreeq.2} also fails, as it contradicts the second property of nested cubes.
Therefore, $Q_1$ and~$Q'_1$ are nested actually means that $Q_1 \subset Q'_1$.

Furthermore, the condition $\ell(Q_1) < \delta^r \ell(Q'_1)$ implies that
$\ell(Q_1) \leq \delta^r \ell(Q')$ for all~$Q' \in \text{ch}(Q'_1)$.
Since~$Q_1$ is good, then there exists a child~$Q' \in \text{ch}(Q'_1)$ such that
$\rho_1(Q_1,Q'^c) > \delta \mathcal{C}\ell(Q_1)^{\gamma_1}\ell(Q')^{1-\gamma_1}$.
Hence $Q_1 \subset Q' \subset Q'_1$.

The cases involving \emph{adjacent} cubes are a bit more complicated. This is because two adjacent cubes can be contained in each other, overlapping or disjoint.
\begin{figure}[H]
  \begin{subfigure}[b]{0.32\textwidth}
    \includegraphics[width=\textwidth]{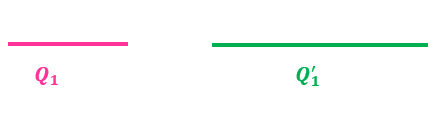}
    \caption{Separated cubes}
    \label{fig:Sep}
  \end{subfigure}
  \begin{subfigure}[b]{0.32\textwidth}
    \includegraphics[width=\textwidth]{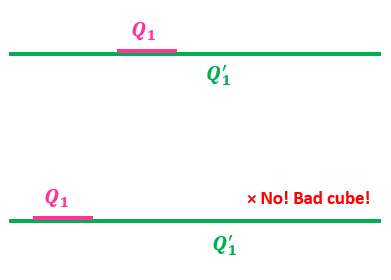}
    \caption{Nested cubes}
    \label{fig:In}
  \end{subfigure}
    \begin{subfigure}[b]{0.32\textwidth}
    \includegraphics[width=\textwidth]{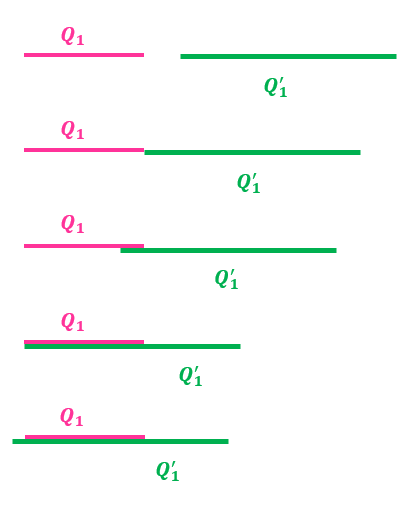}
    \caption{Adjacent cubes}
    \label{fig:Adj}
  \end{subfigure}
  \captionsetup{width=0.8\textwidth}
\caption{Relative positions of separated, nested and adjacent cubes.}
\label{fig:split}
\end{figure}

We preform a similar splitting in the summation~$\ell(Q_2) \leq \ell(Q'_2)$.
Thus, the whole summation is split into nine cases as shown in Table~\ref{tab:case} below. We explicitly deal with the following six cases. Each of the remaining three cases is symmetric to one of these.
\begin{table}[h]
\caption{Splitting of cubes.}\label{tab:case}
\begin{center}
\begin{tabular}{l c c c}
  \hline
  $(Q_1,Q'_1)\backslash (Q_2,Q'_2)$ & \hspace{1cm} separated & \hspace{1cm} adjacent &  \hspace{1cm} nested  \\ \hline
  separated & \hspace{1cm} case~1 & \hspace{1cm} case~2 & \hspace{1cm} case~3 \\
  adjacent &  & \hspace{1cm} case~4 & \hspace{1cm} case~5 \\
  nested &  &  & \hspace{1cm} case~6  \\
  \hline
\end{tabular}
\end{center}
\end{table}

\section{Separated Cubes}\label{sec:separated}
In this section, we deal with the cases involving separated cubes.
In Section~\ref{subsec:sep_sep}, we consider the \emph{Sep/Sep} case when $Q_1, Q'_1$ are separated and $Q_2, Q'_2$ are separated.
In Section~\ref{subsec:sep_in} we consider the \emph{Sep/Nes} case when $Q_1, Q'_1$ are separated and $Q_2, Q'_2$ are nested.
In Section~\ref{subsec:sep_adj} we consider the \emph{Sep/Adj} case when $Q_1, Q'_1$ are separated and $Q_2, Q'_2$ are adjacent.

\subsection{Separated/Separated cubes}\label{subsec:sep_sep}
In this section, we consider the \emph{Sep/Sep} case when $Q_1, Q'_1$ are separated and $Q_2, Q'_2$ are separated. That is,
\begin{eqnarray*}
  &&\ell(Q_1) \leq \ell(Q'_1), \quad \rho_1(Q_1,Q'_1) > \mathcal{C}\ell(Q_1)^{\gamma_1}\ell(Q'_1)^{1-\gamma_1}, \quad \text{and} \\
  &&\ell(Q_2) \leq \ell(Q'_2), \quad \rho_2(Q_2,Q'_2) > \mathcal{C}\ell(Q_2)^{\gamma_2}\ell(Q'_2)^{1-\gamma_2},
\end{eqnarray*}
where $\mathcal{C}:= 2A_0C_QC_K$.

Recall that we need to estimate~\eqref{coreeq.1}, where $|\E\langle Tf_{k,\text{good}}, g_{k,\text{good}}\rangle|$ can be split into 16 terms.
In this \emph{Sep/Sep} case, the terms involving~$h_0^{Q_1}$ or~$h_0^{Q_2}$ will disappear.
This is because for example, the terms involving~$h_0^{Q_1}$ is summed over cubes~$Q_1$ such that~$\text{gen}(Q_1) = m$.
This implies~$\delta^m = \ell(Q_1) \leq \ell(Q'_1) \leq \delta^m$, and so~$\ell(Q_1) = \ell(Q'_1) = \delta^m$.
Thus
\[\rho_1(Q_1,Q'_1) >
  \mathcal{C}\ell(Q_1)^{\gamma_1}\ell(Q'_1)^{1-\gamma_1}
  \geq 2A_0 C_Q\ell(Q_1)
  \geq 2A_0C_Q\delta^m.
 \]
This implies $Q_1 \cap R_{x,m,\delta} = \emptyset$ or~$Q'_1 \cap R_{x,m,\delta} = \emptyset$,
where~$R_{x,m,\delta}$ is the support of the functions~$f$ and~$g$
and so
\[|\langle f, h_{0}^{Q_1} \otimes h_{u_2}^{Q_2}\rangle|
  |\langle g, h_{u'_1}^{Q'_1} \otimes h_{u'_2}^{Q'_2}\rangle| =0.\]

Hence, we are left to bound the expression
\begin{align}\label{sepsepeq.3}
  &\hspace{-0.5cm}\sum_{\substack{Q_1,Q'_1  \\ \text{separated} }} \,
  \sum_{\substack{Q_2,Q'_2\\ \text{separated} }} \,
  \sum_{\substack{u_1,u'_1 \\u_2,u'_2}} \,
  \langle f, h_{u_1}^{Q_1} \otimes h_{u_2}^{Q_2}\rangle
  \langle g, h_{u'_1}^{Q'_1} \otimes h_{u'_2}^{Q'_2}\rangle
  \langle T(h_{u_1}^{Q_1} \otimes h_{u_2}^{Q_2}), h_{u'_1}^{Q'_1} \otimes h_{u'_2}^{Q'_2}\rangle  \noz\\
  &+   \sum_{\substack{Q_1,Q'_1  \\ \text{separated} }} \,
  \sum_{\substack{Q_2,Q'_2 \\ \text{ separated} \\ \text{gen}(Q'_2) = m}} \,
  \sum_{\substack{u_1,u'_1 \\u_2}} \,
  \langle f, h_{u_1}^{Q_1} \otimes h_{u_2}^{Q_2}\rangle
  \langle g, h_{u'_1}^{Q'_1} \otimes h_{0}^{Q'_2}\rangle
  \langle T(h_{u_1}^{Q_1} \otimes h_{u_2}^{Q_2}), h_{u'_1}^{Q'_1} \otimes h_{0}^{Q'_2}\rangle \noz\\
  &+    \sum_{\substack{Q_1,Q'_1  \\ \text{separated} \\ \text{gen}(Q'_1) = m }} \,
  \sum_{\substack{Q_2,Q'_2\\ \text{separated}  }} \,
  \sum_{\substack{u_1\\u_2,u'_2}} \,
  \langle f, h_{u_1}^{Q_1} \otimes h_{u_2}^{Q_2}\rangle
  \langle g, h_{0}^{Q'_1} \otimes h_{u'_2}^{Q'_2}\rangle
  \langle T(h_{u_1}^{Q_1} \otimes h_{u_2}^{Q_2}), h_{0}^{Q'_1} \otimes h_{u'_2}^{Q'_2}\rangle \noz\\
  &+  \sum_{\substack{Q_1,Q'_1  \\ \text{separated} \\ \text{gen}(Q'_1) = m }} \,
  \sum_{\substack{Q_2,Q'_2\\ \text{separated} \\ \text{gen}(Q'_2) = m }} \,
  \sum_{\substack{u_1 \\u_2}} \,
  \langle f, h_{u_1}^{Q_1} \otimes h_{u_2}^{Q_2}\rangle
  \langle g, h_{0}^{Q'_1} \otimes h_{0}^{Q'_2}\rangle
  \langle T(h_{u_1}^{Q_1} \otimes h_{u_2}^{Q_2}), h_{0}^{Q'_1} \otimes h_{0}^{Q'_2}\rangle.
\end{align}

We will estimate the first sum in~\eqref{sepsepeq.3}. The remaining three sums can be handled analogously.
Our proof requires the use of Lemma~\ref{lem.sepsep}, which is also the main result in this section.
For $i = 1,2$ we define $D(Q_i,Q'_i) := \ell(Q_i) + \ell(Q'_i) + \rho_i(Q_i,Q'_i)$.
\begin{lem}\label{lem.sepsep}
Let~$Q_1 \in \Dd_1, Q'_1 \in \Dd'_1, Q_2 \in \Dd_2, Q'_2 \in \Dd'_2$ be such that $Q_1$ and~$Q'_1$ are separated and $Q_2$ and~$Q'_2$ are separated.
Let $f_{Q_1}$, $f_{Q'_1}$, $f_{Q_2}$, $f_{Q'_2}$ be~$L^2(\mu)$ functions supported on cubes $Q_1$, $Q'_1$, $Q_2$, $Q'_2$, respectively. Also assume that
\[\|f_{Q_1}\|_{L^2(\mu_1)} = \|f_{Q'_1}\|_{L^2(\mu_1)}
= \|f_{Q_2}\|_{L^2(\mu_2)} = \|f_{Q'_2}\|_{L^2(\mu_2)} = 1\]
and $\int f_{Q_1}\,d\mu_1 = \int f_{Q_2}\,d\mu_2 = 0$.
Then there holds that
\[|\langle T(f_{Q_1} \otimes f_{Q_2}), g_{Q'_1} \otimes g_{Q'_2}\rangle|
\lesssim A_{Q_1Q'_1}^{\textup{sep}}  A_{Q_2Q'_2}^{\textup{sep}},\]
where
\[A_{Q_1Q'_1}^{\textup{sep}} := \frac{\ell(Q_1)^{\al_1/2}\ell(Q'_1)^{\al_1/2}}{D(Q_1,Q'_1)^{\al_1}
\sup_{z \in Q_1}\lambda_1(z,\rho_1(Q_1,Q'_1))} \mu_1(Q_1)^{1/2}\mu_1(Q'_1)^{1/2}\]
and
\[A_{Q_2Q'_2}^{\textup{sep}} := \frac{\ell(Q_2)^{\al_2/2}\ell(Q'_2)^{\al_2/2}}{D(Q_2,Q'_2)^{\al_2}
\sup_{z \in Q_2}\lambda_2(z,\rho_2(Q_2,Q'_2))} \mu_2(Q_2)^{1/2}\mu_1(Q'_2)^{1/2}.\]
\end{lem}
\begin{proof}
 Since $Q_1$ and $Q'_1$ are separated, for $y_1, z \in Q_1$ and $x_1 \in Q'_1$ we have
  \begin{equation}\label{eq2.sepsep}
  \rho_1(y_1,z) \leq C_Q\ell(Q_1) \leq 2A_0C_Q \ell(Q_1)^{\gamma_1}\ell(Q'_1)^{1-\gamma_1} \leq \frac{\rho_1(Q_1,Q'_1)}{C_K} \leq \frac{\rho_1(x_1,z)}{C_K}.
  \end{equation}
  Similarly, for $y_2,w \in Q_2$ and $x_2 \in Q'_2$ we have
  \[\rho_2(y_2,w) \leq \frac{\rho_2(x_2,w)}{C_K}.\]
Applying Lemma~\ref{lem1:property_kernel} to
$\phi_1 = f_{Q_1}$, $\phi_2 = f_{Q_2}$, $\theta_1 = g_{Q'_1}$ and $\theta_2 = g_{Q'_2}$ we have
\begin{eqnarray}\label{eq2:sepsep_a}
  |\langle T(f_{Q_1} \otimes f_{Q_2}), g_{Q'_1} \otimes g_{Q'_2}\rangle|
   &\ls& \|f_{Q_1}\|_{L^2(\mu_1)}  \|f_{Q_2}\|_{L^2(\mu_2)}
   \mu_1(Q_1)^{1/2} \mu_2(Q_2)^{1/2} \noz\\
   && \times
    \int_{Q'_1}    \frac{C_Q^{\al_1}\ell(Q_1)^{\al_1}}
    {\rho_1(x_1,z)^{\al_1}\lambda_1(z,\rho_1(x_1,z))} |g_{Q'_1}(x_1)| \,d\mu_1(x_1) \noz\\
    && \times
  \int_{Q'_2} \frac{C_Q^{\al_2}\ell(Q_2)^{\al_2}}
  {\rho_2(x_2,w)^{\al_2}\lambda_2(w,\rho_2(x_2,w))}  |g_{Q'_2}(x_2)| \,d\mu_2(x_2).
\end{eqnarray}

 Lemma~\ref{lem1:sepsep} below will be used to estimate the two integrals above.
  \begin{lem}\label{lem1:sepsep}
    Let~$Q_1 \in \Dd_1, Q'_1 \in \Dd'_1, Q_2 \in \Dd_2, Q'_2 \in \Dd'_2$ be such that $Q_1$ and~$Q'_1$ are separated and $Q_2$ and~$Q'_2$ are separated.
    Suppose $z \in Q_1$ and $w \in Q_2$.
    Then the following hold
    \begin{equation}\label{eq3:sepsep}
    \frac{\ell(Q_1)^{\al_1}}{\rho_1(Q_1,Q'_1)^{\al_1}\lambda_1(z,\rho_1(Q_1,Q'_1))} \ls  \frac{\ell(Q_1)^{\al_1/2}\ell(Q'_1)^{\al_1/2}}{D(Q_1,Q'_1)^{\al_1}
\lambda_1(z,D(Q_1,Q'_1))},
    \end{equation}
and
\begin{equation}\label{eq4:sepsep}
\frac{\ell(Q_2)^{\al_2}}{\rho_2(Q_2,Q'_2)^{\al_2}\lambda_2(w,\rho_2(Q_2,Q'_2))} \ls  \frac{\ell(Q_2)^{\al_2/2}\ell(Q'_2)^{\al_2/2}}{D(Q_2,Q'_2)^{\al_2}
\lambda_2(w,D(Q_2,Q'_2))}.
\end{equation}
  \end{lem}

\noindent\emph{Proof of Lemma~\ref{lem1:sepsep}.}
  Consider the case~$\rho_1(Q_1,Q'_1) \geq \ell(Q'_1)$, where we have
  \begin{eqnarray*}
    D(Q_1,Q'_1) &=& \ell(Q_1)+ \ell(Q'_1) + \rho_1(Q_1,Q'_1)
    \leq  \ell(Q_1)+ \rho_1(Q_1,Q'_1) + \rho_1(Q_1,Q'_1) \\
    &\leq&  \ell(Q'_1)+ 2\rho_1(Q_1,Q'_1)
    \leq 3\rho_1(Q_1,Q'_1).
  \end{eqnarray*}
Since $\lambda_1$ is non-decreasing with respect to its second variable and is doubling,
  \[\lambda_1(z,\rho_1(Q_1,Q'_1))\geq \lambda_1\bigg(z,\frac{D(Q_1,Q'_1)}{3}\bigg) \geq C_{\lambda_1}^{-(\log_2 3 +1)}\lambda_1(z,D(Q_1,Q'_1)). \]
Hence when $\rho_1(Q_1,Q'_1) \geq \ell(Q'_1)$ we have
\begin{equation}\label{sepsepeq.1}
   \frac{\ell(Q_1)^{\al_1}}{\rho_1(Q_1,Q'_1)^{\al_1}\lambda_1(z,\rho_1(Q_1,Q'_1))} \ls  \frac{\ell(Q_1)^{\al_1/2}\ell(Q'_1)^{\al_1/2}}{D(Q_1,Q'_1)^{\al_1}
   \lambda_1(z,D(Q_1,Q'_1))}.
\end{equation}
Consider the case~$\rho_1(Q_1,Q'_1) < \ell(Q'_1)$.
Here
\begin{equation*}
  D(Q_1,Q'_1) = \ell(Q_1)+ \ell(Q'_1) + \rho_1(Q_1,Q'_1)
   < \ell(Q'_1) + \ell(Q'_1) + \ell(Q'_1) = 3\ell(Q'_1).
\end{equation*}
Again, since $\lambda_1$ is non-decreasing and is doubling,
\begin{eqnarray*}
  \lambda_1(z, \rho_1(Q_1,Q'_1))
  &>& \lambda_1(z,\mathcal{C}\ell(Q_1)^{\gamma_1}\ell(Q'_1)^{1-\gamma_1})
  > \lambda_1(z,\ell(Q_1)^{\gamma_1}\ell(Q'_1)^{1-\gamma_1})\\
  &\gs& C_{\lambda_1}^{-\gamma_1\log_2\frac{\ell(Q'_1}{\ell(Q_1)})}  \lambda_1(z,\ell(Q'_1))
  = \bigg(\frac{\ell(Q_1)}{\ell(Q'_1)}\bigg)^{t_1\gamma_1}\lambda_1(z,\ell(Q'_1)).
\end{eqnarray*}
The last equality above is due to the facts that $C_{\lambda_1} \geq 1$ and~$t_1 := \log_2 C_{\lambda_1}$, which imply
\[C_{\lambda_1}^{-\gamma_1\log_2\frac{\ell(Q'_1)}{\ell(Q_1)}}
= 2^{-t_1\gamma_1\log_2\frac{\ell(Q'_1)}{\ell(Q_1)}}
= \bigg(\frac{\ell(Q_1)}{\ell(Q'_1)}\bigg)^{t_1\gamma_1}.\]
Combining with the fact that~$\al_1\gamma_1 + t_1\gamma_1 = \al_1/2$ we have
\begin{eqnarray}\label{sepsepeq.2}
   \frac{\ell(Q_1)^{\al_1}}
   {\rho_1(Q_1,Q'_1)^{\al_1}\lambda_1(z,\rho_1(Q_1,Q'_1))}
   &\ls&\frac{\ell(Q_1)^{\al_1}}{\mathcal{C}^{\al_1}\ell(Q_1)^{\al_1\gamma_1}
   \ell(Q'_1)^{(1-\gamma_1)\al_1} \ell(Q_1)^{t_1\gamma_1} \ell(Q'_1)^{-t_1\gamma_1}\lambda_1(z,\ell(Q'_1))} \noz\\
   &=&\frac{\ell(Q_1)^{\al_1-(\al_1\gamma_1+t_1\gamma_1)}
   \ell(Q'_1)^{\al_1\gamma_1+t_1\gamma_1}} {\mathcal{C}^{\al_1}\ell(Q'_1)^{\al_1}\lambda_1(z,\ell(Q'_1))} \noz\\
   &=&\frac{\ell(Q_1)^{\al_1/2}\ell(Q'_1)^{\al_1/2}}
   {\mathcal{C}^{\al_1}\ell(Q'_1)^{\al_1}\lambda_1(z,\ell(Q'_1))} \noz \\
   &\ls&\frac{\ell(Q_1)^{\al_1/2}\ell(Q'_1)^{\al_1/2}}
   {D(Q_1,Q'_1)^{\al_1}\lambda_1(z,D(Q_1,Q'_1))}.
\end{eqnarray}
From~\eqref{sepsepeq.1} and~\eqref{sepsepeq.2}, in both cases we have
\[ \frac{\ell(Q_1)^{\al_1}}{\rho_1(Q_1,Q'_1)^{\al_1}\lambda_1(z,\rho_1(Q_1,Q'_1))} \ls  \frac{\ell(Q_1)^{\al_1/2}\ell(Q'_1)^{\al_1/2}}{D(Q_1,Q'_1)^{\al_1}
\lambda_1(z,D(Q_1,Q'_1))}.\]
Following the same procedure we also obtain
\[ \frac{\ell(Q_2)^{\al_2}}{\rho_2(Q_2,Q'_2)^{\al_2}\lambda_2(w,\rho_2(Q_2,Q'_2))} \ls  \frac{\ell(Q_2)^{\al_2/2}\ell(Q'_2)^{\al_2/2}}{D(Q_2,Q'_2)^{\al_2}
\lambda_2(w,D(Q_2,Q'_2))}.
\]
This completes the proof of Lemma~\ref{lem1:sepsep}.
\(\hfill \Box\)

Using Lemma~\ref{lem1:sepsep} we have
\begin{align*}
\int_{Q'_1}    \frac{C_Q^{\al_1}\ell(Q_1)^{\al_1} |g_{Q'_1}(x_1)| }
    {\rho_1(x_1,z)^{\al_1}\lambda_1(z,\rho_1(x_1,z))} \,d\mu_1(x_1)
 \ls \frac{\ell(Q_1)^{\al_1/2}\ell(Q'_1)^{\al_1/2}}{D(Q_1,Q'_1)^{\al_1}
\lambda_1(z,D(Q_1,Q'_1))} \mu_1(Q'_1)^{1/2}
\ls\frac{A_{Q_1Q'_1}^{\textup{sep}}}{\mu_1(Q_1)^{1/2}},
\end{align*}
and
\begin{align*}
 \int_{Q'_2} \frac{C_Q^{\al_2}\ell(Q_2)^{\al_2}|g_{Q'_2}(x_2)| }
  {\rho_2(x_2,w)^{\al_2}\lambda_2(w,\rho_2(x_2,w))}  \,d\mu_2(x_2)
  \ls \frac{\ell(Q_2)^{\al_2/2}\ell(Q'_2)^{\al_2/2}}{D(Q_2,Q'_2)^{\al_2}
\lambda_2(w,D(Q_2,Q'_2))} \mu_2(Q'_2)^{1/2}
\ls\frac{A_{Q_2Q'_2}^{\textup{sep}}}{\mu_2(Q_2)^{1/2}}.
\end{align*}
Thus, \eqref{eq2:sepsep_a} becomes
\[ |\langle T(f_{Q_1} \otimes f_{Q_2}, g_{Q'_1} \otimes g_{Q'_2})\rangle|
\ls  A_{Q_1Q'_1}^{\textup{sep}}  A_{Q_2Q'_2}^{\textup{sep}},
\]
as required.
\end{proof}\qedhere

The following proposition estimates the sum involving $A_{Q_1Q'_1}^{\textup{sep}}$.
\begin{prop}\label{sepseppop.1}
 For each~$Q_1 \in \Dd_1$ and~$Q'_1 \in \Dd'_1$, given nonnegative constants~$a_{Q_1}, b_{Q'_1} \geq 0$, we have
\[\sum_{\substack{Q_1,Q'_1 \\ \textup{separated}}} A_{Q_1,Q'_1}^{\textup{sep}}a_{Q_1} b_{Q'_1} \ls \bigg(\sum_{Q_1}a_{Q_1}^2\bigg)^{1/2} \bigg(\sum_{Q'_1}b_{Q'_1}^2\bigg)^{1/2}.\]
In fact,
\[\left(\sum_{Q'_1}\left[\sum_{Q_1: \ell(Q_1) \leq \ell(Q'_1)}A_{Q_1,Q'_1}^{\textup{sep}}a_{Q_1} \right]^2\right)^{1/2} \ls \bigg(\sum_{Q_1} a_{Q_1}^2\bigg)^{1/2}.\]
\end{prop}
\begin{proof}
  The proof of Proposition~\ref{sepseppop.1} is a generalisation of Proposition 6.3 in \cite{HM12a}  from non-homogeneous metric spaces to non-homogeneous quasimetric spaces. The proof given there goes through unchanged, as it does not depend on the quasitriangle inequality of quasimetric~$\rho_1$.

  There is also an analogue of Proposition~\ref{sepseppop.1} in non-homogeneous Euclidean spaces stated as Theorem~6.5 in~\cite{NTV03}.
\end{proof}

Now, applying Lemma~\ref{lem.sepsep} to $f_{Q_1} = h_{u_1}^{Q_1}$, $f_{Q_2} = h_{u_2}^{Q_2}$, $f_{Q'_1} = h_{u'_1}^{Q'_1}$, $f_{Q'_2} = h_{u'_2}^{Q'_2}$,
and using Proposition~\ref{sepseppop.1}, the first sum in~\eqref{sepsepeq.3} is bounded by
\begin{eqnarray*}
\lefteqn{\sum_{\ell(Q_1) \leq \ell(Q'_1)} \sum_{\ell(Q_2) \leq \ell(Q'_2)}
\sum_{\substack{u_1,u'_1 \\u_2,u'_2}}
  A_{Q_1,Q'_1}^{\text{sep}} A_{Q_2,Q'_2}^{\text{sep}}
   |\langle f, h_{u_1}^{Q_1} \otimes h_{u_2}^{Q_2}\rangle|
  |\langle g, h_{u'_1}^{Q'_1} \otimes h_{u'_2}^{Q'_2}\rangle|}\\
   &\ls&  \sum_{\ell(Q_1) \leq \ell(Q'_1)}\sum_{u_1,u'_1} A_{Q_1,Q'_1}^{\text{sep}}
   \Big(\sum_{Q_2}\sum_{u_2}|\langle f, h_{u_1}^{Q_1} \otimes h_{u_2}^{Q_2}\rangle|^2\Big)^{1/2}
   \Big(\sum_{Q'_2}\sum_{u'_2}|\langle g, h_{u'_1}^{Q'_1} \otimes h_{u'_2}^{Q'_2}\rangle|^2\Big)^{1/2}\\
  &\ls&\Big(\sum_{Q_1}\sum_{u_1}\sum_{Q_2}\sum_{u_2}|\langle f, h_{u_1}^{Q_1} \otimes h_{u_2}^{Q_2}\rangle|^2\Big)^{1/2}
   \Big(\sum_{Q'_1}\sum_{u'_1}\sum_{Q'_2}\sum_{u'_2}|\langle g, h_{u'_1}^{Q'_1} \otimes h_{u'_2}^{Q'_2}\rangle|^2\Big)^{1/2}\\
   &\ls& \|f\|_{L^2(\mu)} \|g\|_{L^2(\mu)}.
\end{eqnarray*}
Similarly, we obtain the same result for other sums in~\eqref{sepsepeq.3}.
This completes the \emph{Sep/Sep} case.

\subsection{Separated/Nested cubes}\label{subsec:sep_in}
In this section, we consider the \emph{Sep/Nes} case. That is, we estimate~\eqref{coreeq.1} when $Q_1, Q'_1$ are separated and $Q_2, Q'_2$ are nested:
\begin{eqnarray*}
  && \ell(Q_1) \leq \ell(Q'_1), \quad \rho_1(Q_1,Q'_1) > \mathcal{C}\ell(Q_1)^{\gamma_1}\ell(Q'_1)^{1-\gamma_1}, \quad \text{and} \\
  && \ell(Q_2) < \delta^r \ell(Q'_2), \quad \rho_2(Q_2,Q'_2) \leq \mathcal{C}\ell(Q_2)^{\gamma_2}\ell(Q'_2)^{1-\gamma_2},
\end{eqnarray*}
where $\mathcal{C}:= 2A_0C_QC_K$.
With the same reason as in the \emph{Sep/Sep} case, the terms in~\eqref{coreeq.1} which involves~$h_0^{Q_1}$ vanish. Furthermore,
the terms involving~$h_0^{Q_2}$ also vanish. This is because $\ell(Q_2) = \delta^m$ implies that $\delta^{-r}\ell(Q_2) = \delta^{m-r} < \ell(Q'_2) \leq \delta^{m}$, which is impossible.

We are left to bound
\begin{align}\label{sepnesteq.1}
  &\hspace{-0.5cm}\sum_{\substack{Q_1,Q'_1  \\ \text{separated} }} \,
  \sum_{\substack{Q_2,Q'_2\\ \text{nested} }} \,
  \sum_{\substack{u_1,u'_1 \\u_2,u'_2}} \,
  \langle f, h_{u_1}^{Q_1} \otimes h_{u_2}^{Q_2}\rangle
  \langle g, h_{u'_1}^{Q'_1} \otimes h_{u'_2}^{Q'_2}\rangle
  \langle T(h_{u_1}^{Q_1} \otimes h_{u_2}^{Q_2}), h_{u'_1}^{Q'_1} \otimes h_{u'_2}^{Q'_2}\rangle \noz\\
  &+  \sum_{\substack{Q_1,Q'_1  \\ \text{separated} }} \,
  \sum_{\substack{Q_2,Q'_2 \\ \text{ nested} \\ \text{gen}(Q'_2) = m}} \,
  \sum_{\substack{u_1,u'_1 \\u_2}} \,
  \langle f, h_{u_1}^{Q_1} \otimes h_{u_2}^{Q_2}\rangle
  \langle g, h_{u'_1}^{Q'_1} \otimes h_{0}^{Q'_2}\rangle
  \langle T(h_{u_1}^{Q_1} \otimes h_{u_2}^{Q_2}), h_{u'_1}^{Q'_1} \otimes h_{0}^{Q'_2}\rangle  \noz\\
  &+  \sum_{\substack{Q_1,Q'_1  \\ \text{separated} \\ \text{gen}(Q'_1) = m }} \,
  \sum_{\substack{Q_2,Q'_2\\ \text{nested}  }} \,
  \sum_{\substack{u_1\\u_2,u'_2}} \,
  \langle f, h_{u_1}^{Q_1} \otimes h_{u_2}^{Q_2}\rangle
  \langle g, h_{0}^{Q'_1} \otimes h_{u'_2}^{Q'_2}\rangle
  \langle T(h_{u_1}^{Q_1} \otimes h_{u_2}^{Q_2}), h_{0}^{Q'_1} \otimes h_{u'_2}^{Q'_2}\rangle \noz\\
  &+  \sum_{\substack{Q_1,Q'_1  \\ \text{separated} \\ \text{gen}(Q'_1) = m }} \,
  \sum_{\substack{Q_2,Q'_2\\ \text{nested} \\ \text{gen}(Q'_2) = m }} \,
  \sum_{\substack{u_1 \\u_2}} \,
  \langle f, h_{u_1}^{Q_1} \otimes h_{u_2}^{Q_2}\rangle
  \langle g, h_{0}^{Q'_1} \otimes h_{0}^{Q'_2}\rangle
  \langle T(h_{u_1}^{Q_1} \otimes h_{u_2}^{Q_2}), h_{0}^{Q'_1} \otimes h_{0}^{Q'_2}\rangle .
\end{align}

We will handle the first two terms in~\eqref{sepnesteq.1}. The last two terms can be estimated analogously.
We write~$\langle T(h_{u_1}^{Q_1} \otimes h_{u_2}^{Q_2}), h_{u'_1}^{Q'_1} \otimes h_{u'_2}^{Q'_2}\rangle$ as the sum
\begin{eqnarray*}
  && (\langle T(h_{u_1}^{Q_1} \otimes h_{u_2}^{Q_2}), h_{u'_1}^{Q'_1} \otimes h_{u'_2}^{Q'_2}\rangle
- \langle h_{u'_2}^{Q'_2}\rangle_{Q_2}
\langle T(h_{u_1}^{Q_1} \otimes h_{u_2}^{Q_2}), h_{y'_1}^{Q'_1} \otimes 1\rangle) \\
  && \qquad + \, \langle h_{u'_2}^{Q'_2}\rangle_{Q_2}
\langle T(h_{u_1}^{Q_1} \otimes h_{u_2}^{Q_2}), h_{u'_1}^{Q'_1} \otimes 1\rangle.
\end{eqnarray*}
Here $\langle h \rangle_{Q} := \mu_2(Q)^{-1} \int_{Q} h \,d\mu_2$.
Let~$Q'_{2,1} \in \text{ch}(Q'_2)$ be the child of~$Q'_2$ such that $Q_2 \subset Q'_{2,1}$.
As noted in Section~\ref{subsec:outline}, there is such a child~$Q'_{2,1}$ because~$Q_2$ and~$Q'_2$ are nested and~$Q_2$ is good.
Since~$Q_2 \subset Q'_{2,1}$, and~$h_{u'_2}^{Q'_2}$ is constant on~$Q'_{2,1}$, we have
\[\langle h_{u'_2}^{Q'_2}\rangle_{Q_2}
= \langle h_{u'_2}^{Q'_2}\rangle_{Q'_{2,1}} = C. \]

We first consider the matrix element
\begin{eqnarray*}
  \lefteqn{\langle T(h_{u_1}^{Q_1} \otimes h_{u_2}^{Q_2}), h_{u'_1}^{Q'_1} \otimes h_{u'_2}^{Q'_2}\rangle
- \langle h_{u'_2}^{Q'_2}\rangle_{Q_2}
\langle T(h_{u_1}^{Q_1} \otimes h_{u_2}^{Q_2}), h_{u'_1}^{Q'_1} \otimes 1\rangle} \\
 &=& \sum_{\substack{Q' \in \text{ch}(Q'_2) \\ Q' \subset Q'_2 \backslash Q'_{2,1}}} \langle T(h_{u_1}^{Q_1} \otimes h_{u_2}^{Q_2}), h_{u'_1}^{Q'_1} \otimes h_{u'_2}^{Q'_2} \chi_{Q'}\rangle
+ \langle T(h_{u_1}^{Q_1} \otimes h_{u_2}^{Q_2}), h_{u'_1}^{Q'_1} \otimes h_{u'_2}^{Q'_2} \chi_{Q'_{2,1}}\rangle \\
   &&-  \langle h_{u'_2}^{Q'_2}\rangle_{Q'_{2,1}}
\langle T(h_{u_1}^{Q_1} \otimes h_{u_2}^{Q_2}), h_{u'_1}^{Q'_1} \otimes \chi_{Q^{'c}_{2,1}}\rangle
- \langle h_{u'_2}^{Q'_2}\rangle_{Q'_{2,1}}
\langle T(h_{u_1}^{Q_1} \otimes h_{u_2}^{Q_2}), h_{u'_1}^{Q'_1} \otimes \chi_{Q'_{2,1}}\rangle \\
  &=& \sum_{\substack{Q' \in \text{ch}(Q'_2) \\ Q' \subset Q'_2 \backslash Q'_{2,1}}} \langle T(h_{u_1}^{Q_1} \otimes h_{u_2}^{Q_2}), h_{u'_1}^{Q'_1} \otimes h_{u'_2}^{Q'_2} \chi_{Q'}\rangle
-  \, \langle h_{u'_2}^{Q'_2}\rangle_{Q'_{2,1}}
\langle T(h_{u_1}^{Q_1} \otimes h_{u_2}^{Q_2}), h_{u'_1}^{Q'_1} \otimes \chi_{Q^{'c}_{2,1}}\rangle.
\end{eqnarray*}
These two terms will be estimated in Lemmas~\ref{lem1:sepnest} and~\ref{lem2:sepnest} below. Before that, we would like to mention the following useful claim.
\begin{claim}\label{claim:sep_nest}
Suppose $Q_2$ and~$Q'_2$ are nested, $Q'_{2,1} \in \text{ch}(Q'_2)$ such that $Q_2 \subset Q'_{2,1}$, and $Q' \in \text{ch}(Q'_2)$ such that $Q' \subset Q'_2 \backslash Q'_{2,1}$.
The following properties hold:\\
   \indent (i) for $r$ sufficiently large, $\rho_{2}(Q_2,Q^{'c}_{2,1}) \geq \mathcal{C}\ell(Q_2)$
   and  $\rho_{2}(Q_2,Q') \geq \mathcal{C}\ell(Q_2)$;\\
  \indent (ii)  $\rho_{2}(Q_2,Q^{'c}_{2,1}) \geq \mathcal{C}\ell(Q_2)^{1/2} \ell(Q'_{2,1})^{1/2}$; and\\
  \indent (iii) for all $w \in Q_2$, $B(w,\rho_2(Q_2,Q^{'c}_{2,1})) \subset Q^{'}_{2,1}$.
\end{claim}
\begin{proof}
We start with~(i).
As $Q_2$ and $Q'_2$ are nested, $\ell(Q_2) < \delta^r \ell(Q'_2)$. Thus
$$\ell(Q_2) \leq \delta^{r+1} \ell(Q'_2) = \delta^r \ell(Q'_{2,1}) = \delta^r \ell(Q').$$
As~$Q_2$ is good, $Q_2 \subset Q'_{2,1}$ and $\ell(Q_2) \leq \delta^r \ell(Q'_{2,1})$ we have
\[\rho_{2}(Q_2,Q^{'c}_{2,1}) \geq \mathcal{C}\ell(Q_2)^{\gamma_2}\ell(Q'_{2,1})^{1-\gamma_2}
\geq \mathcal{C}\ell(Q_2) \delta^{-r(1-\gamma_2)},\]
where $\mathcal{C} = 2A_0C_QC_K$.
Similarly, since~$Q_2$ is good, $Q' \cap Q'_{2,1} = \emptyset$ and $\ell(Q_2) \leq \delta^r\ell(Q')$ we have
\[\rho_{2}(Q_2,Q') \geq \mathcal{C}\ell(Q_2)^{\gamma_2}\ell(Q'_{2,1})^{1-\gamma_2}
\geq \mathcal{C}\ell(Q_2) \delta^{-r(1-\gamma_2)}.\]
Choose~$r$ large enough such that $\delta^{-r(1-\gamma_2)} \geq 1$. Then $\rho_{2}(Q_2,Q^{'c}_{2,1}) \geq \mathcal{C}\ell(Q_2)$, as required.

Moreover, applying Proposition~\ref{prop2:sep_nest} below to $A = \ell(Q_2)$, $B = \ell(Q'_{2,1})$ and $s = \gamma_2 = \al_2 / 2(\al_2 + t_2) \leq 1/2$ we also obtain~(ii):
\begin{equation}\label{sepnest.eq.1.1}
\rho_{2}(Q_2,Q^{'c}_{2,1}) \geq \mathcal{C}\ell(Q_2)^{\gamma_2}\ell(Q'_{2,1})^{1-\gamma_2}
\geq \mathcal{C}\ell(Q_2)^{1/2} \ell(Q'_{2,1})^{1/2}.
\end{equation}

\begin{prop}\label{prop2:sep_nest}
If $0 < A < B$ and $s \in (0,1/2]$, then $A^s B^{1-s} \geq A^{1/2} B^{1/2}$.
\end{prop}
\begin{proof}
  Let $k(s) := A^s B^{1-s} = e^{s\log A} e^{(1-s)\log B}$.
  Thus
  \begin{eqnarray*}
    k'(s) = A^s B^{1-s}\log A  - A^s B^{1-s}\log B
    =  (\log A - \log B)  A^s B^{1-s}
    < 0 \quad \text{for } s \in (0,1/2].
  \end{eqnarray*}
  This implies that $k(s)$ is a decreasing function, and so $k(s) \geq k(1/2)$.
\end{proof}

Too see~(iii), suppose $B(w,\rho_2(Q_2,Q^{'c}_{2,1})) \nsubseteq Q'_{2,1}$. Then there exists~$z$ such that $z \in B(w,\rho_2(Q_2,Q^{'c}_{2,1}))$ but $z \notin Q'_{2,1}$.
Since $z \in B(w,\rho_2(Q_2,Q^{'c}_{2,1}))$,  $\rho_2(w,z) < \rho_2(Q_2,Q^{'c}_{2,1})$.
Since $z \in Q^{'c}_{2,1}$, $\rho_2(z,w) \geq \rho_2(Q_2,Q^{'c}_{2,1})$.
Thus, $B(w,\rho_2(Q_2,Q^{'c}_{2,1})) \subset Q^{'c}_{2,1}$ by contradiction.
This completes the proof of Claim~\ref{claim:sep_nest}.
\end{proof}


\begin{lem}\label{lem1:sepnest}
Suppose $Q_1$, $Q'_1$ are separated,
$Q_2$, $Q'_2$ are nested, and $Q'_{2,1} \in \textup{ch}(Q'_2)$ such that $Q_2 \subset Q'_{2,1}$.
There holds
\[\Big|\langle h_{u'_2}^{Q'_2}\rangle_{Q'_{2,1}}
\langle T(h_{u_1}^{Q_1} \otimes h_{u_2}^{Q_2}), h_{u'_1}^{Q'_1} \otimes \chi_{Q^{'c}_{2,1}}\rangle \Big|
\ls A^{\textup{sep}}_{Q_1,Q'_1} A^{\textup{in}}_{Q_2,Q'_2},\]
where $A^{\textup{sep}}_{Q_1,Q'_1}$ is defined as in Lemma~\ref{lem.sepsep} above, and
\[A^{\textup{in}}_{Q_2,Q'_2} := \left(\frac{\ell(Q_2)}{\ell(Q'_2)}\right)^{\al_2/2}
\left( \frac{\mu_2(Q_2)}{\mu_2(Q'_{2,1})}\right)^{1/2}. \]
\end{lem}

\begin{proof}
First, using property~\eqref{eq2:pro_Haar} of Haar functions we get
\begin{equation}\label{eq6:sepnest}
\langle h_{u'_2}^{Q'_2}\rangle_{Q'_{2,1}}
\ls \mu_2(Q'_{2,1})^{-1/2}.
\end{equation}

Second, we consider $\langle T(h_{u_1}^{Q_1} \otimes h_{u_2}^{Q_2}), h_{u'_1}^{Q'_1} \otimes \chi_{Q^{'c}_{2,1}}\rangle$.
As $Q_1$ and $Q'_1$ are separated,~\eqref{eq2.sepsep} holds for all $y_1, z \in Q_1$ and $x_1 \in Q'_1$:
 \[\rho_1(y_1,z) \leq C_Q\ell(Q_1) \leq 2A_0C_Q\ell(Q_1)^{\gamma_1}\ell(Q'_1)^{1-\gamma_1}
 < \frac{\rho_1(Q_1,Q'_1)}{C_K} \leq \frac{\rho_1(x_1,z)}{C_K}.\]
Moreover, using property~(i) of Claim~\ref{claim:sep_nest},  for all $y_2, w \in Q_2$, $x_2 \in Q^{'c}_{2,1}$ we have
 \[\rho_2(y_2,w) \leq C_Q \ell(Q_2) \leq \frac{\rho_2(x_2,w)}{C_K}.\]
Since $Q_1 \cap Q'_1 = \emptyset$ and $Q_2 \cap Q^{'c}_{2,1} = \emptyset$,
we can follow a similar calculation as in Lemma~\ref{lem.sepsep} in the \emph{Sep/Sep} case. In particular,
applying Lemma~\ref{lem1:property_kernel} to
$\phi_1 = h^{Q_1}_{u_1}$, $\phi_2 = h^{Q_2}_{u_2}$, $\theta_1 = h^{Q'_1}_{u'_1}$ and $\theta_2 = \chi_{Q^{'c}_{2,1}}$ we have
\begin{eqnarray}\label{eq2:SepIn_a}
  |\langle T(h_{u_1}^{Q_1} \otimes h_{u_2}^{Q_2}), h_{u'_1}^{Q'_1} \otimes \chi_{Q^{'c}_{2,1}}\rangle|
   &\ls& \|h^{Q_1}_{u_1}\|_{L^2(\mu_1)}  \|h^{Q_2}_{u_2}\|_{L^2(\mu_2)}
   \mu_1(Q_1)^{1/2} \mu_2(Q_2)^{1/2} \noz\\
   && \times
    \int_{Q'_1}    \frac{C_Q^{\al_1}\ell(Q_1)^{\al_1}}
    {\rho_1(x_1,z)^{\al_1}\lambda_1(z,\rho_1(x_1,z))} |h^{Q'_1}_{u'_1}(x_1)| \,d\mu_1(x_1) \noz\\
    && \times
  \int_{Q^{'c}_{2,1}} \frac{C_Q^{\al_2}\ell(Q_2)^{\al_2}}
  {\rho_2(x_2,w)^{\al_2}\lambda_2(w,\rho_2(x_2,w))} \,d\mu_2(x_2).
\end{eqnarray}
 Using equation~\eqref{eq3:sepsep} of Lemma~\ref{lem1:sepsep}, the first integral above is bounded by
 \[\frac{\ell(Q_1)^{\al_1/2}\ell(Q'_1)^{\al_1/2}}{D(Q_1,Q'_1)^{\al_1}
\lambda_1(z,D(Q_1,Q'_1))} \mu_1(Q'_1)^{1/2}
\ls \frac{A^{\text{sep}}_{Q_1,Q'_1}}{\mu_1(Q_1)^{1/2}}.\]
 Using property~(iii) of Claim~\ref{claim:sep_nest}, the estimate in Lemma~\ref{upper_dbl_lem1},  property~(ii) of Claim~\ref{claim:sep_nest} and the fact that $\ell(Q^{'}_{2,1}) < \ell(Q'_2)$, the second integral above is bounded by
\begin{align}\label{eq6:sep_nest}
 \ell(Q_2)^{\al_2} \int_{Q^{'c}_{2,1}}  \frac{\rho_2(x_2,w)^{-\al_2}}{\lambda_2(w,\rho_2(x,w))} \,d\mu_2(x_2)
&\leq \ell(Q_2)^{\al_2}\int_{X_2 \backslash B(w,\rho_2(Q_2,Q^{'c}_{2,1}))}  \frac{\rho_2(x_2,w)^{-\al_2}}{\lambda_2(w,\rho_2(x,w))} \,d\mu_2(x_2) \noz \\
&\ls \ell(Q_2)^{\al_2}\rho_2(Q_2, Q^{'c}_{2,1})^{-\al_2} \ls \ell(Q_2)^{\al_2/2} \ell(Q'_2)^{-\al_2/2}.
\end{align}
Thus, \eqref{eq2:SepIn_a} becomes
\begin{eqnarray*}
 |\langle T(h_{u_1}^{Q_1} \otimes h_{u_2}^{Q_2}), h_{u'_1}^{Q'_1} \otimes \chi_{Q^{'c}_{2,1}}\rangle|
   &\ls&  \mu_1(Q_1)^{1/2} \mu_2(Q_2)^{1/2}
    \frac{A^{\text{sep}}_{Q_1,Q'_1}}{\mu_1(Q_1)^{1/2}}
    \ell(Q_2)^{\al_2/2} \ell(Q'_2)^{-\al_2/2}\\
    &=& A^{\textup{sep}}_{Q_1,Q'_1} \left( \frac{\ell(Q_2)}{\ell(Q'_2)}\right)^{\al_2/2} \mu_2(Q_2)^{1/2}.
\end{eqnarray*}
This together with~\eqref{eq6:sepnest} establishes Lemma~\ref{lem1:sepnest}.
\end{proof}

\begin{lem}\label{lem2:sepnest}
Suppose $Q_1$ and~$Q'_1$ are separated,
$Q_2$ and~$Q'_2$ are nested,
$Q' \in \textup{ch}(Q'_2)$ and~$Q' \subset Q'_2\backslash Q'_{2,1}$. Then
 \[|\langle T(h_{u_1}^{Q_1} \otimes h_{u_2}^{Q_2}), h_{u'_1}^{Q'_1} \otimes h_{u'_2}^{Q'_2} \chi_{Q'}\rangle| \ls  A^{\textup{sep}}_{Q_1,Q'_1} A^{\textup{in}}_{Q_2,Q'_2}.\]
\end{lem}

\begin{proof}
Since $Q_1 \cap Q'_1 = \emptyset$ and $Q_2 \cap Q' = \emptyset$,
we can follow the same proof structure as in Lemma~\ref{lem1:sepnest}.
Specifically, we apply Lemma~\ref{lem1:property_kernel} to
$\phi_1 = h^{Q_1}_{u_1}$, $\phi_2 = h^{Q_2}_{u_2}$, $\theta_1 = h^{Q'_1}_{u'_1}$ and $\theta_2 =h_{u'_2}^{Q'_2} \chi_{Q'}$ to get
\begin{align}\label{eq2:SepIn_b}
  |\langle T(h_{u_1}^{Q_1} \otimes h_{u_2}^{Q_2}), h_{u'_1}^{Q'_1} \otimes h_{u'_2}^{Q'_2} \chi_{Q'}\rangle|
   &\ls \|h^{Q_1}_{u_1}\|_{L^2(\mu_1)}  \|h^{Q_2}_{u_2}\|_{L^2(\mu_2)}
   \mu_1(Q_1)^{1/2} \mu_2(Q_2)^{1/2} \noz\\
   & \times
    \int_{Q'_1}    \frac{C_Q^{\al_1}\ell(Q_1)^{\al_1}}
    {\rho_1(x_1,z)^{\al_1}\lambda_1(z,\rho_1(x_1,z))} |h^{Q'_1}_{u'_1}(x_1)| \,d\mu_1(x_1) \noz\\
    & \times
  \int_{Q'} \frac{C_Q^{\al_2}\ell(Q_2)^{\al_2}}
  {\rho_2(x_2,w)^{\al_2}\lambda_2(w,\rho_2(x_2,w))} |h_{u'_2}^{Q'_2}(x_2)| \,d\mu_2(x_2).
\end{align}
 Again, using equation~\eqref{eq3:sepsep} of Lemma~\ref{lem1:sepsep}, the first integral above is bounded by
 \[\frac{\ell(Q_1)^{\al_1/2}\ell(Q'_1)^{\al_1/2}}{D(Q_1,Q'_1)^{\al_1}
\lambda_1(z,D(Q_1,Q'_1))} \mu_1(Q'_1)^{1/2}
\ls \frac{A^{\text{sep}}_{Q_1,Q'_1}}{\mu_1(Q_1)^{1/2}}.\]

Recall that $Q_2$ is good,  $Q_2 \subset Q'_{2,1}$ and $\ell(Q_2) \leq \delta^r \ell(Q'_{2,1})$. For~$x_2 \in Q'$, $w \in Q_2$ we have
\[\rho_2(x_2,w) \geq \rho_2(Q',Q_2) \geq \rho_2(Q^{'c}_{2,1}, Q_2)
\gs \ell(Q_2)^{\gamma_2}\ell(Q'_{2,1})^{1-\gamma_2}
\sim_{\delta}  \ell(Q_2)^{\gamma_2}\ell(Q'_{2})^{1-\gamma_2}.\]
Also, as $\lambda_2$ is doubling and $t_2 := \log_2 C_{\lambda_2}$, we have
\begin{eqnarray*}
\lambda_2(w,\rho_2(x_2,w)) &\gs& \lambda_2(w,\ell(Q_2)^{\gamma_2}\ell(Q'_{2})^{1-\gamma_2}) \\
&\gs& C_{\lambda_2}^{-\gamma_2 \log_2 \frac{\ell(Q'_2)}{\ell(Q_2)}}\lambda_2(w, \ell(Q'_2))
= \bigg(\frac{\ell(Q'_2)}{\ell(Q_2)} \bigg)^{-t_2\gamma_2}
\lambda_2(w, \ell(Q'_2)).
\end{eqnarray*}
Combining this with the fact that~$\al_2\gamma_2 + t_2\gamma_2 = \al_2/2$ we have
\begin{eqnarray*}
  \frac{ \ell(Q_2)^{\al_2} \rho_2(x_2,w)^{-\al_2}}{\lambda_2(w,\rho_2(x,w))} &\ls& \frac{ \ell(Q_2)^{\al_2} \ell(Q_2)^{-\al_2\gamma_2}\ell(Q'_{2})^{-\al_2(1-\gamma_2)} \ell(Q'_2)^{t_2\gamma_2} \ell(Q_2)^{-t_2\gamma_2}}
  {\lambda_2(w, \ell(Q'_2))}\\
  &=&  \left( \frac{\ell(Q_2)}{\ell(Q'_2)}\right)^{\al_2/2}
     \frac{1}{\lambda_2(w,\ell(Q'))}.
\end{eqnarray*}
Therefore the second integral in~\eqref{eq2:SepIn_b} is bounded by
\[ \int_{Q'} \frac{C_Q^{\al_2}\ell(Q_2)^{\al_2}}
  {\rho_2(x_2,w)^{\al_2}\lambda_2(w,\rho_2(x_2,w))} |h_{u'_2}^{Q'_2}(x_2)| \,d\mu_2(x_2)
  \ls  \left( \frac{\ell(Q_2)}{\ell(Q'_2)}\right)^{\al_2/2}
     \frac{\mu_2(Q')^{1/2}}{\lambda_2(w,\ell(Q'))}.
\]
Furthermore, recall that $\lambda_2$ is doubling and $\lambda_2(x,r) \ls \lambda_2(y,r)$ if $\rho_2(x,y) \leq r$.
Let $z'_2$ be the centre of $Q'_2$. Then $\rho(z'_2,w) \leq C_Q\ell(Q'_2)$ as $w \in Q'_2$. Consider
\begin{align*}
\mu_2(Q') \leq \mu_2(Q'_2)
&\leq \mu_2(B(z'_2,C_Q\ell(Q'_2)))
\leq \lambda_2(z'_2,C_Q\ell(Q'_2))
\ls \lambda_2(w,C_Q\ell(Q'_2))\\
&= \lambda_2(w,C_Q\delta^{-1}\ell(Q'))
\ls C(C_{\lambda_2}, C_Q, \delta)\lambda_2(w,\ell(Q')) .
\end{align*}
Thus
\begin{align*}
\frac{\mu_2(Q')^{1/2}}{\lambda_2(w,\ell(Q'))}
\ls \frac{1}{\lambda_2(z'_{2},C_Q\ell(Q'))^{1/2}}\leq \frac{1}{\mu_2(B((z'_{2},C_Q\ell(Q')))^{1/2}}
\leq \frac{1}{\mu_2(Q'_{2})^{1/2}}
< \frac{1}{\mu_2(Q'_{2,1})^{1/2}}.
\end{align*}
Hence, the second integral in~\eqref{eq2:SepIn_b} is bounded further by
\begin{equation}\label{eq5:SepIn_b}
\int_{Q'} \frac{C_Q^{\al_2}\ell(Q_2)^{\al_2}}
  {\rho_2(x_2,w)^{\al_2}\lambda_2(w,\rho_2(x_2,w))} |h_{u'_2}^{Q'_2}(x_2)| \,d\mu_2(x_2)
  \ls  \left( \frac{\ell(Q_2)}{\ell(Q'_2)}\right)^{\al_2/2}
  \frac{1}{\mu_2(Q'_{2,1})^{1/2}}.
\end{equation}

Now, \eqref{eq2:SepIn_b} becomes
\begin{align*}
 |\langle T(h_{u_1}^{Q_1} \otimes h_{u_2}^{Q_2}), h_{u'_1}^{Q'_1} \otimes  h_{u'_2}^{Q'_2} \chi_{Q'}\rangle|
   &\ls  \mu_1(Q_1)^{1/2} \mu_2(Q_2)^{1/2}
    \frac{A^{\text{sep}}_{Q_1,Q'_1}}{\mu_1(Q_1)^{1/2}}
    \left( \frac{\ell(Q_2)}{\ell(Q'_2)}\right)^{\al_2/2}
  \frac{1}{\mu_2(Q'_{2,1})^{1/2}} \\
  &= A^{\textup{sep}}_{Q_1,Q'_1}  A^{\textup{in}}_{Q_2,Q'_2}.
\end{align*}
This completes the proof of Lemma~\ref{lem2:sepnest}.
\end{proof}

The following proposition estimates the sum involving~$A^{\textup{in}}_{Q_2,Q'_2}$.
\begin{prop}\label{prop:A_in}
The matrix~$A^{\textup{in}}_{Q_2,Q'_2}$ generates a bounded operator on~$\ell^2$.
That is, for each $Q_2 \in \Dd_2$ and $Q'_2 \in \Dd'_2$, given nonnegative constants $a_{Q_2}$ and $b_{Q'_2}$, we have
\[\sum_{\substack{Q_2, Q'_2 \\ \textup{nested}}}
A^{\textup{in}}_{Q_2,Q'_2} a_{Q_2} b_{Q'_2}
 \ls \big( \sum_{Q_2} a_{Q_2}^2\big)^{1/2}
\big( \sum_{Q'_2} b_{Q'_2}^2\big)^{1/2}.\]
In fact, there holds that
\[\Big(\sum_{Q'_2} \big( \sum_{\substack{Q_2, Q'_2 \\ \textup{nested} }} A^{\textup{in}}_{Q_2,Q'_2} a_{Q_2} \big)^2\Big)^{1/2}
\ls \big(\sum_{Q_2} a_{Q_2}^2 \big)^{1/2}.\]
\end{prop}

\begin{proof}
As for Proposition~\ref{sepseppop.1},
    the proof of Proposition~\ref{prop:A_in} is a generalisation of Proposition 7.3 in \cite{HM12a} from non-homogeneous metric spaces, and of Lemma~7.4 in~\cite{NTV03} from non-homogeneous Euclidean spaces to non-homogeneous quasimetric spaces. The proofs given there go through unchanged, as it does not depend on the quasitriangle inequality of quasimetric~$\rho_1$.
\end{proof}

Now combining Lemmas~\ref{lem1:sepnest} and~\ref{lem2:sepnest}, Propositions~\ref{sepseppop.1} and~\ref{prop:A_in}, we consider
the summation as in the first term of~\eqref{sepnesteq.1}, but with the term
$$\langle T(h_{u_1}^{Q_1} \otimes h_{u_2}^{Q_2}), h_{u'_1}^{Q'_1} \otimes h_{u'_2}^{Q'_2}\rangle
- \langle h_{u'_2}^{Q'_2}\rangle_{Q_2}
\langle T(h_{u_1}^{Q_1} \otimes h_{u_2}^{Q_2}), h_{u'_1}^{Q'_1} \otimes 1\rangle.$$
We have
\begin{eqnarray*}
  \lefteqn{  \sum_{\substack{Q_1, Q'_1 \\ \textup{separated}  }}
  \sum_{\substack{Q_2, Q'_2 \\ \textup{nested} }}
  \sum_{\substack{u_1,u'_1 \\u_2,u'_2}}
  |\langle f, h_{u_1}^{Q_1} \otimes h_{u_2}^{Q_2}\rangle|
  |\langle g, h_{u'_1}^{Q'_1} \otimes h_{u'_2}^{Q'_2}\rangle|
 }\\
  &&  \qquad \times  \,|\langle T(h_{u_1}^{Q_1} \otimes h_{u_2}^{Q_2}), h_{u'_1}^{Q'_1} \otimes h_{u'_2}^{Q'_2}\rangle
- \langle h_{u'_2}^{Q'_2}\rangle_{Q_2}
\langle T(h_{u_1}^{Q_1} \otimes h_{u_2}^{Q_2}), h_{u'_1}^{Q'_1} \otimes 1\rangle |   \\
  &\ls& \sum_{\substack{Q_1, Q'_1 \\ \textup{separated}}}
  \sum_{\substack{Q_2, Q'_2 \\ \textup{nested} }}
  \sum_{\substack{u_1,u'_1 \\u_2,u'_2}}
  A^{\textup{sep}}_{Q_1,Q'_1} A^{\textup{in}}_{Q_2,Q'_2}
  |\langle f, h_{u_1}^{Q_1} \otimes h_{u_2}^{Q_2}\rangle|
  |\langle g, h_{u'_1}^{Q'_1} \otimes h_{u'_2}^{Q'_2}\rangle| \\
  &\ls& \|f\|_{L^2(\mu)} \|g\|_{L^2(\mu)}.
\end{eqnarray*}

Applying the same proof technique to the second term of~\eqref{sepnesteq.1}, we obtain
\begin{eqnarray*}
  \lefteqn{  \sum_{\substack{Q_1, Q'_1 \\ \textup{separated}  }}
  \sum_{\substack{Q_2, Q'_2 \\ \textup{nested}\\ \text{gen}(Q'_2)=m }}
  \sum_{\substack{u_1,u'_1 \\u_2}}
  |\langle f, h_{u_1}^{Q_1} \otimes h_{u_2}^{Q_2}\rangle|
  |\langle g, h_{u'_1}^{Q'_1} \otimes h_{0}^{Q'_2}\rangle|
 }\\
  &&  \qquad \times  \,|\langle T(h_{u_1}^{Q_1} \otimes h_{u_2}^{Q_2}), h_{u'_1}^{Q'_1} \otimes h_{0}^{Q'_2}\rangle
- \langle h_{0}^{Q'_2}\rangle_{Q_2}
\langle T(h_{u_1}^{Q_1} \otimes h_{u_2}^{Q_2}), h_{u'_1}^{Q'_1} \otimes 1\rangle |   \\
  &\ls& \|f\|_{L^2(\mu)} \|g\|_{L^2(\mu)}.
\end{eqnarray*}

We are left to consider the same summations as in the first and second terms of~\eqref{sepnesteq.1} but involving the elements
$\langle h_{u'_2}^{Q'_2}\rangle_{Q_2}
\langle T(h_{u_1}^{Q_1} \otimes h_{u_2}^{Q_2}), h_{u'_1}^{Q'_1} \otimes 1\rangle$
and
$\langle h_{0}^{Q'_2}\rangle_{Q_2}
\langle T(h_{u_1}^{Q_1} \otimes h_{u_2}^{Q_2}), h_{u'_1}^{Q'_1} \otimes 1\rangle$, respectively.
That is,
\begin{align}\label{sepnesteq.3}
& \sum_{\substack{Q_1, Q'_1 \\ \textup{separated} }} \,
  \sum_{\substack{Q_2, Q'_2 \\ \textup{nested}}} \,
  \sum_{\substack{u_1,u'_1 \\u_2,u'_2}}
\langle f, h_{u_1}^{Q_1} \otimes h_{u_2}^{Q_2}\rangle
  \langle g, h_{u'_1}^{Q'_1} \otimes h_{u'_2}^{Q'_2}\rangle
  \langle h_{u'_2}^{Q'_2}\rangle_{Q_2}
\langle T(h_{u_1}^{Q_1} \otimes h_{u_2}^{Q_2}), h_{u'_1}^{Q'_1} \otimes 1\rangle\noz \\
&+
\sum_{\substack{Q_1, Q'_1 \\ \textup{separated} }} \,
  \sum_{\substack{Q_2, Q'_2 \\ \textup{nested}\\ \text{gen}(Q'_2)=m }} \,
  \sum_{\substack{u_1,u'_1 \\u_2}}
\langle f, h_{u_1}^{Q_1} \otimes h_{u_2}^{Q_2}\rangle
  \langle g, h_{u'_1}^{Q'_1} \otimes h_{0}^{Q'_2}\rangle
  \langle h_{0}^{Q'_2}\rangle_{Q_2}
\langle T(h_{u_1}^{Q_1} \otimes h_{u_2}^{Q_2}), h_{u'_1}^{Q'_1} \otimes 1\rangle.
\end{align}

Lemma~\ref{lem3.sepnest} below will let us rewrite~\eqref{sepnesteq.3} in terms of a paraproduct.
Notice that in Lemma~\ref{lem3.sepnest}, the cubes~$Q_1$ and~$Q'_1$ are not necessarily separated.
\begin{lem}\label{lem3.sepnest}
The sum
\begin{align}\label{sepnesteq.4}
&\sum_{Q_1, Q'_1 } \,
  \sum_{\substack{Q_2, Q'_2 \\ \textup{nested}}} \,
  \sum_{\substack{u_1,u'_1 \\u_2,u'_2}}
\langle f, h_{u_1}^{Q_1} \otimes h_{u_2}^{Q_2}\rangle
  \langle g, h_{u'_1}^{Q'_1} \otimes h_{u'_2}^{Q'_2}\rangle
  \langle h_{u'_2}^{Q'_2}\rangle_{Q_2}
\langle T(h_{u_1}^{Q_1} \otimes h_{u_2}^{Q_2}), h_{u'_1}^{Q'_1} \otimes 1\rangle  \noz \\
&+
\sum_{Q_1, Q'_1 } \,
  \sum_{\substack{Q_2, Q'_2 \\ \textup{nested}\\ \textup{gen}(Q'_2)=m }} \,
  \sum_{\substack{u_1,u'_1 \\u_2}}
\langle f, h_{u_1}^{Q_1} \otimes h_{u_2}^{Q_2}\rangle
  \langle g, h_{u'_1}^{Q'_1} \otimes h_{0}^{Q'_2}\rangle
  \langle h_{0}^{Q'_2}\rangle_{Q_2}
\langle T(h_{u_1}^{Q_1} \otimes h_{u_2}^{Q_2}), h_{u'_1}^{Q'_1} \otimes 1\rangle
\end{align}
equals
\[
\bigg\langle
\sum_{Q_1, Q'_1 \textup{ good}} \quad
\sum_{\substack{u_1,u'_1 \\u_2}} \quad
 h_{u'_1}^{Q'_1} \otimes  \big(\Pi_{b^{u_1 u'_1}_{Q_1 Q'_1}}^{u_2} \big)^*f^{Q_1,u_1},
 g_{\textup{good}}
\bigg\rangle,
\]
where
\begin{eqnarray*}
  f^{Q_1,u_1} &:=& \langle f, h_{u_1}^{Q_1}\rangle_1 = \int_{X_1} f(x,y)h_{u_1}^{Q_1}(x) \,d\mu_1(x), \\
  b^{u_1 u'_1}_{Q_1 Q'_1} &:=& \langle T^*(h_{u'_1}^{Q'_1} \otimes 1), h_{u_1}^{Q_1}\rangle_1, \quad \text{and}\\
  \Pi_{a}^{u_2} \omega &:=&  \sum_{Q'_2 \in \Dd'_2} \quad
\sum_{\substack{Q_2 \textup{ good} \\Q_2 \subset Q'_2 \\ \ell(Q_2) = \delta^r\ell(Q'_2)}}
\langle \omega\rangle_{Q'_2}
\langle a, h_{u_2}^{Q_2}\rangle h_{u_2}^{Q_2}.
\end{eqnarray*}
\end{lem}

\begin{proof}
Since~$Q_2$ is good, $Q_2 \subset Q'_2$ and $\ell(Q_2) < \delta^r \ell(Q'_2)$, there exists a unique cube $S(Q_2) \in \Dd'_2$ such that $\ell(Q_2) = \delta^{r}\ell(S(Q_2))$ and $Q_2 \subset S(Q_2) \subset Q'_2$.
This implies
\[\langle h_{u'_2}^{Q'_2}\rangle_{Q_2}
= \langle h_{u'_2}^{Q'_2}\rangle_{S(Q_2)} \quad \text{and}\quad
 \langle h_{0}^{Q'_2}\rangle_{Q_2}
=  \langle h_{0}^{Q'_2}\rangle_{S(Q_2)}.\]
Recall that all the cubes appearing here are good.
In fact, the goodness of $Q'_2$ is unfavourable here. Therefore, we define
\[
 \langle g_{\text{good}}, h_{u'_1}^{Q'_1} \otimes h_{u'_2}^{Q'_2} \rangle :=
\left\{
  \begin{array}{ll}
    \langle g, h_{u'_1}^{Q'_1} \otimes h_{u'_2}^{Q'_2} \rangle,
     & \hbox{\text{if $Q'_2$ is good};} \\
    0, & \hbox{\text{if $Q'_2$ is bad}.}
  \end{array}
\right.
\]
Now we can add all bad cubes~$Q'_2$ to the summation~\eqref{sepnesteq.4} and write
\[\langle g, h_{u'_1}^{Q'_1} \otimes h_{u'_2}^{Q'_2} \rangle = \langle g_{\text{good}}, h_{u'_1}^{Q'_1} \otimes h_{u'_2}^{Q'_2} \rangle.\]
Below, we explicitly write the fact that $Q_1, Q'_1$ and~$Q_2$ are good, because we want to highlight the fact that~$Q'_2$ is not necessarily a good cube.
Also, we add the condition~$\ell(Q'_2) \leq \delta^m$ back, which we have suppressed in Section~\ref{subsec:outline}.
Note that~$Q_2$ and~$Q'_2$ are nested.
The sum~\eqref{sepnesteq.4} becomes
\begin{align*}
&\sum_{\substack{Q_1, Q'_1 \\ \textup{good} }} \,
  \sum_{Q_2 \textup{ good}} \,
  \sum_{\substack{Q'_2 \\ \ell(Q'_2) \leq \delta^m}} \,
  \sum_{\substack{u_1,u'_1 \\u_2,u'_2}}
  \langle g_{\textup{good}}, h_{u'_1}^{Q'_1} \otimes h_{u'_2}^{Q'_2}\rangle \langle h_{u'_2}^{Q'_2}\rangle_{S(Q_2)} \noz \\
&\hspace{2cm}\times \,
\langle f, h_{u_1}^{Q_1} \otimes h_{u_2}^{Q_2}\rangle
\langle T(h_{u_1}^{Q_1} \otimes h_{u_2}^{Q_2}), h_{u'_1}^{Q'_1} \otimes 1\rangle \noz \\
&+
\sum_{\substack{Q_1, Q'_1 \\ \textup{good} }} \,
  \sum_{Q_2 \textup{ good}} \,
  \sum_{\substack{Q'_2 \\ \ell(Q'_2) = \delta^m}}
  \sum_{\substack{u_1,u'_1 \\u_2}}
  \langle g_{\textup{ good}}, h_{u'_1}^{Q'_1} \otimes h_{0}^{Q'_2}\rangle\langle h_{0}^{Q'_2}\rangle_{S(Q_2)} \noz \\
&\hspace{2cm}\times \,
\langle f, h_{u_1}^{Q_1} \otimes h_{u_2}^{Q_2}\rangle
\langle T(h_{u_1}^{Q_1} \otimes h_{u_2}^{Q_2}), h_{u'_1}^{Q'_1} \otimes 1\rangle,
\end{align*}
which equals
\begin{align}\label{eq8:sep_nest}
&\sum_{Q_1, Q'_1 \textup{ good} } \,
\sum_{Q_2 \textup{ good}}
\sum_{\substack{u_1,u'_1 \\u_2}}
\bigg\langle
\sum_{\substack{Q'_2 \\ \ell(Q'_2) \leq \delta^m}}\,
\sum_{u'_2} \langle g_{\text{good}}, h_{u'_1}^{Q'_1} \otimes h_{u'_2}^{Q'_2}\rangle h_{u'_2}^{Q'_2}
\bigg\rangle_{S(Q_2)} \noz \\
&\hspace{2cm}\times \, \langle f, h_{u_1}^{Q_1} \otimes h_{u_2}^{Q_2}\rangle
\langle T(h_{u_1}^{Q_1} \otimes h_{u_2}^{Q_2}), h_{u'_1}^{Q'_1} \otimes 1 \rangle \noz\\
&+ \sum_{Q_1, Q'_1 \textup{ good} } \,
\sum_{Q_2 \textup{ good}}
\sum_{\substack{u_1,u'_1 \\u_2}}
\bigg\langle
\sum_{\substack{Q'_2 \\ \ell(Q'_2) = \delta^m}}
\langle g_{\text{good}}, h_{u'_1}^{Q'_1} \otimes h_{0}^{Q'_2}\rangle h_{0}^{Q'_2}
\bigg\rangle_{S(Q_2)} \noz \\
&\hspace{2cm}\times \, \langle f, h_{u_1}^{Q_1} \otimes h_{u_2}^{Q_2}\rangle
\langle T(h_{u_1}^{Q_1} \otimes h_{u_2}^{Q_2}), h_{u'_1}^{Q'_1} \otimes 1 \rangle.
\end{align}
Define
\[g_{\text{good}}^{Q'_1,u'_1}(y)
:= \Big\langle g_{\text{good}},h_{u'_1}^{Q'_1}\Big\rangle_1 (y)
:= \int_{X_1} g_{\text{good}}(x,y)h_{u'_1}^{Q'_1}(x) \,d\mu_1(x). \]
We write
\begin{align*}
  &\Big\langle g_{\text{good}}, h_{u'_1}^{Q'_1} \otimes h_{u'_2}^{Q'_2}\Big\rangle
  = \int_{X_2}g_{\text{good}}^{Q'_1,u'_1}(y) h_{u'_2}^{Q'_2}(y) \,d\mu_2(y)
  = \Big\langle g_{\text{good}}^{Q'_1,u'_1},h_{u'_2}^{Q'_2} \Big\rangle, \quad \text{and} \\
  &\Big\langle g_{\text{good}}, h_{u'_1}^{Q'_1} \otimes h_{0}^{Q'_2}\Big\rangle
  = \int_{X_2}g_{\text{good}}^{Q'_1,u'_1}(y) h_{0}^{Q'_2}(y) \,d\mu_2(y)
  = \Big\langle g_{\text{good}}^{Q'_1,u'_1},h_{0}^{Q'_2} \Big\rangle.
\end{align*}
Notice that
\begin{align*}
 &\hspace{-0.5cm}  \sum_{\substack{Q'_2 \\ \ell(Q'_2) \leq \delta^m}}\,
\sum_{u'_2} \langle g_{\text{good}}, h_{u'_1}^{Q'_1} \otimes h_{u'_2}^{Q'_2}\rangle h_{u'_2}^{Q'_2}
 + \sum_{\substack{Q'_2 \\ \ell(Q'_2) = \delta^m}}
\langle g_{\text{good}}, h_{u'_1}^{Q'_1} \otimes h_{0}^{Q'_2}\rangle h_{0}^{Q'_2}\\
&= \sum_{\substack{Q'_2 \\ \ell(Q'_2) \leq \delta^m}}\,
\sum_{u'_2} \Big\langle g_{\text{good}}^{Q'_1,u'_1},h_{u'_2}^{Q'_2} \Big\rangle h_{u'_2}^{Q'_2}
 + \sum_{\substack{Q'_2 \\ \ell(Q'_2) = \delta^m}}
\Big\langle g_{\text{good}}^{Q'_1,u'_1},h_{0}^{Q'_2} \Big\rangle h_{0}^{Q'_2}\\
& = g_{\text{good}}^{Q'_1,u'_1}.
\end{align*}
Hence, \eqref{eq8:sep_nest} equals
\begin{align*}
&\hspace{-0.5cm}\sum_{Q_1, Q'_1 \textup{ good} } \,
\sum_{Q_2 \textup{ good}}
\sum_{\substack{u_1,u'_1 \\u_2}}
\Big\langle
g_{\text{good}}^{Q'_1,u'_1}
\Big\rangle_{S(Q_2)} \langle f, h_{u_1}^{Q_1} \otimes h_{u_2}^{Q_2}\rangle
\langle T(h_{u_1}^{Q_1} \otimes h_{u_2}^{Q_2}), h_{u'_1}^{Q'_1} \otimes 1 \rangle \\
&= \, \sum_{Q_1, Q'_1 \textup{ good}} \,
\sum_{Q'_2} \,
\sum_{\substack{Q_2 \textup{ good} \\Q_2 \subset Q'_2 \\ \ell(Q_2) = \delta^r\ell(Q'_2)}} \,
\sum_{\substack{u_1,u'_1 \\u_2}} \,
\Big\langle  g_{\text{good}}^{Q'_1,u'_1} \Big\rangle_{Q'_2}
\langle f, h_{u_1}^{Q_1} \otimes h_{u_2}^{Q_2}\rangle
\langle T(h_{u_1}^{Q_1} \otimes h_{u_2}^{Q_2}), h_{u'_1}^{Q'_1} \otimes 1 \rangle .
\end{align*}

Define $f^{Q_1,u_1} := \langle f, h_{u_1}^{Q_1}\rangle_1$,
$b^{u_1 u'_1}_{Q_1 Q'_1} := \langle T^*(h_{u'
_1}^{Q'_1} \otimes 1), h_{u_1}^{Q_1}\rangle_1$ and
\[
\Pi_{a}^{u_2} \omega :=  \sum_{Q'_2 \in \Dd'_2} \quad
\sum_{\substack{Q_2 \textup{ good} \\Q_2 \subset Q'_2 \\ \ell(Q_2) = \delta^r\ell(Q'_2)}}
\langle \omega\rangle_{Q'_2}
\langle a, h_{u_2}^{Q_2}\rangle h_{u_2}^{Q_2}, \qquad u_2 \neq 0.
\]
We can manipulate each term of the summand as follows:
\begin{align*}
  \Big \langle g_{\text{good}}^{Q'_1,u'_1} \Big\rangle_{Q'_2}
  &= \intav_{Q'_2} \int_{X_1} g_{\text{good}}(x,y) h_{u'_1}^{Q'_1} (x) \,d\mu_1(x) \,d\mu_2(y) \\
  &=\int_{X_1} \intav_{Q'_2} g_{\text{good}}(x,y)\,d\mu_2(y) \,h_{u'_1}^{Q'_1} (x) \,d\mu_1(x)
  = \Big \langle h_{u'_1}^{Q'_1}, \langle g_{\text{good}} \rangle_{Q'_2} \Big\rangle, \\
  \langle f, h_{u_1}^{Q_1} \otimes h_{u_2}^{Q_2}\rangle
&= \Big\langle \langle f,h_{u_1}^{Q_1} \rangle_1, h_{u_2}^{Q_2}\Big\rangle
= \langle f^{Q_1,u_1}, h_{u_2}^{Q_2}\rangle,\quad \text{and}\\
\langle T(h_{u_1}^{Q_1} \otimes h_{u_2}^{Q_2}), h_{u'_1}^{Q'_1} \otimes 1 \rangle
&= \Big\langle \langle T^*( h_{u'_1}^{Q'_1} \otimes 1 ),h_{u_1}^{Q_1} \rangle_1, h_{u_2}^{Q_2}\Big\rangle
= \Big\langle b^{u_1 u'_1}_{Q_1 Q'_1}, h_{u_2}^{Q_2}\Big\rangle.
\end{align*}

Thus
\begin{eqnarray*}
 \lefteqn{ \Big\langle  g_{\text{good}}^{Q'_1,u'_1} \Big\rangle_{Q'_2}
\langle f, h_{u_1}^{Q_1} \otimes h_{u_2}^{Q_2}\rangle
\langle T(h_{u_1}^{Q_1} \otimes h_{u_2}^{Q_2}), h_{u'_1}^{Q'_1} \otimes 1 \rangle }\\
&=& \Big \langle h_{u'_1}^{Q'_1}, \langle g_{\text{good}} \rangle_{Q'_2} \Big\rangle
\langle f^{Q_1,u_1}, h_{u_2}^{Q_2}\rangle
\Big\langle b^{u_1 u'_1}_{Q_1 Q'_1}, h_{u_2}^{Q_2}\Big\rangle \\
&=& \bigg\langle h_{u'_1}^{Q'_1} \otimes  f^{Q_1,u_1},
\Big\langle  g_{\text{good}}\Big\rangle_{Q'_2}
\Big\langle b^{u_1 u'_1}_{Q_1 Q'_1}, h_{u_2}^{Q_2}\Big\rangle
h_{u_2}^{Q_2}
\bigg\rangle.
\end{eqnarray*}

So we are left with
\begin{eqnarray*}
&&
\sum_{Q_1, Q'_1 \textup{ good}} \quad
\sum_{\substack{u_1,u'_1 \\u_2}} \quad
\bigg\langle h_{u'_1}^{Q'_1} \otimes  f^{Q_1,u_1},
\sum_{Q'_2} \quad
\sum_{\substack{Q_2 \textup{ good} \\Q_2 \subset Q'_2 \\ \ell(Q_2) = \delta^r\ell(Q'_2)}}
\Big\langle  g_{\text{good}}\Big\rangle_{Q'_2}
\Big\langle b^{u_1 u'_1}_{Q_1 Q'_1}, h_{u_2}^{Q_2}\Big\rangle
h_{u_2}^{Q_2}
\bigg\rangle\\
  &=&
\sum_{Q_1, Q'_1 \textup{ good}} \quad
\sum_{\substack{u_1,u'_1 \\u_2}} \quad
\bigg\langle h_{u'_1}^{Q'_1} \otimes  f^{Q_1,u_1},
\Pi_{b^{u_1 u'_1}_{Q_1 Q'_1}}^{u_2} g_{\text{good}}
\bigg\rangle \\
&=&\bigg\langle
\sum_{Q_1, Q'_1 \textup{ good}} \quad
\sum_{\substack{u_1,u'_1 \\u_2}} \quad
 h_{u'_1}^{Q'_1} \otimes  \big(\Pi_{b^{u_1 u'_1}_{Q_1 Q'_1}}^{u_2} \big)^*f^{Q_1,u_1},
 g_{\text{good}}
\bigg\rangle.
\end{eqnarray*}
This completes the proof of Lemma~\ref{lem3.sepnest}.
\end{proof}

Now using Lemma~\ref{lem3.sepnest}, the sum~\eqref{sepnesteq.3} is given by
\[\bigg\langle
\sum_{\substack{Q_1, Q'_1 \textup{ good}\\ \textup{separated} }} \quad
\sum_{\substack{u_1,u'_1 \\u_2}} \quad
 h_{u'_1}^{Q'_1} \otimes  \big(\Pi_{b^{u_1 u'_1}_{Q_1 Q'_1}}^{u_2} \big)^*f^{Q_1,u_1},
 g_{\text{good}}
\bigg\rangle.\]
There holds
\begin{eqnarray*}
  \lefteqn{\bigg|\bigg\langle
\sum_{\substack{Q_1, Q'_1 \textup{ good}\\ \textup{separated} }} \quad
\sum_{\substack{u_1,u'_1 \\u_2}} \quad
 h_{u'_1}^{Q'_1} \otimes  \big(\Pi_{b^{u_1 u'_1}_{Q_1 Q'_1}}^{u_2} \big)^*f^{Q_1,u_1},
 g_{\text{good}}
\bigg\rangle\bigg|}\\
   &\leq&
   \bigg\| \sum_{\substack{Q_1, Q'_1 \textup{ good}\\ \textup{separated} }} \quad
   \sum_{\substack{u_1,u'_1 \\u_2}} \quad
 h_{u'_1}^{Q'_1} \otimes  \big(\Pi_{b^{u_1 u'_1}_{Q_1 Q'_1}}^{u_2} \big)^*f^{Q_1,u_1} \bigg\|_{L^2(\mu)} \|g_{\text{good}}\|_{L^2(\mu)} \\
 &\leq&
 \bigg(\sum_{Q'_1 \textup{ good}}
  \bigg\|
  \sum_{\substack{Q_1 \textup{ good}\\ Q_1, Q'_1\textup{separated} }} \quad
  \sum_{\substack{u_1,u'_1 \\u_2}} \quad
  \big(\Pi_{b^{u_1 u'_1}_{Q_1 Q'_1}}^{u_2} \big)^*f^{Q_1,u_1} \bigg\|_{L^2(\mu_2)}^2 \bigg)^{1/2}
  \|g\|_{L^2(\mu)} \\
  &\ls&
  \bigg(\sum_{Q'_1} \bigg[
  \sum_{\substack{Q_1 \textup{ good}\\Q_1, Q'_1 \textup{separated} }} \quad
  \sum_{u_1,u'_1} \quad
  \| b^{u_1 u'_1}_{Q_1 Q'_1}\|_{\bmo^2_{C_K}(\mu_2)}
  \|f^{Q_1,u_1}\|_{L^2(\mu_2)}
   \bigg]^2 \bigg)^{1/2}
   \|g\|_{L^2(\mu)}.
\end{eqnarray*}
We recall Definition~\ref{defn: BMO_k_p} of function spaces~$\bmo^p_{\kappa}$. Here~$C_K >1$ is fixed in Section~\ref{subsec:assume} (see Assumptions~\ref{assum_2} and~\ref{assum_4}).
These inequalities hold due to
 the orthonormality of the functions $h_{u'_1}^{Q'_1} \in L^2(\mu_1)$ and the $L^2$-boundedness of the paraproduct~$\Pi_a^{\kappa}$, stated in Lemma~\ref{lem4.sepnest} below.

\begin{lem}\label{lem4.sepnest}
  For any $M >1$, we have
  \[\|\Pi_a^{\kappa}\|_{L^2(\mu_2) \rightarrow L^2(\mu_2)} \ls \|a\|_{\bmo_M^2(\mu_2)}.\]
\end{lem}
\begin{proof}
  For $\omega \in L^2(\mu_2)$, we have
  \begin{align*}
    \|\Pi_a^{\kappa} \omega\|_{L^2(\mu_2)}
    = &  \sum_{Q'_2 \in \Dd'_2} \quad |\langle \omega\rangle_{Q'_2}|^2
\sum_{\substack{Q_2 \textup{ good} \\Q_2 \subset Q'_2 \\ \ell(Q_2) = \delta^r\ell(Q'_2)}}
|\langle a, h_{\kappa}^{Q_2}\rangle|^2.
  \end{align*}

Given a cube $Q \in \Dd'_2$, set
$$ \alpha_{Q}:={1\over \mu_2(Q)}\sum_{\substack{Q_2 \textup{ good} \\Q_2 \subset Q \\ \ell(Q_2) = \delta^r\ell(Q)}}
|\langle a, h_{\kappa}^{Q_2}\rangle|^2.$$
We claim that for each fixed cube~$Q'_2 \in \Dd'_2$ we have
\begin{equation}\label{new claim paraproduct}
\sum_{Q\in \Dd'_2, Q\subset Q'_2} \alpha_Q  \mu_2(Q) \lesssim \mu_2(Q'_2) \|a\|^2_{\bmo_M^2(\mu_2)}.
\end{equation}
Then by using the Carleson's embedding theorem~\cite[Theorem~5.2]{HM12a}, we get that
  \begin{align*}
    \|\Pi_a^{\kappa} \omega\|_{L^2(\mu_2)}
    = &  \sum_{Q'_2 \in \Dd'_2} \quad |\langle \omega\rangle_{Q'_2}|^2 \ \alpha_{Q'_2}\ \mu_2(Q'_2)\lesssim
    \|a\|^2_{\bmo_M^2(\mu_2)}  \|\omega\|_{L^2(\mu_2)}^2.
  \end{align*}
Hence, Lemma~\ref{lem4.sepnest} is established.

Now we prove~\eqref{new claim paraproduct}. We consider the left-hand side of~\eqref{new claim paraproduct}
\begin{equation}\label{new claim paraproduct 1}
\sum_{\substack{Q\in \Dd'_2\\ Q\subset Q'_2}} \alpha_Q  \mu_2(Q)
= \sum_{\substack{Q\in \Dd'_2\\ Q\subset Q'_2}}
\sum_{\substack{Q_2 \textup{ good} \\Q_2 \subset Q \\ \ell(Q_2) = \delta^r\ell(Q)}}
|\langle a, h_{\kappa}^{Q_2}\rangle|^2
\leq \sum_{\substack{Q_2 \textup{ good} \\Q_2 \subset Q'_2 \\ \ell(Q_2) \leq \delta^r\ell(Q'_2)}}
|\langle a, h_{\kappa}^{Q_2}\rangle|^2.
\end{equation}
Moreover, using the facts that $\int h_{\kappa}^{Q_2} \,d\mu_2 = 0$,
$\|h_{\kappa}^{Q_2} \|_{L^2(\mu_2)} \sim 1$, $\supp h_{\kappa}^{Q_2} \subset Q_2$, $Q_2 \subset B(Q_2):= B(z_{Q_2}, C_Q\ell(Q_2))$ and $a \in \bmo_M^2(\mu_2)$ we have
\begin{align}\label{new claim paraproduct 2}
  |\langle a, h_{\kappa}^{Q_2}\rangle|^2
  & =  |\langle (a-a_{B(Q_2)}), h_{\kappa}^{Q_2}\rangle|^2
   =  |\langle (a-a_{B(Q_2)})\chi_{B(Q_2)}, h_{\kappa}^{Q_2}\rangle|^2 \noz\\
  &\leq  \|(a-a_{B(Q_2)})\chi_{B(Q_2)}\|^2_{L^2(\mu_2)}
  \| h_{\kappa}^{Q_2}\|^2_{L^2(\mu_2)} \noz\\
  &\ls  \int_{B(Q_2)} |a-a_{B(Q_2)}|^2 \,d\mu_2 \noz\\
  & \ls  \mu_2(M B(Q_2)) \|a\|^2_{\bmo^2_M(\mu_2)}.
\end{align}
Inequalities~\eqref{new claim paraproduct 1} and~\eqref{new claim paraproduct 2} give
\[
\sum_{\substack{Q\in \Dd'_2\\ Q\subset Q'_2}} \alpha_Q  \mu_2(Q)
\ls \sum_{\substack{Q_2 \textup{ good} \\Q_2 \subset Q'_2 \\ \ell(Q_2) \leq \delta^r\ell(Q'_2)}}
 \mu_2(M B(Q_2)) \|a\|^2_{\bmo^2_M(\mu_2)}.
\]
To show~\eqref{new claim paraproduct}, it is sufficient to show that every point~$q \in Q'_2$ is covered by at most~$N$ balls~$MB(Q_2)$ for~$N$ independent of~$Q_2$, as it implies
\[
\sum_{\substack{Q_2 \textup{ good} \\Q_2 \subset Q'_2 \\ \ell(Q_2) \leq \delta^r\ell(Q'_2)}}
 \mu_2(M B(Q_2)) \leq N\mu_2(Q'_2).
\]
To prove this, we adapt a part of the proof of \cite[Lemma~7.1]{HM12a}, where they prove the same result in the setting of metric spaces~$(X,d)$.
For convenience, we note the differences in our notations in Table~\ref{tab:notation1}.
\begin{table}[H]
\caption{ Differences between our notation and that in~\cite{HM12a}.}\label{tab:notation1}
\begin{center}
\begin{tabular}{|c| c|}
  \hline
    Our notation & Notation in~\cite{HM12a} \\ \hline
   $Q'_2$ & $Q$  \\ \hline
   $Q_2$ & $R$,\\ \hline
   $C_Q$ & $C_0$ \\ \hline
   $c_Q$ & $C_1$ \\ \hline
   $M$ & $\kappa$ \\ \hline
\end{tabular}
\end{center}
\end{table}

The proof in~\cite{HM12a} goes through almost unchanged.
The idea is that we first show that every~$x \in Q'_2$ can belong to only $\ls 1$ balls~$MB(Q_2)$ associated with a fixed generation~$k \geq \text{gen}(Q'_2) + r$.
We then show that
if $MB(Q_2^k) \cap MB(Q_2^l)$, where $k = \text{gen}(Q_2^k)$ and $l = \text{gen}(Q_2^l)$, then $|k-l| \ls 1$.

The only difference is that, instead of obtaining that if~$R$ is a chosen cube for which $\text{gen}(R) = k > \text{gen}(Q) + r$, then for all~$x \in \kappa B_R$ we have
\[C_0\kappa\delta^k \leq d(x,X\backslash Q) \leq 3CC_0\kappa \delta^{k-1};\]
we will get that if~$Q_2$ is a chosen cube for which $\text{gen}(Q_2) = k > \text{gen}(Q'_2) + r$, then for all~$x \in MB(Q_2)$ we have
\[A_0 C_Q M \delta^k \leq \rho_2(x, X_2 \backslash Q'_2) \leq 4A_0^2 C_K C_QM\delta^{k^-1},\]
where $A_0$ is the quasitriangle constant of~$\rho_0$ and $C_K$ is fixed in the kernel estimates in Assumptions~\ref{assum_2} and~\ref{assum_4}.
This completes the proof of Lemma~\ref{lem4.sepnest}.
\end{proof}
The $\bmo$ norms above are estimated as follows.
\begin{lem}\label{lem5.sepnest}
   Suppose $Q_1$, $Q'_1$ are separated, meaning $\ell(Q_1) \leq \ell(Q'_1)$ and $\rho_1(Q_1,Q'_1) > \mathcal{C}\ell(Q_1)^{\gamma_1}\ell(Q'_1)^{1-\gamma_1}$.
  Then there holds that
  \[\| b^{u_1 u'_1}_{Q_1 Q'_1}\|_{\bmo^2_{C_K}(\mu_2)} \ls A_{Q_1,Q'_1}^{\textup{sep}}, \quad \text{where } C_K >1.\]
\end{lem}
\begin{proof}
  By duality, it is sufficient to show $|\langle b^{u_1 u'_1}_{Q_1 Q'_1}, a\rangle| \ls A_{Q_1,Q'_1}^{\textup{sep}}$ for all $2$-atom~$a$ in Hardy space~$H^1(\mu_2)$.
  Fix a ball $J = B(x_J,r_J) \subset X_2$ and an atom~$a$ such that $\supp a \subset J$, $\int_{J} a \,d\mu_2 = 0$ and $\|a\|_{L^2(\mu_1)} \leq \mu_1(J)^{-1/2}$ .
  Consider
  \begin{align*}
  \Big \langle b^{u_1 u'_1}_{Q_1 Q'_1}, a\Big\rangle
    =  \Big \langle  \langle T^*(h_{u'_1}^{Q'_1} \otimes 1), h_{u_1}^{Q_1}\rangle_1, a\Big\rangle
    =  \Big \langle   T^*(h_{u'_1}^{Q'_1} \otimes 1), h_{u_1}^{Q_1} \otimes a\Big\rangle
    =  \Big \langle   T( h_{u_1}^{Q_1} \otimes a),h_{u'_1}^{Q'_1} \otimes 1\Big\rangle.
    \end{align*}
  Let~$V := C_KJ = B(x_J,C_Kr_J)$.
  We need to show that
\begin{align}\label{eq4.sepnest}
    \Big|\Big \langle   T( h_{u_1}^{Q_1} \otimes a),h_{u'_1}^{Q'_1} \otimes 1\Big\rangle\Big|
     &\leq \Big|\Big \langle   T( h_{u_1}^{Q_1} \otimes a),h_{u'_1}^{Q'_1} \otimes \chi_V\Big\rangle\Big|
     + \Big|\Big \langle   T( h_{u_1}^{Q_1} \otimes a),h_{u'_1}^{Q'_1} \otimes \chi_{V^c}\Big\rangle\Big|\\
     &\ls  A_{Q_1,Q'_1}^{\textup{sep}}. \noz
  \end{align}
We will consider each term in~\eqref{eq4.sepnest}, individually.

Since~$Q_1$ and~$Q'_1$ are separated, inequality~\eqref{eq2.sepsep} holds for all $y_1,z \in Q_1$ and $x_1 \in Q'_1$.
Applying Lemma~\ref{lem:property_kernel} to $\phi_1 =  h_{u_1}^{Q_1}$, $\phi_2 = a$, $\theta_1= h_{u'_1}^{Q'_1} $ and $\theta_2 = \chi_V$ we have
\begin{align*}
    \Big|\Big \langle   T( h_{u_1}^{Q_1} \otimes a),h_{u'_1}^{Q'_1} \otimes \chi_V\Big\rangle\Big|
     &\ls \|h_{u_1}^{Q_1}\|_{L^2(\mu_1)} \|a\|_{L^2(\mu_2)}
     \mu_1(Q_1)^{1/2} \|\chi_V\|_{L^2(\mu_2)}\\
     & \times \int_{Q'_1}\frac{C_K^{\al_1} \ell(Q_1)^{\al_1}}
     {\rho_1(x_1,y_1)^{\al_1}\lambda_1(x_1,\rho_1(x_1,y_1))}
     |h_{u'_1}^{Q'_1}(x_1) |
     \,d\mu_1(x_1).
  \end{align*}
Using inequality~\eqref{eq3:sepsep} in Lemma~\ref{lem1:sepsep}, the integral above is bounded by
\[
\frac{\ell(Q_1)^{\al_1/2}\ell(Q'_1)^{\al_1/2}}{D(Q_1,Q'_1)^{\al_1}
\lambda_1(z,D(Q_1,Q'_1))}
\mu_1(Q'_1)^{1/2}
= \frac{A_{Q_1,Q'_1}^{\textup{sep}}}{\mu_1(Q_1)^{1/2}}.
\]
Therefore,
  \begin{eqnarray*}
    \Big|\Big \langle   T( h_{u_1}^{Q_1} \otimes a),h_{u'_1}^{Q'_1} \otimes \chi_V\Big\rangle\Big|
     \ls \|a\|_{L^2(\mu_2)} \mu_2(V)^{1/2}
      \mu_1(Q_1)^{1/2}
      \frac{A_{Q_1,Q'_1}^{\textup{sep}}}{\mu_1(Q_1)^{1/2}}
\ls A_{Q_1,Q'_1}^{\textup{sep}}..
  \end{eqnarray*}

  Next we consider the second term. For all $y_2, x_J \in J$ and $x_2 \in V^c$, we have
  \begin{equation}\label{eq5.sepnest}
    \rho_2(x_2,x_J) \geq C_K r_J \geq C_K \rho_2(x_J,y_2).
  \end{equation}
  Since~$Q_1\cap Q'_1 = \emptyset$, and $\supp a \cap \supp \chi_{V^c} = \emptyset$, we can follow a similar proof structure to that in Lemma~\ref{lem1:sepnest}. That is,  applying Lemma~\ref{lem1:property_kernel} to
$\phi_1 = h^{Q_1}_{u_1}$, $\phi_2 = a$, $\theta_1 = h^{Q'_1}_{u'_1}$ and $\theta_2 = \chi_{V^c}$ we have
\begin{align}\label{eq2:SepIn_c}
  &|\langle T(h_{u_1}^{Q_1} \otimes a), h_{u'_1}^{Q'_1} \otimes \chi_{V^c}\rangle| \noz\\
   &\hspace{0.5cm}\ls \|h^{Q_1}_{u_1}\|_{L^2(\mu_1)}  \|a\|_{L^2(\mu_2)}
   \mu_1(Q_1)^{1/2} \mu_2(J)^{1/2}
    \int_{Q'_1}    \frac{C_Q^{\al_1}\ell(Q_1)^{\al_1}}
    {\rho_1(x_1,z)^{\al_1}\lambda_1(z,\rho_1(x_1,z))} |h^{Q'_1}_{u'_1}(x_1)| \,d\mu_1(x_1) \noz\\
    & \hspace{1cm}\times
  \int_{V^c} \frac{r_J^{\al_2}}
  {\rho_2(x_2,x_J)^{\al_2}\lambda_2(x_J,\rho_2(x_2,x_J))} \,d\mu_2(x_2).
\end{align}
 Using inequality~\eqref{eq3:sepsep} of Lemma~\ref{lem1:sepsep}, the first integral above is bounded by
 \[\frac{\ell(Q_1)^{\al_1/2}\ell(Q'_1)^{\al_1/2}}{D(Q_1,Q'_1)^{\al_1}
\lambda_1(z,D(Q_1,Q'_1))} \mu_1(Q'_1)^{1/2}
\ls \frac{A^{\text{sep}}_{Q_1,Q'_1}}{\mu_1(Q_1)^{1/2}}.\]
 Using Lemma~\ref{upper_dbl_lem1},  the second integral above is bounded by
\begin{eqnarray*}
 r_J^{\al_2}
 \int_{X_2 \backslash B(x_J, C_K r_J)}  \frac{\rho_2(x_2,x_J)^{-\al_2}}{\lambda_2(x_J,\rho_2(x,x_J))} \,d\mu_2(x_2) \ls r^{\al_2} (C_K r_J)^{-\al_2}
 \ls 1.
\end{eqnarray*}
Thus, \eqref{eq2:SepIn_c} becomes
\begin{eqnarray*}
 |\langle T(h_{u_1}^{Q_1} \otimes h_{u_2}^{Q_2}), h_{u'_1}^{Q'_1} \otimes \chi_{V^c}\rangle|
   \ls    \|a\|_{L^2(\mu_2)}
   \mu_1(Q_1)^{1/2} \mu_2(J)^{1/2}
    \frac{A^{\text{sep}}_{Q_1,Q'_1}}{\mu_1(Q_1)^{1/2}}
    = A^{\textup{sep}}_{Q_1,Q'_1}.
\end{eqnarray*}
This establishes Lemma~\ref{lem5.sepnest}.
\end{proof}

Now, using Lemma~\ref{lem5.sepnest} and Proposition~\ref{sepseppop.1} we go back to considering
\begin{eqnarray*}
  \lefteqn{  \sum_{Q'_1} \bigg[
  \sum_{\substack{Q_1\textup{ good}\\Q_1, Q'_1 \textup{separated} }} \,
  \sum_{u_1,u'_1}
  \| b^{u_1 u'_1}_{Q_1 Q'_1}\|_{\bmo^2_{C_K}(\mu_2)}
  \|f^{Q_1,u_1}\|_{L^2(\mu_2)}
   \bigg]^2 }\hspace{0.5cm}\\
   &\ls&   \sum_{Q'_1} \bigg[
  \sum_{\substack{Q_1 \textup{ good}\\Q_1, Q'_1 \textup{separated} }} \,
  \sum_{u_1}
  A_{Q_1,Q'_1}^{\textup{sep}}
  \|f^{Q_1,u_1}\|_{L^2(\mu_2)}
   \bigg]^2
   \ls \sum_{Q_1} \sum_{u_1}   \|f^{Q_1,u_1}\|^2_{L^2(\mu_2)} \\
   &=& \sum_{Q_1} \sum_{u_1}  \int_{X_2} \Big|\int_{X_1} f(x_1,x_2) h_{u_1}^{Q_1}(x_1) \,d\mu_1(x_1) \Big|^2 \,d\mu_2(x_2)\\
   &=& \int_{X_2} \sum_{Q_1} \sum_{u_1}   |\langle f(\cdot, x_2),h_{u_1}^{Q_1} \rangle|^2 \,d\mu_2(x_2) \\
   &=& \int_{X_2} \|f(\cdot,x_2)\|_{L^2(\mu_1)}^2 \,d\mu_2(x_2)
   = \int_{X_2}\int_{X_1} |f(x_1,x_2)|^2 \,d\mu = \|f\|_{L^2(\mu)}^2.
\end{eqnarray*}
This lets us control the first sum in~\eqref{sepnesteq.3} by $\|f\|_{L^2(\mu)} \|g\|_{L^2(\mu)}$. Other sum in~\eqref{sepnesteq.3} can be controlled in the same way, completing the \emph{Sep/Nes} case.
\subsection{Separated/Adjacent cubes}\label{subsec:sep_adj}
In this section, we consider the \emph{Sep/Adj} case. That is, we estimate~\eqref{coreeq.1} when $Q_1, Q'_1$ are separated and $Q_2, Q'_2$ are adjacent:
\begin{eqnarray*}
  &&\ell(Q_1) \leq \ell(Q'_1), \quad \rho_1(Q_1,Q'_1) > \mathcal{C}\ell(Q_1)^{\gamma_1}\ell(Q'_1)^{1-\gamma_1}, \quad \text{and} \\
  &&\delta^r \ell(Q'_2) \leq \ell(Q_2) \leq \ell(Q'_2), \quad \rho_2(Q_2,Q'_2) \leq \mathcal{C}\ell(Q_2)^{\gamma_2}\ell(Q'_2)^{1-\gamma_2},
\end{eqnarray*}
where $\mathcal{C}:= 2A_0C_QC_K$.
With the same reason as in the \emph{Sep/Sep} case, the terms in~\eqref{coreeq.1} which involves~$h_0^{Q_1}$ vanish.
Therefore, we only need to bound eight (out of 16 terms) in~\eqref{coreeq.1}:
\begin{align}\label{sepadj_eq.1}
 &\hspace{-0.5cm}\sum_{\substack{Q_1,Q'_1  \\ \text{separated} }} \,
  \sum_{\substack{Q_2,Q'_2\\ \text{adjacent} }} \,
  \sum_{\substack{u_1,u'_1 \\u_2,u'_2}} \,
  \langle f, h_{u_1}^{Q_1} \otimes h_{u_2}^{Q_2}\rangle
  \langle g, h_{u'_1}^{Q'_1} \otimes h_{u'_2}^{Q'_2}\rangle
  \langle T(h_{u_1}^{Q_1} \otimes h_{u_2}^{Q_2}), h_{u'_1}^{Q'_1} \otimes h_{u'_2}^{Q'_2}\rangle \noz\\
  &+ \sum_{\substack{Q_1,Q'_1  \\ \text{separated} }} \,
  \sum_{\substack{Q_2,Q'_2 \\ \text{adjacent} \\ \text{gen}(Q'_2) = m}} \,
  \sum_{\substack{u_1,u'_1 \\u_2}} \,
  \langle f, h_{u_1}^{Q_1} \otimes h_{u_2}^{Q_2}\rangle
  \langle g, h_{u'_1}^{Q'_1} \otimes h_{0}^{Q'_2}\rangle
  \langle T(h_{u_1}^{Q_1} \otimes h_{u_2}^{Q_2}), h_{u'_1}^{Q'_1} \otimes h_{0}^{Q'_2}\rangle\noz\\
  &+  \sum_{\substack{Q_1,Q'_1  \\ \text{separated} \\ \text{gen}(Q'_1) = m }} \,
  \sum_{\substack{Q_2,Q'_2\\ \text{adjacent}  }} \,
  \sum_{\substack{u_1\\u_2,u'_2}} \,
  \langle f, h_{u_1}^{Q_1} \otimes h_{u_2}^{Q_2}\rangle
  \langle g, h_{0}^{Q'_1} \otimes h_{u'_2}^{Q'_2}\rangle
  \langle T(h_{u_1}^{Q_1} \otimes h_{u_2}^{Q_2}), h_{0}^{Q'_1} \otimes h_{u'_2}^{Q'_2}\rangle \noz\\
  &+  \sum_{\substack{Q_1,Q'_1  \\ \text{separated} \\ \text{gen}(Q'_1) = m }} \,
  \sum_{\substack{Q_2,Q'_2\\ \text{adjacent} \\ \text{gen}(Q'_2) = m }} \,
  \sum_{\substack{u_1 \\u_2}} \,
  \langle f, h_{u_1}^{Q_1} \otimes h_{u_2}^{Q_2}\rangle
  \langle g, h_{0}^{Q'_1} \otimes h_{0}^{Q'_2}\rangle
  \langle T(h_{u_1}^{Q_1} \otimes h_{u_2}^{Q_2}), h_{0}^{Q'_1} \otimes h_{0}^{Q'_2}\rangle \noz\\
  &+ \sum_{\substack{Q_1,Q'_1  \\ \text{separated} }} \,
  \sum_{\substack{Q_2,Q'_2\\ \text{adjacent} }} \,
  \sum_{\substack{u_1,u'_1 \\u'_2}} \,
  \langle f, h_{u_1}^{Q_1} \otimes h_{0}^{Q_2}\rangle
  \langle g, h_{u'_1}^{Q'_1} \otimes h_{u'_2}^{Q'_2}\rangle
  \langle T(h_{u_1}^{Q_1} \otimes h_{0}^{Q_2}), h_{u'_1}^{Q'_1} \otimes h_{u'_2}^{Q'_2}\rangle \noz\\
  &+ \sum_{\substack{Q_1,Q'_1  \\ \text{separated} }} \,
  \sum_{\substack{Q_2,Q'_2 \\ \text{adjacent} \\ \text{gen}(Q_2) = \text{gen}(Q'_2) = m}} \,
  \sum_{u_1,u'_1} \,
  \langle f, h_{u_1}^{Q_1} \otimes h_{0}^{Q_2}\rangle
  \langle g, h_{u'_1}^{Q'_1} \otimes h_{0}^{Q'_2}\rangle
  \langle T(h_{u_1}^{Q_1} \otimes h_{0}^{Q_2}), h_{u'_1}^{Q'_1} \otimes h_{0}^{Q'_2}\rangle \noz\\
  &+   \sum_{\substack{Q_1,Q'_1  \\ \text{separated} \\ \text{gen}(Q'_1) = m }} \,
  \sum_{\substack{Q_2,Q'_2\\ \text{adjacent} \\\text{gen}(Q_2) = m  }} \,
  \sum_{\substack{u_1\\u'_2}} \,
  \langle f, h_{u_1}^{Q_1} \otimes h_{0}^{Q_2}\rangle
  \langle g, h_{0}^{Q'_1} \otimes h_{u'_2}^{Q'_2}\rangle
  \langle T(h_{u_1}^{Q_1} \otimes h_{0}^{Q_2}), h_{0}^{Q'_1} \otimes h_{u'_2}^{Q'_2}\rangle\noz\\
  &+  \sum_{\substack{Q_1,Q'_1  \\ \text{separated} \\ \text{gen}(Q'_1) = m }} \,
  \sum_{\substack{Q_2,Q'_2\\ \text{adjacent} \\ \text{gen}(Q_2) = \text{gen}(Q'_2) = m }} \,
  \sum_{u_1} \,
  \langle f, h_{u_1}^{Q_1} \otimes h_{0}^{Q_2}\rangle
  \langle g, h_{0}^{Q'_1} \otimes h_{0}^{Q'_2}\rangle
  \langle T(h_{u_1}^{Q_1} \otimes h_{0}^{Q_2}), h_{0}^{Q'_1} \otimes h_{0}^{Q'_2}\rangle.
\end{align}
We will consider the first term in~\eqref{sepadj_eq.1}. The other terms can be estimated analogously.

Again, as $Q_1$ and~$Q'_1$ are separated, inequality~\eqref{eq2.sepsep} holds for all $y_1,z \in Q_1$ and~$x_1 \in Q'_1$. Applying Lemma~\ref{lem:property_kernel} to
$\phi_1 = h_{u_1}^{Q_1}$, $\phi_2 = h_{u_2}^{Q_2}$, $\theta_1 = h_{u'_1}^{Q'_1}$ and $\theta_2 = h_{u'_2}^{Q'_2}$ we obtain
  \begin{eqnarray*}
    |\langle T(h_{u_1}^{Q_1} \otimes h_{u_2}^{Q_2}), h_{u'_1}^{Q'_1} \otimes h_{u'_2}^{Q'_2}\rangle|
     &\ls& \|h_{u_1}^{Q_1}\|_{L^2(\mu_1)} \|h_{u_2}^{Q_2}\|_{L^2(\mu_2)}  \mu_1(Q_1)^{1/2} \|h_{u'_2}^{Q'_2}\|_{L^2(\mu_2)} \\
     &&\times \int_{Q'_1} \frac{ C_Q^{\al_1}\ell(Q_1)^{\al_1}}{ \rho_1(x_1,z)^{\al_1}\lambda_1(z,\rho_1(x_1,z))}
     |h_{u'_1}^{Q'_1}(x_1)|\, d\mu_1(x_1).
  \end{eqnarray*}
 Using equation~\eqref{eq3:sepsep} of Lemma~\ref{lem1:sepsep}, the integral above is bounded by
 \[\frac{\ell(Q_1)^{\al_1/2}\ell(Q'_1)^{\al_1/2}}{D(Q_1,Q'_1)^{\al_1}
\lambda_1(z,D(Q_1,Q'_1))} \mu_1(Q'_1)^{1/2}
\ls \frac{A^{\text{sep}}_{Q_1,Q'_1}}{\mu_1(Q_1)^{1/2}}.\]
Therefore
\[
|\langle T(h_{u_1}^{Q_1} \otimes h_{u_2}^{Q_2}), h_{u'_1}^{Q'_1} \otimes h_{u'_2}^{Q'_2}\rangle|
\ls A^{\text{sep}}_{Q_1,Q'_1}.
\]

Note that for a given~$Q_2$, there are finitely many~$Q'_2$ such that
\[\delta^r \ell(Q'_2) \leq \ell(Q_2) \leq \ell(Q'_2), \quad \rho_2(Q_2,Q'_2) \leq \mathcal{C}\ell(Q_2)^{\gamma_2}\ell(Q'_2)^{1-\gamma_2}.\]
Using Proposition~\ref{sepseppop.1} we can now estimate the first term in~\eqref{sepadj_eq.1} by
\begin{eqnarray*}
\lefteqn{
  \sum_{\substack{Q_1,Q'_1  \\ \text{separated} }} \quad
  \sum_{\substack{Q_2,Q'_2\\ \text{adjacent} }} \quad
  \sum_{\substack{u_1,u'_1 \\u_2,u'_2}}
  A_{Q_1,Q'_1}^{\textup{sep}}
  |\langle f, h_{u_1}^{Q_1} \otimes h_{u_2}^{Q_2}\rangle|
  |\langle g, h_{u'_1}^{Q'_1} \otimes h_{u'_2}^{Q'_2}\rangle|}\\
  &\leq& \sum_{\substack{u_1,u'_1 \\u_2,u'_2}} \quad
  \sum_{\ell(Q_1) \leq \ell(Q'_1)}
  A_{Q_1,Q'_1}^{\textup{sep}}
  \bigg(\sum_{Q_2} |\langle  f, h_{u_1}^{Q_1} \otimes h_{u_2}^{Q_2} \rangle|^2 \bigg)^{1/2}
  \bigg(\sum_{Q'_2}|\langle g, h_{u'_1}^{Q'_1} \otimes h_{u'_2}^{Q'_2}\rangle|^2  \bigg)^{1/2} \\
  &\ls&
   \bigg(\sum_{Q_1}\sum_{u_1}\sum_{Q_2}\sum_{u_2} |\langle  f, h_{u_1}^{Q_1} \otimes h_{u_2}^{Q_2} \rangle|^2 \bigg)^{1/2}
  \bigg(\sum_{Q'_1}\sum_{u'_1}\sum_{Q'_2}\sum_{u'_2}|\langle g, h_{u'_1}^{Q'_1} \otimes h_{u'_2}^{Q'_2}\rangle|^2  \bigg)^{1/2} \\
  &\ls& \|f\|_{L^2(\mu)} \|g\|_{L^2(\mu)}.
\end{eqnarray*}
This completes the \emph{Sep/Adj} case.
\section{Adjacent Cubes}\label{sec:adjacent}
In this section, we handle the cases involving adjacent cubes.
In Section~\ref{subsec:surgery}, we present the so-called \emph{surgery} for adjacent cubes, which lets us split the cubes into smaller parts, which are easier to deal with.
In Section~\ref{subsec:ball_cover}, we recall the construction of the \emph{random almost-covering by balls}, which will further allow us to decompose the intersection~$\Delta$ of two dyadic cubes from different dyadic grids into smaller parts. This construction is first introduced in~\cite[Section~8]{HM12a}.
In Section~\ref{subsec:adj_in}, we consider the \emph{Adj/Nes} case when $Q_1, Q'_1$ are adjacent and $Q_2, Q'_2$ are nested.
In Section~\ref{subsec:adj_adj},  we consider the \emph{Adj/Adj} case when $Q_1, Q'_1$ are adjacent and $Q_2, Q'_2$ are adjacent.

\subsection{Surgery for adjacent cubes}\label{subsec:surgery}
Handling the cases involving adjacent cubes requires a further splitting, namely the surgery for adjacent cubes. This idea was also used by other authors who work on non-homogeneous settings. See for example \cite{NTV03, HM12a, HM14}.

The setting of this section is a non-homogeneous (one-parameter) quasimetric spaces~$(X,\rho,\mu)$.
We can think of~$X$ as the factor~$X_1$ or~$X_2$ of our product space.
We drop the subscript in all symbols for simplicity.
In this section, we present the surgery of adjacent cubes on $(X,\rho,\mu)$. This surgery is derived from~\cite[Section~9]{HM12a}.
Later, we will apply this surgery to the cubes in each factor~$X_1$ and~$X_2$ of the bi-parameter space~$\mathcal{X}$.

Let~$\Dd, \Dd'$ be two independent random dyadic grids in~$X$.
Let~$Q \in \Dd$ and~$Q' \in \Dd'$ be two adjacent cubes. That is
\[\delta^r \ell(Q') \leq \ell(Q) \leq \ell(Q'),
 \quad \rho_1(Q,Q') \leq \mathcal{C}\ell(Q)^{\gamma}\ell(Q')^{1-\gamma},
 \quad \text{where } \mathcal{C} = 2A_0C_QC_K.\]

Given a fixed small~$\epsilon \in (0,1)$, we define the following sets:
\begin{eqnarray*}
  \Delta &:=& Q \cap Q', \\
  \delta_Q &:=& \{x: \rho(x,Q) \leq \epsilon \ell(Q) \textup{ and }
  \rho(x, X\backslash Q) \leq \epsilon\ell(Q)\}, \\
   \delta_{Q'} &:=& \{x: \rho(x,Q') \leq \epsilon \ell(Q') \textup{ and }
  \rho(x, X\backslash Q') \leq \epsilon\ell(Q')\}, \\
  Q_b &:=& Q \cap \bigcup_{\substack{Q'\in \Dd' \\\delta^r \ell(Q') \leq \ell(Q) \leq \ell(Q') }} \delta_{Q'} ,\\
   Q'_b &:=& Q' \cap \bigcup_{\substack{Q\in \Dd \\\delta^r \ell(Q') \leq \ell(Q) \leq \ell(Q') }} \delta_{Q} \\
  Q_s &:=& Q \backslash (\Delta \cup \delta_{Q'}), \\
  Q'_s &:=& Q' \backslash (\Delta \cup \delta_{Q}), \\
  Q_{\partial} &:=& Q \backslash (\Delta \cup Q_s), \\
  Q'_{\partial} &:=& Q \backslash (\Delta \cup Q'_s), \text{ and } \\
  \widetilde{\Delta} &:=& \Delta \backslash (\delta_Q \cup \delta_{Q'}).
\end{eqnarray*}

Figure~\ref{fig:surgery_1} illustrates the sets $\Delta$, $\delta_{Q}$, $\delta_{Q'}$, $Q_s$, $Q_{\partial}$ and $\widetilde{\Delta}$ in~$\R$ to give some intuition as to what is going on in the more general case.
\begin{figure}[H]
	\centering
\captionsetup{width=0.8\textwidth}
  	\includegraphics[width=15cm]{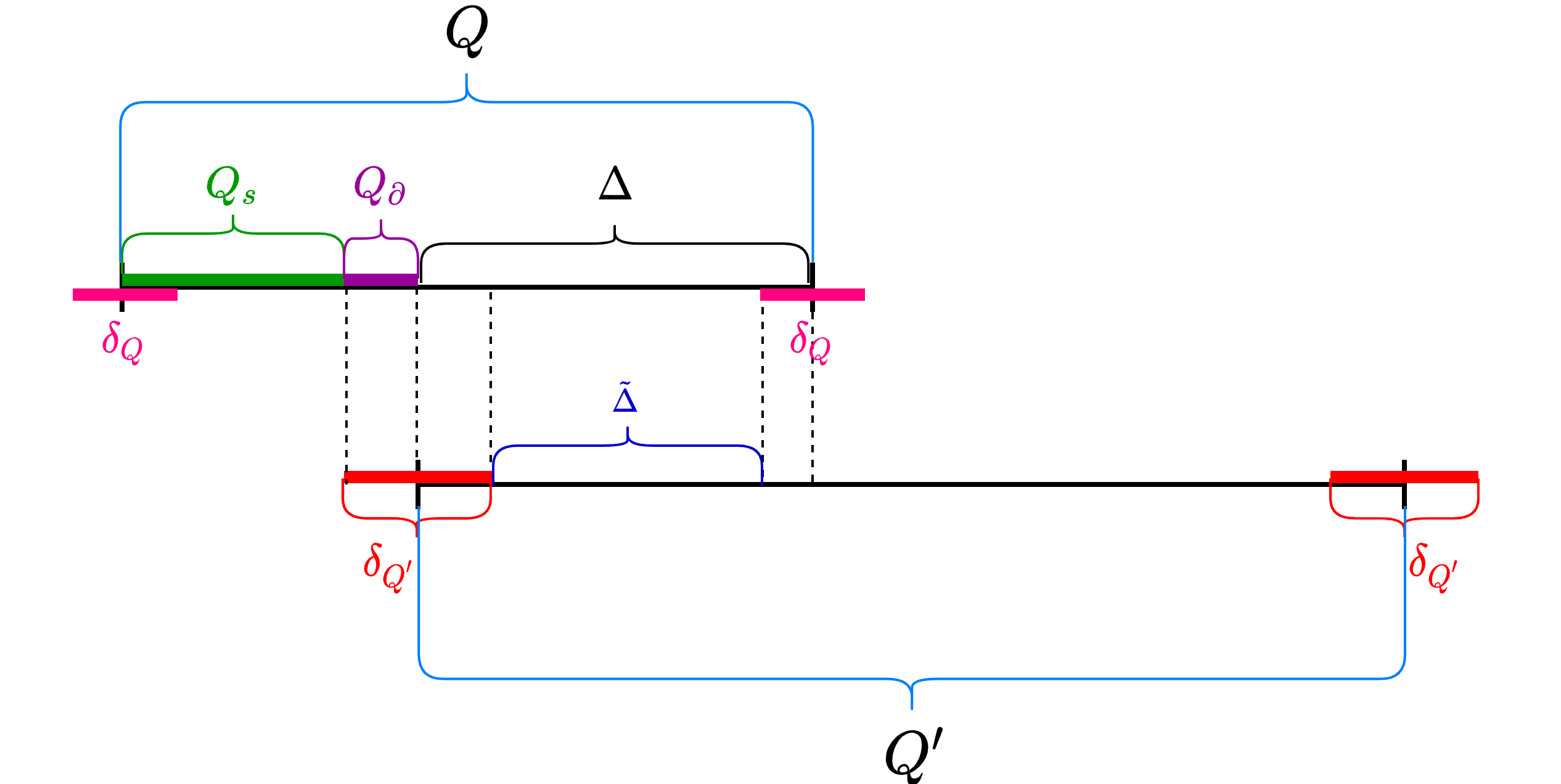}
   	\caption{Sets  $\Delta$, $\delta_{Q'}$, $Q_s$ $Q_{\partial}$ and $\widetilde{\Delta}$  in~$\R$.}
  	\label{fig:surgery_1}
\end{figure}

The set $\Delta$ is the intersection of the two cubes~$Q$ and~$Q'$,
$\delta_Q$ is the set of points close to the boundary of~$Q$,
$Q_b$ is the set of points in~$Q$ which are close to the boundary of at least one cube~$Q' \in \Dd'$ which is adjacent to~$Q$,
$Q_s$ is the set of points in~$Q$ which are strictly separated from~$Q'$,
$Q_{\partial}$ is the set of points in~$Q$ which do not belong to~$Q'$ but are close to the boundary of~$Q'$,
and $\widetilde{\Delta}$ is the set of points that belong to both cubes~$Q$ and~$Q'$ but are far from the boundary of both cubes.
The sets~$\delta_{Q'}$, $Q'_b$, $Q'_s$ and~$Q'_{\partial}$ are described analogously.

\begin{rem}
  As we may notice from the definitions of the sets above, in order to perform surgery on the cubes~$Q$ and~$Q'$, it is sufficient if $\delta^r \ell(Q') \leq \ell(Q) \leq \ell(Q')$, which is one (of two) conditions satisfied by adjacent cubes.
\end{rem}

By \emph{surgery on adjacent cubes~$Q$, $Q'$} we mean the decomposition of
each of the cubes~$Q$ and~$Q'$ into three disjoint sets:
\begin{equation}\label{surgery_eq1}
  Q = Q_s \cup Q_{\partial} \cup \Delta \quad \text{and} \quad
  Q' = Q'_s \cup Q'_{\partial} \cup \Delta.
\end{equation}

Notice also that the following inclusion property holds.
\begin{equation}\label{surgery_eq2}
  Q_{\partial} \subset \delta_{Q'} \subset Q_b \quad \text{and} \quad
  Q'_{\partial} \subset \delta_{Q} \subset Q'_b.
\end{equation}

The following lemma will be useful in the calculations of~\eqref{eq:a} and~\eqref{eq:c} in Lemma~\ref{adj_in_lem3} below.
\begin{lem}\label{lem1:surgery}
Suppose $\epsilon \in (0,1)$ is fixed as in the definition of surgery above.
  For all $x \in Q$, $y \in Q'_s$, we have
  \begin{equation}\label{surgery_eq3}
  \lambda(x, \rho(x,y)) \gs C(\epsilon) \mu(Q)^{1/2} \mu(Q')^{1/2}.
  \end{equation}
  Similarly, equation~\eqref{surgery_eq3} also holds for all $x \in Q_s$ and $y \in Q'$.
\end{lem}
\begin{proof}
We will consider the first case when $x \in Q$, $y \in Q'_s$.
The second case when $x \in Q_s$, $y \in Q'$ can be proved analogously.

Note that for each $x \in Q$, $y \in Q'_s$, we have
$\rho(x,y) \geq \rho(y,\partial Q) > \epsilon \ell(Q) \geq \epsilon \delta^r \ell(Q').$
This together with~\eqref{eq1:up_dbl} gives us
\begin{align*}
  \lambda(x,\rho(x,y)) & \sim \lambda(x,\rho(x,y))^{1/2} \lambda(y,\rho(x,y)) ^{1/2} \\
   & \geq \lambda(x,\epsilon \ell(Q))^{1/2} \lambda(y, \epsilon \delta^r \ell(Q'))^{1/2}.
\end{align*}
Let~$x^*$ be the centre of the cube~$Q$, so $\rho(x^*,x) \leq C_Q \ell(Q)$.
Using the doubling property of~$\lambda$, inequality~\eqref{eq2:up_dbl} and the upper doubling property of~$\mu$ we have
\begin{align*}
\lambda(x,\epsilon \ell(Q)) &\geq C(\epsilon) \lambda(x,C_Q \ell(Q))
\geq C(\epsilon) \lambda (x^*, C_Q \ell(Q))
\geq C(\epsilon) \mu(B(x^*, C_Q \ell(Q))) \geq C(\epsilon) \mu(Q).
\end{align*}
Similarly, we also have
$
 \lambda(y, \epsilon \delta^r \ell(Q')) \geq C(\epsilon) \mu(Q').
$
Therefore,
 \[
 \lambda(x, \rho(x,y)) \gs C(\epsilon) \mu(Q)^{1/2} \mu(Q')^{1/2},
  \]
  as required.
\end{proof}

We also recall a result in~\cite{HM12a} stated as Lemma~\ref{lem2:surgery} below. This lemma will be used in the calculation of estimate~\eqref{eq6:Adj_In}.
\begin{lem}[\text{\cite[Lemma~10.1]{HM12a}}]
\label{lem2:surgery}
There exists $\eta>0$ such that for each fixed $x \in X$ and dyadic cube~$Q = Q^k_{\al}$ we have
\[\mathbb{P}(x \in \delta_{Q^k_{\al}}) \ls \epsilon^{\eta}.\]
\end{lem}

By Lemma~\ref{lem2:surgery} and the definition of~$Q_b$, we have that $\mathbb{E}(\mu(Q_b)) \ls \epsilon^{\eta} \mu(Q)$, where the expectation~$\mathbb{E}$ is taken over dyadic grids~$\Dd'$.

\subsection{Random almost-covering by balls}\label{subsec:ball_cover}
The setting of this section is the same as that of Section~\ref{subsec:surgery} above. That is, non-homogeneous one-parameter quasimetric spaces~$(X,\rho,\mu)$.

We will introduce the construction of a probabilistic covering~$\mathcal{B}$ of a large portion of the space with balls, starting from some fixed size of the radii of balls, and then decreasing the radii but only for some controlled amount.
This construction of a covering is originally developed in~\cite[Section~8]{HM12a}.
The covering~$\mathcal{B}$ is used as a substitute for a certain auxiliary third dyadic grid used in [NTV, Section 10.2] and in~\cite[Section 5.1]{HM14}.
Once this covering is available, we can carry out further decomposition of the characteristic function~$\chi_{\Delta}$, where $\Delta = Q_1 \cap Q'_1$ as defined in Section~\ref{subsec:surgery}
(see \eqref{eq3:adj} and~\eqref{eq4:adj} below).
Furthermore, at the end of this section we also present some observations and Lemma~\ref{lem1:ball_cover}, which are  properties related to the covering~$\mathcal{B}$. These properties will be used in later proofs.

Proposition~\ref{prop1:adj} summarises the construction of the covering.
Please refer to~\cite[Section~8]{HM12a} for the detail of the construction.
\begin{prop}[\text{\cite[Proposition 8.2]{HM12a}}] \label{prop1:adj}
Suppose~$(X,\rho,\mu)$ is a non-homogeneous one-parameter quasimetric space.
Let $0 < \vartheta < A_0^{-4}/32$, where~$A_0$ is the quasitriangle constant, and let $k \in \Z$ and $\upsilon \in (0,1)$ be given.
Then we may randomly construct a family~$\mathcal{B}$ of disjoint quasiballs as follows: there exist constants $c(\vartheta, \upsilon)$ and $C$ with $0 < c(\vartheta, \upsilon) \leq C < 1 $ so that
 if $B, \widehat{B} \in \mathcal{B}$ are two different balls, then
\[c(\vartheta,\upsilon)\vartheta^k \leq r_B \leq C\vartheta^k, \qquad \rho(B,\widehat{B}) \geq c(\vartheta,\upsilon)\vartheta^k. \]
and for every~$x \in X$,
\[\mathbb{P}\bigg( x \in \bigcup_{B \in \mathcal{B}}B\bigg) > 1 - \upsilon.\]
Here~$\mathbb{P}$ is a probability measure on the collection of all such families~$\mathcal{B}$.
\end{prop}
From Proposition~\ref{prop1:adj} we know that all balls~$B \in \mathcal{B}$ have roughly the same radius, and are $c(\vartheta, \upsilon)$ separated.

Now, we will use Proposition~\ref{prop1:adj} to decompose the function~$\chi_{\Delta}$.
Recall that we have fixed~$\epsilon \in (0,1)$ in Section~\ref{subsec:surgery}.
We now fix $\delta = A_0^{-4}/1000$, where~$A_0$ is the quasitriangle constant. Then fix the smallest~$k$ for which
\begin{equation}\label{eq2:adj}
\delta^k \leq (\Lambda)^{-1} (2A_0)^{-2}\epsilon \min (\ell(Q),\ell(Q')),
\end{equation}
where~$\Lambda > 1$ is from Assumptions~\ref{assum_5} and~\ref{assum_6} of weak boundedness properties of the operator~$T$.

\begin{rem}
Notice that our choice of~$k$ here is slightly different from that in~\cite{HM12a},
in which $k$ is chosen so that
\begin{equation}\label{eq1:adj}
\delta^k \leq \Lambda^{-1}(8^{-1}\epsilon \min (\ell(Q,\ell(Q'))))^{\beta}, \quad \text{where } \beta \geq 1.
\end{equation}
Recall that in~\cite{HM12a}, they initially work with quasimetrics~$\rho$. Then, using a result of Mac\'{i}as and Segovia, they reduce to working with an equivalent metric~$d$ so that for all~$x,y \in X$
\[2^{-\beta}d(x,y)^{\beta} \leq \rho(x,y) \leq 4^{\beta}d(x,y)^{\beta}.\]
The extra power~$\beta$ in~\eqref{eq1:adj} is to cope with this reduction.
Since we always work with quasimetrics, we can choose~$k$ as in~\eqref{eq2:adj}.
\end{rem}

Consider some small enough $\upsilon \in (0,1)$, and set $\vartheta = \delta$.
By Proposition~\ref{prop1:adj} we can randomly construct a family~$\mathcal{B}$ of disjoint  balls starting from the fixed level~$k$ with parameter~$\upsilon$, so that
for all $x \in X$ we have
$\mathbb{P}( x \in \bigcup_{B \in \mathcal{B}}B) > 1 - \upsilon$.
This implies
$\mathbb{E}(\mu(\widetilde{\Delta}\backslash \bigcup_{B \in \mathcal{B}} B)) <\upsilon \mu(\widetilde{\Delta}),$ where the expectation~$\mathbb{E}$ is taken over families~$\mathcal{B}$.
So we may now fix some such covering~$\mathcal{B}$ for which
$\mu(\widetilde{\Delta}\backslash \bigcup_{B \in \mathcal{B}} B) <\upsilon \mu(\widetilde{\Delta})$.
We now remove from the collection~$\mathcal{B}$ those balls that do not touch~$\widetilde{\Delta}$.

We claim that for all~$B \in \mathcal{B}$, we have $\Lambda B := B(x_B, \Lambda r_B) \subset \Delta$.
To see this, fix a ball $B = B(x_B,r_B) \in \mathcal{B}$.
By Proposition~\ref{prop1:adj}, $r_B \leq C\vartheta^k = C\delta^k$, where~$C<1$.
As $B \cap \widetilde{\Delta} \neq \emptyset$, there exists $x \in B$ so that $\rho(x, X \backslash \Delta) \geq \epsilon \min (\ell(Q), \ell(Q'))$.
Fix $w \in \Lambda B$, then with the choice of~$k$ as in~\eqref{eq2:adj} we have
\begin{eqnarray*}
  \rho(x, X \backslash \Delta) - A_0\rho(w, X \backslash \Delta)
  &\leq& A_0\rho(w,x)
  \leq A_0^2 (\rho(w,x_B) + \rho(x_B,x))
  \leq 2A_0^2 \Lambda C  \delta^k \\
  &\leq& 2^{-1} \epsilon \min (\ell(Q), \ell(Q')).
\end{eqnarray*}
Therefore,
$\rho(w, X \backslash \Delta) \geq (1-2^{-1})A_0^{-1} \epsilon \min (\ell(Q), \ell(Q'))$. Since this is true for all $w \in \Lambda B$, we have
 $\rho(\Lambda B , X \backslash \Delta) \geq (2A_0)^{-1} \epsilon \min (\ell(Q), \ell(Q')) > 0$, and so $\Lambda B \subset \Delta$.

Since any two balls~$B$ and~$\widehat{B} \in \mathcal{B}$ are separated by a distance comparable to their radii, namely $\rho(B,\widehat{B}) \geq c(\vartheta,\upsilon)\vartheta^k$ and $c(\vartheta,\upsilon)\vartheta^k \leq r_B$, the same is still true for some expanded balls~$(1+\epsilon')B$, which contain the support of the function $\widetilde{\chi}_B := \widetilde{\chi}_{B,\epsilon'}$.
Recall that $\widetilde{\chi}_{B,\epsilon'} \in C^{\eta}(\mathcal{X})$ is a function which is more regular than the rough function~$\chi_B$, and for all~$x \in X$ we have $\chi_B(x) \leq \widetilde{\chi}_{B,\epsilon'}(x) \leq \chi_{(1+\epsilon')B}(x)$.
As~$\vartheta$ and~$k$ are fixed, the parameter~$\epsilon'$ only depends on~$\upsilon$.

We write
\begin{equation}\label{eq3:adj}
\chi_{\Delta} = \sum_{B \in \mathcal{B}} \widetilde{\chi}_B + \widetilde{\chi}_{\Delta\backslash \cup B},
\end{equation}
where
\[
\widetilde{\chi}_{\Delta\backslash \cup B}
:= \chi_{\Delta} - \sum_{B \in \mathcal{B}} \widetilde{\chi}_B
\leq \chi_{\Delta} - \sum_{B \in \mathcal{B}} \chi_B
= \chi_{\Delta\backslash \cup B}.
\]
Note also that $\widetilde{\chi}_{\Delta\backslash \cup B} \geq 0$ since
$(1+\epsilon')B \subset \Lambda B \subset \Delta$.
We further decompose
\begin{eqnarray}\label{eq4:adj}
  \widetilde{\chi}_{\Delta\backslash \cup B}
  = \widetilde{\chi}_{\Delta\backslash \cup B}
  (\chi_{\Om_i} + \chi_{\Om_Q} + \chi_{\Om_{Q'}})
  =: \widetilde{\chi}_{\Om_i} + \widetilde{\chi}_{\Om_Q} + \widetilde{\chi}_{\Om_{Q'}},
\end{eqnarray}
where $\Delta\backslash \cup B$ is a disjoint union of $\Om_i := \widetilde{\Delta}\backslash \cup B$, and some sets $\Om_Q \subset Q_b$ and $\Om_{Q'} \subset Q'_b$.
Figure~\ref{fig:surgery_2} presents the sets $\Omega_i$, $\Omega_Q$ and $\Omega_{Q'}$ in~$\R$ for intuition in the general case.
\begin{figure}[H]
	\centering
\captionsetup{width=0.8\textwidth}
  	\includegraphics[width=14cm]{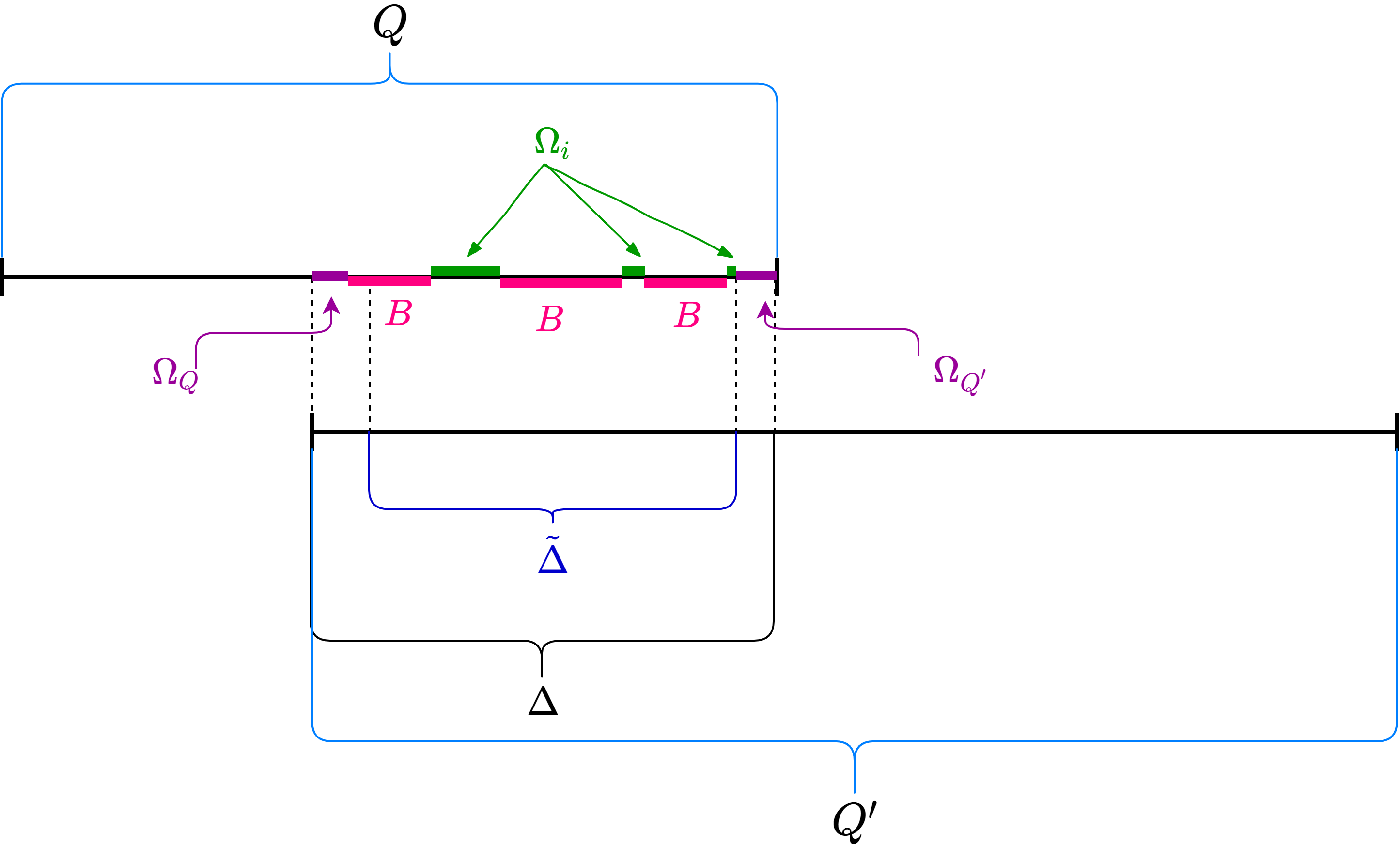}
   	\caption{Sets $\Omega_i$, $\Omega_Q$ and $\Omega_{Q'}$ in~$\R$.}
  	\label{fig:surgery_2}
\end{figure}

The following observations will be useful for later calculations:
\begin{align}
   &   \|\widetilde{\chi}_{\Om_i}\|_{L^2}
  \leq \|\chi_{\widetilde{\Delta}\backslash \cup B}\|_{L^2}
  \leq \mu(\widetilde{\Delta}\backslash \cup B)^{1/2}
  \leq \upsilon^{1/2}\mu(\widetilde{\Delta})^{1/2}, \label{ball_cover_eq1}\\
  &   \Big\|\sum_{B \in \mathcal{B}} \widetilde{\chi}_B\Big\|_{L^2}
  \leq \|\chi_{\Delta}\|_{L^2}
  \leq \mu(\Delta)^{1/2}, \label{ball_cover_eq2} \\
  &   \|\widetilde{\chi}_{\Om_Q}\|_{L^2}
  \leq \|\chi_{\Om_Q}\|_{L^2} \leq \mu(\Om_Q)^{1/2} \leq \mu(Q_b)^{1/2}, \label{ball_cover_eq3}
\end{align}
and similarly,
\begin{equation}\label{ball_cover_eq4}
  \|\widetilde{\chi}_{\Om_{Q'}}\|_{L^2}
  \leq \mu(Q'_b)^{1/2}.
\end{equation}

Lemma~\ref{lem1:ball_cover} shows another property of the collection~$\mathcal{B}$.
This lemma will be used in the calculation of~\eqref{eq:g2} in Lemma~\ref{adj_in_lem3}.
\begin{lem}\label{lem1:ball_cover}
  Fix the balls $B, \widehat{B} \in \mathcal{B}$ such that $B \neq \widehat{B}$.
  For all $x \in (1+\epsilon')B =: B_{\epsilon'}$ and $y \in (1+\epsilon')\widehat{B}=: \widehat{B}_{\epsilon'}$ we have
  \[\lambda(x,\rho(x,y)) \gs C(\upsilon) \mu(B_{\epsilon'})^{1/2} \mu(\widehat{B}_{\epsilon'})^{1/2},\]
  where $\epsilon' \in (0,1)$ is fixed in the surgery, and $\upsilon \in (0,1)$ is given in the construction of the covering~$\mathcal{B}$.
\end{lem}
\begin{proof}
Note that $\vartheta$ is already fixed and $\epsilon'$ only depends on $\upsilon$.
 Using~\eqref{eq1:up_dbl}, Proposition~\ref{prop1:adj}, and the facts that $\lambda$ is doubling and $\mu$ is upper doubling, we get
 \begin{align*}
  \lambda(x,\rho(x,y)) & \sim \lambda(x,\rho(x,y))^{1/2} \lambda(y,\rho(x,y)) ^{1/2}
   \geq \lambda(x,\rho(B_{\epsilon'},\widehat{B}_{\epsilon'}))^{1/2} \lambda(y,\rho(B_{\epsilon'},\widehat{B}_{\epsilon'})) ^{1/2} \\
  & \geq \lambda(x,c(\vartheta, \upsilon) \vartheta^k)^{1/2} \lambda(y,c(\vartheta, \upsilon) \vartheta^k) ^{1/2}
   \gs \lambda(x,C(\upsilon) r_{B_{\epsilon'}})^{1/2} \lambda(y,C(\upsilon) r_{\widehat{B}_{\epsilon'}}) ^{1/2} \\
  & \geq C(\upsilon) \mu(B_{\epsilon'})^{1/2} \mu(\widehat{B}_{\epsilon'})^{1/2}, \quad \text{  as required.} \qedhere
  \end{align*}
\end{proof}

To sum up,
given two adjacent cubes~$Q \in \Dd$ and~$Q' \in \Dd'$, we can perform the decompositions as in~\eqref{surgery_eq1}, \eqref{eq3:adj} and~\eqref{eq4:adj}.
\subsection{Adjacent/Nested cubes}\label{subsec:adj_in}
In this section, we consider the \emph{Adj/Nes} case. That is, we estimate~\eqref{coreeq.1} when $Q_1, Q'_1$ are adjacent and $Q_2, Q'_2$ are nested:
\begin{eqnarray*}
\delta^r \ell(Q'_1) \leq \ell(Q_1) \leq \ell(Q'_1),
 && \rho_1(Q_1,Q'_1) \leq \mathcal{C}\ell(Q_1)^{\gamma_1}\ell(Q'_1)^{1-\gamma_1}, \quad \text{and} \\
 \ell(Q_2) < \delta^r \ell(Q'_2),
&& \rho_2(Q_2,Q'_2) \leq \mathcal{C}\ell(Q_2)^{\gamma_2}\ell(Q'_2)^{1-\gamma_2},
\end{eqnarray*}
where $\mathcal{C}:= 2A_0C_QC_K$.
With the same reason as in the \emph{Sep/Nes} case, the terms in~\eqref{coreeq.1} which involves~$h_0^{Q_2}$ vanish.
Thus, we only need to estimate eight (out of 16 terms) in~\eqref{coreeq.1}:
\begin{align}\label{adj_in_eq.1}
 &\hspace{-0.5cm}\sum_{\substack{Q_1,Q'_1  \\ \text{adjacent} }} \,
  \sum_{\substack{Q_2,Q'_2\\ \text{nested} }} \,
  \sum_{\substack{u_1,u'_1 \\u_2,u'_2}} \,
  \langle f, h_{u_1}^{Q_1} \otimes h_{u_2}^{Q_2}\rangle
  \langle g, h_{u'_1}^{Q'_1} \otimes h_{u'_2}^{Q'_2}\rangle
  \langle T(h_{u_1}^{Q_1} \otimes h_{u_2}^{Q_2}), h_{u'_1}^{Q'_1} \otimes h_{u'_2}^{Q'_2}\rangle \noz\\
  &+  \sum_{\substack{Q_1,Q'_1  \\ \text{adjacent} }} \,
  \sum_{\substack{Q_2,Q'_2 \\ \text{nested} \\ \text{gen}(Q'_2) = m}} \,
  \sum_{\substack{u_1,u'_1 \\u_2}} \,
  \langle f, h_{u_1}^{Q_1} \otimes h_{u_2}^{Q_2}\rangle
  \langle g, h_{u'_1}^{Q'_1} \otimes h_{0}^{Q'_2}\rangle
  \langle T(h_{u_1}^{Q_1} \otimes h_{u_2}^{Q_2}), h_{u'_1}^{Q'_1} \otimes h_{0}^{Q'_2}\rangle \noz\\
  &+ \sum_{\substack{Q_1,Q'_1  \\ \text{adjacent} \\ \text{gen}(Q'_1) = m }} \,
  \sum_{\substack{Q_2,Q'_2\\ \text{nested}  }} \,
  \sum_{\substack{u_1\\u_2,u'_2}} \,
  \langle f, h_{u_1}^{Q_1} \otimes h_{u_2}^{Q_2}\rangle
  \langle g, h_{0}^{Q'_1} \otimes h_{u'_2}^{Q'_2}\rangle
  \langle T(h_{u_1}^{Q_1} \otimes h_{u_2}^{Q_2}), h_{0}^{Q'_1} \otimes h_{u'_2}^{Q'_2}\rangle \noz\\
  &+   \sum_{\substack{Q_1,Q'_1  \\ \text{adjacent} \\ \text{gen}(Q'_1) = m }} \,
  \sum_{\substack{Q_2,Q'_2\\ \text{nested} \\ \text{gen}(Q'_2) = m }} \,
  \sum_{\substack{u_1 \\u_2}} \,
  \langle f, h_{u_1}^{Q_1} \otimes h_{u_2}^{Q_2}\rangle
  \langle g, h_{0}^{Q'_1} \otimes h_{0}^{Q'_2}\rangle
  \langle T(h_{u_1}^{Q_1} \otimes h_{u_2}^{Q_2}), h_{0}^{Q'_1} \otimes h_{0}^{Q'_2}\rangle \noz\\
  &+\sum_{\substack{Q_1,Q'_1  \\ \text{adjacent} \\ \text{gen}(Q_1) = m}} \,
  \sum_{\substack{Q_2,Q'_2\\ \text{nested} }} \,
  \sum_{\substack{u'_1 \\u_2,u'_2}} \,
  \langle f, h_{0}^{Q_1} \otimes h_{u_2}^{Q_2}\rangle
  \langle g, h_{u'_1}^{Q'_1} \otimes h_{u'_2}^{Q'_2}\rangle
  \langle T(h_{0}^{Q_1} \otimes h_{u_2}^{Q_2}), h_{u'_1}^{Q'_1} \otimes h_{u'_2}^{Q'_2}\rangle\noz\\
  &+ \sum_{\substack{Q_1,Q'_1  \\ \text{adjacent} \\ \text{gen}(Q_1) = m }} \,
  \sum_{\substack{Q_2,Q'_2 \\ \text{nested} \\ \text{gen}(Q'_2) = m}} \,
  \sum_{\substack{u'_1 \\u_2}} \,
  \langle f, h_{0}^{Q_1} \otimes h_{u_2}^{Q_2}\rangle
  \langle g, h_{u'_1}^{Q'_1} \otimes h_{0}^{Q'_2}\rangle
  \langle T(h_{0}^{Q_1} \otimes h_{u_2}^{Q_2}), h_{u'_1}^{Q'_1} \otimes h_{0}^{Q'_2}\rangle\noz\\
  &+ \sum_{\substack{Q_1,Q'_1  \\ \text{adjacent} \\ \text{gen}(Q_1) = \text{gen}(Q'_1) = m }} \,
  \sum_{\substack{Q_2,Q'_2\\ \text{nested}  }} \,
  \sum_{u_2,u'_2} \,
  \langle f, h_{0}^{Q_1} \otimes h_{u_2}^{Q_2}\rangle
  \langle g, h_{0}^{Q'_1} \otimes h_{u'_2}^{Q'_2}\rangle
  \langle T(h_{0}^{Q_1} \otimes h_{u_2}^{Q_2}), h_{0}^{Q'_1} \otimes h_{u'_2}^{Q'_2}\rangle\noz\\
  &+ \sum_{\substack{Q_1,Q'_1  \\ \text{adjacent} \\ \text{gen}(Q_1) = \text{gen}(Q'_1) = m }} \,
  \sum_{\substack{Q_2,Q'_2\\ \text{nested} \\ \text{gen}(Q'_2) = m }} \,
  \sum_{u_2} \,
  \langle f, h_{0}^{Q_1} \otimes h_{u_2}^{Q_2}\rangle
  \langle g, h_{0}^{Q'_1} \otimes h_{0}^{Q'_2}\rangle
  \langle T(h_{0}^{Q_1} \otimes h_{u_2}^{Q_2}), h_{0}^{Q'_1} \otimes h_{0}^{Q'_2}\rangle .
\end{align}

We will estimate the pair of the first two sums in~\eqref{adj_in_eq.1}. The remaining pairs can be handled using the same proof technique.
The proof strategy in this case is the same as that in the \emph{Sep/Nes} case (Section~\ref{subsec:sep_in}). Namely,
we write~$\langle T(h_{u_1}^{Q_1} \otimes h_{u_2}^{Q_2}), h_{u'_1}^{Q'_1} \otimes h_{u'_2}^{Q'_2}\rangle$ as the sum
\begin{eqnarray*}
(\langle T(h_{u_1}^{Q_1} \otimes h_{u_2}^{Q_2}), h_{u'_1}^{Q'_1} \otimes h_{u'_2}^{Q'_2}\rangle
- \langle h_{u'_2}^{Q'_2}\rangle_{Q_2}
\langle T(h_{u_1}^{Q_1} \otimes h_{u_2}^{Q_2}), h_{u'_1}^{Q'_1} \otimes 1\rangle)  \langle h_{u'_2}^{Q'_2}\rangle_{Q_2}
\langle T(h_{u_1}^{Q_1} \otimes h_{u_2}^{Q_2}), h_{u'_1}^{Q'_1} \otimes 1\rangle,
\end{eqnarray*}
and then consider the first summation in~\eqref{adj_in_eq.1} involving each of these terms.
Here $\langle h \rangle_{Q} := \mu_2(Q)^{-1} \int_{Q} h \,d\mu_2$.
Let~$Q'_{2,1} \in \text{ch}(Q'_2)$ be such that~$Q_2 \subset Q'_{2,1}$.

As in the \emph{Sep/Nes} case, we also have
\begin{eqnarray*}
  \lefteqn{\langle T(h_{u_1}^{Q_1} \otimes h_{u_2}^{Q_2}), h_{u'_1}^{Q'_1} \otimes h_{u'_2}^{Q'_2}\rangle
- \langle h_{u'_2}^{Q'_2}\rangle_{Q_2}
\langle T(h_{u_1}^{Q_1} \otimes h_{u_2}^{Q_2}), h_{u'_1}^{Q'_1} \otimes 1\rangle} \\
  &=& \sum_{\substack{Q' \in \text{ch}(Q'_2) \\ Q' \subset Q'_2 \backslash Q'_{2,1}}} \langle T(h_{u_1}^{Q_1} \otimes h_{u_2}^{Q_2}), h_{u'_1}^{Q'_1} \otimes h_{u'_2}^{Q'_2} \chi_{Q'}\rangle
-   \langle h_{u'_2}^{Q'_2}\rangle_{Q'_{2,1}}
\langle T(h_{u_1}^{Q_1} \otimes h_{u_2}^{Q_2}), h_{u'_1}^{Q'_1} \otimes \chi_{Q^{'c}_{2,1}}\rangle.
\end{eqnarray*}
These two terms will be estimated in Lemmas~\ref{lem1:adj_in} and~\ref{lem2:adj_in} below. These  two lemmas are substitutes for Lemmas~\ref{lem1:sepnest} and~\ref{lem2:sepnest} from the \emph{Sep/Nes} case.
The main differences are the use of Lemma~\ref{lem:property_kernel} instead of Lemma~\ref{lem1:property_kernel}.

\begin{lem}\label{lem1:adj_in}
Suppose $Q_1$, $Q'_1$ are adjacent,
$Q_2$, $Q'_2$ are nested, and $Q'_{2,1} \in \textup{ch}(Q'_2)$ such that $Q_2 \subset Q'_{2,1}$.
  There holds
  \[|\langle h_{u'_2}^{Q'_2}\rangle_{Q'_{2,1}}
\langle T(h_{u_1}^{Q_1} \otimes h_{u_2}^{Q_2}), h_{u'_1}^{Q'_1} \otimes \chi_{Q^{'c}_{2,1}}\rangle| \ls A_{Q_2,Q'_2}^{\textup{in}},\]
where
\[A^{\textup{in}}_{Q_2,Q'_2} := \left(\frac{\ell(Q_2)}{\ell(Q'_2)}\right)^{\al_2/2}
\left( \frac{\mu_2(Q_2)}{\mu_2(Q'_{2,1})}\right)^{1/2}. \]
\end{lem}
\begin{proof}
First, using property~\eqref{eq2:pro_Haar} of Haar functions we get
  \begin{equation}\label{eq1:lem1:adj_in}
  \langle h_{u'_2}^{Q'_2}\rangle_{Q'_{2,1}}
\ls \mu_2(Q'_{2})^{-1/2}.
  \end{equation}
Second, we consider the term $\langle T(h_{u_1}^{Q_1} \otimes h_{u_2}^{Q_2}), h_{u'_1}^{Q'_1} \otimes \chi_{Q^{'c}_{2,1}}\rangle$.
Recall that $Q_2 \cap Q^{'c}_{2,1} = \emptyset$,
and by property~(i) of Claim~\ref{claim:sep_nest},  for all $y_2, w \in Q_2$, $x_2 \in Q^{'c}_{2,1}$ we have
 \[\rho_2(y_2,w) \leq C_Q \ell(Q_2) \leq \frac{\rho_2(x_2,w)}{C_K}.\]
Applying Lemma~\ref{lem:property_kernel} to
$\phi_1 = h_{u_1}^{Q_1} $, $\phi_2 =  h_{u_2}^{Q_2}$, $\theta_1 = h_{u'_1}^{Q'_1} $ and $\theta_2 = \chi_{Q^{'c}_{2,1}}$ we have
  \begin{align*}
    |\langle T(h_{u_1}^{Q_1} \otimes h_{u_2}^{Q_2}), h_{u'_1}^{Q'_1} \otimes \chi_{Q^{'c}_{2,1}}\rangle|
    &\ls \|h_{u_1}^{Q_1}\|_{L^2(\mu_1)} \|h_{u'_1}^{Q'_1}\|_{L^2(\mu_1)}
    \|h_{u_2}^{Q_2}\|_{L^2(\mu_2)}
    \mu_2(Q_2)^{1/2} \\
   &\hspace{0.5cm}\times \, \int_{Q^{'c}_{2,1}}   \frac{C_Q^{\al_2} \ell(Q_2)^{\al_2}}
    {\rho_2(x_2,w)^{\al_2}\lambda_2(w,\rho_2(x_2,w))^{\al_2}} \,d\mu_2(x_2).
  \end{align*}
Moreover, as shown in~\eqref{eq6:sep_nest}, the integral above is bounded by
 $\ell(Q_2)^{\al_2/2} \ell(Q'_{2})^{-\al_2/2}$.
Therefore,
\begin{equation}\label{eq2:lem1:adj_in}
  |\langle T(h_{u_1}^{Q_1} \otimes h_{u_2}^{Q_2}), h_{u'_1}^{Q'_1} \otimes \chi_{Q^{'c}_{2,1}}\rangle|
  \ls \bigg(\frac{\ell(Q_2)}{\ell(Q'_2)} \bigg)^{\al_2/2} \mu_2(Q_2)^{1/2}.
\end{equation}

Combining~\eqref{eq1:lem1:adj_in} and~\eqref{eq2:lem1:adj_in}, we establish Lemma~\ref{lem1:adj_in}.
\end{proof}

\begin{lem}\label{lem2:adj_in}
Suppose $Q_1$ and~$Q'_1$ are adjacent,
$Q_2$ and~$Q'_2$ are nested,
and $Q' \in \textup{ch}(Q'_2)$ such that $Q' \subset Q'_2 \backslash Q'_{2,1}$. Then
\[|\langle T(h_{u_1}^{Q_1} \otimes h_{u_2}^{Q_2}), h_{u'_1}^{Q'_1} \otimes h_{u'_2}^{Q'_2} \chi_{Q'}\rangle| \ls A_{Q_2,Q'_2}^{\textup{in}}.\]
\end{lem}
\begin{proof}
  Note that $Q_2 \cap Q' = \emptyset$.
  Moreover, using property~(i) of Claim~\ref{claim:sep_nest},  for all $y_2, w \in Q_2$, $x_2 \in Q'$ we have
 \[\rho_2(y_2,w) \leq C_Q \ell(Q_2) \leq  \frac{\rho_2(Q_2,Q')}{2A_0C_K} \leq \frac{\rho_2(x_2,w)}{C_K}.\]
  Similar to Lemma~\ref{lem1:adj_in}, we apply  Lemma~\ref{lem:property_kernel} to
$\phi_1 = h_{u_1}^{Q_1} $, $\phi_2 =  h_{u_2}^{Q_2}$, $\theta_1 = h_{u'_1}^{Q'_1} $ and $\theta_2 = h_{u'_2}^{Q'_2} \chi_{Q'}$ to get
  \begin{align*}
    |\langle T(h_{u_1}^{Q_1} \otimes h_{u_2}^{Q_2}), h_{u'_1}^{Q'_1} \otimes h_{u'_2}^{Q'_2} \chi_{Q'}\rangle|
    &\ls  \|h_{u_1}^{Q_1}\|_{L^2(\mu_1)} \|h_{u'_1}^{Q'_1}\|_{L^2(\mu_1)}
    \|h_{u_2}^{Q_2}\|_{L^2(\mu_2)}
    \mu_2(Q_2)^{1/2}\\
    &\hspace{0.5cm}\times\int_{Q'}   \frac{C_Q^{\al_2} \ell(Q_2)^{\al_2} }
    {\rho_2(x_2,w)^{\al_2}\lambda_2(w,\rho_2(x_2,w))} |h_{u'_2}^{Q'_2}(x_2)|  \,d\mu_2(x_2).
  \end{align*}
The integral above was estimated in~\eqref{eq5:SepIn_b}:
\[\int_{Q'} \frac{C_Q^{\al_2}\ell(Q_2)^{\al_2}}
  {\rho_2(x_2,w)^{\al_2}\lambda_2(w,\rho_2(x_2,w))} |h_{u'_2}^{Q'_2}(x_2)| \,d\mu_2(x_2)
  \ls  \left( \frac{\ell(Q_2)}{\ell(Q'_2)}\right)^{\al_2/2}
  \frac{1}{\mu_2(Q'_{2,1})^{1/2}}.\]
Thus we have
  \begin{eqnarray*}
    |\langle T(h_{u_1}^{Q_1} \otimes h_{u_2}^{Q_2}), h_{u'_1}^{Q'_1} \otimes h_{u'_2}^{Q'_2} \chi_{Q'}\rangle|
        \ls \left(\frac{\ell(Q_2)}{\ell(Q'_2)}\right)^{\al_2/2} \left( \frac{\mu_2(Q_2)}{\mu_2(Q'_{2,1})}\right)^{1/2} = A^{\textup{in}}_{Q_2,Q'_2}
  \end{eqnarray*}
as required. \qedhere
\end{proof}

Lemmas~\ref{lem1:adj_in} and~\ref{lem2:adj_in} give us
\begin{align*}
  &\hspace{-0.5cm}|\langle T(h_{u_1}^{Q_1} \otimes h_{u_2}^{Q_2}), h_{u'_1}^{Q'_1} \otimes h_{u'_2}^{Q'_2}\rangle
- \langle h_{u'_2}^{Q'_2}\rangle_{Q_2}
\langle T(h_{u_1}^{Q_1} \otimes h_{u_2}^{Q_2}), h_{u'_1}^{Q'_1} \otimes 1\rangle| \\
 &\ls \sum_{\substack{Q' \in \text{ch}(Q'_2) \\ Q' \subset Q'_2 \backslash Q'_{2,1}}} A_{Q_2,Q'_2}^{\textup{in}}
 + A_{Q_2,Q'_2}^{\textup{in}}
 \ls A_{Q_2,Q'_2}^{\textup{in}}.
\end{align*}
Recall that for a given~$Q_1$, there are only finitely many~$Q'_1$ such that $Q_1$ and $Q'_1$ are adjacent, meaning
\[\delta^r \ell(Q'_1) \leq \ell(Q_1) \leq \ell(Q'_1),
 \quad \rho_1(Q_1,Q'_1) \leq \mathcal{C}\ell(Q_1)^{\gamma_1}\ell(Q'_1)^{1-\gamma_1}.\]
Combining this fact with Lemmas~\ref{lem1:adj_in}, \ref{lem2:adj_in} and Proposition~\ref{prop:A_in}, we have
 \begin{eqnarray}\label{eq5a:Adj_In}
   \lefteqn{\sum_{\substack{Q_1,Q'_1  \\ \text{adjacent} }} \,
  \sum_{\substack{Q_2,Q'_2\\ \text{nested} }}\,
  \sum_{\substack{u_1,u'_1 \\u_2,u'_2}}\,
 | \langle T(h_{u_1}^{Q_1} \otimes h_{u_2}^{Q_2}), h_{u'_1}^{Q'_1} \otimes h_{u'_2}^{Q'_2}\rangle
- \langle h_{u'_2}^{Q'_2}\rangle_{Q_2}
\langle T(h_{u_1}^{Q_1} \otimes h_{u_2}^{Q_2}), h_{u'_1}^{Q'_1} \otimes 1\rangle |} \noz\\
  &&\hspace{1.5cm} \times |\langle f, h_{u_1}^{Q_1} \otimes h_{u_2}^{Q_2}\rangle|
 |\langle g, h_{u'_1}^{Q'_1} \otimes h_{u'_2}^{Q'_2}\rangle| \noz\\
   &\ls& \sum_{\substack{Q_1,Q'_1  \\ \text{adjacent} }} \,
  \sum_{\substack{Q_2,Q'_2\\ \text{nested} }}\,
  \sum_{\substack{u_1,u'_1 \\u_2,u'_2}}\,
  A_{Q_2,Q'_2}^{\textup{in}}
  |\langle f, h_{u_1}^{Q_1} \otimes h_{u_2}^{Q_2}\rangle|
  |\langle g, h_{u'_1}^{Q'_1} \otimes h_{u'_2}^{Q'_2}\rangle| \noz\\
  &\ls& \sum_{\substack{Q_1,Q'_1  \\ \text{adjacent} }} \,
  \sum_{u_1,u'_1}\,
  \left( \sum_{Q_2}|\sum_{u_2}\langle f, h_{u_1}^{Q_1} \otimes h_{u_2}^{Q_2}\rangle|^2 \right)^{1/2}
  \left( \sum_{Q'_2}|\sum_{u'_2}\langle g, h_{u'_1}^{Q'_1} \otimes h_{u'_2}^{Q'_2}\rangle|^2\right)^{1/2} \noz\\
  &\ls& \left(\sum_{Q_1}\sum_{u_1}\sum_{Q_2}\sum_{u_2}|\langle f, h_{u_1}^{Q_1} \otimes h_{u_2}^{Q_2}\rangle|^2\right)^{1/2}
   \left(\sum_{Q'_1}\sum_{u'_1}\sum_{Q'_2}\sum_{u'_2}|\langle g, h_{u'_1}^{Q'_1} \otimes h_{u'_2}^{Q'_2}\rangle|^2\right)^{1/2}\noz\\
   &\ls& \|f\|_{L^2(\mu)} \|g\|_{L^2(\mu)}.
 \end{eqnarray}

 Similarly, we also obtain the same result for the second sum in~\eqref{adj_in_eq.1}. Specifically,
\begin{align}
   &\sum_{\substack{Q_1,Q'_1  \\ \text{adjacent} }} \,
  \sum_{\substack{Q_2,Q'_2\\ \text{nested}\\ \text{gen}(Q'_2) = m }}\,
  \sum_{\substack{u_1,u'_1 \\u_2}}\,
 | \langle T(h_{u_1}^{Q_1} \otimes h_{u_2}^{Q_2}), h_{u'_1}^{Q'_1} \otimes h_{0}^{Q'_2}\rangle
- \langle h_{0}^{Q'_2}\rangle_{Q_2}
\langle T(h_{u_1}^{Q_1} \otimes h_{u_2}^{Q_2}), h_{u'_1}^{Q'_1} \otimes 1\rangle | \noz\\
  &\hspace{1.5cm} \times |\langle f, h_{u_1}^{Q_1} \otimes h_{u_2}^{Q_2}\rangle|
 |\langle g, h_{u'_1}^{Q'_1} \otimes h_{0}^{Q'_2}\rangle| \noz\\
   &\ls \|f\|_{L^2(\mu)} \|g\|_{L^2(\mu)}. \label{eq9:Adj_In}
\end{align}

 Now we only need to consider the same summation as in the first and the second term of~\eqref{adj_in_eq.1}, but with the elements $\langle h_{u'_2}^{Q'_2}\rangle_{Q_2}
\langle T(h_{u_1}^{Q_1} \otimes h_{u_2}^{Q_2}), h_{u'_1}^{Q'_1} \otimes 1\rangle$
and $\langle h_{0}^{Q'_2}\rangle_{Q_2}
\langle T(h_{u_1}^{Q_1} \otimes h_{u_2}^{Q_2}), h_{u'_1}^{Q'_1} \otimes 1\rangle$, respectively. That is
\begin{align}
  &\sum_{\substack{Q_1,Q'_1  \\ \text{adjacent} }} \,
  \sum_{\substack{Q_2,Q'_2\\ \text{nested} }}\,
  \sum_{\substack{u_1,u'_1 \\u_2,u'_2}}
  \langle f, h_{u_1}^{Q_1} \otimes h_{u_2}^{Q_2}\rangle
   \langle g, h_{u'_1}^{Q'_1} \otimes h_{u'_2}^{Q'_2}\rangle
  \langle h_{u'_2}^{Q'_2}\rangle_{Q_2}
\langle T(h_{u_1}^{Q_1} \otimes h_{u_2}^{Q_2}), h_{u'_1}^{Q'_1} \otimes 1\rangle \noz\\
&\hspace{0.5cm}+
\sum_{\substack{Q_1,Q'_1  \\ \text{adjacent} }} \,
  \sum_{\substack{Q_2,Q'_2\\ \text{nested}\\ \text{gen}(Q'_2) = m }}\,
  \sum_{\substack{u_1,u'_1 \\u_2}}
  \langle f, h_{u_1}^{Q_1} \otimes h_{u_2}^{Q_2}\rangle
   \langle g, h_{u'_1}^{Q'_1} \otimes h_{0}^{Q'_2}\rangle
  \langle h_{0}^{Q'_2}\rangle_{Q_2}
\langle T(h_{u_1}^{Q_1} \otimes h_{u_2}^{Q_2}), h_{u'_1}^{Q'_1} \otimes 1\rangle. \label{eq10:Adj_In}
\end{align}
Using Lemma~\ref{lem3.sepnest}, these two sums equal
\begin{eqnarray*}
  \lefteqn{\bigg\langle
\sum_{\substack{Q_1, Q'_1 \textup{ good}\\ \textup{adjacent} }} \quad
\sum_{\substack{u_1,u'_1 \\u_2}} \quad
 h_{u'_1}^{Q'_1} \otimes  \big(\Pi_{b^{u_1 u'_1}_{Q_1 Q'_1}}^{u_2} \big)^*f^{Q_1,u_1},
 g_{\text{good}}
\bigg\rangle}\hspace{0.5cm}\\
 &=& \sum_{\substack{Q_1, Q'_1 \textup{ good}\\ \textup{adjacent} }} \quad
 \sum_{\substack{u_1,u'_1 \\u_2}} \quad
\bigg\langle
\big(\Pi_{b^{u_1 u'_1}_{Q_1 Q'_1}}^{u_2} \big)^*f^{Q_1,u_1},
 \langle g_{\text{good}}, h_{u'_1}^{Q'_1} \rangle_1
\bigg\rangle \\
&=& \sum_{\substack{Q_1, Q'_1 \textup{ good}\\ \textup{adjacent} }} \quad
\sum_{\substack{u_1,u'_1 \\u_2}} \quad
\bigg\langle
\big(\Pi_{b^{u_1 u'_1}_{Q_1 Q'_1}}^{u_2} \big)^*f^{Q_1,u_1},
 g_{\text{good}}^{Q'_1,u'_1}
\bigg\rangle, \\
\end{eqnarray*}
where
\begin{eqnarray*}
  f^{Q_1,u_1} &:=& \langle f, h_{u_1}^{Q_1}\rangle_1 = \int_{X_1} f(x,y)h_{u_1}^{Q_1}(x) \,d\mu_1(x), \\
  g_{\text{good}}^{Q'_1,u'_1} &:=& \langle g_{\text{good}}, h_{u_1}^{Q_1}\rangle_1, \\
  b^{u_1 u'_1}_{Q_1 Q'_1} &:=& \langle T^*(h_{u'_1}^{Q'_1} \otimes 1), h_{u_1}^{Q_1}\rangle_1, \text{ and} \\
  \Pi_{a}^{u_2} \omega &:=&  \sum_{Q'_2 \in \Dd'_2} \quad
\sum_{\substack{Q_2 \textup{ good} \\Q_2 \subset Q'_2 \\ \ell(Q_2) = \delta^r\ell(Q'_2)}}
\langle \omega\rangle_{Q'_2}
\langle a, h_{u_2}^{Q_2}\rangle h_{u_2}^{Q_2}, \qquad u_2 \neq 0.
\end{eqnarray*}

For fixed $u_1, u'_1, u_2$, we can then estimate
\begin{eqnarray}\label{eq5:Adj_In}
 && \bigg| \sum_{\substack{Q_1, Q'_1 \textup{ good}\\ \textup{adjacent} }} \quad
\bigg\langle
\big(\Pi_{b^{u_1 u'_1}_{Q_1 Q'_1}}^{u_2} \big)^*f^{Q_1,u_1},
 g_{\text{good}}^{Q'_1,u'_1}
\bigg\rangle\bigg| \noz\\
 &\leq& \sum_{\substack{Q_1, Q'_1 \textup{ good}\\ \textup{adjacent} }} \quad
 \Big\|\big(\Pi_{b^{u_1 u'_1}_{Q_1 Q'_1}}^{u_2} \big)^*f^{Q_1,u_1}\Big\|_{L^2(\mu_2)} \,
 \big\|g_{\text{good}}^{Q'_1,u'_1}\big\|_{L^2(\mu_2)} \noz\\
 &\ls& \sum_{\substack{Q_1, Q'_1 \\ \textup{adjacent} }} \quad
 \Big\|b^{u_1 u'_1}_{Q_1 Q'_1}\Big\|_{\bmo_{C_K}^{2}(\mu_2)} \,
  \big\|f^{Q_1,u_1}\big\|_{L^2(\mu_2)} \,
 \big\|g^{Q'_1,u'_1}\big\|_{L^2(\mu_2)},
\end{eqnarray}
where we use the paraproduct estimate in Lemma~\ref{lem4.sepnest} and the bound
\[\big\|g_{\text{good}}^{Q'_1,u'_1}\big\|_{L^2(\mu_2)}
= \big\|\langle g_{\text{good}}, h^{Q'_1}_{u'_1}\rangle_1 \big \|_{L^2(\mu_2)}
\leq \big\|\langle g, h^{Q'_1}_{u'_1}\rangle_1 \big \|_{L^2(\mu_2)}
= \big\|g^{Q'_1,u'_1}\big\|_{L^2(\mu_2)}.\]

In Lemma~\ref{adj_in_lem3} below, we will estimate $\Big\|b^{u_1 u'_1}_{Q_1 Q'_1}\Big\|_{\bmo_{C_K}^{2}(\mu_2)} $.

\begin{lem}\label{adj_in_lem3}
  If $Q_1$ and $Q'_1$ are adjacent, then there holds that
  \[\Big\|b^{u_1 u'_1}_{Q_1 Q'_1}\Big\|_{\bmo_{C_K}^{2}(\mu_2)}
  \ls C(\epsilon,\upsilon) + \upsilon^{1/2} \|T\|
  + \|T\| \Big( H_{Q_1,b}^{u_1} +  H_{Q'_1,b}^{u'_1}\Big),\]
  where $\epsilon \in (0,1)$ is fixed in the surgery of adjacent cubes (Section~\ref{subsec:surgery}), $\upsilon \in (0,1)$ is fixed in the random almost-covering by balls (Section~\ref{subsec:ball_cover}), and
  \begin{align*}
   H_{Q_1,b}^{u_1} &:= \Big( \sum_{Q \in \textup{ch}(Q_1)}
  |\langle h_{u_1}^{Q_1}\rangle_Q|^2 \mu_1(Q_b)\Big)^{1/2}, \text{ and}\quad
  H_{Q'_1,b}^{u'_1} := \Big( \sum_{Q' \in \textup{ch}(Q'_1)}
  |\langle h_{u'_1}^{Q'_1}\rangle_{Q'}|^2 \mu_1(Q'_b)\Big)^{1/2}.
  \end{align*}
\end{lem}
\begin{proof}
  Fix a cube~$J \subset X_2$ and a $2$-atom~$a$ of the Hardy space $H^{1,2}(\mu_2)$. That is $\supp a \subset J$, $\int_{J} a \,d\mu = 0$ and $\|a\|_{L^2(\mu_2)} \leq \mu_2(J)^{-1/2}$.
  Let $x_J$ be the centre of~$J$. Then $J \subset B(x_J, C_Q\ell(J)) =: B(J)$.
  Let $V \subset X_2$ be an arbitrary ball such that
  $C_K B(J) = B(x_J, C_KC_Q\ell(J)) \subset V.$
  We need to show that
  \begin{eqnarray}\label{eqa:Adj_In}
    \langle b^{u_1 u'_1}_{Q_1 Q'_1},a \rangle &=&
    |\langle T(h_{u_1}^{Q_1} \otimes a), h_{u'_1}^{Q'_1} \otimes 1 \rangle| \noz\\
    &\leq& |\langle T(h_{u_1}^{Q_1} \otimes a), h_{u'_1}^{Q'_1} \otimes \chi_V\rangle|
    + |\langle T(h_{u_1}^{Q_1} \otimes a), h_{u'_1}^{Q'_1} \otimes \chi_{V^c} \rangle| \\
    &\ls& C(\epsilon,\upsilon) + \upsilon^{1/2} \|T\|
  + \|T\| \big( H_{Q_1,b}^{u_1} +  H_{Q'_1,b}^{u'_1}\big). \noz
    \end{eqnarray}
  We will consider each term in~\eqref{eqa:Adj_In}.
  As the reader may notice, the idea of the proof of Lemma~\ref{adj_in_lem3} is similar to that of Lemma~\ref{lem5.sepnest}, in which we estimate $\big\|b^{u_1 u'_1}_{Q_1 Q'_1}\big\|_{\bmo_{C_K}^{2}(\mu_2)}$ when $Q_1$ and~$Q'_1$ are separated.

  We start with $|\langle T(h_{u_1}^{Q_1} \otimes a), h_{u'_1}^{Q'_1} \otimes \chi_{V^c} \rangle|$.
  Recall that $\supp a \cap \supp  \chi_{V^c} = \emptyset$ and for $y_2\in J$ and~$x_2 \in V^c$, we have
  $$\rho_2(x_J,x_2) \geq C_KC_Q \ell(J) \geq C_K \rho_2(y_2,x_J).$$
  Thus, we can apply Lemma~\ref{lem:property_kernel} to
  $\phi_1 = h^{Q_1}_{u_1}$, $\phi_2 = a$, $\theta_1 = h^{Q'_1}_{u'_1}$ and $\theta_2 = \chi_{V^c}$. We have
\begin{eqnarray*}
  \lefteqn{|\langle T(h_{u_1}^{Q_1} \otimes a), h_{u'_1}^{Q'_1} \otimes \chi_{V^c}\rangle|} \noz\\
   &\ls& \|h^{Q_1}_{u_1}\|_{L^2(\mu_1)}  \|h^{Q'_1}_{u'_1}\|_{L^2(\mu_1)}  \|a\|_{L^2(\mu_2)}  \mu_2(J)^{1/2}
  \int_{V^c} \frac{C_Q^{\al_2}\ell(J)^{\al_2}}
  {\rho_2(x_2,x_J)^{\al_2}\lambda_2(x_J,\rho_2(x_2,x_J))} \,d\mu_2(x_2).
\end{eqnarray*}
 Using Lemma~\ref{upper_dbl_lem1}, the integral above is bounded by
\begin{eqnarray*}
 \ell(J)^{\al_2}
 \int_{X_2 \backslash B(x_J, C_K C_Q \ell(J))}  \frac{\rho_2(x_2,x_J)^{-\al_2}}{\lambda_2(x_J,\rho_2(x_2,x_J))} \,d\mu_2(x_2) \noz\ls \ell(J)^{\al_2} (C_K C_Q \ell(J))^{-\al_2}
 \ls 1.
\end{eqnarray*}
This implies
\begin{equation}\label{eq3:AdjIn}
  |\langle T(h_{u_1}^{Q_1} \otimes a), h_{u'_1}^{Q'_1} \otimes \chi_{V^c}\rangle|
  \ls
  \|a\|_{L^2(\mu_2)}  \mu_2(V)^{1/2} \ls 1.
\end{equation}

Next, we consider the term $|\langle T(h_{u_1}^{Q_1} \otimes a), h_{u'_1}^{Q'_1} \otimes \chi_{V} \rangle|$.
We write the Haar function~$h_{u_1}^{Q_1}$ as a combination of~$\chi_Q$, where $Q \in \text{ch}(Q_1)$.
Similarly,  write~$h_{u'_1}^{Q'_1}$ as a combination of~$\chi_{Q'}$, where $Q' \in \text{ch}(Q'_1)$.
Then we have
\begin{align*}
   &\hspace{-0.5cm} \Big|\Big\langle T\Big(\sum_{Q \in \text{ch}(Q_1)} \langle h_{u_1}^{Q_1}\rangle_{Q} \chi_Q \otimes a\Big),
    \sum_{Q' \in \text{ch}(Q'_1)} \langle h_{u'_1}^{Q'_1}\rangle_{Q'} \chi_{Q'} \otimes \chi_{V} \Big\rangle \Big| \\
   & \leq \sum_{Q \in \text{ch}(Q_1)} \sum_{Q' \in \text{ch}(Q'_1)}
   |\langle h_{u_1}^{Q_1}\rangle_{Q}| |\langle h_{u'_1}^{Q'_1}\rangle_{Q'} |
   |\langle T(\chi_Q \otimes a), \chi_{Q'} \otimes \chi_V \rangle|.
\end{align*}
Notice that, since~$Q_1$ and~$Q'_1$ are adjacent, $\delta^r \ell(Q'_1) \leq \ell(Q_1) \leq \ell(Q'_1)$. Thus, we also have $\delta^r \ell(Q') \leq \ell(Q) \leq \ell(Q')$.
Now for fixed~$Q$ and~$Q'$, we focus on the matrix element $\langle T(\chi_Q \otimes a), \chi_{Q'} \otimes \chi_V \rangle$.
We apply surgery to the pair of cubes $Q$ and~$Q'$ as shown in~\eqref{surgery_eq1}, \eqref{eq3:adj} and~\eqref{eq4:adj} (see Sections~\ref{subsec:surgery} and~\ref{subsec:ball_cover}):
\begin{align*}
  Q = Q_s \cup Q_{\partial} \cup \Delta,  \quad &\text{and} \quad Q' = Q'_s \cap Q'_{\partial} \cap \Delta, \\
   \chi_{\Delta} = \sum_{B \in \mathcal{B}} \widetilde{\chi}_B + \widetilde{\chi}_{\Delta\backslash \cup B}, \quad &\text{and} \quad \widetilde{\chi}_{\Delta\backslash \cup B} =  \widetilde{\chi}_{\Om_i} + \widetilde{\chi}_{\Om_Q} + \widetilde{\chi}_{\Om_{Q'}}.
\end{align*}
Then $\langle T(\chi_Q \otimes a), \chi_{Q'} \otimes \chi_V \rangle$ is decomposed into
\begin{align}
    &\hspace{-0.5cm} \langle T(\chi_Q \otimes a), \chi_{Q'_{s}} \otimes \chi_V \rangle \label{eq:a}\\
   & + \langle T(\chi_Q \otimes a), \chi_{Q'_{\partial}} \otimes \chi_V \rangle \label{eq:b}\\
   & + \langle T(\chi_{Q_s} \otimes a), \chi_{\Delta} \otimes \chi_V \rangle \label{eq:c}\\
   & + \langle T(\chi_{Q_{\partial}} \otimes a), \chi_{\Delta} \otimes \chi_V \rangle \label{eq:d}\\
   & + \langle T(\chi_{\Delta} \otimes a), \widetilde{\chi}_{\Delta\backslash \cup B} \otimes \chi_V \rangle \label{eq:e}\\
   & + \langle T(\widetilde{\chi}_{\Delta\backslash \cup B} \otimes a), \sum_{B \in \mathcal{B}} \widetilde{\chi}_B \otimes \chi_V \rangle \label{eq:f}\\
   & + \langle T(\sum_{B \in \mathcal{B}} \widetilde{\chi}_B  \otimes a), \sum_{B \in \mathcal{B}} \widetilde{\chi}_B \otimes \chi_V \rangle. \label{eq:g}
\end{align}
Equations~\eqref{eq:e}--\eqref{eq:g} can be decomposed further:
\begin{align}
  \langle T(\chi_{\Delta} \otimes a), \widetilde{\chi}_{\Delta\backslash \cup B} \otimes \chi_V \rangle
  = {} & \langle T(\chi_{\Delta} \otimes a),
   \widetilde{\chi}_{\Omega_i} \otimes \chi_V \rangle \label{eq:e1} \\
   & + \langle T(\chi_{\Delta} \otimes a),
    \widetilde{\chi}_{\Omega_Q} \otimes \chi_V \rangle \label{eq:e2}\\
   & + \langle T(\chi_{\Delta} \otimes a),
   \widetilde{\chi}_{\Omega_{Q'}} \otimes \chi_V \rangle, \label{eq:e3}
\end{align}
\begin{align}
  \langle T(\widetilde{\chi}_{\Delta\backslash \cup B} \otimes a), \sum_{B \in \mathcal{B}} \widetilde{\chi}_B \otimes \chi_V \rangle
  = {} & \langle T(\widetilde{\chi}_{\Omega_i} \otimes a), \sum_{B \in \mathcal{B}} \widetilde{\chi}_B \otimes \chi_V \rangle \label{eq:f1}\\
   & + \langle T(\widetilde{\chi}_{\Omega_Q} \otimes a), \sum_{B \in \mathcal{B}} \widetilde{\chi}_B \otimes \chi_V \rangle \label{eq:f2}\\
   & + \langle T(\widetilde{\chi}_{\Omega_{Q'}} \otimes a), \sum_{B \in \mathcal{B}} \widetilde{\chi}_B \otimes \chi_V \rangle, \label{eq:f3}
\end{align}
\begin{align}
  \langle T(\sum_{B \in \mathcal{B}} \widetilde{\chi}_B  \otimes a), \sum_{B \in \mathcal{B}} \widetilde{\chi}_B \otimes \chi_V \rangle
  = {} & \sum_{B} \langle T( \widetilde{\chi}_B  \otimes a), \widetilde{\chi}_B \otimes \chi_V \rangle \label{eq:g1} \\
   & + \sum_{B \neq \widehat{B}} \langle T( \widetilde{\chi}_B  \otimes a), \widetilde{\chi}_{\widehat{B}} \otimes \chi_V \rangle. \label{eq:g2}
\end{align}

Using properties~\eqref{surgery_eq2}, \eqref{ball_cover_eq4} and~\eqref{ball_cover_eq2}, we can show that~\eqref{eq:b}, \eqref{eq:e3} and~\eqref{eq:f3} are dominated by
\[
\|T\| \|a\|_{L^2(\mu_2)} \mu_1(Q)^{1/2} \mu_1(Q'_b)^{1/2} \mu_2(V)^{1/2}.
\]
Similarly, using properties~\eqref{surgery_eq2}, \eqref{ball_cover_eq3} and~\eqref{ball_cover_eq2}, we can show that \eqref{eq:d}, \eqref{eq:e2} and~\eqref{eq:f2} are dominated by
\[
\|T\| \|a\|_{L^2(\mu_2)} \mu_1(Q_b)^{1/2} \mu_1(Q')^{1/2} \mu_2(V)^{1/2}.
\]
Using properties~\eqref{ball_cover_eq1} and~\eqref{ball_cover_eq2}, we can show that~\eqref{eq:e1} and~\eqref{eq:f1} are dominated by
\[
\upsilon^{1/2}\|T\| \|a\|_{L^2(\mu_2)} \mu_1(Q)^{1/2} \mu_1(Q')^{1/2} \mu_2(V)^{1/2}.
\]

We claim that~\eqref{eq:a} and~\eqref{eq:c} are dominated by
\[
C(\epsilon) \|a\|_{L^2(\mu_2)}\mu_1(Q)^{1/2} \mu_1(Q')^{1/2} \mu_2(V)^{1/2}.
\]
We present the proof for~\eqref{eq:a}. The same argument works for~\eqref{eq:c}.
Since $Q \cap Q'_s = \emptyset$, we can use Assumptions~\ref{assum_3} and~\ref{assum_4} to obtain
\begin{align*}
    |\langle T(\chi_Q \otimes a), \chi_{Q'_{s}} \otimes \chi_V \rangle |
   & \leq \int_{Q'_s} \int_{Q} |K_{a,\chi_V}(x_1,y_1)| \,d\mu_1(x_1) \,d\mu_1(y_1) \\
   & \leq C(a,\chi_V)  \int_{Q'_s} \int_{Q}
    \frac{1}{\lambda_1(x_1, \rho_1(x_1,y_1))}  \,d\mu_1(x_1) \,d\mu_1(y_1) \\
   & \ls \|a\|_{L^2(\mu_1)} \|\chi_V\|_{L^2(\mu_1)}
   \int_{Q'_s} \int_{Q} \frac{1}{\lambda_1(x_1, \rho_1(x_1,y_1))}  \,d\mu_1(x_1) \,d\mu_1(y_1).
\end{align*}
Using Lemma~\ref{lem1:surgery}, the integral above is bounded by
\[C(\epsilon)  \mu_1(Q)^{-1/2}\mu_1(Q')^{-1/2} \mu_1(Q) \mu_1(Q'_s)
\leq C(\epsilon)\mu_1(Q)^{1/2}\mu_1(Q')^{1/2}. \]
Thus,
\[
|\langle T(\chi_Q \otimes a), \chi_{Q'_{s}} \otimes \chi_V \rangle |
\ls C(\epsilon) \|a\|_{L^2(\mu_2)}\mu_1(Q)^{1/2} \mu_1(Q')^{1/2} \mu_2(V)^{1/2}.
\]

We claim that~\eqref{eq:g1} and~\eqref{eq:g2} are dominated by
\[C(\epsilon,\upsilon) \|a\|_{L^2(\mu_2)}\mu_1(Q)^{1/2} \mu_1(Q')^{1/2} \mu_2(V)^{1/2}.\]
The constant depends on~$\epsilon$ because the balls $B \in \mathcal{B}$ depend on~$\widetilde{\Delta}$, which in turn depends on~$\epsilon$.
We start with~\eqref{eq:g1}. Using Assumption~\ref{assum_6} and the fact that $\Lambda B \subset \Delta$, we have
\begin{align*}
  |\sum_{B} \langle T( \widetilde{\chi}_B  \otimes a), \widetilde{\chi}_B \otimes \chi_V \rangle|
  \ls & \, \, C(\epsilon,\upsilon)\|a\|_{L^2(\mu_2)}  \mu_2(V)^{1/2} \mu_1(\Lambda B)\\
  \leq &  \, \, C(\epsilon,\upsilon)\|a\|_{L^2(\mu_2)}  \mu_2(V)^{1/2} \mu_1(Q) \\
  \leq & \, \, C(\epsilon,\upsilon)\|a\|_{L^2(\mu_2)}  \mu_2(V)^{1/2} \mu_1(Q)^{1/2} \mu_1(Q')^{1/2}.
\end{align*}

For~\eqref{eq:g2}, we can use Assumptions~\ref{assum_3} and~\ref{assum_4}, and Lemma~\ref{lem1:ball_cover}, to get
\begin{align*}
    & \hspace{-1cm} |\sum_{B \neq \widehat{B}} \langle T( \widetilde{\chi}_B  \otimes a), \widetilde{\chi}_{\widehat{B}} \otimes \chi_V \rangle| \\
    \leq & \sum_{B \neq \widehat{B}}
   | \langle T(\chi_{(1+\epsilon')B}  \otimes a), \chi_{(1+\epsilon')\widehat{B}} \otimes \chi_V \rangle| \\
   \leq & \sum_{B \neq \widehat{B}}  \int_{(1+\epsilon')B} \int_{(1+\epsilon')\widehat{B}} |K_{a,\chi_V}(x_1,y_1)| \,d\mu_1(x_1) \,d\mu_1(y_1) \\
   \leq & \sum_{B \neq \widehat{B}} C(a,\chi_V)
    \int_{(1+\epsilon')B} \int_{(1+\epsilon')\widehat{B}}
    \frac{1}{\lambda_1(x_1, \rho_1(x_1,y_1))}  \,d\mu_1(x_1) \,d\mu_1(y_1) \\
   \ls &  \sum_{B \neq \widehat{B}} \|a\|_{L^2(\mu_1)} \|\chi_V\|_{L^2(\mu_1)}
   \int_{(1+\epsilon')B} \int_{(1+\epsilon')\widehat{B}} \frac{1}{\lambda_1(x_1, \rho_1(x_1,y_1))}  \,d\mu_1(x_1) \,d\mu_1(y_1)\\
    \ls  & \, \, C(\epsilon, \upsilon) \|a\|_{L^2(\mu_1)} \mu_2(V)^{1/2}
   \mu_1((1+\epsilon')B)^{1/2} \mu_1((1+\epsilon')\widehat{B})^{1/2}\\
  \leq  & \, \,   C(\epsilon,\upsilon) \mu_2(J)^{-1/2} \mu_2(V)^{1/2} \mu_1(Q)^{1/2} \mu_1(Q')^{1/2} \\
   \ls  & \, \,   C(\epsilon,\upsilon) \mu_1(Q)^{1/2} \mu_1(Q')^{1/2} .
\end{align*}

Combining all the estimates above, we obtain
\begin{align*}
 & \hspace{-1cm}  |\langle T(h_{u_1}^{Q_1} \otimes a, h_{u'_1}^{Q'_1} \otimes \chi_{V}) \rangle| \\
   \leq & \sum_{Q \in \text{ch}(Q_1)} \sum_{Q' \in \text{ch}(Q'_1)}
   |\langle h_{u_1}^{Q_1}\rangle_{Q}| |\langle h_{u'_1}^{Q'_1}\rangle_{Q'} |
   |\langle T(\chi_Q \otimes a), \chi_{Q'} \otimes \chi_V \rangle|\\
   \ls & \,\,
   \sum_{Q \in \text{ch}(Q_1)} \sum_{Q' \in \text{ch}(Q'_1)}
  |\langle h_{u_1}^{Q_1}\rangle_{Q}| |\langle h_{u'_1}^{Q'_1}\rangle_{Q'} |
   \\
  & \times \Big[ \|T\| \mu_1(Q)^{1/2} \mu_1(Q'_b)^{1/2}
  + \|T\| \mu_1(Q_b)^{1/2} \mu_1(Q')^{1/2}
  + \upsilon^{1/2} \|T\| \mu_1(Q)^{1/2} \mu_1(Q')^{1/2} \\
  &\hspace{0.5cm} + C(\epsilon, \upsilon)  \mu_1(Q)^{1/2} \mu_1(Q')^{1/2} \Big]\\
  =: & \, \,
    \sum_{Q \in \text{ch}(Q_1)} \sum_{Q' \in \text{ch}(Q'_1)}
  |\langle h_{u_1}^{Q_1}\rangle_{Q}| |\langle h_{u'_1}^{Q'_1}\rangle_{Q'} |
   \big[(\mathcal F_1) + (\mathcal F_2) + (\mathcal F_3) + (\mathcal F_4)\big].
\end{align*}
Using property~\eqref{eq2:pro_Haar} of Haar functions, we have
 \[
 |\langle h_{u_1}^{Q_1}\rangle_{Q}| \ls \mu_1(Q)^{-1/2}
  \quad \text{and}  \quad |\langle h_{u'_1}^{Q'_1}\rangle_{Q'} | \ls \mu_1(Q')^{-1/2}
  .\]
Thus, the sum involving~$(\mathcal F_3)$ is dominated by
$\upsilon^{1/2} \|T\| $,
and the sum involving~$(\mathcal F_4)$ is dominated by
$C(\epsilon, \upsilon)$.
Next, using H\"{o}lder's inequality, the sum involving~$(\mathcal F_1)$ is estimates as
\begin{align*}
    & \hspace{-1cm} \|T\| \sum_{Q \in \text{ch}(Q_1)} \sum_{Q' \in \text{ch}(Q'_1)}
   |\langle h_{u'_1}^{Q'_1}\rangle_{Q'} |  \mu_1(Q'_b)^{1/2} \\
   & =\|T\|  M_{Q_1}   \sum_{Q' \in \text{ch}(Q'_1)}
   |\langle h_{u'_1}^{Q'_1}\rangle_{Q'} |  \mu_1(Q'_b)^{1/2}\\
   & \leq \|T\|  M_{Q_1} \Big(\sum_{Q' \in \text{ch}(Q'_1)}
   |\langle h_{u'_1}^{Q'_1}\rangle_{Q'} |^2  \mu_1(Q'_b) \Big)^{1/2}
   \Big(\sum_{Q' \in \text{ch}(Q'_1)} 1 \Big)^{1/2}
    \ls \|T\|   H^{u'_1}_{Q'_1,b}.
\end{align*}
Similarly, the sum involving~$(\mathcal F_2)$ is reduced to $\|T\| H^{u_1}_{Q_1,b}$.
This implies
\begin{align}\label{eq4:Adj_In}
  |\langle T(h_{u_1}^{Q_1} \otimes a, h_{u'_1}^{Q'_1} \otimes \chi_{V}) \rangle| \ls  \Big[C(\epsilon,\upsilon) + \upsilon^{1/2} \|T\|
  + \|T\| \big( H_{Q_1,b}^{u_1} +  H_{Q'_1,b}^{u'_1}\big)\Big] \|a\|_{L^2(\mu_2)} \mu_2(V)^{1/2}.
\end{align}
Hence, equations~\eqref{eq3:AdjIn} and~\eqref{eq4:Adj_In} have established Lemma~\ref{adj_in_lem3}.
\end{proof}

Now, using Lemma~\ref{adj_in_lem3}, we return to considering~\eqref{eq5:Adj_In}.
We note that $ \big\|f^{Q_1,u_1}\big\|_{L^2(\mu_2)} \ls \|f\|_{L^2(\mu)}$,
 $ \big\|g^{Q'_1,u'_1}\big\|_{L^2(\mu_2)} \ls \|g\|_{L^2(\mu)}$, and for each $Q_1 \in \Dd_1$, there are only finitely many~$Q'_1 \in \Dd'_1$ adjacent to~$Q_1$.
Thus
\begin{align}\label{eq6:Adj_In}
  &\hspace{-1cm} \sum_{\substack{Q_1, Q'_1 \\ \textup{adjacent} }} \quad
 \Big\|b^{u_1 u'_1}_{Q_1 Q'_1}\Big\|_{\bmo_{C_K}^{2}(\mu_2)} \,
  \big\|f^{Q_1,u_1}\big\|_{L^2(\mu_2)} \,
 \big\|g^{Q'_1,u'_1}\big\|_{L^2(\mu_2)} \noz\\
  \ls & \,\, C(\epsilon,\upsilon) \|f\|_{L^2(\mu)} \|g\|_{L^2(\mu)}
  + \upsilon^{1/2} \|T\| \|f\|_{L^2(\mu)} \|g\|_{L^2(\mu)} \noz\\
  & + \|T\| \sum_{\substack{Q_1, Q'_1 \\ \textup{adjacent} }}
   H_{Q_1,b}^{u_1}  \big\|f^{Q_1,u_1}\big\|_{L^2(\mu_2)} \,
 \big\|g^{Q'_1,u'_1}\big\|_{L^2(\mu_2)} \noz\\
 & + \|T\| \sum_{\substack{Q_1, Q'_1 \\ \textup{adjacent} }}
   H_{Q'_1,b}^{u'_1}  \big\|f^{Q_1,u_1}\big\|_{L^2(\mu_2)} \,
 \big\|g^{Q'_1,u'_1}\big\|_{L^2(\mu_2)} \noz\\
 =: & \, (\widehat{\mathcal F}_1) + (\widehat{\mathcal F}_2)  + (\widehat{\mathcal F}_3)  + (\widehat{\mathcal F}_4) .
\end{align}
Term~$(\widehat{\mathcal F}_3)$ is estimated by
\[\|T\| \Big( \sum_{Q_1 \in \Dd_1} ( H_{Q_1,b}^{u_1})^2  \|f^{Q_1,u_1}\|^2_{L^2(\mu_2)} \Big)^{1/2} \|g\|_{L^2(\mu)}.\]
By Lemma~\ref{lem2:surgery}, we have $\mathbb{E}(\mu_1(Q_b)) \ls \epsilon^{\eta} \mu_1(Q)$, where $\epsilon >0$ and $\eta>0$. Thus,
\begin{equation}\label{eq8:Adj_In}
\mathbb{E}[(H_{Q_1,b}^{u_1})^2] \ls \epsilon^{\eta}  |\langle h_{u_1}^{Q_1}\rangle_{Q}|^2 \mu_1(Q) \ls \epsilon^{\eta}.
\end{equation}
This implies
\[\mathbb{E}\Big( \sum_{Q_1 \in \Dd_1} ( H_{Q_1,b}^{u_1})^2  \|f^{Q_1,u_1}\|^2_{L^2(\mu_2)} \Big)^{1/2}
\ls \epsilon^{\eta/2} \Big( \sum_{Q_1 \in \Dd_1} \|f^{Q_1,u_1}\|^2_{L^2(\mu_2)} \Big)^{1/2}
= \epsilon^{\eta/2} \|f\|_{L^2(\mu)}. \]
Therefore, $(\widehat{\mathcal F}_3)$ is dominated by $\epsilon^{\eta/2} \|T\| \|f\|_{L^2(\mu)} \|g\|_{L^2(\mu)}.$
Similarly, we can obtain that $(\widehat{\mathcal F}_4)$ is also dominated by $\epsilon^{\eta/2} \|T\| \|f\|_{L^2(\mu)} \|g\|_{L^2(\mu)}.$

Fixing $\epsilon$ and $\upsilon$ sufficiently small, \eqref{eq6:Adj_In} becomes
\begin{align}\label{eq7:Adj_In}
  &\hspace{-1cm}  \sum_{\substack{Q_1, Q'_1 \\ \textup{adjacent} }} \quad
 \Big\|b^{u_1 u'_1}_{Q_1 Q'_1}\Big\|_{\bmo_{C_K}^{2}(\mu_2)} \,
  \big\|f^{Q_1,u_1}\big\|_{L^2(\mu_2)} \,
 \big\|g^{Q'_1,u'_1}\big\|_{L^2(\mu_2)} \noz \\
  &\ls \, \, C \|f\|_{L^2(\mu)} \|g\|_{L^2(\mu)} + 0.01 \|T\|  \|f\|_{L^2(\mu)} \|g\|_{L^2(\mu)} \noz \\
   &\ls  \, \,  (C + 0.01 \|T\|) \|f\|_{L^2(\mu)} \|g\|_{L^2(\mu)},
\end{align}
completing the estimation of~\eqref{eq10:Adj_In}.

Hence, equations~\eqref{eq5a:Adj_In}, \eqref{eq9:Adj_In} together with~\eqref{eq7:Adj_In} has completed the \emph{Adj/Nes} case.
\subsection{Adjacent/Adjacent cubes}\label{subsec:adj_adj}
In this section, we consider the \emph{Adj/Adj} case. That is, we estimate~\eqref{coreeq.1} when $Q_1, Q'_1$ are adjacent and $Q_2, Q'_2$ are adjacent:
\begin{eqnarray*}
  \delta^r \ell(Q'_1) \leq \ell(Q_1) \leq \ell(Q'_1),
  &&  \quad \rho_1(Q_1,Q'_1) \leq \mathcal{C}\ell(Q_1)^{\gamma_1}\ell(Q'_1)^{1-\gamma_1}, \quad \text{and} \\
  \delta^r \ell(Q'_2) \leq \ell(Q_2) \leq \ell(Q'_2),
  && \quad \rho_2(Q_2,Q'_2) \leq \mathcal{C}\ell(Q_2)^{\gamma_2}\ell(Q'_2)^{1-\gamma_2},
\end{eqnarray*}
where $\mathcal{C}:= 2A_0C_QC_K$.
There are 16 terms in~\eqref{coreeq.1}, which we need to bound. We will give the proof of the first term. The other terms can be handled analogously.
We are to bound
\begin{equation}\label{adj_adj_eq.1}
  \sum_{\substack{Q_1,Q'_1  \\ \text{adjacent} }} \quad
  \sum_{\substack{Q_2,Q'_2\\ \text{adjacent} }} \quad
  \sum_{\substack{u_1,u'_1 \\u_2,u'_2}}
  \langle f, h_{u_1}^{Q_1} \otimes h_{u_2}^{Q_2}\rangle
  \langle g, h_{u'_1}^{Q'_1} \otimes h_{u'_2}^{Q'_2}\rangle
  \langle T(h_{u_1}^{Q_1} \otimes h_{u_2}^{Q_2}), h_{u'_1}^{Q'_1} \otimes h_{u'_2}^{Q'_2}\rangle.
\end{equation}

Lemma~\ref{lem1:AdjAdj} estimates the matrix element $ |\langle T(h_{u_1}^{Q_1} \otimes h_{u_2}^{Q_2}), h_{u'_1}^{Q'_1} \otimes h_{u'_2}^{Q'_2}\rangle |$.
\begin{lem}\label{lem1:AdjAdj}
  Given pairs of adjacent cubes $(Q_1,Q_1')$ and $(Q_2,Q'_2)$, we have
  \begin{align*}
      |\langle T(h_{u_1}^{Q_1} \otimes h_{u_2}^{Q_2}), h_{u'_1}^{Q'_1} \otimes h_{u'_2}^{Q'_2}\rangle |
    \ls
  C(\epsilon,\upsilon) + \upsilon^{1/2} \|T\|
  + \|T\| \Big( H_{Q_1,b}^{u_1} +  H_{Q'_1,b}^{u'_1}
  +  H_{Q_2,b}^{u_2} +  H_{Q'_2,b}^{u'_2}\Big),
 \end{align*}
  where $\epsilon \in (0,1)$ is fixed in the surgery of adjacent cubes (Section~\ref{subsec:surgery}), $\upsilon \in (0,1)$ is fixed in the random almost-covering by balls (Section~\ref{subsec:ball_cover}), and for $i = 1,2$
  \begin{align*}
   H_{Q_i,b}^{u_i} &:= \Big( \sum_{Q \in \textup{ch}(Q_i)}
  |\langle h_{u_i}^{Q_i}\rangle_Q|^2 \mu_i(Q_b)\Big)^{1/2}, \text{ and}\quad
  H_{Q'_i,b}^{u'_i} := \Big( \sum_{Q' \in \textup{ch}(Q'_i)}
  |\langle h_{u'_1}^{Q'_i}\rangle_{Q'}|^2 \mu_i(Q'_b)\Big)^{1/2}.
  \end{align*}
\end{lem}
\begin{proof}
  Let $Q \in \text{ch}(Q_1)$, $Q' \in \text{ch}(Q'_1)$, $P \in \text{ch}(Q_2)$ and $P' \in \text{ch}(Q'_2)$.
  We write
  \begin{align*}
     \langle T(h_{u_1}^{Q_1} \otimes h_{u_2}^{Q_2}), h_{u'_1}^{Q'_1} \otimes h_{u'_2}^{Q'_2}\rangle
    & = \sum_{\substack{Q, Q' \\ P,P'}}
    \langle T(h_{u_1}^{Q_1}\chi_Q \otimes h_{u_2}^{Q_2}\chi_P), h_{u'_1}^{Q'_1}\chi_{Q'} \otimes h_{u'_2}^{Q'_2} \chi_{P'}\rangle \\
    & = \sum_{\substack{Q, Q' \\ P,P'}}
    \langle h_{u_1}^{Q_1}\rangle_{Q}\, \langle h_{u_2}^{Q_2}\rangle_{P} \,
    \langle h_{u'_1}^{Q'_1}\rangle_{Q'}\, \langle h_{u'_2}^{Q'_2} \rangle_{P'}\,
    \langle T(\chi_Q \otimes \chi_P), \chi_{Q'} \otimes \chi_{P'}\rangle.
  \end{align*}
  Now, we focus on the term $ \langle T(\chi_Q \otimes \chi_P), \chi_{Q'} \otimes \chi_{P'}\rangle$.

  Recall that in the \emph{Adj/Nes} case when we consider $\langle T(\chi_Q \otimes a), \chi_{Q'} \otimes \chi_{V}\rangle$, we perform surgery on the pair $(Q,Q')$.
  In addition to that, here in the \emph{Adj/Adj} case we also perform surgery on the pair of cubes $(P,P')$.
  In particular, we will use the following decompositions:
  \begin{align*}
  Q = Q_s \cup Q_{\partial} \cup \Delta_1,  \quad &\text{and} \quad Q' = Q'_s \cup Q'_{\partial} \cup \Delta_1, \\
   P= P_s \cup P_{\partial} \cup \Delta_2,  \quad &\text{and} \quad P' = P'_s \cup P'_{\partial} \cup \Delta_2, \\
   \chi_{\Delta_1} = \sum_{B_1 \in \mathcal{B}_1} \widetilde{\chi}_{B_1} + \widetilde{\chi}_{\Delta_1\backslash \cup B_1}, \quad &\text{and} \quad \widetilde{\chi}_{\Delta_1\backslash \cup B_1} =  \widetilde{\chi}_{\Om_i} + \widetilde{\chi}_{\Om_Q} + \widetilde{\chi}_{\Om_{Q'}}, \\
   \chi_{\Delta_2} = \sum_{B_2 \in \mathcal{B}_2} \widetilde{\chi}_{B_2} + \widetilde{\chi}_{\Delta_2\backslash \cup B_2}, \quad &\text{and} \quad \widetilde{\chi}_{\Delta_2\backslash \cup B_2} =  \widetilde{\chi}_{\Om_j} + \widetilde{\chi}_{\Om_P} + \widetilde{\chi}_{\Om_{P'}}.
\end{align*}
  Here $\Delta_1 := Q \cap Q'$ and $\Delta_2 := P \cap P'$. The sets $P_s$, $P_{\partial}$, $P'_s$, $P'_{\partial}$ are defined analogously to $Q_s$, $Q_{\partial}$, $Q'_s$, $Q'_{\partial}$, respectively (see Section~\ref{subsec:surgery}).
  $\mathcal{B}_1$ and $\mathcal{B}_2$ are random almost-coverings by balls of each factor~$X_1$ and $X_2$.
  The sets $\Omega_j$, $\Omega_P$ and $\Omega_{P'}$ are defined analogously to
  $\Omega_i$, $\Omega_Q$ and $\Omega_{Q'}$, respectively (see Section~\ref{subsec:ball_cover}).

  Then $ \langle T(\chi_Q \otimes \chi_P), \chi_{Q'} \otimes \chi_{P'}\rangle$ can be split into
  \begin{align}
    &\hspace{-0.5cm} \langle T(\chi_Q \otimes \chi_P), \chi_{Q'_{s}} \otimes \chi_{P'} \rangle \label{eq:1_AdjAdj}\\
   & + \langle T(\chi_Q \otimes \chi_P), \chi_{Q'_{\partial}} \otimes \chi_{P'} \rangle \label{eq:2_AdjAdj}\\
   & + \langle T(\chi_{Q_s} \otimes \chi_P), \chi_{\Delta_1} \otimes \chi_{P'}  \rangle \label{eq:3_AdjAdj}\\
   & + \langle T(\chi_{Q_{\partial}} \otimes \chi_P), \chi_{\Delta_1} \otimes \chi_{P'}  \rangle \label{eq:4_AdjAdj}\\
   & + \langle T(\chi_{\Delta_1} \otimes \chi_P), \chi_{\Delta_1} \otimes \chi_{P'_s}  \rangle \label{eq:5_AdjAdj}\\
   & + \langle T(\chi_{\Delta_1} \otimes \chi_P), \chi_{\Delta_1} \otimes \chi_{P'_{\partial}}  \rangle \label{eq:6_AdjAdj}\\
   & + \langle T(\chi_{\Delta_1} \otimes \chi_{P_s}), \chi_{\Delta_1} \otimes \chi_{\Delta_2}  \rangle \label{eq:7_AdjAdj}\\
   & + \langle T(\chi_{\Delta_1} \otimes \chi_{P_{\partial}}), \chi_{\Delta_1} \otimes \chi_{\Delta_2}  \rangle \label{eq:8_AdjAdj}\\
   & + \langle T(\chi_{\Delta_1} \otimes \chi_{\Delta_2}), \widetilde{\chi}_{\Delta_1\backslash \cup B_1} \otimes \chi_{\Delta_2}  \rangle \label{eq:9_AdjAdj}\\
   & + \Big\langle T(\widetilde{\chi}_{\Delta_1\backslash \cup B_1} \otimes \chi_{\Delta_2}), \sum_{B_1 \in \mathcal{B}_1} \widetilde{\chi}_{B_1} \otimes \chi_{\Delta_2}  \Big\rangle \label{eq:10_AdjAdj}\\
   & + \Big\langle T( \sum_{B_1 \in \mathcal{B}_1} \widetilde{\chi}_{B_1} \otimes \chi_{\Delta_2}), \sum_{B_1 \in \mathcal{B}_1} \widetilde{\chi}_{B_1} \otimes \widetilde{\chi}_{\Delta_2\backslash \cup B_2}   \Big\rangle \label{eq:11_AdjAdj}\\
   & + \Big\langle T( \sum_{B_1 \in \mathcal{B}_1} \widetilde{\chi}_{B_1} \otimes \widetilde{\chi}_{\Delta_2\backslash \cup B_2} ),
   \sum_{B_1 \in \mathcal{B}_1} \widetilde{\chi}_{B_1} \otimes
   \sum_{B_2 \in \mathcal{B}_2} \widetilde{\chi}_{B_2} \Big\rangle \label{eq:12_AdjAdj} \\
   & + \Big\langle T( \sum_{B_1 \in \mathcal{B}_1} \widetilde{\chi}_{B_1} \otimes \sum_{B_2 \in \mathcal{B}_2} \widetilde{\chi}_{B_2}) ,
   \sum_{B_1 \in \mathcal{B}_1} \widetilde{\chi}_{B_1} \otimes
   \sum_{B_2 \in \mathcal{B}_2} \widetilde{\chi}_{B_2} \Big\rangle. \label{eq:13_AdjAdj}
\end{align}
Terms~\eqref{eq:9_AdjAdj}--\eqref{eq:13_AdjAdj} can be further decomposed:
\begin{align}
   \langle T(\chi_{\Delta_1} \otimes \chi_{\Delta_2}), \widetilde{\chi}_{\Delta_1\backslash \cup B_1} \otimes \chi_{\Delta_2}  \rangle
  & =  \, \langle T(\chi_{\Delta_1} \otimes \chi_{\Delta_2}), \widetilde{\chi}_{\Omega_i} \otimes \chi_{\Delta_2}  \rangle \label{eq:9.1_AdjAdj}\\
   &\hspace{0.5cm} + \langle T(\chi_{\Delta_1} \otimes \chi_{\Delta_2}), \widetilde{\chi}_{\Omega_Q} \otimes \chi_{\Delta_2}  \rangle \label{eq:9.2_AdjAdj}\\
   &\hspace{0.5cm} + \langle T(\chi_{\Delta_1} \otimes \chi_{\Delta_2}), \widetilde{\chi}_{\Omega_{Q'}} \otimes \chi_{\Delta_2}  \rangle,
   \label{eq:9.3_AdjAdj}
\end{align}
\begin{align}
  &\hspace{-1cm}\Big\langle T(\widetilde{\chi}_{\Delta_1\backslash \cup B_1} \otimes \chi_{\Delta_2}), \sum_{B_1 \in \mathcal{B}_1} \widetilde{\chi}_{B_1} \otimes \chi_{\Delta_2}  \Big\rangle \noz \\
  =& \, \Big\langle T(\widetilde{\chi}_{\Omega_i} \otimes \chi_{\Delta_2}), \sum_{B_1 \in \mathcal{B}_1} \widetilde{\chi}_{B_1} \otimes \chi_{\Delta_2}  \Big \rangle  \label{eq:10.1_AdjAdj}\\
  &\hspace{+0.5cm} +  \Big\langle T(\widetilde{\chi}_{\Omega_Q} \otimes \chi_{\Delta_2}), \sum_{B_1 \in \mathcal{B}_1} \widetilde{\chi}_{B_1} \otimes \chi_{\Delta_2}  \Big\rangle  \label{eq:10.2_AdjAdj}\\
  &\hspace{+0.5cm} +  \Big\langle T(\widetilde{\chi}_{\Omega_{Q'}} \otimes \chi_{\Delta_2}), \sum_{B_1 \in \mathcal{B}_1} \widetilde{\chi}_{B_1} \otimes \chi_{\Delta_2}  \Big\rangle  \label{eq:10.3_AdjAdj},
\end{align}
\begin{align}
  &\hspace{-1cm}  \Big\langle T\Big( \sum_{B_1 \in \mathcal{B}_1} \widetilde{\chi}_{B_1} \otimes \chi_{\Delta_2}\Big), \sum_{B_1 \in \mathcal{B}_1} \widetilde{\chi}_{B_1} \otimes \widetilde{\chi}_{\Delta_2\backslash \cup B_2}   \Big\rangle  \noz \\
  = & \, \Big \langle T\Big( \sum_{B_1 \in \mathcal{B}_1} \widetilde{\chi}_{B_1} \otimes \chi_{\Delta_2}\Big), \sum_{B_1 \in \mathcal{B}_1} \widetilde{\chi}_{B_1} \otimes \widetilde{\chi}_{\Omega_j}   \Big\rangle \label{eq:11.1_AdjAdj} \\
  & + \Big\langle T\Big( \sum_{B_1 \in \mathcal{B}_1} \widetilde{\chi}_{B_1} \otimes \chi_{\Delta_2}\Big), \sum_{B_1 \in \mathcal{B}_1} \widetilde{\chi}_{B_1} \otimes \widetilde{\chi}_{\Omega_P}   \Big\rangle \label{eq:11.2_AdjAdj} \\
  & + \Big\langle T\Big( \sum_{B_1 \in \mathcal{B}_1} \widetilde{\chi}_{B_1} \otimes \chi_{\Delta_2}\Big), \sum_{B_1 \in \mathcal{B}_1} \widetilde{\chi}_{B_1} \otimes \widetilde{\chi}_{\Omega_{P'}}   \Big\rangle \label{eq:11.3_AdjAdj},
\end{align}
\begin{align}
   &\hspace{-1cm} \Big\langle T\Big( \sum_{B_1 \in \mathcal{B}_1} \widetilde{\chi}_{B_1} \otimes \widetilde{\chi}_{\Delta_2\backslash \cup B_2}\Big) ,
   \sum_{B_1 \in \mathcal{B}_1} \widetilde{\chi}_{B_1} \otimes
   \sum_{B_2 \in \mathcal{B}_2} \widetilde{\chi}_{B_2} \Big\rangle  \noz \\
  = & \, \Big\langle T\Big( \sum_{B_1 \in \mathcal{B}_1} \widetilde{\chi}_{B_1} \otimes \widetilde{\chi}_{\Omega_j}\Big) ,
   \sum_{B_1 \in \mathcal{B}_1} \widetilde{\chi}_{B_1} \otimes
   \sum_{B_2 \in \mathcal{B}_2} \widetilde{\chi}_{B_2} \Big\rangle \label{eq:12.1_AdjAdj} \\
   & + \Big\langle T\Big( \sum_{B_1 \in \mathcal{B}_1} \widetilde{\chi}_{B_1} \otimes \widetilde{\chi}_{\Omega_P} \Big),
   \sum_{B_1 \in \mathcal{B}_1} \widetilde{\chi}_{B_1} \otimes
   \sum_{B_2 \in \mathcal{B}_2} \widetilde{\chi}_{B_2} \Big\rangle \label{eq:12.2_AdjAdj} \\
   & + \Big\langle T\Big( \sum_{B_1 \in \mathcal{B}_1} \widetilde{\chi}_{B_1} \otimes \widetilde{\chi}_{\Omega_{P'}}\Big) ,
   \sum_{B_1 \in \mathcal{B}_1} \widetilde{\chi}_{B_1} \otimes
   \sum_{B_2 \in \mathcal{B}_2} \widetilde{\chi}_{B_2} \Big\rangle, \label{eq:12.3_AdjAdj}
\end{align}
\begin{align}
  &\hspace{-1cm}  \Big\langle T\Big( \sum_{B_1 \in \mathcal{B}_1} \widetilde{\chi}_{B_1} \otimes \sum_{B_2 \in \mathcal{B}_2} \widetilde{\chi}_{B_2} \Big),
   \sum_{B_1 \in \mathcal{B}_1} \widetilde{\chi}_{B_1} \otimes
   \sum_{B_2 \in \mathcal{B}_2} \widetilde{\chi}_{B_2} \Big\rangle \noz \\
  = & \, \sum_{B_1, B_2} \Big\langle T\Big( \widetilde{\chi}_{B_1} \otimes  \widetilde{\chi}_{B_2}\Big) ,
    \widetilde{\chi}_{B_1} \otimes
    \widetilde{\chi}_{B_2} \Big\rangle \label{eq:13.1_AdjAdj} \\
  & + \sum_{B_1, B_2 \neq \widehat{B}_2} \Big\langle T\Big( \widetilde{\chi}_{B_1} \otimes  \widetilde{\chi}_{B_2} \Big),
    \widetilde{\chi}_{B_1} \otimes
    \widetilde{\chi}_{\widehat{B}_2} \Big\rangle \label{eq:13.2_AdjAdj} \\
  & + \sum_{B_1 \neq \widehat{B}_1, B_2} \Big\langle T( \widetilde{\chi}_{B_1} \otimes  \widetilde{\chi}_{B_2}) ,
    \widetilde{\chi}_{\widehat{B}_1} \otimes
    \widetilde{\chi}_{B_2} \Big\rangle \label{eq:13.3_AdjAdj} \\
  & + \sum_{B_1 \neq \widehat{B}_1, B_2 \neq \widehat{B}_2} \Big\langle T\Big( \widetilde{\chi}_{B_1} \otimes  \widetilde{\chi}_{B_2} \Big),
    \widetilde{\chi}_{\widehat{B}_1} \otimes
    \widetilde{\chi}_{\widehat{B}_2} \Big\rangle \label{eq:13.4_AdjAdj}.
\end{align}

Using properties~\eqref{surgery_eq2}, \eqref{ball_cover_eq4} and~\eqref{ball_cover_eq2}, we can show that~\eqref{eq:2_AdjAdj}, \eqref{eq:9.3_AdjAdj} and \eqref{eq:10.3_AdjAdj} are dominated by
\[\|T\| \mu_1(Q)^{1/2} \mu_2(P)^{1/2} \mu_1(Q'_b)^{1/2} \mu_2(P')^{1/2};\]
and \eqref{eq:6_AdjAdj}, \eqref{eq:11.3_AdjAdj} and \eqref{eq:12.3_AdjAdj} by
\[\|T\| \mu_1(Q)^{1/2} \mu_2(P)^{1/2} \mu_1(Q')^{1/2} \mu_2(P'_b)^{1/2}.\]

Similarly, using properties~\eqref{surgery_eq2}, \eqref{ball_cover_eq3} and~\eqref{ball_cover_eq2}, we can show that \eqref{eq:4_AdjAdj}, \eqref{eq:9.2_AdjAdj} and \eqref{eq:10.2_AdjAdj} are dominated by
\[\|T\| \mu_1(Q_b)^{1/2} \mu_2(P)^{1/2} \mu_1(Q')^{1/2} \mu_2(P')^{1/2};\]
and \eqref{eq:8_AdjAdj}, \eqref{eq:11.2_AdjAdj} and \eqref{eq:12.2_AdjAdj} by
\[\|T\| \mu_1(Q)^{1/2} \mu_2(P_b)^{1/2} \mu_1(Q')^{1/2} \mu_2(P')^{1/2}.\]

Using properties~\eqref{ball_cover_eq1} and~\eqref{ball_cover_eq2}, we can show that \eqref{eq:9.1_AdjAdj}, \eqref{eq:10.1_AdjAdj}, \eqref{eq:11.1_AdjAdj} and \eqref{eq:12.1_AdjAdj} are dominated by
\[\upsilon^{1/2} \|T\| \mu_1(Q)^{1/2} \mu_2(P)^{1/2} \mu_1(Q')^{1/2} \mu_2(P')^{1/2}.\]

We can also show that~\eqref{eq:1_AdjAdj}, \eqref{eq:3_AdjAdj}, \eqref{eq:5_AdjAdj} and \eqref{eq:7_AdjAdj} are dominated by
\[C(\epsilon) \mu_1(Q)^{1/2}\mu_2(P)^{1/2} \mu_1(Q')^{1/2} \mu_2(P')^{1/2}.\]
Their proofs are similar to that of~\eqref{eq:a} in the \emph{Adj/Nes} case, where we require the use of Assumption~\ref{assum_3}, Assumption~\ref{assum_4} and Lemma~\ref{lem1:surgery}.

We claim that terms \eqref{eq:13.1_AdjAdj}--\eqref{eq:13.4_AdjAdj} are estimated by
\[C(\epsilon,\upsilon) \mu_1(Q)^{1/2} \mu_2(P)^{1/2} \mu_1(Q')^{1/2} \mu_2(P')^{1/2}.\]
For~\eqref{eq:13.1_AdjAdj}, we use Assumption~\ref{assum_5} of the full weak boundedness property to get
\begin{align*}
  \big|\sum_{B_1, B_2} \langle T( \widetilde{\chi}_{B_1} \otimes  \widetilde{\chi}_{B_2} ,
    \widetilde{\chi}_{B_1} \otimes
    \widetilde{\chi}_{B_2} \rangle\big|
    \leq & \,\,\sum_{B_1, B_2} C(\epsilon, \upsilon) \mu_1(\Lambda B_1) \mu_2(\Lambda B_2) \\
  \leq  & \,\,C(\epsilon, \upsilon) \mu_1(Q) \mu_2(P) \\
  \leq & \,\,C(\epsilon,\upsilon) \mu_1(Q)^{1/2} \mu_2(P)^{1/2} \mu_1(Q')^{1/2} \mu_2(P')^{1/2}.
\end{align*}
The proofs of \eqref{eq:13.2_AdjAdj} and \eqref{eq:13.3_AdjAdj} are similar to that of~\eqref{eq:g1} in the \emph{Adj/Nes} case, where we use Assumption~\ref{assum_6} of the partial weak boundedness property and the fact that $\Lambda B \subset \Delta$.
The proof of~\eqref{eq:13.4_AdjAdj} is similar to that of~\eqref{eq:g2} in the \emph{Adj/Nes} case, where we use Assumptions~\ref{assum_3} and~\ref{assum_4} of the partial weak boundedness property, and Lemma~\ref{lem1:ball_cover}.

Combining all the estimates above, we obtain
 \begin{align*}
    &\hspace{-1.5cm} |\langle T(h_{u_1}^{Q_1} \otimes h_{u_2}^{Q_2}), h_{u'_1}^{Q'_1} \otimes h_{u'_2}^{Q'_2}\rangle| \\
     \hspace{-0.7cm}\ls & { }\sum_{\substack{Q, Q' \\ P,P'}}
    |\langle h_{u_1}^{Q_1}\rangle_{Q}|\,
    |\langle h_{u_2}^{Q_2}\rangle_{P}| \,
    |\langle h_{u'_1}^{Q'_1}\rangle_{Q'}|\,
    |\langle h_{u'_2}^{Q'_2} \rangle_{P'}|\, \\
    & \times \, \Big[ \|T\| \mu_1(Q)^{1/2} \mu_2(P)^{1/2} \mu_1(Q'_b)^{1/2} \mu_2(P')^{1/2} \\
    &\hspace{0.7cm} +  \|T\| \mu_1(Q)^{1/2} \mu_2(P)^{1/2} \mu_1(Q')^{1/2} \mu_2(P'_b)^{1/2} \\
    & \hspace{0.7cm}+ \|T\| \mu_1(Q_b)^{1/2} \mu_2(P)^{1/2} \mu_1(Q')^{1/2} \mu_2(P')^{1/2} \\
    &\hspace{0.7cm} + \|T\| \mu_1(Q)^{1/2} \mu_2(P_b)^{1/2} \mu_1(Q')^{1/2} \mu_2(P')^{1/2} \\
    &\hspace{0.7cm} + \upsilon^{1/2} \|T\| \mu_1(Q)^{1/2} \mu_2(P)^{1/2} \mu_1(Q')^{1/2} \mu_2(P')^{1/2} \\
    &\hspace{0.7cm} + C(\epsilon,\upsilon) \mu_1(Q)^{1/2} \mu_2(P)^{1/2} \mu_1(Q')^{1/2} \mu_2(P')^{1/2}.
    \Big] \\
    =: & { } \sum_{\substack{Q, Q' \\ P,P'}}
    |\langle h_{u_1}^{Q_1}\rangle_{Q}|\,
    |\langle h_{u_2}^{Q_2}\rangle_{P}| \,
    |\langle h_{u'_1}^{Q'_1}\rangle_{Q'}|\,
    |\langle h_{u'_2}^{Q'_2} \rangle_{P'}|\, \\
    & \times \big[(\widetilde{\mathcal F}_1) + (\widetilde{\mathcal F}_2) + (\widetilde{\mathcal F}_3) + (\widetilde{\mathcal F}_4) + (\widetilde{\mathcal F}_5) + (\widetilde{\mathcal F}_6)\big].
  \end{align*}
Using property~\eqref{eq2:pro_Haar} of Haar functions, we have
\begin{align*}
 &\text{$|\langle h_{u_1}^{Q_1}\rangle_{Q}| \ls \mu_1(Q)^{-1/2}$,\quad
 $|\langle h_{u_2}^{Q_2}\rangle_{P}| \ls \mu_2(P)^{-1/2}$,} \\
 &\text{$|\langle h_{u'_1}^{Q'_1}\rangle_{Q'} | \ls \mu_1(Q')^{-1/2}$, \quad and \quad
 $|\langle h_{u'_2}^{Q'_2}\rangle_{P'} | \ls \mu_2(P')^{-1/2}$.}
 \end{align*}
Thus, the sum involving~$(\widetilde{\mathcal F}_5)$ is controlled by $\upsilon^{1/2} \|T\| $,
and the sum involving~$(\widetilde{\mathcal F}_6)$ by
$C(\epsilon, \upsilon)$.
Moreover, the sum involving~$(\widetilde{\mathcal F}_1)$ is estimated by
\[
 \|T\|   \sum_{Q' \in \text{ch}(Q'_1)}
   |\langle h_{u'_1}^{Q'_1}\rangle_{Q'} |  \mu_1(Q'_b)^{1/2}
   \ls \|T\|  \Big(\sum_{Q' \in \text{ch}(Q'_1)}
   |\langle h_{u'_1}^{Q'_1}\rangle_{Q'} |^2  \mu_1(Q'_b) \Big)^{1/2}
   = \|T\|   H^{u'_1}_{Q'_1,b}.
\]
Similarly, the sum involving~$(\widetilde{\mathcal F}_2)$ is reduced to $\|T\| H^{u'_2}_{Q'_2,b}$,
the sum involving~$(\widetilde{\mathcal F}_3)$ is reduced to $\|T\| H^{u_1}_{Q_1,b}$, and
the sum involving~$(\widetilde{\mathcal F}_4)$ is reduced to $\|T\| H^{u_2}_{Q_2,b}$.
Therefore, Lemma~\ref{lem1:AdjAdj} is established.
\end{proof}

Now, using Lemma~\ref{lem1:AdjAdj}, we go back to considering~\eqref{adj_adj_eq.1}.
Recall that as shown in~\eqref{eq8:Adj_In}, we have
$\mathbb{E}[(H_{Q_1,b}^{u_1})^2] \ls \epsilon^{\eta}$, where $\eta >0$.
Similarly, we also have
\[\text{
$\mathbb{E}[(H_{Q'_1,b}^{u'_1})^2] \ls \epsilon^{\eta}$,\quad
$\mathbb{E}[(H_{Q_2,b}^{u_2})^2] \ls \epsilon^{\eta}$, and \quad
$\mathbb{E}[(H_{Q'_2,b}^{u'_2})^2] \ls \epsilon^{\eta}$.}\]
Thus, for fixed $u_1$, $u'_1$, $u_2$ and $u'_2$, \eqref{adj_adj_eq.1} becomes
\begin{align*}
 &\hspace{-1cm}\sum_{\substack{Q_1,Q'_1  \\ \text{adjacent} }} \quad
  \sum_{\substack{Q_2,Q'_2\\ \text{adjacent} }}
  |\langle f, h_{u_1}^{Q_1} \otimes h_{u_2}^{Q_2}\rangle|
  |\langle g, h_{u'_1}^{Q'_1} \otimes h_{u'_2}^{Q'_2}\rangle| \big[ C(\epsilon,\upsilon) + \upsilon^{1/2} \|T\|
  + \epsilon^{\eta/2}\|T\| \big]\\
  \leq & \,\, C(\epsilon,\upsilon)\|f\|_{L^2(\mu)} \|g\|_{L^2(\mu)}
  + \upsilon^{1/2} \|T\| \|f\|_{L^2(\mu)} \|g\|_{L^2(\mu)}
  + \epsilon^{\eta/2} \|T\| \|f\|_{L^2(\mu)} \|g\|_{L^2(\mu)}.
\end{align*}
Fixing $\epsilon$ and~$\upsilon$ sufficiently small, we establish the bound
$(C + 0.01 \|T\|) \|f\|_{L^2(\mu)} \|g\|_{L^2(\mu)}$,
completing the estimation of the \emph{Adj/Adj} case.

\section{Nested Cubes}
\subsection{Nested/Nested cubes}
In this section, we consider the \emph{Nes/Nes} case. That is, we estimate~\eqref{coreeq.1} when $Q_1, Q'_1$ are nested and $Q_2, Q'_2$ are nested:
\begin{eqnarray*}
  && \ell(Q_1) < \delta^r \ell(Q'_1), \quad \rho_1(Q_1,Q'_1) \leq \mathcal{C}\ell(Q_1)^{\gamma_1}\ell(Q'_1)^{1-\gamma_1}, \quad \text{and} \\
  && \ell(Q_2) < \delta^r \ell(Q'_2), \quad \rho_2(Q_2,Q'_2) \leq \mathcal{C}\ell(Q_2)^{\gamma_2}\ell(Q'_2)^{1-\gamma_2},
\end{eqnarray*}
where $\mathcal{C}:= 2A_0C_QC_K$.
As explained in the \emph{Sep/Nes} case, when $Q_2, Q'_2$ are nested, the terms in~\eqref{coreeq.1} which involves~$h_0^{Q_2}$ vanish. Here, with the same reasoning,  the terms in~\eqref{coreeq.1} which involves~$h_0^{Q_1}$ also vanish.
We are left to bound
\begin{align}\label{InIn_eq.1}
  &\hspace{-0.5cm}\sum_{\substack{Q_1,Q'_1  \\ \text{nested} }} \,
  \sum_{\substack{Q_2,Q'_2\\ \text{nested} }} \,
  \sum_{\substack{u_1,u'_1 \\u_2,u'_2}} \,
  \langle f, h_{u_1}^{Q_1} \otimes h_{u_2}^{Q_2}\rangle
  \langle g, h_{u'_1}^{Q'_1} \otimes h_{u'_2}^{Q'_2}\rangle
  \langle T(h_{u_1}^{Q_1} \otimes h_{u_2}^{Q_2}), h_{u'_1}^{Q'_1} \otimes h_{u'_2}^{Q'_2}\rangle \noz\\
  &+   \sum_{\substack{Q_1,Q'_1  \\ \text{nested} }} \,
  \sum_{\substack{Q_2,Q'_2 \\ \text{ nested} \\ \text{gen}(Q'_2) = m}} \,
  \sum_{\substack{u_1,u'_1 \\u_2}} \,
  \langle f, h_{u_1}^{Q_1} \otimes h_{u_2}^{Q_2}\rangle
  \langle g, h_{u'_1}^{Q'_1} \otimes h_{0}^{Q'_2}\rangle
  \langle T(h_{u_1}^{Q_1} \otimes h_{u_2}^{Q_2}), h_{u'_1}^{Q'_1} \otimes h_{0}^{Q'_2}\rangle \noz\\
  &+   \sum_{\substack{Q_1,Q'_1  \\ \text{nested} \\ \text{gen}(Q'_1) = m }} \,
  \sum_{\substack{Q_2,Q'_2\\ \text{nested}  }} \,
  \sum_{\substack{u_1\\u_2,u'_2}} \,
  \langle f, h_{u_1}^{Q_1} \otimes h_{u_2}^{Q_2}\rangle
  \langle g, h_{0}^{Q'_1} \otimes h_{u'_2}^{Q'_2}\rangle
  \langle T(h_{u_1}^{Q_1} \otimes h_{u_2}^{Q_2}), h_{0}^{Q'_1} \otimes h_{u'_2}^{Q'_2}\rangle \noz\\
  &+  \sum_{\substack{Q_1,Q'_1  \\ \text{nested} \\ \text{gen}(Q'_1) = m }} \,
  \sum_{\substack{Q_2,Q'_2\\ \text{nested} \\ \text{gen}(Q'_2) = m }} \,
  \sum_{\substack{u_1 \\u_2}} \,
  \langle f, h_{u_1}^{Q_1} \otimes h_{u_2}^{Q_2}\rangle
  \langle g, h_{0}^{Q'_1} \otimes h_{0}^{Q'_2}\rangle
  \langle T(h_{u_1}^{Q_1} \otimes h_{u_2}^{Q_2}), h_{0}^{Q'_1} \otimes h_{0}^{Q'_2}\rangle.
\end{align}

For~$i = 1,2$, let~$Q_{i,1}' \in \text{ch}(Q'_i)$ be such that $Q_i \subset Q_{i,1}'$, and~$Q_{i,a}' \in \text{ch}(Q'_i)$ be such that $Q'_{i,a} \in Q'_i \backslash Q'_{i,1}$. Define
\begin{align*}
&S^{u'_i}_{Q_iQ'_i} := \chi_{Q^{'c}_{i,1}} \big(h_{u'_i}^{Q'_i} - \langle h_{u'_i}^{Q'_i} \rangle_{Q'_{i,1}}\big)
= - \langle h_{u'_i}^{Q'_i} \rangle_{Q'_{i,1}} \chi_{Q^{'c}_{i,1}}
+ \sum_{Q'_{i,a}} h_{u'_i}^{Q'_i} \chi_{Q^{'}_{i,a}}, \quad \text{and} \\
&S^{0}_{Q_iQ'_i} := \chi_{Q^{'c}_{i,1}} \big(h_{0}^{Q'_i} - \langle h_{0}^{Q'_i} \rangle_{Q'_{i,1}}\big)
= - \langle h_{0}^{Q'_i} \rangle_{Q'_{i,1}} \chi_{Q^{'c}_{i,1}}
+ \sum_{Q'_{i,a}} h_{0}^{Q'_i} \chi_{Q^{'}_{i,a}}.
\end{align*}

We perform the split
\begin{align}
   &\hspace{-1cm} \langle T(h_{u_1}^{Q_1} \otimes h_{u_2}^{Q_2}), h_{u'_1}^{Q'_1} \otimes h_{u'_2}^{Q'_2}\rangle \noz \\
   &\hspace{-0.5cm}= \, \langle T(h_{u_1}^{Q_1} \otimes h_{u_2}^{Q_2}), S^{u'_1}_{Q_1Q'_1} \otimes S^{u'_2}_{Q_2Q'_2}\rangle \label{InIn_eq1a} \\
   &+ \,\langle h_{u'_1}^{Q'_1} \rangle_{Q_1} \langle T(h_{u_1}^{Q_1} \otimes h_{u_2}^{Q_2}), 1\otimes S^{u'_2}_{Q_2Q'_2}\rangle \label{InIn_eq1b}\\
   &+\,  \langle h_{u'_2}^{Q'_2} \rangle_{Q_2}   \langle T(h_{u_1}^{Q_1} \otimes h_{u_2}^{Q_2}), S^{u'_1}_{Q_1Q'_1} \otimes 1 \rangle \label{InIn_eq1c}\\
   &+ \,  \langle h_{u'_1}^{Q'_1} \rangle_{Q_1} \langle h_{u'_2}^{Q'_2} \rangle_{Q_2}   \langle T(h_{u_1}^{Q_1} \otimes h_{u_2}^{Q_2}), 1 \rangle. \label{InIn_eq1d}
\end{align}
Similarly, we write
\begin{align}
   &\hspace{-1cm} \langle T(h_{u_1}^{Q_1} \otimes h_{u_2}^{Q_2}), h_{u'_1}^{Q'_1} \otimes h_{0}^{Q'_2}\rangle \noz \\
   &\hspace{-0.5cm}= \, \langle T(h_{u_1}^{Q_1} \otimes h_{u_2}^{Q_2}), S^{u'_1}_{Q_1Q'_1} \otimes S^{0}_{Q_2Q'_2}\rangle \label{InIn_eq2a} \\
   &+ \,\langle h_{u'_1}^{Q'_1} \rangle_{Q_1} \langle T(h_{u_1}^{Q_1} \otimes h_{u_2}^{Q_2}), 1\otimes S^{0}_{Q_2Q'_2}\rangle \label{InIn_eq2b}\\
   &+\,  \langle h_{0}^{Q'_2} \rangle_{Q_2}   \langle T(h_{u_1}^{Q_1} \otimes h_{u_2}^{Q_2}), S^{u'_1}_{Q_1Q'_1} \otimes 1 \rangle \label{InIn_eq2c}\\
   &+ \,  \langle h_{u'_1}^{Q'_1} \rangle_{Q_1} \langle h_{0}^{Q'_2} \rangle_{Q_2}   \langle T(h_{u_1}^{Q_1} \otimes h_{u_2}^{Q_2}), 1 \rangle, \label{InIn_eq2d}
\end{align}

\begin{align}
   &\hspace{-1cm} \langle T(h_{u_1}^{Q_1} \otimes h_{u_2}^{Q_2}), h_{0}^{Q'_1} \otimes h_{u'_2}^{Q'_2}\rangle \noz \\
   &\hspace{-0.5cm}= \, \langle T(h_{u_1}^{Q_1} \otimes h_{u_2}^{Q_2}), S^{0}_{Q_1Q'_1} \otimes S^{u'_2}_{Q_2Q'_2}\rangle \label{InIn_eq3a} \\
   &+ \,\langle h_{0}^{Q'_1} \rangle_{Q_1} \langle T(h_{u_1}^{Q_1} \otimes h_{u_2}^{Q_2}), 1\otimes S^{u'_2}_{Q_2Q'_2}\rangle \label{InIn_eq3b}\\
   &+\,  \langle h_{u'_2}^{Q'_2} \rangle_{Q_2}   \langle T(h_{u_1}^{Q_1} \otimes h_{u_2}^{Q_2}), S^{0}_{Q_1Q'_1} \otimes 1 \rangle \label{InIn_eq3c}\\
   &+ \,  \langle h_{0}^{Q'_1} \rangle_{Q_1} \langle h_{u'_2}^{Q'_2} \rangle_{Q_2}   \langle T(h_{u_1}^{Q_1} \otimes h_{u_2}^{Q_2}), 1 \rangle, \label{InIn_eq3d} \quad\text{and}
\end{align}

\begin{align}
   &\hspace{-1cm} \langle T(h_{u_1}^{Q_1} \otimes h_{u_2}^{Q_2}), h_{0}^{Q'_1} \otimes h_{0}^{Q'_2}\rangle \noz \\
   &\hspace{-0.5cm}= \, \langle T(h_{u_1}^{Q_1} \otimes h_{u_2}^{Q_2}), S^{0}_{Q_1Q'_1} \otimes S^{0}_{Q_2Q'_2}\rangle \label{InIn_eq4a} \\
   &+ \,\langle h_{0}^{Q'_1} \rangle_{Q_1} \langle T(h_{u_1}^{Q_1} \otimes h_{u_2}^{Q_2}), 1\otimes S^{0}_{Q_2Q'_2}\rangle \label{InIn_eq4b}\\
   &+\,  \langle h_{0}^{Q'_2} \rangle_{Q_2}   \langle T(h_{u_1}^{Q_1} \otimes h_{u_2}^{Q_2}), S^{0}_{Q_1Q'_1} \otimes 1 \rangle \label{InIn_eq4c}\\
   &+ \,  \langle h_{0}^{Q'_1} \rangle_{Q_1} \langle h_{0}^{Q'_2} \rangle_{Q_2}   \langle T(h_{u_1}^{Q_1} \otimes h_{u_2}^{Q_2}), 1 \rangle \label{InIn_eq4d}.
\end{align}
Our plan is the following. First, we estimate the sum involving the element~\eqref{InIn_eq1a}. Then the sums involving \eqref{InIn_eq2a}, \eqref{InIn_eq3a}, \eqref{InIn_eq4a} can be estimated using a similar proof structure.
Second, we group the sums involving~\eqref{InIn_eq1b} and~\eqref{InIn_eq2b} and estimate them at once. Then the pairs of sums involving \eqref{InIn_eq3b} and~\eqref{InIn_eq4b}; \eqref{InIn_eq1c} and~\eqref{InIn_eq2c};~\eqref{InIn_eq3c} and~\eqref{InIn_eq4c} can be estimated analogously.
Finally, we group the sums involving~\eqref{InIn_eq1d}, \eqref{InIn_eq2d}, \eqref{InIn_eq3d}, and~\eqref{InIn_eq4d} and deal with them all at the same time.

We start with the sum involving~\eqref{InIn_eq1a}. The following lemma gives its bound.

\begin{lem}\label{lem1:InIn}
  Suppose $Q_1$, $Q'_1$ are nested and $Q_2$, $Q'_2$ are nested.
  There holds that
  \[| \langle T(h_{u_1}^{Q_1} \otimes h_{u_2}^{Q_2}), S^{u'_1}_{Q_1Q'_1} \otimes S^{u'_u}_{Q_2Q'_2}\rangle|
  \ls A^{\textup{in}}_{Q_1,Q'_1} A^{\textup{in}}_{Q_2,Q'_2}, \quad \text{where} \]
  \[A^{\textup{in}}_{Q_i,Q'_i} := \left(\frac{\ell(Q_i)}{\ell(Q'_i)}\right)^{\al_i/2}
\left( \frac{\mu_i(Q_i)}{\mu_i(Q'_{i,1})}\right)^{1/2}, \quad i = 1,2. \]
\end{lem}

\begin{proof}
The proof of Lemma~\ref{lem1:InIn} requires some results that we have shown in the \emph{Sep/Nes} cases. Specifically, Claim~\ref{claim:sep_nest}, equations~\eqref{eq6:sep_nest} and~\eqref{eq5:SepIn_b} were shown to be true for cubes~$Q_2$ and~$Q'_2$ which are nested. We note that these results also holds for cubes~$Q_1$ and~$Q'_1$ which are nested.

We write $\langle T(h_{u_1}^{Q_1} \otimes h_{u_2}^{Q_2}), S^{u'_1}_{Q_1Q'_1} \otimes S^{u'_u}_{Q_2Q'_2}\rangle$ as
\begin{align}
  &\hspace{-0,5cm} \langle h_{u'_1}^{Q'_1} \rangle_{Q'_{1,1}}
  \langle h_{u'_2}^{Q'_2} \rangle_{Q'_{2,1}}
  \langle T(h_{u_1}^{Q_1} \otimes h_{u_2}^{Q_2}), \chi_{Q^{'c}_{1,1}} \otimes \chi_{Q^{'c}_{2,1}} \rangle \label{InIn_eq5} \\
  & - \, \langle h_{u'_1}^{Q'_1} \rangle_{Q'_{1,1}}
  \sum_{Q'_{2,a}}
  \langle T(h_{u_1}^{Q_1} \otimes h_{u_2}^{Q_2}), \chi_{Q^{'c}_{1,1}} \otimes h_{u'_2}^{Q'_2} \chi_{Q^{'}_{2,a}} \rangle \label{InIn_eq6} \\
  & - \, \langle h_{u'_2}^{Q'_2} \rangle_{Q'_{2,1}}
  \sum_{Q'_{1,a}}
  \langle T(h_{u_1}^{Q_1} \otimes h_{u_2}^{Q_2}), h_{u'_1}^{Q'_1} \chi_{Q^{'}_{1,a}} \otimes \chi_{Q^{'c}_{2,1}} \rangle \label{InIn_eq7}\\
  & + \, \sum_{Q'_{1,a}} \sum_{Q'_{2,a}}
  \langle T(h_{u_1}^{Q_1} \otimes h_{u_2}^{Q_2}), h_{u'_1}^{Q'_1} \chi_{Q^{'}_{1,a}} \otimes h_{u'_2}^{Q'_2} \chi_{Q^{'}_{2,a}}  \rangle. \label{InIn_eq8}
\end{align}
We will apply Lemma~\ref{lem1:property_kernel} to the terms~\eqref{InIn_eq5}--\eqref{InIn_eq8} and show that they are bounded by $A^{\textup{in}}_{Q_1,Q'_1} A^{\textup{in}}_{Q_2,Q'_2}$.
We will show the calculation for~\eqref{InIn_eq5}. The bound for three other terms can be obtained similarly.

Note that
\begin{equation}\label{InIn_eq9}
|\langle h_{u'_1}^{Q'_1} \rangle_{Q'_{1,1}}| \ls \mu_1(Q'_{1,1})^{-1/2}
\quad \text{and} \quad |\langle h_{u'_2}^{Q'_2} \rangle_{Q'_{2,1}}|\ls \mu_2(Q'_{2,1})^{-1/2}.
\end{equation}
Using property~(i) in Claim~\ref{claim:sep_nest}, for all $y_1,z \in Q_1$, $x_1 \in Q^{'c}_{1,1}$, $y_2,w \in Q_2$ and $x_2 \in Q^{'c}_{2,1}$ we have
\[
\rho_1(y_1,z) \leq \frac{\rho_1(x_1,z)}{C_K}
\quad \text{and} \quad
\rho_2(y_2,w) \leq \frac{\rho_2(x_2,w)}{C_K}.
\]
Applying Lemma~\ref{lem1:property_kernel} to
$\phi_1 = h^{Q_1}_{u_1}$, $\phi_2 = h^{Q_2}_{u_2}$, $\theta_1 = \chi_{Q^{'c}_{1,1}}$ and $\theta_2 = \chi_{Q^{'c}_{2,1}}$ we have
\begin{eqnarray*}
  |\langle T(h_{u_1}^{Q_1} \otimes h_{u_2}^{Q_2}), \chi_{Q^{'c}_{1,1}} \otimes \chi_{Q^{'c}_{2,1}}\rangle|
   &\ls& \|h^{Q_1}_{u_1}\|_{L^2(\mu_1)}  \|h^{Q_2}_{u_2}\|_{L^2(\mu_2)}
   \mu_1(Q_1)^{1/2} \mu_2(Q_2)^{1/2} \noz\\
   && \times
    \int_{Q^{'c}_{1,1}}    \frac{C_Q^{\al_1}\ell(Q_1)^{\al_1}}
    {\rho_1(x_1,z)^{\al_1}\lambda_1(z,\rho_1(x_1,z))} \,d\mu_1(x_1) \noz\\
    && \times
  \int_{Q^{'c}_{2,1}} \frac{C_Q^{\al_2}\ell(Q_2)^{\al_2}}
  {\rho_2(x_2,w)^{\al_2}\lambda_2(w,\rho_2(x_2,w))} \,d\mu_2(x_2).
\end{eqnarray*}
As shown in equation~\eqref{eq6:sep_nest}, the two integrals above are bounded by $\ell(Q_1)^{\al_1/2} \ell(Q'_1)^{-\al_1/2}$ and $\ell(Q_2)^{\al_2/2} \ell(Q'_2)^{-\al_2/2}$, respectively.
These together with~\eqref{InIn_eq9} imply that~\eqref{InIn_eq5} is bounded by $A^{\textup{in}}_{Q_1,Q'_1} A^{\textup{in}}_{Q_2,Q'_2}$.
This establishes Lemma~\ref{lem1:InIn}.
\end{proof}

Using Lemma~\ref{lem1:InIn} and Proposition~\ref{prop:A_in},  the term in~\eqref{InIn_eq.1} involving element~\eqref{InIn_eq1a} is dominated by $\|f\|_{L^2(\mu)} \|g\|_{L^2(\mu)}$:
\begin{align*}
   &\hspace{-0.5cm}\sum_{\substack{Q_1,Q'_1  \\ \text{nested} }} \,
  \sum_{\substack{Q_2,Q'_2\\ \text{nested} }} \,
  \sum_{\substack{u_1,u'_1 \\u_2,u'_2}} \,
  |\langle f, h_{u_1}^{Q_1} \otimes h_{u_2}^{Q_2}\rangle|
  |\langle g, h_{u'_1}^{Q'_1} \otimes h_{u'_2}^{Q'_2}\rangle|
  |\langle T(h_{u_1}^{Q_1} \otimes h_{u_2}^{Q_2}), S^{u'_1}_{Q_1Q'_1} \otimes S^{u'_u}_{Q_2Q'_2}\rangle |\noz\\
  & \ls \sum_{\substack{Q_1,Q'_1  \\ \text{nested} }} \,
  \sum_{\substack{Q_2,Q'_2\\ \text{nested} }} \,
  \sum_{\substack{u_1,u'_1 \\u_2,u'_2}} \,
  A^{\textup{in}}_{Q_1,Q'_1} A^{\textup{in}}_{Q_2,Q'_2}
  |\langle f, h_{u_1}^{Q_1} \otimes h_{u_2}^{Q_2}\rangle|
  |\langle g, h_{u'_1}^{Q'_1} \otimes h_{u'_2}^{Q'_2}\rangle| \\
  & \ls  \sum_{\substack{Q_1,Q'_1  \\ \text{nested} }}
  A^{\textup{in}}_{Q_1,Q'_1}
  \left(\sum_{Q_2} \sum_{u_2} |\langle f, h_{u_1}^{Q_1} \otimes h_{u_2}^{Q_2}\rangle|^2 \right)^{1/2}
  \left(\sum_{Q'_2} \sum_{u'_2}  |\langle g, h_{u'_1}^{Q'_1} \otimes h_{u'_2}^{Q'_2}\rangle|^2 \right)^{1/2} \\
  &\ls \left(\sum_{Q_1}\sum_{u_1} \sum_{Q_2} \sum_{u_2} |\langle f, h_{u_1}^{Q_1} \otimes h_{u_2}^{Q_2}\rangle|^2 \right)^{1/2}
   \left(\sum_{Q'_1}\sum_{u'_1} \sum_{Q'_2} \sum_{u'_2}  |\langle g, h_{u'_1}^{Q'_1} \otimes h_{u'_2}^{Q'_2}\rangle|^2 \right)^{1/2} \\
  &\ls \|f\|_{L^2(\mu)} \|g\|_{L^2(\mu)}.
\end{align*}

Next, we group the sums involving~\eqref{InIn_eq1b} and~\eqref{InIn_eq2b} and estimate them together. To do so, we write these sums in terms of a paraproduct.
\begin{lem}\label{lem2:InIn}
The sum
\begin{align*}
&\sum_{\substack{Q_1, Q'_1 \\ \textup{nested} }} \,
  \sum_{\substack{Q_2, Q'_2 \\ \textup{nested}}} \,
  \sum_{\substack{u_1,u'_1 \\u_2,u'_2}}
\langle f, h_{u_1}^{Q_1} \otimes h_{u_2}^{Q_2}\rangle
  \langle g, h_{u'_1}^{Q'_1} \otimes h_{u'_2}^{Q'_2}\rangle
  \langle h_{u'_2}^{Q'_2}\rangle_{Q_2}
\langle T(h_{u_1}^{Q_1} \otimes h_{u_2}^{Q_2}), S^{u'_1}_{Q_1Q'_1} \otimes 1\rangle  \noz \\
&+
\sum_{\substack{Q_1, Q'_1 \\ \textup{nested}} } \,
  \sum_{\substack{Q_2, Q'_2 \\ \textup{nested}\\ \textup{gen}(Q'_2)=m }} \,
  \sum_{\substack{u_1,u'_1 \\u_2}}
\langle f, h_{u_1}^{Q_1} \otimes h_{u_2}^{Q_2}\rangle
  \langle g, h_{u'_1}^{Q'_1} \otimes h_{0}^{Q'_2}\rangle
  \langle h_{0}^{Q'_2}\rangle_{Q_2}
\langle T(h_{u_1}^{Q_1} \otimes h_{u_2}^{Q_2}), S^{u'_1}_{Q_1Q'_1}  \otimes 1\rangle
\end{align*}
equals
\[
\bigg\langle
\sum_{Q_1, Q'_1 \textup{ good}} \quad
\sum_{\substack{u_1,u'_1 \\u_2}} \quad
 h_{u'_1}^{Q'_1} \otimes  \big(\Pi_{c^{u_1 u'_1}_{Q_1 Q'_1}}^{u_2} \big)^*f^{Q_1,u_1},
 g_{\textup{good}}
\bigg\rangle,
\]
where
\begin{eqnarray*}
  f^{Q_1,u_1} &:=& \langle f, h_{u_1}^{Q_1}\rangle_1 = \int_{X_1} f(x,y)h_{u_1}^{Q_1}(x) \,d\mu_1(x), \\
  c^{u_1 u'_1}_{Q_1 Q'_1} &:=& \langle T^*(S^{u'_1}_{Q_1Q'_1}  \otimes 1), h_{u_1}^{Q_1}\rangle_1, \quad \text{and}\\
  \Pi_{a}^{u_2} \omega &:=&  \sum_{Q'_2 \in \Dd'_2} \quad
\sum_{\substack{Q_2 \textup{ good} \\Q_2 \subset Q'_2 \\ \ell(Q_2) = \delta^r\ell(Q'_2)}}
\langle \omega\rangle_{Q'_2}
\langle a, h_{u_2}^{Q_2}\rangle h_{u_2}^{Q_2}.
\end{eqnarray*}
\end{lem}
\begin{proof}
  Notice that Lemma~\ref{lem2:InIn} is similar to Lemma~\ref{lem3.sepnest} in the \emph{Sep/Nes} case, except that~$S^{u'_1}_{Q_1Q'_1} $ has replaced~$h_{u'_1}^{Q'_1}$ and~$c^{u_1 u'_1}_{Q_1 Q'_1}$ has replaced~$b^{u_1 u'_1}_{Q_1 Q'_1}$.
  Thus, the proof of Lemma~\ref{lem2:InIn} is similar to that of Lemma~\ref{lem3.sepnest}, incorporating the above mentioned differences.
\end{proof}

Using Lemma~\ref{lem2:InIn} and Lemma~\ref{lem4.sepnest} we have
\begin{align}\label{InIn:eq14}
   & \bigg| \bigg\langle
\sum_{Q_1, Q'_1 \textup{ good}} \quad
\sum_{\substack{u_1,u'_1 \\u_2}} \quad
 h_{u'_1}^{Q'_1} \otimes  \big(\Pi_{c^{u_1 u'_1}_{Q_1 Q'_1}}^{u_2} \big)^*f^{Q_1,u_1},
 g_{\textup{good}}
\bigg\rangle \bigg| \noz\\
  & \ls \bigg( \sum_{Q'_1 \textup{ good}}
  \bigg [\sum_{Q_1 \textup{ good}} \quad
\sum_{\substack{u_1,u'_1 \\u_2}} \quad
\|c^{u_1 u'_1}_{Q_1 Q'_1}\|_{\bmo^2_{C_K}(\mu_2)}
\|f^{Q_1,u_1}\|_{L^2(\mu_2)}
 \bigg]^2 \bigg)^{1/2} \|g\|_{L^2(\mu)}.
\end{align}
We will estimated the term $\|c^{u_1 u'_1}_{Q_1 Q'_1}\|_{\bmo^2_{C_K}(\mu_2)}$ below.

\begin{lem}\label{lem3:InIn}
  If $Q_1$, $Q'_1$ are nested,
  then there holds that
  \[\| c^{u_1 u'_1}_{Q_1 Q'_1}\|_{\bmo^2_{C_K}(\mu_2)} \ls A_{Q_1,Q'_1}^{\textup{in}}, \quad \text{where } C_K >1.\]
\end{lem}
\begin{proof}
  By duality, it is sufficient to show $|\langle c^{u_1 u'_1}_{Q_1 Q'_1}, a\rangle| \ls A_{Q_1,Q'_1}^{\textup{sep}}$ for all $2$-atoms~$a$ in the Hardy space~$H^1(\mu_2)$.
  Fix a ball $J = B(x_J,r_J) \subset X_2$ and an atom~$a$ such that $\supp a \subset J$, $\int_{J} a \,d\mu_2 = 0$ and $\|a\|_{L^2(\mu_1)} \leq \mu_1(J)^{-1/2}$ .
  Consider
  \begin{align*}
  \Big \langle c^{u_1 u'_1}_{Q_1 Q'_1}, a\Big\rangle
    &=  \Big \langle   T( h_{u_1}^{Q_1} \otimes a),S^{u'_1}_{Q_1Q'_1} \otimes 1\Big\rangle \\
    &=  -\langle h_{u'_1}^{Q'_1} \rangle_{Q'_{1,1}} \Big \langle
    T( h_{u_1}^{Q_1} \otimes a),\chi_{Q^{'c}_{1,1}} \otimes 1 \Big\rangle
    +  \sum_{Q'_{1,a}} \Big \langle   T( h_{u_1}^{Q_1} \otimes a),h_{u'_1}^{Q'_1} \chi_{Q'_{1,a}} \otimes 1\Big\rangle.
    \end{align*}
  Let~$V := C_KJ = B(x_J,C_Kr_J)$.
  We need to show that
  \begin{eqnarray}
    \Big|\Big \langle   T( h_{u_1}^{Q_1} \otimes a),S^{u'_1}_{Q_1Q'_1} \otimes 1\Big\rangle\Big|
     &\leq& \Big| \langle h_{u'_1}^{Q'_1} \rangle_{Q'_{1,1}} \Big \langle   T( h_{u_1}^{Q_1} \otimes a),\chi_{Q^{'c}_{1,1}} \otimes \chi_V\Big\rangle\Big| \label{InIn:eq10} \\
     &&+ \Big| \langle h_{u'_1}^{Q'_1} \rangle_{Q'_{1,1}} \Big \langle   T( h_{u_1}^{Q_1} \otimes a),\chi_{Q^{'c}_{1,1}} \otimes \chi_{V^c}\Big\rangle\Big| \label{InIn:eq11} \\
     &&+ \sum_{Q'_{1,a}} \Big|\Big \langle   T( h_{u_1}^{Q_1} \otimes a),h_{u'_1}^{Q'_1} \chi_{Q'_{1,a}} \otimes \chi_{V}\Big\rangle\Big| \label{InIn:eq12} \\
     &&+ \sum_{Q'_{1,a}} \Big|\Big \langle   T( h_{u_1}^{Q_1} \otimes a),h_{u'_1}^{Q'_1} \chi_{Q'_{1,a}} \otimes \chi_{V^c}\Big\rangle\Big| \label{InIn:eq13}\\
     &\ls&  A_{Q_1,Q'_1}^{\textup{in}}. \noz
  \end{eqnarray}
  To do so, we will apply Lemma~\ref{lem:property_kernel} to the terms~\eqref{InIn:eq10} and~\eqref{InIn:eq12}, and apply Lemma~\ref{lem1:property_kernel} to the terms~\eqref{InIn:eq11} and~\eqref{InIn:eq13}.

  For~\eqref{InIn:eq10}, applying Lemma~\ref{lem:property_kernel} to $\phi_1 =  h_{u_1}^{Q_1}$, $\phi_2 = a$, $\theta_1= \chi_{Q^{'c}_{1,1}}$ and $\theta_2 = \chi_V$ we have
  \begin{eqnarray*}
    \Big|\Big \langle   T( h_{u_1}^{Q_1} \otimes a),\chi_{Q^{'c}_{1,1}} \otimes \chi_V\Big\rangle\Big|
     &\ls& \|h_{u_1}^{Q_1}\|_{L^2(\mu_1)} \|a\|_{L^2(\mu_2)}
     \mu_1(Q_1)^{1/2} \|\chi_V\|_{L^2(\mu_2)} \\ &&\times\int_{Q^{'c}_{1,1}}\frac{C_K^{\al_1} \ell(Q_1)^{\al_1}}
     {\rho_1(x_1,z)^{\al_1}\lambda_1(z,\rho_1(x_1,z))}
     |\chi_{Q^{'c}_{1,1}}|
     \,d\mu_1(x_1).
  \end{eqnarray*}
Following equation~\eqref{eq6:sep_nest}, the integral above is bounded by $\ell(Q_1)^{\al_1/2} \ell(Q'_1)^{-\al_1/2}$. This together with~\eqref{InIn_eq9} implies that~\eqref{InIn:eq10} is dominated by~$A_{Q_1,Q'_1}^{\textup{in}}$.

For~\eqref{InIn:eq11}, applying Lemma~\ref{lem1:property_kernel} to
$\phi_1 = h^{Q_1}_{u_1}$, $\phi_2 = a$, $\theta_1 = \chi_{Q^{'c}_{1,1}}$ and $\theta_2 = \chi_{V^c}$ we have
\begin{eqnarray*}
  |\langle T(h_{u_1}^{Q_1} \otimes a, \chi_{Q^{'c}_{1,1}} \otimes \chi_{V^c}\rangle|
   &\ls& \|h^{Q_1}_{u_1}\|_{L^2(\mu_1)}  \|a\|_{L^2(\mu_2)}
   \mu_1(Q_1)^{1/2} \mu_2(J)^{1/2} \\
    && \times \int_{Q^{'c}_{1,1}}    \frac{C_Q^{\al_1}\ell(Q_1)^{\al_1}}
    {\rho_1(x_1,z)^{\al_1}\lambda_1(z,\rho_1(x_1,z))} \,d\mu_1(x_1) \\
  && \times  \int_{V^c} \frac{r_J^{\al_2}}
  {\rho_2(x_2,x_J)^{\al_2}\lambda_2(x_J,\rho_2(x_2,x_J))} \,d\mu_2(x_2).
\end{eqnarray*}
The first integral above is bounded by $\ell(Q_1)^{\al_1/2} \ell(Q'_1)^{-\al_1/2}$. Using Lemma~\ref{upper_dbl_lem1}, the second integral is bounded by
\[r_J^{\al_2} \int_{X_2 \backslash B(x_J,C_kr_J)}
 \frac{\rho_2(x_2,x_J)^{-\al_2}}
  {\lambda_2(x_J,\rho_2(x_2,x_J))} \,d\mu_2(x_2)
  \ls r_J^{\al_2} (C_K r_J)^{-\al_2} \sim 1.
  \]
 Hence, \eqref{InIn:eq11} is dominated by~$A_{Q_1,Q'_1}^{\textup{in}}$.

 For \eqref{InIn:eq12}, for each fixed cube~$Q'_{1,a}$ we apply Lemma~\ref{lem:property_kernel} to $\phi_1 =  h_{u_1}^{Q_1}$, $\phi_2 = a$, $\theta_1= h_{u'_1}^{Q'_1} \chi_{Q'_{1,a}}$ and $\theta_2 = \chi_V$ to have
  \begin{eqnarray*}
    \Big|\Big \langle   T( h_{u_1}^{Q_1} \otimes a),h_{u'_1}^{Q'_1} \chi_{Q'_{1,a}} \otimes \chi_V\Big\rangle\Big|
     &\ls& \|h_{u_1}^{Q_1}\|_{L^2(\mu_1)} \|a\|_{L^2(\mu_2)}
     \mu_1(Q_1)^{1/2} \|\chi_V\|_{L^2(\mu_2)} \\
     && \times \int_{Q'_{1,a}}\frac{C_K^{\al_1} \ell(Q_1)^{\al_1}}
     {\rho_1(x_1,z)^{\al_1}\lambda_1(z,\rho_1(x_1,z))}
     |h_{u'_1}^{Q'_1} \chi_{Q'_{1,a}}|
     \,d\mu_1(x_1).
  \end{eqnarray*}
 Following equation~\eqref{eq5:SepIn_b}, the integral above in bounded by
 $\ell(Q_1)^{\al_1/2} \ell(Q'_1)^{-\al_1/2} \mu_1(Q'_{1,1})^{-1/2}$.
 Thus, \eqref{InIn:eq12} is dominated by~$A_{Q_1,Q'_1}^{\textup{in}}$.

  For \eqref{InIn:eq13}, for each fixed cube~$Q'_{1,a}$ we apply Lemma~\ref{lem1:property_kernel} to $\phi_1 =  h_{u_1}^{Q_1}$, $\phi_2 = a$, $\theta_1= h_{u'_1}^{Q'_1} \chi_{Q'_{1,a}}$ and $\theta_2 = \chi_{V^c}$ giving
\begin{eqnarray*}
  |\langle T(h_{u_1}^{Q_1} \otimes a, h_{u'_1}^{Q'_1} \chi_{Q'_{1,a}}\otimes \chi_{V^c}\rangle|
   &\ls& \|h^{Q_1}_{u_1}\|_{L^2(\mu_1)}  \|a\|_{L^2(\mu_2)}
   \mu_1(Q_1)^{1/2} \mu_2(J)^{1/2}\\
   && \times \int_{Q'_{1,a}}    \frac{C_Q^{\al_1}\ell(Q_1)^{\al_1}}
    {\rho_1(x_1,z)^{\al_1}\lambda_1(z,\rho_1(x_1,z))}|h_{u'_1}^{Q'_1} \chi_{Q'_{1,a}}| \,d\mu_1(x_1) \noz\\
    && \times
  \int_{V^c} \frac{r_J^{\al_2}}
  {\rho_2(x_2,x_J)^{\al_2}\lambda_2(x_J,\rho_2(x_2,x_J))} \,d\mu_2(x_2).
\end{eqnarray*}
The first integral above is bounded by $\ell(Q_1)^{\al_1/2} \ell(Q'_1)^{-\al_1/2} \mu_1(Q'_{1,1})^{-1/2}$, while the second integral is bounded by 1.
Hence, \eqref{InIn:eq13} is dominated by~$A_{Q_1,Q'_1}^{\textup{in}}$.
This completes the proof of Lemma~\ref{lem3:InIn}.
\end{proof}

Now we go back to considering~\eqref{InIn:eq14}. We have that
\begin{align*}
   & \sum_{Q'_1 \textup{ good}}
  \bigg [\sum_{Q_1 \textup{ good}} \quad
\sum_{\substack{u_1,u'_1 \\u_2}} \quad
\|c^{u_1 u'_1}_{Q_1 Q'_1}\|_{\bmo^2_{C_K}(\mu_2)}
\|f^{Q_1,u_1}\|_{L^2(\mu_2)}
 \bigg]^2 \\
 & \ls \sum_{Q'_1 \textup{ good}}
  \bigg [\sum_{Q_1 \textup{ good}} \quad
\sum_{\substack{u_1,u'_1 \\u_2}} \quad
A_{Q_1,Q'_1}^{\textup{in}}
\|f^{Q_1,u_1}\|_{L^2(\mu_2)}
 \bigg]^2  \\
 &\ls \sum_{Q_1}\sum_{u_1} \|f^{Q_1,u_1}\|_{L^2(\mu_2)} \\
 & = \sum_{Q_1}\sum_{u_1}
 \int_{X_2} \bigg| \int_{X_1} f(x_1,x_2) h_{u_1}^{Q_1}(x_1) \,d\mu_1(x_1)\bigg|^2 \,d\mu_2(x_2)
  = \int_{X_2} \sum_{Q_1}\sum_{u_1}
 |\langle f(\cdot,x_2), h_{u_1}^{Q_1}\rangle|^2 \,d\mu_2(x_2) \\
 & = \int_{X_2} \|f(\cdot,x_2)\|_{L^2(\mu_1)}^2 \,d\mu_2(x_2)
  = \int_{X_2} \int_{X_1} |f(x_1,x_2)|^2 \,d\mu = \|f\|_{L^2(\mu)}^2.
\end{align*}
This implies that  the sums involving~\eqref{InIn_eq1b} and~\eqref{InIn_eq2b} are dominated by $\|f\|_{L^2(\mu)}^2 \|g\|_{L^2(\mu)}^2$.

\subsection{Full Paraproduct}\label{subsec:full_para}
We are left to consider the sums involving~\eqref{InIn_eq1d}, \eqref{InIn_eq2d} , \eqref{InIn_eq3d} and~\eqref{InIn_eq4d}:
\begin{align}\label{InIn_eq.15}
  &\hspace{-0.5cm}\sum_{\substack{Q_1,Q'_1  \\ \text{nested} }} \,
  \sum_{\substack{Q_2,Q'_2\\ \text{nested} }} \,
  \sum_{\substack{u_1,u'_1 \\u_2,u'_2}} \,
  \langle f, h_{u_1}^{Q_1} \otimes h_{u_2}^{Q_2}\rangle
  \langle g, h_{u'_1}^{Q'_1} \otimes h_{u'_2}^{Q'_2}\rangle
  \langle h_{u'_1}^{Q'_1} \rangle_{Q_1} \langle h_{u'_2}^{Q'_2} \rangle_{Q_2}   \langle T(h_{u_1}^{Q_1} \otimes h_{u_2}^{Q_2}), 1 \rangle \noz\\
  &+   \sum_{\substack{Q_1,Q'_1  \\ \text{nested} }} \,
  \sum_{\substack{Q_2,Q'_2 \\ \text{ nested} \\ \text{gen}(Q'_2) = m}} \,
  \sum_{\substack{u_1,u'_1 \\u_2}} \,
  \langle f, h_{u_1}^{Q_1} \otimes h_{u_2}^{Q_2}\rangle
  \langle g, h_{u'_1}^{Q'_1} \otimes h_{0}^{Q'_2}\rangle
  \langle h_{u'_1}^{Q'_1} \rangle_{Q_1} \langle h_{0}^{Q'_2} \rangle_{Q_2}   \langle T(h_{u_1}^{Q_1} \otimes h_{u_2}^{Q_2}), 1 \rangle \noz\\
  &+   \sum_{\substack{Q_1,Q'_1  \\ \text{nested} \\ \text{gen}(Q'_1) = m }} \,
  \sum_{\substack{Q_2,Q'_2\\ \text{nested}  }} \,
  \sum_{\substack{u_1\\u_2,u'_2}} \,
  \langle f, h_{u_1}^{Q_1} \otimes h_{u_2}^{Q_2}\rangle
  \langle g, h_{0}^{Q'_1} \otimes h_{u'_2}^{Q'_2}\rangle
  \langle h_{0}^{Q'_1} \rangle_{Q_1} \langle h_{u'_2}^{Q'_2} \rangle_{Q_2}   \langle T(h_{u_1}^{Q_1} \otimes h_{u_2}^{Q_2}), 1 \rangle \noz\\
  &+  \sum_{\substack{Q_1,Q'_1  \\ \text{nested} \\ \text{gen}(Q'_1) = m }} \,
  \sum_{\substack{Q_2,Q'_2\\ \text{nested} \\ \text{gen}(Q'_2) = m }} \,
  \sum_{\substack{u_1 \\u_2}} \,
  \langle f, h_{u_1}^{Q_1} \otimes h_{u_2}^{Q_2}\rangle
  \langle g, h_{0}^{Q'_1} \otimes h_{0}^{Q'_2}\rangle
  \langle h_{0}^{Q'_1} \rangle_{Q_1} \langle h_{0}^{Q'_2} \rangle_{Q_2}   \langle T(h_{u_1}^{Q_1} \otimes h_{u_2}^{Q_2}), 1 \rangle.
\end{align}
Since~$Q_1$ is good, $Q_1 \subset Q'_1$ and $\ell(Q_1) < \delta^r \ell(Q'_1)$, there exists a unique cube $S(Q_1) \in \Dd'_1$ such that $\ell(Q_1) = \delta^{r}\ell(S(Q_1))$ and $Q_1 \subset S(Q_1) \subset Q'_1$.
Similarly, there exists a unique cube $S(Q_2) \in \Dd'_2$ such that $\ell(Q_2) = \delta^{r}\ell(S(Q_2))$ and $Q_2 \subset S(Q_2) \subset Q'_2$.
This implies
\[
\langle h_{u'_1}^{Q'_1}\rangle_{Q_1}
= \langle h_{u'_1}^{Q'_1}\rangle_{S(Q_1)}, \,
 \langle h_{0}^{Q'_1}\rangle_{Q_1}
=  \langle h_{0}^{Q'_1}\rangle_{S(Q_1)}, \,
\langle h_{u'_2}^{Q'_2}\rangle_{Q_2}
= \langle h_{u'_2}^{Q'_2}\rangle_{S(Q_2)} \, \text{and}\,
 \langle h_{0}^{Q'_2}\rangle_{Q_2}
=  \langle h_{0}^{Q'_2}\rangle_{S(Q_2)}.\]
Recall that all the cubes appearing here are good.
We define
\[
 \langle g_{\text{good}}, h_{u'_1}^{Q'_1} \otimes h_{u'_2}^{Q'_2} \rangle :=
\left\{
  \begin{array}{ll}
    \langle g, h_{u'_1}^{Q'_1} \otimes h_{u'_2}^{Q'_2} \rangle,
     & \hbox{\text{if $Q'_1$ is good and $Q'_2$ is good};} \\
    0, & \hbox{\text{otherwise}.}
  \end{array}
\right.
\]
Now we can add all bad cubes~$Q'_2$ to the summation~\eqref{sepnesteq.4} and write
\[\langle g, h_{u'_1}^{Q'_1} \otimes h_{u'_2}^{Q'_2} \rangle = \langle g_{\text{good}}, h_{u'_1}^{Q'_1} \otimes h_{u'_2}^{Q'_2} \rangle.\]
Below, we rewrite~\eqref{InIn_eq.15} in terms of~$g_{\text{good}}$.  We explicitly write the fact that $Q_1$ and~$Q_2$ are good, because we want to highlight the fact that~$Q'_1$ and~$Q'_2$ are not necessarily good cubes.
Also, we add the conditions~$\ell(Q'_1) \leq \delta^m$ and~$\ell(Q'_2) \leq \delta^m$ back into the sum, which we have suppressed in Section~\ref{subsec:outline}.

\begin{align*}
  &\sum_{Q_1, Q'_1 \textup{ good} } \,
\sum_{Q_2 \textup{ good}}
\sum_{\substack{u_1\\u_2}}
\bigg\langle
\sum_{\substack{Q'_1 \\ \ell(Q'_1) \leq \delta^m}}\,
\sum_{\substack{Q'_2 \\ \ell(Q'_2) \leq \delta^m}}\,
\sum_{\substack{u'_1 \\u'_2}} \langle g_{\text{good}}, h_{u'_1}^{Q'_1} \otimes h_{u'_2}^{Q'_2}\rangle h_{u'_1}^{Q'_1} \otimes h_{u'_2}^{Q'_2}
\bigg\rangle_{S(Q_1) \otimes S(Q_2)} \noz \\
&\hspace{2cm}\times \, \langle f, h_{u_1}^{Q_1} \otimes h_{u_2}^{Q_2}\rangle
\langle T^*1, h_{u_1}^{Q_1} \otimes h_{u_2}^{Q_2}\rangle \noz\\
&+ \sum_{Q_1, Q'_1 \textup{ good} } \,
\sum_{Q_2 \textup{ good}}
\sum_{\substack{u_1\\u_2}}
\bigg\langle
\sum_{\substack{Q'_1 \\ \ell(Q'_1) = \delta^m}}\,
\sum_{\substack{Q'_2 \\ \ell(Q'_2) \leq \delta^m}}\,
\sum_{u'_2} \langle g_{\text{good}}, h_{0}^{Q'_1} \otimes h_{u'_2}^{Q'_2}\rangle h_{0}^{Q'_1} \otimes h_{u'_2}^{Q'_2}
\bigg\rangle_{S(Q_1) \otimes S(Q_2)} \noz \\
&\hspace{2cm}\times \, \langle f, h_{u_1}^{Q_1} \otimes h_{u_2}^{Q_2}\rangle
\langle T^*1, h_{u_1}^{Q_1} \otimes h_{u_2}^{Q_2}\rangle \\
& + \sum_{Q_1, Q'_1 \textup{ good} } \,
\sum_{Q_2 \textup{ good}}
\sum_{\substack{u_1\\u_2}}
\bigg\langle
\sum_{\substack{Q'_1 \\ \ell(Q'_1) \leq \delta^m}}\,
\sum_{\substack{Q'_2 \\ \ell(Q'_2) = \delta^m}}\,
\sum_{u'_1} \langle g_{\text{good}}, h_{u'_1}^{Q'_1} \otimes h_{0}^{Q'_2}\rangle h_{u'_1}^{Q'_1} \otimes h_{0}^{Q'_2}
\bigg\rangle_{S(Q_1) \otimes S(Q_2)} \noz \\
&\hspace{2cm}\times \, \langle f, h_{u_1}^{Q_1} \otimes h_{u_2}^{Q_2}\rangle
\langle T^*1, h_{u_1}^{Q_1} \otimes h_{u_2}^{Q_2}\rangle \noz\\
& + \sum_{Q_1, Q'_1 \textup{ good} } \,
\sum_{Q_2 \textup{ good}}
\sum_{\substack{u_1\\u_2}}
\bigg\langle
\sum_{\substack{Q'_1 \\ \ell(Q'_1) = \delta^m}}\,
\sum_{\substack{Q'_2 \\ \ell(Q'_2) = \delta^m}}\,
 \langle g_{\text{good}}, h_{0}^{Q'_1} \otimes h_{0}^{Q'_2}\rangle h_{0}^{Q'_1} \otimes h_{0}^{Q'_2}
\bigg\rangle_{S(Q_1) \otimes S(Q_2)} \noz \\
&\hspace{2cm}\times \, \langle f, h_{u_1}^{Q_1} \otimes h_{u_2}^{Q_2}\rangle
\langle T^*1, h_{u_1}^{Q_1} \otimes h_{u_2}^{Q_2}\rangle.
\end{align*}
Note that
\begin{align*}
 g_{\text{good}} =& \sum_{\substack{Q'_1 \\ \ell(Q'_1) \leq \delta^m}}\,
\sum_{\substack{Q'_2 \\ \ell(Q'_2) \leq \delta^m}}\,
\sum_{\substack{u'_1 \\u'_2}} \langle g_{\text{good}}, h_{u'_1}^{Q'_1} \otimes h_{u'_2}^{Q'_2}\rangle h_{u'_1}^{Q'_1} \otimes h_{u'_2}^{Q'_2} \\
& + \sum_{\substack{Q'_1 \\ \ell(Q'_1) = \delta^m}}\,
\sum_{\substack{Q'_2 \\ \ell(Q'_2) \leq \delta^m}}\,
\sum_{u'_2} \langle g_{\text{good}}, h_{0}^{Q'_1} \otimes h_{u'_2}^{Q'_2}\rangle h_{0}^{Q'_1} \otimes h_{u'_2}^{Q'_2} \\
& + \sum_{\substack{Q'_1 \\ \ell(Q'_1) \leq \delta^m}}\,
\sum_{\substack{Q'_2 \\ \ell(Q'_2) = \delta^m}}\,
\sum_{u'_2} \langle g_{\text{good}}, h_{u'_1}^{Q'_1} \otimes h_{0}^{Q'_2}\rangle h_{u'_1}^{Q'_1} \otimes h_{0}^{Q'_2} \\
& + \sum_{\substack{Q'_1 \\ \ell(Q'_1) = \delta^m}}\,
\sum_{\substack{Q'_2 \\ \ell(Q'_2) = \delta^m}}\,
\langle g_{\text{good}}, h_{0}^{Q'_1} \otimes h_{0}^{Q'_2}\rangle h_{0}^{Q'_1} \otimes h_{0}^{Q'_2}.
\end{align*}
Hence for fixed~$u_1$ and~$u_2$ we are to bound
\begin{align*}
   &\sum_{Q_1, Q'_1 \textup{ good} } \,
\sum_{Q_2 \textup{ good}}
\sum_{\substack{u_1\\u_2}}
\langle g_{\text{good}}\rangle_{S(Q_1) \otimes S(Q_2)}
 \langle f, h_{u_1}^{Q_1} \otimes h_{u_2}^{Q_2}\rangle
\langle T^*1, h_{u_1}^{Q_1} \otimes h_{u_2}^{Q_2}\rangle \noz \\
   & = \sum_{Q'_1} \sum_{\substack{Q_1 \textup{ good} \\ Q_1 \subset Q'_1 \\ \ell(Q_1) = \delta^r \ell(Q'_1)}} \,
   \sum_{Q'_2} \sum_{\substack{Q_2 \textup{ good} \\ Q_2 \subset Q'_2 \\ \ell(Q_2) = \delta^r \ell(Q'_2)}} \,
\sum_{\substack{u_1\\u_2}}
\langle g_{\text{good}}\rangle_{Q_1 \otimes Q_2}
 \langle f, h_{u_1}^{Q_1} \otimes h_{u_2}^{Q_2}\rangle
\langle T^*1, h_{u_1}^{Q_1} \otimes h_{u_2}^{Q_2}\rangle \noz \\
  & = \Big \langle \Pi_{T^*1}^{u_1 u_2} g_{\text{good}},f \Big \rangle,
\end{align*}
where
\begin{equation}\label{eq:full paraproduct}
\Pi_{b}^{u_1 u_2} w := \sum_{Q'_1} \sum_{\substack{Q_1 \textup{ good} \\ Q_1 \subset Q'_1 \\ \ell(Q_1) = \delta^r \ell(Q'_1)}} \,
   \sum_{Q'_2} \sum_{\substack{Q_2 \textup{ good} \\ Q_2 \subset Q'_2 \\ \ell(Q_2) = \delta^r \ell(Q'_2)}} \,
   \langle w\rangle_{Q_1 \otimes Q_2}
   \langle b, h_{u_1}^{Q_1} \otimes h_{u_2}^{Q_2}\rangle
   h_{u_1}^{Q_1} \otimes h_{u_2}^{Q_2}.
\end{equation}
Using Lemma~\ref{lem4:InIn} below, we conclude that the sums involving~\eqref{InIn_eq1d}, \eqref{InIn_eq2d} , \eqref{InIn_eq3d} and~\eqref{InIn_eq4d} are dominated by
$$| \langle \Pi_{T^*1}^{u_1 u_2} g_{\text{good}},f \rangle|
\leq \|\Pi_{T^*1}^{u_1 u_2} g_{\text{good}}\|_{L^2(\mu)}
\|f\|_{L^2(\mu)}
\leq \|T^*1\|_{\bmo_{\textup{prod}}(\mu)} \|g\|_{L^2(\mu)} \|f\|_{L^2(\mu)}.
$$

\begin{lem}\label{lem4:InIn}
  There holds that
  \[
  \|\Pi_{b}^{u_1 u_2}\|_{L^2(\mu) \rightarrow L^2(\mu)} \ls \|b\|_{\bmo_{\textup{prod}}(\mu)}.
  \]
\end{lem}

\begin{proof}
  Suppose $w \in L^2(\mu)$.
  We recall
  the strong maximal function $M_{\Dd'}$ with respect~$\Dd'$, being defined as
  \[M_{\Dd'}w(x) := \sup_{\substack{R' \in \Dd' \\ R' \ni x}} \frac{1}{\mu(5R')} \int_{R'} |w(y)| \, d\mu(y).\]
  Consider
  \begin{align*}
    \|\Pi_{b}^{u_1 u_2} w\|^2_{L^2(\mu_2)}
    & \leq \sum_{S \in \Dd'} \sum_{\substack{R = Q_1 \times Q_2 \subset \Dd'{\text{-good}} \\ R \subset S \\ \text{gen}(R) = \text{gen}(S) + (r,r)}}
    |\langle w\rangle_{S}|^2
    |\langle b, h_{u_1}^{Q_1} \otimes h_{u_2}^{Q_2}\rangle|^2 \\
    & = 2 \int_{0}^{\infty} \sum_{\substack{S \in \Dd' \\|\langle w\rangle_{S}| >t }}
     \sum_{\substack{R = Q_1 \times Q_2 \subset \Dd'{\text{-good}} \\ R \subset S \\ \text{gen}(R) = \text{gen}(S) + (r,r)}}
     |\langle b, h_{u_1}^{Q_1} \otimes h_{u_2}^{Q_2}\rangle|^2 \,tdt \\
    & \leq 2 \int_{0}^{\infty} \sum_{\substack{S \in \Dd' \\S \subset \{M_{\Dd'} w > t \} }}
     \sum_{\substack{R = Q_1 \times Q_2 \subset \Dd'{\text{-good}} \\ R \subset S \\ \text{gen}(R) = \text{gen}(S) + (r,r)}}
     |\langle b, h_{u_1}^{Q_1} \otimes h_{u_2}^{Q_2}\rangle|^2 \,tdt \\
    & \ls \|b\|^2_{\bmo_{\textup{prod}}(\mu)}
     \int_{0}^{\infty} \mu( \{M_{\Dd'} w > t \}) \, tdt \\
    & = \|b\|^2_{\bmo_{\textup{prod}}(\mu)}
    \|M_{\Dd'} w\|^2_{L^2(\mu)} \\
    & \ls \|b\|^2_{\bmo_{\textup{prod}}(\mu)}
    \|w\|^2_{L^2(\mu)}. \qedhere
  \end{align*}
\end{proof}

\section{Mixed Paraproducts}\label{sec:mix_para}
So far, we have completed the estimation of $|\E\langle Tf_{k,\text{good}}, g_{k,\text{good}}\rangle|$ in~\eqref{coreeq.1} corresponding to the subseries when $\ell(Q_1) \leq \ell(Q_1')$ and~$\ell(Q_2) \leq \ell(Q_2')$.

In this section, we discuss the difference which arises in estimating the second subseries $\ell(Q_1) \leq \ell(Q_1')$ and~$\ell(Q_2) > \ell(Q_2')$. Briefly, instead of the full paraproduct $\Pi_{b}^{u_1 u_2}$ being defined in~\eqref{eq:full paraproduct}, we will obtain a mixed paraproduct~$\Pi_{\text{mixed},b}^{u_1u'_2}$ as shown in~\eqref{eq:mixed paraproduct} below.
To see that, we start by estimating the following sums, which is a substitution of~\eqref{InIn_eq.15} in the full paraproduct case,
\begin{align}\label{Mix:eq1}
  &\hspace{-0.5cm}\sum_{\substack{Q_1,Q'_1  \\ \text{nested} }} \,
  \sum_{\substack{Q'_2,Q_2\\ \text{nested} }} \,
  \sum_{\substack{u_1,u'_1 \\u_2,u'_2}} \,
  \langle f, h_{u_1}^{Q_1} \otimes h_{u_2}^{Q_2}\rangle
  \langle g, h_{u'_1}^{Q'_1} \otimes h_{u'_2}^{Q'_2}\rangle
  \langle h_{u'_1}^{Q'_1} \rangle_{Q_1} \langle h_{u_2}^{Q_2} \rangle_{Q'_2}   \langle T^*_1(h_{u_1}^{Q_1} \otimes h_{u'_2}^{Q'_2}), 1 \rangle \\
  &+  \sum_{\substack{Q_1,Q'_1  \\ \text{nested} }} \,
  \sum_{\substack{Q'_2,Q_2 \\ \text{ nested} \\ \text{gen}(Q_2) = m}} \,
  \sum_{\substack{u_1,u'_1 \\u'_2}} \,
  \langle f, h_{u_1}^{Q_1} \otimes h_{0}^{Q_2}\rangle
  \langle g, h_{u'_1}^{Q'_1} \otimes h_{u'_2}^{Q'_2}\rangle
  \langle h_{u'_1}^{Q'_1} \rangle_{Q_1} \langle h_{0}^{Q_2} \rangle_{Q'_2}   \langle T^*_1(h_{u_1}^{Q_1} \otimes h_{u'_2}^{Q'_2}), 1 \rangle  \noz\\
  &+  \sum_{\substack{Q_1,Q'_1  \\ \text{nested} \\ \text{gen}(Q'_1) = m }} \,
  \sum_{\substack{Q'_2,Q_2\\ \text{nested}  }} \,
  \sum_{\substack{u_1\\u_2,u'_2}} \,
  \langle f, h_{u_1}^{Q_1} \otimes h_{u_2}^{Q_2}\rangle
  \langle g, h_{0}^{Q'_1} \otimes h_{u'_2}^{Q'_2}\rangle
  \langle h_{0}^{Q'_1} \rangle_{Q_1} \langle h_{u_2}^{Q_2} \rangle_{Q'_2}   \langle T^*_1(h_{u_1}^{Q_1} \otimes h_{u'_2}^{Q'_2}), 1 \rangle \noz\\
  &+ \sum_{\substack{Q_1,Q'_1  \\ \text{nested} \\ \text{gen}(Q'_1) = m }} \,
  \sum_{\substack{Q'_2,Q_2\\ \text{nested} \\ \text{gen}(Q_2) = m }} \,
  \sum_{\substack{u_1 \\u'_2}} \,
  \langle f, h_{u_1}^{Q_1} \otimes h_{0}^{Q_2}\rangle
  \langle g, h_{0}^{Q'_1} \otimes h_{u'_2}^{Q'_2}\rangle
  \langle h_{0}^{Q'_1} \rangle_{Q_1} \langle h_{0}^{Q_2} \rangle_{Q'_2}   \langle T^*_1(h_{u_1}^{Q_1} \otimes h_{u'_2}^{Q'_2}), 1 \rangle.\noz
\end{align}
Since~$Q_1$ is nested in~$Q'_1$, there exists a unique cube $S(Q_1) \in \Dd'_1$ such that $\ell(Q_1) = \delta^{r}\ell(S(Q_1))$ and $Q_1 \subset S(Q_1) \subset Q'_1$.
Note that in this case, $Q'_2$ is nested in~$Q_2$, meaning $\ell(Q'_2) < \delta^r \ell(Q_2)$ and $\rho_2(Q_2,Q'_2) \leq \mathcal{C}\ell(Q'_2)^{\gamma_2} \ell(Q_2)^{1-\gamma_2}$.
Hence, there exists a unique cube $S(Q'_2) \in \Dd_2$ such that $\ell(Q'_2) = \delta^{r}\ell(S(Q'_2))$ and $Q'_2 \subset S(Q'_2) \subset Q_2$.
Therefore,
\[
\langle h_{u'_1}^{Q'_1}\rangle_{Q_1}
= \langle h_{u'_1}^{Q'_1}\rangle_{S(Q_1)}, \,
 \langle h_{0}^{Q'_1}\rangle_{Q_1}
=  \langle h_{0}^{Q'_1}\rangle_{S(Q_1)}, \,
\langle h_{u_2}^{Q_2}\rangle_{Q'_2}
= \langle h_{u_2}^{Q_2}\rangle_{S(Q'_2)} \, \text{and}\,
 \langle h_{0}^{Q_2}\rangle_{Q'_2}
=  \langle h_{0}^{Q_2}\rangle_{S(Q_2)}.\]
We define
\[
 \langle f_{\text{good}}, h_{u_1}^{Q_1} \otimes h_{u_2}^{Q_2} \rangle :=
\left\{
  \begin{array}{ll}
    \langle f, h_{u_1}^{Q_1} \otimes h_{u_2}^{Q_2} \rangle,
     & \hbox{\text{if $Q_2$ is good};} \\
    0, & \hbox{\text{if $Q_2$ is bad}}
  \end{array}
\right.
\]
and
\[
 \langle g_{\text{good}}, h_{u'_1}^{Q'_1} \otimes h_{u'_2}^{Q'_2} \rangle :=
\left\{
  \begin{array}{ll}
    \langle g, h_{u'_1}^{Q'_1} \otimes h_{u'_2}^{Q'_2} \rangle,
     & \hbox{\text{if $Q'_1$ is good};} \\
    0, & \hbox{\text{if $Q'_1$ is bad}.}
  \end{array}
\right.
\]

Below, we rewrite the four terms in \eqref{Mix:eq1} via $f_{\text{good}}$ and~$g_{\text{good}}$.  We explicitly write the fact that $Q_1$ and~$Q'_2$ are good, because we want to highlight the fact that~$Q'_1$ and~$Q_2$ are not necessarily good cubes.
Also, we add the conditions~$\ell(Q'_1) \leq \delta^m$ and~$\ell(Q_2) \leq \delta^m$ back, which we have suppressed in Section~\ref{subsec:outline}.  This then gives

\begin{align*}
    &\hspace{-1cm} \eqref{Mix:eq1} = \sum_{Q_1\text{ good} } \,
  \sum_{Q'_2 \text{ good} } \,
   \langle T_1(1), h_{u_1}^{Q_1} \otimes h_{u'_2}^{Q'_2} \rangle \\
   & \hspace{-0.5cm}\times \bigg\{
  \sum_{\substack{Q'_1 \\ \ell(Q'_1) \leq \delta^m}} \,
  \sum_{\substack{Q_2 \\ \ell(Q_2) \leq \delta^m}} \,
  \sum_{\substack{u_1,u'_1 \\u_2,u'_2}} \,
  \langle f_{\text{good}}, h_{u_1}^{Q_1} \otimes h_{u_2}^{Q_2}\rangle
  \langle g_{\text{good}}, h_{u'_1}^{Q'_1} \otimes h_{u'_2}^{Q'_2}\rangle
  \langle h_{u'_1}^{Q'_1} \rangle_{S(Q_1)} \langle h_{u_2}^{Q_2} \rangle_{S(Q'_2)}   \noz\\
  &+
  \sum_{\substack{Q'_1 \\ \ell(Q'_1) \leq \delta^m}} \,
  \sum_{\substack{Q_2 \\ \ell(Q_2) = \delta^m}} \,
  \sum_{\substack{u_1,u'_1 \\u'_2}} \,
  \langle f_{\text{good}}, h_{u_1}^{Q_1} \otimes h_{0}^{Q_2}\rangle
  \langle g_{\text{good}}, h_{u'_1}^{Q'_1} \otimes h_{u'_2}^{Q'_2}\rangle
  \langle h_{u'_1}^{Q'_1} \rangle_{S(Q_1)} \langle h_{0}^{Q_2} \rangle_{S(Q'_2)}   \noz\\
  &+
  \sum_{\substack{Q'_1 \\ \ell(Q'_1) = \delta^m}} \,
  \sum_{\substack{Q_2 \\ \ell(Q_2) \leq \delta^m}} \,
  \sum_{\substack{u_1\\u_2,u'_2}} \,
  \langle f_{\text{good}}, h_{u_1}^{Q_1} \otimes h_{u_2}^{Q_2}\rangle
  \langle g_{\text{good}}, h_{0}^{Q'_1} \otimes h_{u'_2}^{Q'_2}\rangle
  \langle h_{0}^{Q'_1} \rangle_{S(Q_1)} \langle h_{u_2}^{Q_2} \rangle_{S(Q'_2)}   \noz\\
  &+
  \sum_{\substack{Q'_1 \\ \ell(Q'_1) = \delta^m}} \,
  \sum_{\substack{Q_2 \\ \ell(Q_2) = \delta^m}} \,
  \sum_{\substack{u_1 \\u'_2}} \,
  \langle f_{\text{good}}, h_{u_1}^{Q_1} \otimes h_{0}^{Q_2}\rangle
  \langle g_{\text{good}}, h_{0}^{Q'_1} \otimes h_{u'_2}^{Q'_2}\rangle
  \langle h_{0}^{Q'_1} \rangle_{S(Q_1)} \langle h_{0}^{Q_2} \rangle_{S(Q'_2)}   \bigg\}
\end{align*}
which equals
\begin{align*}
    &\hspace{-1cm} \sum_{Q_1\text{ good} } \,
  \sum_{Q'_2 \text{ good} } \,
  \sum_{\substack{u_1 \\u'_2}} \,
   \langle T_1(1), h_{u_1}^{Q_1} \otimes h_{u'_2}^{Q'_2} \rangle \\
   & \hspace{-0.5cm}\times \bigg\{
  \Big \langle \sum_{\substack{Q_2 \\ \ell(Q_2) \leq \delta^m}} \,
  \sum_{u_2}
  \langle f_{\text{good}}, h_{u_1}^{Q_1} \otimes h_{u_2}^{Q_2}\rangle
   h_{u_2}^{Q_2} \Big\rangle_{S(Q'_2)}
  \Big \langle \sum_{\substack{Q'_1 \\ \ell(Q'_1) \leq \delta^m}} \,
  \sum_{u'_1}
  \langle g_{\text{good}}, h_{u'_1}^{Q'_1} \otimes h_{u'_2}^{Q'_2} \rangle
   h_{u'_1}^{Q'_1} \Big\rangle_{S(Q_1)}   \noz\\
  &+
  \Big \langle \sum_{\substack{Q_2 \\ \ell(Q_2) = \delta^m}} \,
  \langle f_{\text{good}}, h_{u_1}^{Q_1} \otimes h_{0}^{Q_2}\rangle
   h_{0}^{Q_2} \Big\rangle_{S(Q'_2)}
  \Big \langle \sum_{\substack{Q'_1 \\ \ell(Q'_1) \leq \delta^m}} \,
  \sum_{u'_1}
  \langle g_{\text{good}}, h_{u'_1}^{Q'_1} \otimes h_{u'_2}^{Q'_2} \rangle
   h_{u'_1}^{Q'_1} \Big\rangle_{S(Q_1)}   \noz\\
  &+
  \Big \langle \sum_{\substack{Q_2 \\ \ell(Q_2) \leq \delta^m}} \,
  \sum_{u_2}
  \langle f_{\text{good}}, h_{u_1}^{Q_1} \otimes h_{u_2}^{Q_2}\rangle
   h_{u_2}^{Q_2} \Big\rangle_{S(Q'_2)}
  \Big \langle \sum_{\substack{Q'_1 \\ \ell(Q'_1) = \delta^m}} \,
  \langle g_{\text{good}}, h_{0}^{Q'_1} \otimes h_{u'_2}^{Q'_2} \rangle
   h_{0}^{Q'_1} \Big\rangle_{S(Q_1)}   \noz\\
  &+
  \Big \langle \sum_{\substack{Q_2 \\ \ell(Q_2) = \delta^m}} \,
  \langle f_{\text{good}}, h_{u_1}^{Q_1} \otimes h_{0}^{Q_2}\rangle
   h_{0}^{Q_2} \Big\rangle_{S(Q'_2)}
  \Big \langle \sum_{\substack{Q'_1 \\ \ell(Q'_1) = \delta^m}} \,
  \langle g_{\text{good}}, h_{0}^{Q'_1} \otimes h_{u'_2}^{Q'_2} \rangle
   h_{0}^{Q'_1} \Big\rangle_{S(Q_1)}     \bigg\}.
\end{align*}
Define
\[f_{\text{good}}^{Q_1,u_1}(y)
:= \Big\langle f_{\text{good}},h_{u_1}^{Q_1}\Big\rangle_1 (y)
\quad\text{ and }\quad
g_{\text{good}}^{Q'_2,u'_2}(y)
:= \Big\langle g_{\text{good}},h_{u'_2}^{Q'_2}\Big\rangle_2 (y). \]
Then the above expression becomes
\begin{align*}
    &\hspace{-1cm} \sum_{Q_1\text{ good} } \,
  \sum_{Q'_2 \text{ good} } \,
  \sum_{\substack{u_1 \\u'_2}} \,
   \langle T_1(1), h_{u_1}^{Q_1} \otimes h_{u'_2}^{Q'_2} \rangle \\
   & \hspace{-0.3cm}\times \bigg\{
  \Big \langle \sum_{\substack{Q_2 \\ \ell(Q_2) \leq \delta^m}} \,
  \sum_{u_2}
  \langle f_{\text{good}}^{Q_1,u_1}, h_{u_2}^{Q_2}\rangle
   h_{u_2}^{Q_2} \Big\rangle_{S(Q'_2)}
  \Big \langle \sum_{\substack{Q'_1 \\ \ell(Q'_1) \leq \delta^m}} \,
  \sum_{u'_1}
  \langle g_{\text{good}}^{Q'_2,u'_2}, h_{u'_1}^{Q'_1} \rangle
   h_{u'_1}^{Q'_1} \Big\rangle_{S(Q_1)}   \noz\\
  &+
  \Big \langle \sum_{\substack{Q_2 \\ \ell(Q_2) = \delta^m}} \,
  \langle f_{\text{good}}^{Q_1,u_1}, h_{0}^{Q_2}\rangle
   h_{0}^{Q_2} \Big\rangle_{S(Q'_2)}
  \Big \langle \sum_{\substack{Q'_1 \\ \ell(Q'_1) \leq \delta^m}} \,
  \sum_{u'_1}
  \langle g_{\text{good}}^{Q'_2,u'_2}, h_{u'_1}^{Q'_1} \rangle
   h_{u'_1}^{Q'_1} \Big\rangle_{S(Q_1)}   \noz\\
  &+
  \Big \langle \sum_{\substack{Q_2 \\ \ell(Q_2) \leq \delta^m}} \,
  \sum_{u_2}
  \langle f_{\text{good}}^{Q_1,u_1}, h_{u_2}^{Q_2}\rangle
   h_{u_2}^{Q_2} \Big\rangle_{S(Q'_2)}
  \Big \langle \sum_{\substack{Q'_1 \\ \ell(Q'_1) = \delta^m}} \,
  \langle g_{\text{good}}^{Q'_2,u'_2}, h_{0}^{Q'_1} \rangle
   h_{0}^{Q'_1} \Big\rangle_{S(Q_1)}   \noz\\
  &+
  \Big \langle \sum_{\substack{Q_2 \\ \ell(Q_2) = \delta^m}} \,
  \langle f_{\text{good}}^{Q_1,u_1}, h_{0}^{Q_2}\rangle
   h_{0}^{Q_2} \Big\rangle_{S(Q'_2)}
  \Big \langle \sum_{\substack{Q'_1 \\ \ell(Q'_1) = \delta^m}} \,
  \langle g_{\text{good}}^{Q'_2,u'_2}, h_{0}^{Q'_1} \rangle
   h_{0}^{Q'_1} \Big\rangle_{S(Q_1)}     \bigg\} \\
    &\hspace{-0.5cm}= \sum_{Q_1\text{ good} } \,
  \sum_{Q'_2 \text{ good} } \,
  \sum_{\substack{u_1 \\u'_2}} \,
   \langle T_1(1), h_{u_1}^{Q_1} \otimes h_{u'_2}^{Q'_2} \rangle \\
   & \hspace{-0.3cm}\times \bigg\{
  \Big \langle
  \sum_{\substack{Q_2 \\ \ell(Q_2) \leq \delta^m}} \,
  \sum_{u_2}
  \langle f_{\text{good}}^{Q_1,u_1}, h_{u_2}^{Q_2}\rangle
   h_{u_2}^{Q_2}
   +  \sum_{\substack{Q_2 \\ \ell(Q_2) = \delta^m}} \,
  \langle f_{\text{good}}^{Q_1,u_1}, h_{0}^{Q_2}\rangle
   h_{0}^{Q_2}
   \Big\rangle_{S(Q'_2)} \\
  & \times \Big \langle
   \sum_{\substack{Q'_1 \\ \ell(Q'_1) \leq \delta^m}} \,
  \sum_{u'_1}
  \langle g_{\text{good}}^{Q'_2,u'_2}, h_{u'_1}^{Q'_1} \rangle
   h_{u'_1}^{Q'_1}
    + \sum_{\substack{Q'_1 \\ \ell(Q'_1) = \delta^m}} \,
  \langle g_{\text{good}}^{Q'_2,u'_2}, h_{0}^{Q'_1} \rangle
   h_{0}^{Q'_1}
   \Big\rangle_{S(Q_1)}    \bigg\} \\
   &\hspace{-0.5cm} =  \sum_{Q_1\text{ good} } \,
  \sum_{Q'_2 \text{ good} } \,
  \sum_{\substack{u_1 \\u'_2}} \,
   \langle T_1(1), h_{u_1}^{Q_1} \otimes h_{u'_2}^{Q'_2} \rangle
   \langle f_{\text{good}}^{Q_1,u_1}\rangle_{S(Q'_2)}
   \langle g_{\text{good}}^{Q'_2,u'_2}\rangle_{S(Q_1)} \\
   &\hspace{-0.5cm} =  \sum_{Q'_1 } \,
    \sum_{\substack{Q_1 \textup{ good} \\ Q_1 \subset Q'_1 \\ \ell(Q_1) = \delta^r \ell(Q'_1)}} \,
  \sum_{Q_2} \,
   \sum_{\substack{Q'_2 \textup{ good} \\ Q'_2 \subset Q_2 \\ \ell(Q'_2) = \delta^r \ell(Q_2)}} \,
  \sum_{\substack{u_1 \\u'_2}} \,
   \langle T_1(1), h_{u_1}^{Q_1} \otimes h_{u'_2}^{Q'_2} \rangle
   \langle f_{\text{good}}^{Q_1,u_1}\rangle_{Q_2}
   \langle g_{\text{good}}^{Q'_2,u'_2}\rangle_{Q'_1} \\
   &\hspace{-0.5cm} =  \sum_{Q'_1 } \,
    \sum_{\substack{Q_1 \textup{ good} \\ Q_1 \subset Q'_1 \\ \ell(Q_1) = \delta^r \ell(Q'_1)}} \,
  \sum_{Q_2} \,
   \sum_{\substack{Q'_2 \textup{ good} \\ Q'_2 \subset Q_2 \\ \ell(Q'_2) = \delta^r \ell(Q_2)}} \,
  \sum_{\substack{u_1 \\u'_2}} \,
   \langle T_1(1), h_{u_1}^{Q_1} \otimes h_{u'_2}^{Q'_2} \rangle \\
   & \times \Big\langle f_{\text{good}}, h_{u_1}^{Q_1} \otimes \frac{\chi_{Q_2}}{\mu_2(Q_2)} \Big\rangle
   \Big\langle g_{\text{good}}, \frac{\chi_{Q'_1}}{\mu_1(Q'_1)} \otimes h_{u'_2}^{Q'_2}\Big\rangle \\
   &\hspace{-0.5cm} = \Big\langle \Pi_{\text{mixed},T_1(1)}^{u_1u'_2}
    f_{\text{good}}, g_{\text{good}}  \Big \rangle
\end{align*}
for each fixed~$u_1$ and~$u'_2$, where
\begin{align}\label{eq:mixed paraproduct}
\Pi_{\text{mixed},b}^{u_1u'_2} w
&:= \sum_{Q'_1 } \,
\sum_{\substack{Q_1 \textup{ good} \\ Q_1 \subset Q'_1 \\ \ell(Q_1)=
   \delta^r \ell(Q'_1)}} \,
   \sum_{Q_2} \,
   \sum_{\substack{Q'_2 \textup{ good} \\ Q'_2 \subset Q_2 \\ \ell(Q'_2) = \delta^r \ell(Q_2)}} \,
   \Big\langle w, h_{u_1}^{Q_1} \otimes \frac{\chi_{Q_2}}{\mu_2(Q_2)}
   \Big\rangle \noz\\
   & \times \langle b, h_{u_1}^{Q_1} \otimes h_{u'_2}^{Q'_2} \rangle
   \frac{\chi_{Q'_1}}{\mu_1(Q'_1)} \otimes h_{u'_2}^{Q'_2}.
\end{align}

Using Proposition~\ref{prop1:mixed_para} below, we conclude that the sums in~\eqref{InIn_eq.15} are dominated by
\begin{align*}
\Big|\Big\langle \Pi_{\text{mixed},T_1(1)}^{u_1u'_2}
    f_{\text{good}}^{Q_1,u_1}, g_{\text{good}}^{Q'_2,u'_2}  \Big \rangle\Big|
  &\ls \|T_1(1)\|_{\bmo_{\textup{prod}}(\mu)} \| f_{\text{good}}^{Q_1,u_1}\|_{L^2(\mu)}  \|g_{\text{good}}^{Q'_2,u'_2} \|_{L^2(\mu)}\\
  &\ls \|f\|_{L^2(\mu)} \|g\|_{L^2(\mu)}.
\end{align*}

The final step is to prove Proposition~\ref{prop1:mixed_para}.

\begin{prop}\label{prop1:mixed_para}
  Given functions~$w, v \in L^2(\mu)$ and $b \in \bmo_{\textup{prod}}(\mu)$,
  there holds that
  \[|\langle \Pi_{\textup{mixed},b}^{u_1u'_2} w, v\rangle| \ls
  \|b\|_{\bmo_{\textup{prod}}(\mu)} \|w\|_{L^2(\mu)}  \|v\|_{L^2(\mu)}.
  \]
\end{prop}
\begin{proof}
We write
  \begin{align*}
    &\langle \Pi_{\text{mixed},b}^{u_1u'_2} w, v\rangle \\
     & = \sum_{Q'_1 } \,
     \sum_{\substack{Q_1 \textup{ good} \\ Q_1 \subset Q'_1 \\ \ell(Q_1)=
   \delta^r \ell(Q'_1)}} \,
   \sum_{Q_2} \,
   \sum_{\substack{Q'_2 \textup{ good} \\ Q'_2 \subset Q_2 \\ \ell(Q'_2) = \delta^r \ell(Q_2)}} \,
   \Big\langle w, h_{u_1}^{Q_1} \otimes \frac{\chi_{Q_2}}{\mu_2(Q_2)}
   \Big\rangle
    \langle b, h_{u_1}^{Q_1} \otimes h_{u'_2}^{Q'_2} \rangle
   \Big\langle v, \frac{\chi_{Q'_1}}{\mu_1(Q'_1)} \otimes h_{u'_2}^{Q'_2} \Big\rangle \\
   & = \bigg \langle b,
   \sum_{Q'_1 } \,
   \sum_{\substack{Q_1 \textup{ good} \\ Q_1 \subset Q'_1 \\ \ell(Q_1)=
   \delta^r \ell(Q'_1)}} \,
   \sum_{Q_2} \,
   \sum_{\substack{Q'_2 \textup{ good} \\ Q'_2 \subset Q_2 \\ \ell(Q'_2) = \delta^r \ell(Q_2)}} \,
   \Big\langle w, h_{u_1}^{Q_1} \otimes \frac{\chi_{Q_2}}{\mu_2(Q_2)}
   \Big\rangle
   \Big\langle v, \frac{\chi_{Q'_1}}{\mu_1(Q'_1)} \otimes h_{u'_2}^{Q'_2} \Big\rangle
    h_{u_1}^{Q_1} \otimes h_{u'_2}^{Q'_2}
   \bigg\rangle \\
   & = \bigg \langle b,
   \sum_{Q_1 \textup{ good} } \,
   \sum_{Q'_2 \textup{ good} } \,
   \Big\langle w, h_{u_1}^{Q_1} \otimes \frac{\chi_{S(Q'_2)}}{\mu_2(S(Q'_2))}
   \Big\rangle
   \Big\langle v, \frac{\chi_{S(Q_1)}}{\mu_1(S(Q_1))} \otimes h_{u'_2}^{Q'_2} \Big\rangle
    h_{u_1}^{Q_1} \otimes h_{u'_2}^{Q'_2}
   \bigg\rangle \\
   & =: \langle b, \varphi\rangle.
  \end{align*}
  Using Proposition~\ref{prop1:dual_prod} we have
  \begin{equation}\label{eq1:mix_para}
  |\langle \Pi_{\text{mixed},b}^{u_1u'_2} w, v\rangle|
  \ls \|b\|_{\bmo_{\textup{prod}}(\mu)} \|S_{\Dd_1 \Dd'_2} \varphi\|_{L^1(\mu)}.
  \end{equation}

  Next we consider $\|S_{\Dd_1 \Dd'_2} \varphi\|_{L^1(\mu)}$. We note that
  \begin{align*}
    \big(S_{\Dd_1 \Dd'_2} \varphi(x_1,x_2)\big)^2
    &\ls  \sum_{Q_1 \textup{ good} } \,
   \sum_{Q'_2 \textup{ good} } \,
   \Big|\Big\langle w, h_{u_1}^{Q_1} \otimes \frac{\chi_{S(Q'_2)}}{\mu_2(S(Q'_2))}
   \Big\rangle \Big|^2
   \Big|\Big\langle v, \frac{\chi_{S(Q_1)}}{\mu_1(S(Q_1))} \otimes h_{u'_2}^{Q'_2} \Big\rangle\Big|^2 \\
   & \hspace{0.2cm} \times   \frac{\chi_{S(Q_1)}(x_1) \otimes \chi_{S(Q'_2)}(x_2)}{\mu_1(S(Q_1))\mu_2(S(Q'_2))} \\
   & \leq \bigg( \sum_{Q_1 \textup{ good}} \, \frac{\chi_{S(Q_1)}(x_1)}{\mu_1(S(Q_1))}
   \otimes \sup_{Q'_2 \textup{ good} } \,
   \Big[ \Big|\Big\langle w, h_{u_1}^{Q_1} \otimes \frac{\chi_{S(Q'_2)}}{\mu_2(S(Q'_2))}
   \Big\rangle \Big|^2 \chi_{S(Q'_2)}(x_2)\Big]
   \bigg)\\
   & \hspace{0.2cm} \times \bigg(
   \sum_{Q'_2 \textup{ good} } \,  \sup_{Q_1 \textup{ good} } \,
   \Big[  \Big|\Big\langle v, \frac{\chi_{S(Q_1)}}{\mu_1(S(Q_1))} \otimes h_{u'_2}^{Q'_2} \Big\rangle\Big|^2 \chi_{S(Q_1)}(x_1) \Big]
   \otimes \frac{\chi_{S(Q'_2)}(x_2)}{\mu_2(S(Q'_2))}
   \bigg).
  \end{align*}
  Moreover,
  \begin{align*}
     &\hspace{-0.5cm} \sup_{Q'_2 \textup{ good} } \,
   \Big[ \Big|\Big\langle w, h_{u_1}^{Q_1} \otimes \frac{\chi_{S(Q'_2)}}{\mu_2(S(Q'_2))}
   \Big\rangle \Big|^2 \chi_{S(Q'_2)}(x_2)\Big] \\
   & = \sup_{Q'_2 \textup{ good} } \,
   \Big[ \Big| \int_{X_2} \int_{X_1}
   w(y_1,y_2) h_{u_1}^{Q_1}(y_1)
    \frac{\chi_{S(Q'_2)}(y_1)}{\mu_2(S(Q'_2))} \,d\mu_1 \,d\mu_2
   \Big|^2 \chi_{S(Q'_2)}(x_2)\Big] \\
   & \leq  \sup_{Q'_2 \textup{ good} } \,
   \Big[ \Big( \frac{1}{\mu_2(Q'_2)}
   \int_{Q'_2}|\langle w,h_{u_1}^{Q_1}\rangle_1(y_2)| \,d\mu_2\Big)^2
   \chi_{S(Q'_2)}(x_2)\Big] \\
   & = \big(M_{\Dd'_2}\langle w, h_{u_1}^{Q_1}\rangle_1\big)^2(x_2) \chi_{S(Q'_2)}(x_2) \\
   &\leq \big( M_{\Dd'_2}\langle w, h_{u_1}^{Q_1}\rangle_1\big)^2(x_2).
  \end{align*}
  Similarly, we also have
  \[\sup_{Q_1 \textup{ good} } \,
   \Big[  \Big|\Big\langle v, \frac{\chi_{S(Q_1)}}{\mu_1(S(Q_1))} \otimes h_{u'_2}^{Q'_2} \Big\rangle\Big|^2 \chi_{S(Q_1)}(x_1) \Big]
   \leq  \big(M_{\Dd_1}\langle v, h_{u'_2}^{Q'_2}\rangle_2\big)^2(x_1).\]
   Hence,
   \begin{align*}
   \big(S_{\Dd_1 \Dd'_2} \varphi(x_1,x_2)\big)^2
   &\ls
    \bigg( \sum_{Q_1 \textup{ good}} \, \frac{\chi_{S(Q_1)}(x_1)}{\mu_1(S(Q_1))}
   \otimes  \big( M_{\Dd'_2}\langle w, h_{u_1}^{Q_1}\rangle-1\big)^2(x_2) \bigg) \\
   & \hspace{0.5cm} \times
   \bigg( \sum_{Q'_2 \text{ good}}
    \big(M_{\Dd_1}\langle v, h_{u'_2}^{Q'_2}\rangle_2\big)^2(x_1)
     \otimes \frac{\chi_{S(Q'_2)}(x_2)}{\mu_2(S(Q'_2))}
   \bigg).
   \end{align*}
Next, using Cauchy--Schwartz inequality we get
\begin{align} \label{eq2:mix_para}
  \|S_{\Dd_1 \Dd'_2} \varphi\|_{L^1(\mu)}
  &\leq \bigg\|\sum_{Q_1 \textup{ good}} \, \frac{\chi_{S(Q_1)}}{\mu_1(S(Q_1))}
   \otimes  \big( M_{\Dd'_2}\langle w, h_{u_1}^{Q_1}\rangle_1\big)^2\bigg\|^{1/2}_{L^1(\mu)}\noz\\
  &\hspace{0.5cm} \times\bigg\|\sum_{Q'_2 \text{ good}}
    \big(M_{\Dd_1}\langle v, h_{u'_2}^{Q'_2}\rangle_2\big)^2
     \otimes \frac{\chi_{S(Q'_2)}}{\mu_2(S(Q'_2))}\bigg\|^{1/2}_{L^1(\mu)} \noz\\
  & \ls \|w\|_{L^2(\mu)}  \|v\|_{L^2(\mu)}.
\end{align}
The last inequality above holds because
\begin{align*}
   &\hspace{-0.5cm} \bigg\|\sum_{Q_1 \textup{ good}} \, \frac{\chi_{S(Q_1)}}{\mu_1(S(Q_1))}
   \otimes  \big( M_{\Dd'_2}\langle w, h_{u_1}^{Q_1}\rangle_1\big)^2\bigg\|_{L^1(\mu)} \\
   & = \bigg(\int_{X_1} \int_{X_2} \Big| \sum_{Q_1 \textup{ good}} \, \frac{\chi_{S(Q_1)}(x_1)}{\mu_1(S(Q_1))}
   \big(M_{\Dd'_2}\langle w, h_{u_1}^{Q_1}\rangle_1\big)^2(x_2)\Big|^2 \,d\mu_2 \,d\mu_1 \bigg)^{1/2} \\
   & = \sum_{Q_1 \textup{ good}} \int_{X_1} \int_{X_2}
   \frac{\chi_{S(Q_1)}(x_1)}{\mu_1(S(Q_1))}
   \big(M_{\Dd'_2}\langle w, h_{u_1}^{Q_1}\rangle_1\big)^2(x_2)
   \,d\mu_2 \,d\mu_1 \\
   & = \sum_{Q_1 \textup{ good}} \int_{X_1}
   \frac{\chi_{S(Q_1)}(x_1)}{\mu_1(S(Q_1))} \,d\mu_1
   \int_{X_2}  \big(M_{\Dd'_2}\langle w, h_{u_1}^{Q_1}\rangle_1\big)^2(x_2)
   \,d\mu_2 \\
   & \ls \int_{X_2}  \big|\langle w, h_{u_1}^{Q_1}\rangle_1\big|^2
   \,d\mu_2  \ls \|w\|^2_{L^2(\mu)},
\end{align*}
and similarly,
\[
 \bigg\|\sum_{Q'_2 \text{ good}}
    \big(M_{\Dd_1}\langle v, h_{u'_2}^{Q'_2}\rangle_2\big)^2
     \otimes \frac{\chi_{S(Q'_2)}}{\mu_2(S(Q'_2))}\bigg\|^{1/2}_{L^1(\mu)}
 \ls \|v\|^2_{L^2(\mu)}.
\]

Thus inequalities~\eqref{eq1:mix_para} and~\eqref{eq2:mix_para} establish Proposition~\ref{prop1:mixed_para}.
\end{proof}

\section{Proof of Theorem~\ref{thm:T1thm_ver2} - $T(1)$ Theorem under Weaker \emph{A Priori} Assumption}\label{sec:T1 ver 2}
In this section, we establish Theorem~\ref{thm:T1thm_ver2} which is  a $T(1)$ type under the weaker \emph{a priori} assumption that for all functions $f, g \in C^{\eta}_{\textup{c}}(\mathcal{X}) $ which are supported on a rectangle~$R \in \mathcal{X}$, we have
      \[|\langle Tf,g \rangle| \leq C(R) \|g\|_{L^2(\mu)} \|g\|_{L^2(\mu)}.\]
Let $(\mathcal{X},\rho,\mu)$ be a non-homogeneous bi-parameter quasimetric space.
Let $T: L^2(\mu) \rightarrow L^2(\mu)$ be a bi-parameter singular integral operator on~$\mathcal{X}$ defined as in Section~\ref{subsec:assume}.
For each $x = x_1 \times x_2 \in \mathcal{X}$, $m \in \Z$ and~$\delta \in (0,1)$ satisfying $96A_0^6 \delta \leq 1$, define
\[
\mathcal{E} = \mathcal{E}(x,m,\delta)
:= \{f \in \mathcal{X}: \|f\|_{L^2(\mu)} = 1, \supp f \in R_{x,m,\delta}\},
\]
where $R_{x,m,\delta} := Q_{1,x_1,m,\delta} \times Q_{2,x_2,m,\delta}$,
$Q_{i,x,m,\delta}$ is the dyadic cube in~$X_i$, centred at~$x_i$ and at level~$\delta^m$, $i = 1,2$.
Define
\[M(\delta^m) := \sup \{|\langle Tf,g \rangle|: f, g \in \mathcal{E}\}.\]

Fix $f, g \in \mathcal{E}$ such that
$0.7 M(\delta^m) \leq |\langle Tf,g \rangle|$.
We want to show that $M(\delta^m) \leq C$, where the constant~$C$ is independent of~$f$ and~$g$. This then implies that~$\|T\| \leq C$.

Notice that the setup here is the same as that of the proof of Theorem~\ref{thm:T1thm_ver1} in Section~\ref{subsec:initial_reduce}, except that~$\|T\|$ is replaced by~$M(\delta^m)$.
Thus, the proof of Theorem~\ref{thm:T1thm_ver2} is exactly the same as that of Theorem~\ref{thm:T1thm_ver1}, except replacing~$\|T\|$ by~$M(\delta^m)$.

For convenience, we write down below the initial reduction of the proof of Theorem~\ref{thm:T1thm_ver2}. As explained above, this part is similar to Section~\ref{subsec:initial_reduce}, but we use~$M(\delta^m)$ instead of~$\|T\|$.

For each integer $k >m$ we define
\begin{eqnarray*}
  E_kf &:=& \sum_{\substack{Q_1 \in \Dd_1 \\ k>\text{gen}(Q_1) \geq m}} \sum_{u_1=1}^{M_{Q_1}-1}
   h_{u_1}^{Q_1} \otimes \langle f,h_{u_1}^{Q_1} \rangle_1
   + \sum_{\substack{Q_1 \in \Dd_1 \\ \text{gen}(Q_1) = m}}
   h_{0}^{Q_1} \otimes \langle f,h_{0}^{Q_1} \rangle_1\\
   &=:&  E_{k,u} f +E_{k,0} f, \\
  E'_kg &:=& \sum_{\substack{Q'_1 \in \Dd'_1 \\ k>\text{gen}(Q'_1) \geq m}}
  \sum_{u'_1=1}^{M_{Q'_1}-1}
  h_{u'_1}^{Q'_1} \otimes \langle g,h_{u'_1}^{Q'_1} \rangle_1
  + \sum_{\substack{Q'_1 \in \Dd'_1 \\ \text{gen}(Q'_1) = m}}
   h_{0}^{Q'_1} \otimes \langle g,h_{0}^{Q'_1} \rangle_1, \\
   &=:&  E'_{k,u} g +E'_{k,0} g, \\
  \widetilde{E}_kf &:=& \sum_{\substack{Q_2 \in \Dd_2 \\ k>\text{gen}(Q_2) \geq m}}
  \sum_{u_2=1}^{M_{Q_2}-1}
  h_{u_2}^{Q_2} \otimes \langle f,h_{u_2}^{Q_2} \rangle_2
   + \sum_{\substack{Q_2 \in \Dd_2 \\ \text{gen}(Q_2) = m}}
   h_{0}^{Q_2} \otimes \langle f,h_{0}^{Q_2} \rangle_2, \quad \text{and} \\
   &=:& \widetilde{E}_{k,u} f +\widetilde{E}_{k,0} f, \\
  \widetilde{E}'_kg &:=& \sum_{\substack{Q'_2 \in \Dd'_2 \\ k>\text{gen}(Q'_2) \geq m}}
  \sum_{u'_2=1}^{M_{Q'_2}-1}
   h_{u'_2}^{Q'_2} \otimes \langle g,h_{u'_2}^{Q'_2} \rangle_2
   + \sum_{\substack{Q'_2 \in \Dd'_2 \\ \text{gen}(Q'_2) = m}}
   h_{0}^{Q'_2} \otimes \langle g,h_{0}^{Q'_2} \rangle_2\\
   &=:& \widetilde{E}'_{k,u} g+\widetilde{E}'_{k,0} g.
\end{eqnarray*}
Here $\langle f,h_{u_1}^{Q_1} \rangle_1(y_2) = \langle f(\cdot,y_2),h_{u_1}^{Q_1} \rangle = \int f(y_1,y_2)h_{u_1}^{Q_1}(y_1)\,d\mu_1(y_1)$, and $\text{gen}(Q)$ means the generation of the cube~$Q$.

We recall the definitions of good and bad cubes from Section~\ref{subsec:good_bad_cube}. We fix the parameters $\gamma_i$, $\al_i$ and $t_i$, where $i = 1,2$, which are used to defined good and bad cubes.

We also denote $E_{k,u}f = E_{k,u,\text{good}}+ E_{k,u,\text{bad}}$, where
\begin{align*}
E_{k,u,\text{good}} f&:=\sum_{\substack{Q_1 \in \Dd_1, \text{good} \\ k>\text{gen}(Q_1) \geq m}} \sum_{u_1=1}^{M_{Q_1}-1}
   h_{u_1}^{Q_1} \otimes \langle f,h_{u_1}^{Q_1} \rangle_1,\\
E_{k,u,\text{bad}} f&:=\sum_{\substack{Q_1 \in \Dd_1, \text{bad} \\ k>\text{gen}(Q_1) \geq m}} \sum_{u_1=1}^{M_{Q_1}-1}
   h_{u_1}^{Q_1} \otimes \langle f,h_{u_1}^{Q_1} \rangle_1.
\end{align*}
We perform a similar split for other terms~$E_{k,0}f,\ldots,\widetilde{E}'_{k,0}g$ above, and  the terms $E_{k,0,\text{good}} f$, $E_{k,0,\text{bad}} f,\ldots, \widetilde{E}'_{k,0,\text{good}} g$, $\widetilde{E}'_{k,0,\text{bad}} g$ are defined analogously to~$E_{k,u,\text{good}}$ and~$E_{k,u,\text{bad}}$.

Note that the operator~$E_k$ and $E'_k$ in fact are the inhomogeneous Haar expansions of $f$ and $g$ respectively with respect to generation~$k$ in the first variable of functions.
Similarly, $\widetilde{E}_k$ and $\widetilde{E}'_k$ are the inhomogeneous Haar expansions with respect to generation~$k$ in the second variable.  Thus, they satisfy the following properties:
\begin{enumerate}
  \item[(1)] $E_k$, $E'_k$, $\widetilde{E}_k$ and $\widetilde{E}'_k$ are linear operators,
  \item[(2)] $f = \lim_{k \rightarrow \infty}E_kf = \lim_{k \rightarrow \infty}\widetilde{E}_kf$ in~$L^2(\mu)$,
  \item[(3)] $\|E_kf\|_{L^2(\mu)}, \|\widetilde{E}_kf\|_{L^2(\mu)} \leq \|f\|_{L^2(\mu)} = 1$, and
  \item[(4)] $E_k\widetilde{E}_kf = \widetilde{E}_kE_kf$.
\end{enumerate}

We write
\begin{eqnarray*}
  \langle Tf,g\rangle &=& \langle T(f-E_kf),g\rangle + \langle T(E_kf), g-E'_kg\rangle + \langle T(E_kf - \widetilde{E}_kE_kf), E'_kg\rangle \\
  &&+ \langle T(\widetilde{E}_kE_kf) , E'_kg - \widetilde{E}'_kE'_kg\rangle
+\langle T(\widetilde{E}_kE_kf),\widetilde{E}'_kE'_kg\rangle.
\end{eqnarray*}
More briefly,
\[\langle Tf,g\rangle = \langle T(\widetilde{E}_kE_kf),\widetilde{E}'_kE'_kg\rangle + \e_k(\Dd), \quad
\text{where~$\Dd = (\Dd_1, \Dd'_1, \Dd_2, \Dd'_2)$}.\]
Using the the property~(2) listed above we obtain
$\lim_{k \rightarrow \infty} \e_k(\Dd) = 0.$
Then, using the Dominated Convergence Theorem, we deduce that
\begin{equation}\label{eq2:Part3a}
\langle Tf,g\rangle = \lim_{k \rightarrow \infty} \E \langle T(\widetilde{E}_kE_kf),\widetilde{E}'_kE'_kg\rangle.
\end{equation}
Here $ \E := \E_{\Dd} = \E_{(\Dd_1, \Dd'_1, \Dd_2, \Dd'_2)} = \E_{\Dd_1} \E_{\Dd'_1} \E_{\Dd_2} \E_{\Dd'_2}$ is the mathematical expectation taken over all the random dyadic lattices $\Dd_1$, $\Dd'_1$, $\Dd_2$ and $\Dd'_2$ constructed above.
Note that
\begin{eqnarray}\label{initial_eq5a}
\widetilde{E}_kE_kf =\widetilde{E}_{k,u} E_{k,u} f +\widetilde{E}_{k,0} E_{k,u} f+\widetilde{E}_{k,u} E_{k,0} f+
\widetilde{E}_{k,0} E_{k,0} f.
\end{eqnarray}
We continue to decompose $\widetilde{E}_kE_kf$ into smaller sums as follows.

Denote~$f_k := \widetilde{E}_kE_kf$.
Let $f_k = f_{k,\text{good}} + f_{k,\text{bad}}$ where~$f_{k,\text{good}}$ contains the sums over~$Q_1$ and~$Q_2$ where both cubes are good, and~$f_{k,\text{bad}}$ contain the sums over~$Q_1$ and~$Q_2$ where at least one cube is bad:
\begin{eqnarray*}
  f_{k,\text{good}} := \widetilde E_{k,u,\text{good}} E_{k,u,\text{good}}f +\widetilde E_{k,0,\text{good}} E_{k,u,\text{good}}f + \widetilde E_{k,u,\text{good}} E_{k,0,\text{good}}f +\widetilde E_{k,0,\text{good}} E_{k,0,\text{good}}f
\end{eqnarray*}
and
\begin{align*}
  f_{k,\text{bad}} :=\,&
   \widetilde E_{k,u,\text{good}} E_{k,u,\text{bad}}f
   +\widetilde E_{k,0,\text{good}} E_{k,u,\text{bad}}f
   + \widetilde E_{k,u,\text{good}} E_{k,0,\text{bad}}f
   +\widetilde E_{k,0,\text{good}} E_{k,0,\text{bad}}f\\
   &+ \, \widetilde E_{k,u,\text{bad}} E_{k,u}f
   +\widetilde E_{k,0,\text{bad}} E_{k,u}f
   + \widetilde E_{k,u,\text{bad}} E_{k,0}f
   +\widetilde E_{k,0,\text{bad}} E_{k,0}f.
\end{align*}

Notice that
\begin{align*}
  &\|f_k\|_{L^2(\mu)} = \|\widetilde{E}_kE_kf\|_{L^2(\mu)}
  \leq \|E_kf\|_{L^2(\mu)} \leq \|f\|_{L^2(\mu)}, \quad \text{and} \\
  &\|f_k\|_{L^2(\mu)} \leq \|f_{k,\text{good}}\|_{L^2(\mu)}  + \|f_{k,\text{bad}}\|_{L^2(\mu)}.
\end{align*}

Since $\|f\|_{L^2(\mu)} = \|g\|_{L^2(\mu)} = 1$, we have
\begin{eqnarray}
 && \|f_{k,\text{good}}\|_{L^2(\mu)}  \leq \|f\|_{L^2(\mu)} = 1, \quad \text{and}  \label{initial_eq2a} \\
 && \|f_{k,\text{bad}}\|_{L^2(\mu)} \leq \|f\|_{L^2(\mu)} = 1. \label{initial_eq3a}
\end{eqnarray}

Similarly, denote
$g_k := \widetilde{E}'_kE'_kg = g_{k,\text{good}} + g_{k,\text{bad}}.$
We now write
\begin{align}\label{eq3:Part3a}
\langle T(\widetilde{E}_kE_kf),\widetilde{E}'_kE'_kg\rangle
= \langle Tf_k, g_k\rangle
= \langle Tf_{k,\text{good}}, g_{k,\text{good}}\rangle
+ \langle Tf_{k,\text{good}}, g_{k,\text{bad}}\rangle
+ \langle Tf_{k,\text{bad}}, g_{k}\rangle.
\end{align}
Using~\eqref{eq2:Part3a}, \eqref{initial_eq2a}, \eqref{initial_eq3a} and~\eqref{eq3:Part3a} we have
\begin{eqnarray*}
  |\langle Tf,g\rangle| &=& \lim_{k \rightarrow \infty} |\E \langle T(\widetilde{E}_kE_kf),\widetilde{E}'_kE'_kg\rangle|
   = \lim_{k \rightarrow \infty} |\E \langle Tf_k, g_k\rangle| \\
   &\leq&  \lim_{k \rightarrow \infty} \left( \E |\langle Tf_{k,\text{good}}, g_{k,\text{good}}\rangle|
   + \E |\langle Tf_{k,\text{good}}, g_{k,\text{bad}}\rangle|
   + \E |\langle Tf_{k,\text{bad}}, g_{k}\rangle|\right) \\
   &\leq&  \lim_{k \rightarrow \infty}  \left(\E |\langle Tf_{k,\text{good}}, g_{k,\text{good}}\rangle|
   +  M(\delta^m) \E \|g_{k,\text{bad}}\|_{L^2(\mu)}
   +  M(\delta^m) \E \|f_{k,\text{bad}}\|_{L^2(\mu)}\right).
\end{eqnarray*}

Now, we will estimate $\E \|f_{k,\text{bad}}\|_{L^2(\mu)}$.
We start with the first term $\widetilde E_{k,u,\text{good}} E_{k,u,\text{bad}}f$ in the definition of~$f_{k,\text{bad}}$.
Using property~\eqref{eq1:pro_Haar} of Haar functions, we have
\begin{align*}
  &\hspace{-1cm}\|\widetilde E_{k,u,\text{good}} E_{k,u,\text{bad}}f\|_{L^2(\mu)} \\
  &\ls  \bigg( \sum_{\substack{Q_1 \in \Dd_1,\text{good} \\k> \text{gen}(Q_1) \geq m}} \,
\sum_{\substack{Q_2 \in \Dd_2,\text{bad} \\ k>\text{gen}(Q_2) \geq m}}
\sum_{u_1=1}^{M_{Q_1}-1}
\sum_{u_2=1}^{M_{Q_2}-1}
|\langle f, h_{u_1}^{Q_1} \otimes h_{u_2}^{Q_2}\rangle|^2\bigg)^{1/2} \\
  &\leq \bigg( \sum_{\substack{Q_1 \in \Dd_1 \\k> \text{gen}(Q_1) \geq m}} \,
\sum_{\substack{Q_2 \in \Dd_2 \\ k>\text{gen}(Q_2) \geq m}}
\sum_{u_1=1}^{M_{Q_1}-1}
\sum_{u_2=1}^{M_{Q_2}-1}
\chi_{\text{bad}}(Q_2)
|\langle f, h_{u_1}^{Q_1} \otimes h_{u_2}^{Q_2}\rangle|^2\bigg)^{1/2},
\end{align*}
where  the function $\chi_{\text{bad}}(Q_2)$ is defined by
\[ \chi_{\text{bad}}(Q_2) := \left\{
      \begin{array}{l l}
        1, & \quad \text{if $Q_2$ is bad;}\\
        0, & \quad \text{if $Q_2$ is good.}
      \end{array} \right.\]
Hence,
\begin{eqnarray*}
  \lefteqn{\E \|\widetilde E_{k,u,\text{good}} E_{k,u,\text{bad}}f\|_{L^2(\mu)}}\\
  &=& \E_{(\Dd_1,\Dd'_1,\Dd_2,\Dd'_2)} \|\widetilde E_{k,u,\text{good}} E_{k,u,\text{bad}}f\|_{L^2(\mu)} \\
  &=&
 \E_{(\Dd_1,\Dd'_1,\Dd_2)} \bigg( \sum_{\substack{Q_1 \in \Dd_1 \\k> \text{gen}(Q_1) \geq m}} \,
\sum_{\substack{Q_2 \in \Dd_2 \\ k>\text{gen}(Q_2) \geq m}}
\sum_{u_1=1}^{M_{Q_1}-1}
\sum_{u_2=1}^{M_{Q_2}-1}
\Pb(Q_2 \in \Dd'_2\text{-bad})
|\langle f, h_{u_1}^{Q_1} \otimes h_{u_2}^{Q_2}\rangle|^2\bigg)^{1/2} \\
&\ls& c(r),
\end{eqnarray*}
where the last inequality follows from Theorem~\ref{thm.10.2.HM12} with $c(r) \sim \delta^{r\gamma_2\kappa}$, where $\delta \in (0,1)$, $\gamma_2 \in (0,1)$ and $\kappa \in (0,1]$.
It is clear that~$c(r) \rightarrow 0$, when~$r \rightarrow \infty$.
Recall that $ \Dd'_1\text{-bad}$ and $\Dd'_2\text{-bad}$ are collections of all bad cubes $Q_1 \in \Dd_1$ and $Q_2 \in \Dd_2$, respectively (see Section~\ref{subsec:good_bad_cube}).
Here~$\Pb$ is the probability measure on the collection~$\Omega$ of random dyadic grids.

Following a similar calculation, we obtain the same result for other terms in the definition of~$f_{k,\text{bad}}$. Then we obtain the bound
\begin{eqnarray*}
\E  \|f_{k,\text{bad}}\|_{L^2(\mu)} \ls c(r).
\end{eqnarray*}

The same result holds for $\E  \|g_{k,\text{bad}}\|_{L^2(\mu)}$.
Fixing~$r$ to be sufficiently large, we have shown that
\begin{equation}\label{Initialeq.1a}
0.7M(\delta^m) \leq |\langle Tf,g\rangle| \leq \lim_{k \rightarrow \infty} |\E\langle Tf_{k,\text{good}}, g_{k,\text{good}}\rangle| + 0.1 M(\delta^m).
\end{equation}

Notice that $|\E\langle Tf_{k,\text{good}}, g_{k,\text{good}}\rangle|$
can be split into 16 terms, and
our proof is reduced to showing that
\begin{eqnarray}\label{coreeq.1a}
  |\E\langle Tf_{k,\text{good}}, g_{k,\text{good}}\rangle|
 &=&
 |\E \langle T( \widetilde E_{k,u,\text{good}} E_{k,u,\text{good}}f),  \widetilde E'_{k,u,\text{good}} E'_{k,u,\text{good}}g\rangle \noz \\
 &&\,+ \cdots +
 \E \langle T( \widetilde E_{k,0,\text{good}} E_{k,0,\text{good}}f),  \widetilde E'_{k,0,\text{good}} E'_{k,0,\text{good}}g\rangle| \noz \\
  &=&
  \Bigg|\E \sum_{\substack{Q_1,Q_2 \text{ good } \\ k>\text{gen}(Q_1), \text{gen}(Q_2) \geq m }} \,
  \sum_{\substack{Q'_1,Q'_2 \text{ good} \\ k>\text{gen}(Q'_1), \text{gen}(Q'_2) \geq m }}
  \, \sum_{\substack{u_1,u_2 \\u'_1,u'_2}}
  \langle f, h_{u_1}^{Q_1} \otimes h_{u_2}^{Q_2}\rangle \noz\\
  &&\quad\times \langle g, h_{u'_1}^{Q'_1} \otimes h_{u'_2}^{Q'_2}\rangle
  \langle T(h_{u_1}^{Q_1} \otimes h_{u_2}^{Q_2}), h_{u'_1}^{Q'_1} \otimes h_{u'_2}^{Q'_2}\rangle  \noz\\
  && \,+ \cdots +   \E \sum_{\substack{Q_1,Q_2 \text{ good } \\ \text{gen}(Q_1)= \text{gen}(Q_2) = m }} \,
  \sum_{\substack{Q'_1,Q'_2 \text{ good} \\ \text{gen}(Q'_1)=\text{gen}(Q'_2) = m }}
  \langle f, h_{0}^{Q_1} \otimes h_{0}^{Q_2}\rangle \noz\\
  &&\quad\times \langle g, h_{0}^{Q'_1} \otimes h_{0}^{Q'_2}\rangle
  \langle T(h_{0}^{Q_1} \otimes h_{0}^{Q_2}), h_{0}^{Q'_1} \otimes h_{0}^{Q'_2}\rangle \Bigg| \noz\\
    &\leq& (C + 0.1 M(\delta^m)) \|f\|_{L^2(\mu)} \|g\|_{L^2(\mu)},
\end{eqnarray}
where~$C$ is a constant depending only on the assumptions and the $\bmo_{\textup{prod}}(\mu)$ norms of the four $S(1)$, $S \in \{T,T^*,T_1,T^*_1\}$, but independent of the \emph{a priori} bound.

Combining with~\eqref{Initialeq.1a}, we get
$0.6M(\delta^m)\leq C+ 0.1M(\delta^m),$
which implies
$M(\delta^m)\leq 2C.$

\section{Proof of Theorem~\ref{thm:T1thm_ver3} - the Full $T(1)$ Theorem}\label{sec:T1 full}
In this section, we consider the most general version of the~$T(1)$ theorem on~$(\mathcal{X},\rho,\mu)$ in which the assumption of \emph{a priori} boundedness of~$T$ is relaxed (Theorem~\ref{thm:T1thm_ver3}). We only assume that  the bilinear form~$\langle Tf,g \rangle$ is defined for functions~$f$ and~$g \in C_c^{\eta}(\mathcal{X})$.

Let~$T$ be a product non-homogeneous SIO as defined in Section~\ref{subsec:assume}.
Given~$\tau >0$, let~$T_{\tau}$ be the \emph{truncated operator} defined by
\begin{align*}
T_{\tau}f(x) &:= \int_{\rho(x,y)>\tau} K(x,y) f(y) \,d\mu(y) \\
&= \int_{\rho_1(x_1,y_1)>\tau} \int_{\rho_2(x_2,y_2)>\tau} K(x_1,x_2,y_1,y_2) f(y_1,y_2) \,d\mu_2(y_2) d\mu_1(y_1).   \\
\end{align*}
The truncated operator~$T_{\tau}$ is clearly well defined on compactly supported functions and satisfies the weaker \emph{a priori} boundedness. That is, for compactly supported functions~$f$ and~$g$ we have
\[|\langle T_{\tau}f,g \rangle| \leq C(f,g,\tau)\|f\|_{L^2(\mu)} \|g\|_{L^2(\mu)}. \]
As shown in~\cite[p.165]{NTV03}, it is reasonable to think of boundedness of~$T$ as the uniform boundedness of~$T_{\tau}$.
Therefore, to show that the operator~$T$ is bounded on~$L^2(\mu)$ it is sufficient to show that the truncated operator~$T_{\tau}$ is uniformly bounded on~$L^2(\mu)$.
To do so, we will show that if~$T$ satisfies the hypotheses of Theorem~\ref{thm:T1thm_ver3}, then~$T_{\tau}$ satisfies the hypotheses of Theorem~\ref{thm:T1thm_ver2}, which is independent on the choice of~$\tau$.
This result is stated in Theorem~\ref{thm:bilinear imply a priori} below.
Consequently, the sequence~$T_{\tau}$ is uniformly bounded on~$L^2(\mu)$ and hence, $T$ is bounded on~$L^2(\mu)$.

\begin{thm}\label{thm:bilinear imply a priori}
Let~$T$ be a biparameter SIO with the bilinear form~$\langle Tf,g \rangle$ defined for functions~$f$ and~$g \in C_c^{\eta}(\mathcal{X})$. Assume that~$T$ is weakly bounded in the sense of Assumptions~\ref{assum:5b}--\ref{assum:7b}. Let~$T_{\tau}$ be truncated operators, where~$\tau >0$. Then for all $S \in \{T, T^*,T_1,T_1^*\}$, the condition $S(1) \in \bmo_{\textup{prod}}(\mu)$ implies that $S_{\tau}(1) \in \bmo_{\textup{prod}}(\mu)$ uniformly in~$\tau$.
Moreover, $S_{\tau}$ is weakly bounded (with uniform estimates) in the sense of Assumptions~\ref{assum_5} and~\ref{assum_6}.
\end{thm}

\begin{proof}
  We give the proof for~$S = T$. The proofs for~$T^*, T_1,T_1^*$ can be done analogously.

  Let $(\mathcal{X},\rho,\mu) := (X_1 \times X_2, \rho_1 \times \rho_2, \mu_1 \times \mu_2)$ be a non-homogeneous bi-parameter quasimetric space.
Let $\mathcal{D}_1$ and $\mathcal{D}_2$ be fixed dyadic systems in~$X_1$ and~$X_2$, respectively.
Take~$Q_1 \in \Dd_1$ and~$Q_2 \in \Dd_2$.
Let~$h_{Q_1} := h_{u_1}^{Q_1}$ and~$h_{Q_2} := h_{u_2}^{Q_2}$ be Haar functions as defined in Section~\ref{subsec:Haar_func}.
Fix $\tau >0$.
To interpret the condition~$T_{\tau}(1) \in \bmo_{\textup{prod}}(\mu)$, it is enough to define the pairing~$\langle T_{\tau}(1) , h_{Q_1}\otimes h_{Q_2}\rangle$.

Let~$U = B(x_U,r_U) \subset X_1$ and~$V = B(x_V,r_V) \subset X_2$ be arbitrary balls with the property that
$$C_KB(Q_1) := B(x_{Q_1}, C_KC_Q\ell(Q_1)) \subset U  \text{ \,and\, } C_KB(Q_2) := B(x_{Q_2}, C_KC_Q\ell(Q_2)) \subset V.$$
Note that for $x_1 \in Q_1$ and $y_1 \in U^c$ we have
$$\rho_1(x_{Q_1},y_1) \geq C_KC_Q \ell(Q_1) \geq C_K \rho_1(x_{Q_1},x_1).$$
Similarly, for $x_2 \in Q_2$ and $y_2 \in V^c$ we also have
$\rho_2(x_{Q_2},y_2) \geq C_K \rho_2(x_{Q_2},x_2)$.

Set~$\langle T_{\tau}1, h_{Q_1} \otimes h_{Q_2}\rangle := \sum_{i=1}^{4}A_i(U,V,\tau)$ where
\begin{eqnarray*}
  A_1(U,V,\tau) &:=& \langle T_{\tau}(\chi_U \otimes \chi_V), h_{Q_1} \otimes h_{Q_2}\rangle, \\
  A_2(U,V,\tau) &:=& \langle T_{\tau}(\chi_{U^c} \otimes \chi_V), h_{Q_1} \otimes h_{Q_2}\rangle, \\
  A_3(U,V,\tau) &:=& \langle T_{\tau}(\chi_{U} \otimes \chi_{V^c}), h_{Q_1} \otimes h_{Q_2}\rangle, \quad \text{and} \\
  A_4(U,V,\tau) &:=& \langle T_{\tau}(\chi_{U^c} \otimes \chi_{V^c}), h_{Q_1} \otimes h_{Q_2}\rangle.
\end{eqnarray*}
Following the same techniques used for $A_2(U,V)$, $A_3(U,V)$ and $A_4(U,V)$ in Section~\ref{subsec:intepretation}, we can show that $A_2(U,V,\tau)$, $A_3(U,V,\tau)$ and $A_4(U,V,\tau)$ are bounded by $\mu_1(U)^{1/2} \mu_2(V)^{1/2}$.

We are left to estimate $A_1(U,V,\tau)$. We will consider two cases:
$\tau > \max \{r_U, r_V\}$ and $\tau \leq \max \{r_U, r_V\}$.

Suppose $\tau > \max \{r_U, r_V\}$. Using the full standard kernel estimate assumption (Assumption~\ref{assum_2}), the doubling property of~$\lambda_1$ and~$\lambda_2$, and upper doubling property of~$\mu_1$ and~$\mu_2$, for each point $x = x_1 \times x_2 \in U \times V$ we have
\begin{align*}
  &|T_{\tau}(\chi_U \otimes \chi_V)(x)|\\
  &= \bigg| \int_{\rho(x,y) >\tau} K(x,y) \chi_U(y) \chi_V(y) \,d\mu(y)\bigg| \\
  &\leq \int_{U: \rho_1(x_1,y_1) >\tau}  \int_{V: \rho_2(x_2,y_2) >\tau}
  |K(x_1,x_2,y_1,y_1)| \,d\mu_2(y_2) d\mu_1(y_1) \\
  &\leq C \int_{U: \rho_1(x_1,y_1) >\tau}  \int_{V: \rho_2(x_2,y_2) >\tau}
  \frac{1}{\lambda_1(y_1,\rho_1(x_1,y_1))}
  \frac{1}{\lambda_2(y_2,\rho_2(x_2,y_2))} \,d\mu_2(y_2) d\mu_1(y_1) \\
  &< C\int_U \int_V \frac{1}{\lambda_1(y_1,\tau)}
  \frac{1}{\lambda_2(y_2,\tau)}
  \,d\mu_2(y_2) d\mu_1(y_1) \\
  &<C \int_U \int_V
   \frac{1}{\lambda_1(y_1,r_U)} \frac{1}{\lambda_2(y_2,r_V)}
   \,d\mu_2(y_2) d\mu_1(y_1)\\
  &\leq C \mu_1(U) \mu_2(V) \frac{1}{\lambda_1(x_U,r_U)} \frac{1}{\lambda_2(x_U,r_V)} \\
  &\leq C \mu_1(U) \mu_2(V) \frac{1}{\mu_1(U)} \frac{1}{\mu_2(V)} \\
  & = C.
\end{align*}
Hence,
\begin{equation}\label{eq5:reduce}
  \int_{U \times V} |T_{\tau}(\chi_U \otimes \chi_V)|^2 \,d\mu
  \ls \mu_1(U) \mu_2(V).
\end{equation}
Therefore
\begin{align*}
  |A_1(U,V,\tau)|
  &= |\langle T_{\tau}(\chi_U \otimes \chi_V), h_{Q_1} \otimes h_{Q_2}\rangle|
  \leq \| T_{\tau}(\chi_U \otimes \chi_V)\|_{L^2(U\times V)}
  \|  h_{Q_1} \otimes h_{Q_2}\|_{L^2(\mu)}\\
  &\ls \mu_1(U)^{1/2} \mu_2(V)^{1/2}.
\end{align*}

Notice also the proof of inequality~\eqref{eq5:reduce} holds for every balls~$U \in X_1$ and~$V \in X_2$. Hence Assumptions~\ref{assum_5} and~\ref{assum_6} on the weak boundedness properties are satisfied.

Next, we consider the case $\tau \leq \max \{r_U, r_V\}$.
Fix $\epsilon \in (0,1]$. It is sufficient to show that for all~$x \in U\times V$, we have
\begin{equation}\label{eq3:reduce}
 |T_{\tau}(\widetilde{\chi}_{U,\epsilon} \otimes \widetilde{\chi}_{V,\epsilon})(x)| \ls C,
\end{equation}
where $\chi_U \leq \widetilde{\chi}_{U,\epsilon} \leq \chi_{(1+\epsilon)U}$
and $\chi_V \leq \widetilde{\chi}_{V,\epsilon} \leq \chi_{(1+\epsilon)V}$.
This is because~\eqref{eq3:reduce} implies~\eqref{eq5:reduce}, which results in $|A_1(U,V,\tau)| \ls \mu_1(U)^{1/2} \mu_2(V)^{1/2}$ and Assumptions~\ref{assum_5} and~\ref{assum_6} on the weak boundedness properties hold.

Let $N := 3C_K\Lambda$, where $C_K >1$ is from the kernel estimates and $\Lambda >1$ is from the weak boundedness properties.
Fix a point $x = x_{1} \times x_{2} \in \mathcal{X}$, where $x_{1} \in U$ and~$x_{2} \in V$.
Consider a sequence of balls $B^j = B_1^j \times B_2^j := B(x_{1},r_j) \times B(x_{2},r_j)$, where $r_j = N^j\tau$.
Let~$n$ be the smallest number such that either $\mu_1(B_1^n) \leq \mathcal{C}\mu_1(B_1^{n-1})$ and $\mu_2(B_2^n) \leq \mathcal{C}\mu_2(B_2^{n-1})$, where $\mathcal{C} = (2C_{\lambda_1}^{\log_2N}C_{\lambda_2}^{\log_2N})^{1/2}$, or $U\subset B_1^n$ and $ V\subset B_2^n$.
By construction, for $j = 0,\ldots,n-1$ we have
\begin{equation}\label{eq7:reduce}
\mu_1(B^j_1) < \mathcal{C}^{j+1-n}\mu_1(B_1^{n-1})\quad \text{and}\quad
\mu_2(B^j_2) < \mathcal{C}^{j+1-n}\mu_2(B_2^{n-1}).
\end{equation}

We write
\begin{equation}\label{eq1:reduce}
|T_{\tau}(\widetilde{\chi}_{U,\epsilon} \otimes \widetilde{\chi}_{V,\epsilon})|
\leq |T_{\tau}(\widetilde{\chi}_{U,\epsilon} \otimes \widetilde{\chi}_{V,\epsilon})
- T_{r_n}(\widetilde{\chi}_{U,\epsilon} \otimes \widetilde{\chi}_{V,\epsilon})|
+ |T_{r_n}(\widetilde{\chi}_{U,\epsilon} \otimes \widetilde{\chi}_{V,\epsilon})|.
\end{equation}

We will consider each of the two terms on the right-hand side of~\eqref{eq1:reduce}.
Using Assumption~\ref{assum_2}, property~\eqref{eq7:reduce}, the facts that~$\mu_1$,~$\mu_2$ are upper doubling and~$\lambda_1$,~$\lambda_2$ are doubling we obtain
\begin{align*}
  &|T_{\tau}(\widetilde{\chi}_{U,\epsilon} \otimes \widetilde{\chi}_{V,\epsilon})
- T_{r_n}(\widetilde{\chi}_{U,\epsilon} \otimes \widetilde{\chi}_{V,\epsilon})|\\
&= \bigg| \int_{\rho(x,y)>\tau} K(x,y) \widetilde{\chi}_{U,\epsilon} \widetilde{\chi}_{V,\epsilon} \,d\mu(y)
- \int_{\rho(x,y)>3C_KR} K(x,y) \widetilde{\chi}_{U,\epsilon} \widetilde{\chi}_{V,\epsilon} \,d\mu(y) \bigg| \\
  &\leq \int_{B^n\backslash B^0} |K(x,y)| \,d\mu(y)
  \leq \sum_{j=1}^n \int_{B^j \backslash B^{j-1}} |K(x,y)| \,d\mu(y)\\
  &\leq \sum_{j=1}^n \int_{B_1^j \backslash B_1^{j-1}}
  \int_{B_2^j \backslash B_2^{j-1}}
  \frac{1}{\lambda_1(x_1,\rho_1(x_1,y_1))}
  \frac{1}{\lambda_2(x_2,\rho_1(x_2,y_2))} \,d\mu_2(y_2) d\mu_1(y_1) \\
  &\leq \sum_{j=1}^n \int_{B_1^j} \int_{B_2^j}
  \frac{1}{\lambda_1(x_1,r_{j-1})}
  \frac{1}{\lambda_2(x_2,r_{j-1})} \,d\mu_2(y_2) d\mu_1(y_1)
  =  \sum_{j=1}^n
  \frac{\mu_1(B_1^j)}{\lambda_1(x_1,r_{j-1})}
  \frac{\mu_2(B_2^j)}{\lambda_2(x_2,r_{j-1})}\\
  & \ls \sum_{j=1}^{n-1} \mathcal{C}^{2(j+1-n)}
  \frac{\mu_1(B_1^{n-1})}{\lambda_1(x_1,r_{j-1})}
  \frac{\mu_2(B_2^{n-1})}{\lambda_2(x_2,r_{j-1})}
  \leq \sum_{j=1}^{n-1} \mathcal{C}^{2(j+1-n)}
  \frac{\lambda_1(x_1,r_{n-1})}{\lambda_1(x_1,r_{j-1})}
  \frac{\lambda_2(x_2,r_{n-1})}{\lambda_2(x_2,r_{j-1})} \\
  &= \sum_{j=1}^{n-1} \mathcal{C}^{2(j+1-n)}
  C_{\lambda_1}^{\log_2 N^{n-j}} C_{\lambda_2}^{\log_2 N^{n-j}}
  = \sum_{j=1}^{n-1} (2C_{\lambda_1}^{\log_2 N} C_{\lambda_2}^{\log_2 N})^{j+1-n} C_{\lambda_1}^{\log_2 N^{n-j}}
  C_{\lambda_2}^{\log_2 N^{n-j}} \\
  &= 2C_{\lambda_1}^{\log_2 N} C_{\lambda_2}^{\log_2 N}
  \sum_{j=1}^{n-1}  2^{j-n}\\
  &\leq C(C_{\lambda_1}, C_{\lambda_2},N) .
\end{align*}

Next, we consider $|T_{r_n}(\widetilde{\chi}_{U,\epsilon} \otimes \widetilde{\chi}_{V,\epsilon})|$.
If we stop because $U \subset B_1^n$ and $V \subset B_2^n$,
then we have
$r_U \leq r_n$ and $r_V \leq r_n$. Hence, $r_n >\max\{r_U, r_V\}$.
Following the same calculation as in the case $\tau > \max\{r_U, r_V\}$ we get
\[
|T_{r_n}(\widetilde{\chi}_{U,\epsilon} \otimes \widetilde{\chi}_{V,\epsilon})| \leq C.
\]
Thus, we obtain~\eqref{eq3:reduce}.

Suppose we stop because  $\mu_1(B_1^n) \leq \mathcal{C}\mu_1(B_1^{n-1})$ and $\mu_2(B_2^n) \leq \mathcal{C}\mu_2(B_2^{n-1})$, where $\mathcal{C} = (2C_{\lambda_1}^{\log_2N}C_{\lambda_2}^{\log_2N})^{1/2}$.
Define $F_1 := \int_{X_1} \widetilde{\chi}_{B^{n-1}_{1,\epsilon}} \,d\mu_1$ and $F_2 := \int_{X_2} \widetilde{\chi}_{B^{n-1}_{2,\epsilon}} \,d\mu_2$,
where $\chi_{B_1^{n-1}} \leq \widetilde{\chi}_{B^{n-1}_{1,\epsilon}} \leq \chi_{(1+\epsilon)B_1^{n-1}}$ and $\chi_{B_2^{n-1}} \leq \widetilde{\chi}_{B^{n-1}_{2,\epsilon}} \leq \chi_{(1+\epsilon)B_2^{n-1}}$.
Recall that for~$i = 1,2$, we have
$\delta_{x_{i}}(x) = 1$ if $x=x_{i}$ and $0$ otherwise.
Let
\begin{align*}
D_1 &:= \frac{1}{F_1 F_2} \langle \widetilde{\chi}_{B^{n-1}_{1,\epsilon}} \otimes \widetilde{\chi}_{B^{n-1}_{2,\epsilon}},
T(\widetilde{\chi}_{U,\epsilon} \otimes \widetilde{\chi}_{V,\epsilon})
\rangle, \\
D_2 &:= \frac{1}{F_1} \langle \widetilde{\chi}_{B^{n-1}_{1,\epsilon}} \otimes \delta_{x_2},
T(\widetilde{\chi}_{U,\epsilon} \otimes \widetilde{\chi}_{V,\epsilon})
\rangle, \quad \text{and} \\
D_3 &:= \frac{1}{F_2} \langle \delta_{x_1} \otimes \widetilde{\chi}_{B^{n-1}_{2,\epsilon}},
T(\widetilde{\chi}_{U,\epsilon} \otimes \widetilde{\chi}_{V,\epsilon})
\rangle.
\end{align*}
We write
\begin{equation}\label{eq2:reduce}
  |T_{r_n}(\widetilde{\chi}_{U,\epsilon} \otimes \widetilde{\chi}_{V,\epsilon})|
  \leq |T_{r_n}(\widetilde{\chi}_{U,\epsilon} \otimes \widetilde{\chi}_{V,\epsilon}) + D_1 - D_2  - D_3| + |D_1| + |D_2| + |D_3|.
\end{equation}

We will estimate each term in the right-hand side of~\eqref{eq2:reduce}. We begin with $|D_1|$. Notice that $F_1 \geq \mu_1(B_1^{n-1})$ and $F_2 \geq \mu_2(B_2^{n-1})$. This together with the weak boundedness property of~$T$,  the fact that $\mu_1(B_1^n) \leq \mathcal{C}\mu_1(B_1^{n-1})$, $\mu_2(B_2^n) \leq \mathcal{C}\mu_2(B_2^{n-1})$ and $N > \Lambda$ will give bounds for $|D_1|$, $|D_2|$ and $|D_3|$.
\begin{align*}
  |D_1| &= \bigg|\frac{1}{F_1 F_2} \langle \widetilde{\chi}_{B^{n-1}_{1,\epsilon}} \otimes \widetilde{\chi}_{B^{n-1}_{2,\epsilon}},
T(\widetilde{\chi}_{U,\epsilon} \otimes \widetilde{\chi}_{V,\epsilon})
\rangle\bigg|
\leq C \frac{1}{\mu_1(B_1^{n-1}) \mu_2(B_2^{n-1})}
\mu_1(\Lambda B_1^{n-1}) \mu_2(\Lambda B_2^{n-1}) \\
& \leq C\mathcal{C}^2  \frac{1}{\mu_1(B_1^{n}) \mu_2(B_2^{n})}
\mu_1(N B_1^{n-1}) \mu_2(N B_2^{n-1})
 = C\mathcal{C}^2  \frac{1}{\mu_1(B_1^{n}) \mu_2(B_2^{n})}
\mu_1(B_1^{n}) \mu_2(B_2^{n}) \\
&\ls C.
\end{align*}
Below we estimate $|D_2|$. The bound for $|D_3|$ can be found analogously.
\begin{align*}
  |D_2| &= \bigg|\frac{1}{F_1} \langle \widetilde{\chi}_{B^{n-1}_{1,\epsilon}} \otimes \delta_{x_2},
T(\widetilde{\chi}_{U,\epsilon} \otimes \widetilde{\chi}_{V,\epsilon})
\rangle\bigg|
\leq C \frac{1}{\mu_1(B_1^{n-1})}
\mu_1(\Lambda B_1^{n-1}) \\
& \leq C\mathcal{C}  \frac{1}{\mu_1(B_1^{n})}
\mu_1(N B_1^{n-1})
 = C\mathcal{C} \frac{1}{\mu_1(B_1^{n})}
\mu_1(B_1^{n}) \\
&\ls C.
\end{align*}

We are left the term $|T_{r_n}(\widetilde{\chi}_{U,\epsilon} \otimes \widetilde{\chi}_{V,\epsilon}) + D_1(x) - D_2(x) - D_3(x)|$.
Note that
\begin{align*}
  &T_{r_n}(\widetilde{\chi}_{U,\epsilon} \otimes \widetilde{\chi}_{V,\epsilon})\\
  & = \int_{\rho(x,y)>r_n} K(x,y) \widetilde{\chi}_{U,\epsilon} \widetilde{\chi}_{V,\epsilon} \,d\mu
   = \int_{(1+\epsilon)U \backslash B_1^n}  \int_{(1+\epsilon)V \backslash B_2^n}
  T^*(\delta_{x_1} \otimes \delta_{x_2})
  \widetilde{\chi}_{U,\epsilon} \widetilde{\chi}_{V,\epsilon} \,d\mu.
\end{align*}
We also can write
\begin{align*}
  D_1& =
  \frac{1}{F_1 F_2} \langle \widetilde{\chi}_{B^{n-1}_{1,\epsilon}} \otimes \widetilde{\chi}_{B^{n-1}_{2,\epsilon}},
T(\widetilde{\chi}_{U,\epsilon} \otimes \widetilde{\chi}_{V,\epsilon})(\chi^c_{B_1^n} \otimes \chi^c_{B_2^n})
\rangle\\
& \hspace{0.5cm} +
 \frac{1}{F_1 F_2} \langle \widetilde{\chi}_{B^{n-1}_{1,\epsilon}} \otimes \widetilde{\chi}_{B^{n-1}_{2,\epsilon}},
T(\widetilde{\chi}_{U,\epsilon} \otimes \widetilde{\chi}_{V,\epsilon})(\chi_{B_1^n} \otimes \chi^c_{B_2^n})
\rangle\\
& \hspace{0.5cm} +
 \frac{1}{F_1 F_2} \langle \widetilde{\chi}_{B^{n-1}_{1,\epsilon}} \otimes \widetilde{\chi}_{B^{n-1}_{2,\epsilon}},
T(\widetilde{\chi}_{U,\epsilon} \otimes \widetilde{\chi}_{V,\epsilon})(\chi^c_{B_1^n} \otimes \chi_{B_2^n})
\rangle\\
& \hspace{0.5cm} +
 \frac{1}{F_1 F_2} \langle \widetilde{\chi}_{B^{n-1}_{1,\epsilon}} \otimes \widetilde{\chi}_{B^{n-1}_{2,\epsilon}},
T(\widetilde{\chi}_{U,\epsilon} \otimes \widetilde{\chi}_{V,\epsilon})(\chi_{B_1^n} \otimes \chi_{B_2^n})
\rangle\\
&= \int_{(1+\epsilon)U \backslash B_1^n} \int_{(1+\epsilon)V \backslash B_2^n}
T^*(F_1^{-1}\widetilde{\chi}_{B^{n-1}_{1,\epsilon}} \otimes F_2^{-1}\widetilde{\chi}_{B^{n-1}_{2,\epsilon}})
\widetilde{\chi}_{U,\epsilon} \widetilde{\chi}_{V,\epsilon} \,d\mu \\
& \hspace{0.5cm} +
 \frac{1}{F_1 F_2}\int_{\mathcal{X}}\widetilde{\chi}_{B^{n-1}_{1,\epsilon}} \widetilde{\chi}_{B^{n-1}_{2,\epsilon}}
T[(\widetilde{\chi}_{U,\epsilon} \otimes \widetilde{\chi}_{V,\epsilon})
(\chi_{B_1^n} \otimes \chi^c_{B_2^n})]\,d\mu \\
& \hspace{0.5cm} +
 \frac{1}{F_1 F_2}\int_{\mathcal{X}}\widetilde{\chi}_{B^{n-1}_{1,\epsilon}} \widetilde{\chi}_{B^{n-1}_{2,\epsilon}}
T[(\widetilde{\chi}_{U,\epsilon} \otimes \widetilde{\chi}_{V,\epsilon})
(\chi^c_{B_1^n} \otimes \chi_{B_2^n})]\,d\mu \\
& \hspace{0.5cm} +
 \frac{1}{F_1 F_2}\int_{\mathcal{X}}\widetilde{\chi}_{B^{n-1}_{1,\epsilon}} \widetilde{\chi}_{B^{n-1}_{2,\epsilon}}
T[(\widetilde{\chi}_{U,\epsilon} \otimes \widetilde{\chi}_{V,\epsilon})
(\chi_{B_1^n} \otimes \chi_{B_2^n})]\,d\mu \\
&=: D_{1,a} + D_{1,b} + D_{1,c} + D_{1,d}.
\end{align*}
Similarly we have
\begin{align*}
  D_2
  & = \int_{(1+\epsilon)U \backslash B_1^n} \int_{(1+\epsilon)V \backslash B_2^n}
T^*(\delta_{x_1} \otimes F_2^{-1}\widetilde{\chi}_{B^{n-1}_{2,\epsilon}})
\widetilde{\chi}_{U,\epsilon} \widetilde{\chi}_{V,\epsilon} \,d\mu \\
& \hspace{0.5cm} +
 \frac{1}{F_2}\int_{\mathcal{X}}\delta_{x_1} \widetilde{\chi}_{B^{n-1}_{2,\epsilon}}
T[(\widetilde{\chi}_{U,\epsilon} \otimes \widetilde{\chi}_{V,\epsilon})
(\chi_{B_1^n} \otimes \chi^c_{B_2^n})]\,d\mu \\
& \hspace{0.5cm} +
 \frac{1}{F_2}\int_{\mathcal{X}}\delta_{x_1} \widetilde{\chi}_{B^{n-1}_{2,\epsilon}}
T[(\widetilde{\chi}_{U,\epsilon} \otimes \widetilde{\chi}_{V,\epsilon})
(\chi^c_{B_1^n} \otimes \chi_{B_2^n})]\,d\mu \\
& \hspace{0.5cm} +
 \frac{1}{F_2}\int_{\mathcal{X}}\delta_{x_1} \widetilde{\chi}_{B^{n-1}_{2,\epsilon}}
T[(\widetilde{\chi}_{U,\epsilon} \otimes \widetilde{\chi}_{V,\epsilon})
(\chi_{B_1^n} \otimes \chi_{B_2^n})]\,d\mu \\
&=: D_{2,a} + D_{2,b} + D_{2,c} + D_{2,d}, \quad \text{and}
\end{align*}
\begin{align*}
  D_3
  & = \int_{(1+\epsilon)U \backslash B_1^n} \int_{(1+\epsilon)V \backslash B_2^n}
T^*(F_1^{-1}\widetilde{\chi}_{B^{n-1}_{1,\epsilon}} \otimes \delta_{x_2})
\widetilde{\chi}_{U,\epsilon} \widetilde{\chi}_{V,\epsilon} \,d\mu \\
& \hspace{0.5cm} +
 \frac{1}{F_1}\int_{\mathcal{X}}\widetilde{\chi}_{B^{n-1}_{1,\epsilon}} \delta_{x_2}
T[(\widetilde{\chi}_{U,\epsilon} \otimes \widetilde{\chi}_{V,\epsilon})
(\chi_{B_1^n} \otimes \chi^c_{B_2^n})]\,d\mu \\
& \hspace{0.5cm} +
 \frac{1}{F_1}\int_{\mathcal{X}}\widetilde{\chi}_{B^{n-1}_{1,\epsilon}} \delta_{x_2}
T[(\widetilde{\chi}_{U,\epsilon} \otimes \widetilde{\chi}_{V,\epsilon})
(\chi^c_{B_1^n} \otimes \chi_{B_2^n})]\,d\mu \\
& \hspace{0.5cm} +
 \frac{1}{F_1}\int_{\mathcal{X}}\widetilde{\chi}_{B^{n-1}_{1,\epsilon}} \delta_{x_2}
T[(\widetilde{\chi}_{U,\epsilon} \otimes \widetilde{\chi}_{V,\epsilon})
(\chi_{B_1^n} \otimes \chi_{B_2^n})]\,d\mu \\
&=: D_{3,a} + D_{3,b} + D_{3,c} + D_{3,d}.
\end{align*}
Now we have
\begin{align}\label{eq4:reduce}
  &|T_{r_n}(\widetilde{\chi}_{U,\epsilon} \otimes \widetilde{\chi}_{V,\epsilon}) + D_1(x) - D_2(x) - D_3(x)| \noz\\
   \hspace{0.5cm}& \leq
   |T_{r_n}(\widetilde{\chi}_{U,\epsilon} \otimes \widetilde{\chi}_{V,\epsilon}) + D_{1,a} - D_{2,a} - D_{3,a}| \noz\\
   \hspace{1cm}&+ |D_{1,b}| + |D_{1,c}| + |D_{1,d}|
   + |D_{2,b}| + |D_{2,c}| + |D_{2,d}|
   + |D_{3,b}| + |D_{3,c}| + |D_{3,d}|.
\end{align}
We will estimate each term in the right-hand side of~\eqref{eq4:reduce}.
We begin with
\begin{align*}
  &T_{r_n}(\widetilde{\chi}_{U,\epsilon} \otimes \widetilde{\chi}_{V,\epsilon}) + D_{1,a} - D_{2,a} - D_{3,a} \\
  & = \int_{(1+\epsilon)U \backslash B_1^n}  \int_{(1+\epsilon)V \backslash B_2^n}
  T^*\big(\delta_{x_1} \otimes \delta_{x_2}
  + F_1^{-1}\widetilde{\chi}_{B^{n-1}_{1,\epsilon}} \otimes F_2^{-1}\widetilde{\chi}_{B^{n-1}_{2,\epsilon}} \\
  &\hspace{0.5cm}- \delta_{x_1}\otimes F_2^{-1}\widetilde{\chi}_{B^{n-1}_{2,\epsilon}}
  - F_1^{-1}\widetilde{\chi}_{B^{n-1}_{1,\epsilon}} \otimes \delta_{x_2}\big)
  \widetilde{\chi}_{U,\epsilon} \widetilde{\chi}_{V,\epsilon} \,d\mu \\
  & =  \int_{(1+\epsilon)U \backslash B_1^n}  \int_{(1+\epsilon)V \backslash B_2^n}
  T^*\big((\delta_{x_1} - F_1^{-1}\widetilde{\chi}_{B^{n-1}_{1,\epsilon}}) \otimes (\delta_{x_2} - F_2^{-1}\widetilde{\chi}_{B^{n-1}_{2,\epsilon}})\big)
  \widetilde{\chi}_{U,\epsilon} \widetilde{\chi}_{V,\epsilon} \,d\mu \\
  &= \Big\langle
   T^*\big((\delta_{x_1} - F_1^{-1}\widetilde{\chi}_{B^{n-1}_{1,\epsilon}}) \otimes (\delta_{x_2} - F_2^{-1}\widetilde{\chi}_{B^{n-1}_{2,\epsilon}})\big),
  \widetilde{\chi}_{U,\epsilon}\chi^c_{B_1^n} \otimes \widetilde{\chi}_{V,\epsilon}\chi^c_{B_2^n}
   \Big\rangle\\
   & =: \langle T^*(\phi_1 \otimes \phi_2), \theta_1 \otimes \theta_2\rangle.
\end{align*}
Notice that $\int \phi_1 \,d\mu_1 = \int \phi_2 \,d\mu_2 = 0$, $\supp \phi_1 \cap \supp \theta_1 = \emptyset$ and $\supp \phi_2 \cap \supp \theta_2 = \emptyset$. Further more, for $x_1 \in \supp \theta_1$ and $y_1 \in \supp \phi_1$ we have
\[\rho_1(x_1, x_{B_1}) \geq r_n - (1+\delta)r_{n-1}
\geq (N -2) r_{n-1}
\geq (3C_K - 2C_K) r_{n-1} \geq C_K\rho_1(y_1,x_{B_1}).\]
Similarly,  for $x_2 \in \supp \theta_2$ and $y_2 \in \supp \phi_2$ we have
$\rho_2(x_2, x_{B_2}) \geq C_K\rho_2(y_2,x_{B_2}).$
By Lemma~\ref{lem1:property_kernel} we obtain
\begin{align*}
    &\hspace{-0.5cm}|T_{r_n}(\widetilde{\chi}_{U,\epsilon} \otimes \widetilde{\chi}_{V,\epsilon}) + D_{1,a} - D_{2,a} - D_{3,a}|\\
    &\ls \|\phi_1 \|_{L^1(\mu_1)} \|\phi_2 \|_{L^1(\mu_2)} \\
    &\hspace{0.5cm} \times
    \int_{(1+\epsilon)U \backslash B_1^n}
    \frac{\rho_1(y_1,x_{B_1})^{\al_1}}{\rho_1(x_1,x_{B_1})^{\al_1}
    \lambda_1(x_{B_1},\rho_1(x_1,x_{B_1})) } \,d\mu_1(x_1) \\
    &\hspace{0.5cm} \times
    \int_{(1+\epsilon)V \backslash B_2^n}
    \frac{\rho_2(y_2,x_{B_2})^{\al_2}}{\rho_2(x_2,x_{B_2})^{\al_2}
    \lambda_2(x_{B_2},\rho_2(x_2,x_{B_2})) } \,d\mu_2(x_2).
\end{align*}
By Lemma~\ref{upper_dbl_lem1}, the first integral above is dominated by
$ r_{n-1}^{\al_1} r_{n-1}^{-\al_1} = 1.$
Similarly, the second integral above is also bounded by 1.
Furthermore,
\[
\|\phi_1 \|_{L^1(\mu_1)}
= \| \delta_{x_1} - F_1^{-1}\widetilde{\chi}_{B^{n-1}_{1,\epsilon}} \|_{L^1(\mu_1)}
\leq \|\delta_{x_1}\|_{L^1(\mu_1)} + |F_1^{-1}| \|\widetilde{\chi}_{B^{n-1}_{1,\epsilon}}\|_{L^1(\mu_1)}
=1 + |F_1^{-1}| |F_1| \ls 1.
\]
The same result applies to $\|\phi_2 \|_{L^1(\mu_2)}$.
We conclude that
\[
|T_{r_n}(\widetilde{\chi}_{U,\epsilon} \otimes \widetilde{\chi}_{V,\epsilon}) + D_{1,a} - D_{2,a} - D_{3,a}| \ls 1.
\]

The remaining terms $ |D_{1,b}| + |D_{1,c}| + |D_{1,d}|
   + |D_{2,b}| + |D_{2,c}| + |D_{2,d}|
   + |D_{3,b}| + |D_{3,c}| + |D_{3,d}|$
can be estimated by $3|D_1| +3|D_2| + 3|D_3|$.  And since we have the desired estimates on $D_1$, $D_2$ and $D_3$ we are done.
\end{proof}

\section{Application to Carleson Measures on Reproducing Kernel Hilbert Spaces}\label{sec:appln}

\subsection{Equivalence between Certain Testing Conditions}

\begin{assume}\label{assum_3'}\textbf{(Partial kernel representation, v2)}
  If~$f=f_1 \otimes f_2 \in C^{\eta}_{c}(\mathcal{X}) $ and $g=g_1\otimes g_2 \in C^{\eta}_{c}(\mathcal{X}) $, then we assume the \emph{partial kernel representation}
  \[ \langle Tf,g\rangle = \int_{X_1} \int_{X_1} g_1(x_1) \langle g_2, K_1(x_1,y_1) f_2\rangle f_1(y_1)\,d\mu_1(x_1) \,d\mu_1(y_1) \quad {\rm for}\ \supp f_1 \cap \supp g_1 = \emptyset,\]
  \[ \langle Tf,g\rangle = \int_{X_2} \int_{X_2} g_2(x_2) \langle g_1, K_2(x_2,y_2) f_1\rangle g_2(y_2)\,d\mu_2(x_2) \,d\mu_2(y_2) \quad {\rm for}\ \supp f_2 \cap \supp g_2 = \emptyset.\]
\end{assume}
\begin{assume}\label{assum_4'}\textbf{(Partial boundedness $\times$ standard estimates, v2)}
The kernel from Assumption~\ref{assum_3'}
$K_1:(X_1 \times X_1)\backslash \left\{(x_1,y_1)\in X_1 \times X_1: x_1=y_1\right\} \rightarrow \C$
is assumed to satisfy the \emph{size condition}
\[\|K_1(x_1,y_1)\|_{L^2(X_2)\to L^2(X_2)} \leq C\frac{1}{\lambda_1(x_1,\rho_1(x_1,y_1))},\]
for some constant $C(f_2,g_2)$,
and the \emph{H\"{o}lder conditions}
\[\|K_1(x_1,y_1) - K_1(x_1',y_1)\|_{L^2(X_2)\to L^2(X_2)} \leq C \frac{\rho_1(x_1,x_1')^{\al_1}}{\rho_1(x_1,y_1)^{\al_1}\lambda_1(x_1,\rho_1(x_1,y_1))}\]
whenever~$\rho_1(x_1,y_1) \geq C_K \rho_1(x_1,x_1')$, and
\[\|K_1(x_1,y_1) - K_1(x_1,y_1')\|_{L^2(X_2)\to L^2(X_2)} \leq C \frac{\rho_1(y_1,y_1')^{\al_1}}{\rho_1(x_1,y_1)^{\al_1}\lambda_1(x_1,\rho_1(x_1,y_1))}\]
whenever~$\rho_1(x_1,y_1) \geq C_K \rho_1(y_1,y_1')$. Here the constant $C$ is independent of  $x_1$, $x'_1$ and $y_1$.

We also assume the analogous properties with the kernel~$K_2$.
\end{assume}

We point out that if $T$ satisfies a stronger  partial regularity condition, i.e.,  Assumptions \ref{assum_3'} and \ref{assum_4'}, then we have the following equivalent statements.
\begin{prop}\label{re1.2}
Let $(\mathcal{X},\rho,\mu)$ be a non-homogeneous bi-parameter quasimetric space.
Assume that $T: C^{\eta}_{c}(\mathcal{X}) \rightarrow C^{\eta}_{c}(\mathcal{X})'$ satisfies Assumptions \ref{assum_1}, \ref{assum_2}, \ref{assum_3'} and \ref{assum_4'}.
Then
the following are equivalent:

\begin{itemize}
\item[{\rm(i)}] $T$ is bounded on $L^2(\mathcal X)$;

\item[{\rm(ii)}] $T$ satisfies the weak boundedness property \eqref{assum:5b}--\eqref{assum:7b}, and $S(1)\in \bmo_{\textup{prod}}(\mu)$ where $S\in \{T, T^*, T_1, T_1^*\}$;

\item[{\rm(iii)}] the global testing of  $S$ with $S\in \{T, T^*, T_1, T_1^*\}$:
$$\int_{\mathcal X} |S(1_\Omega)|^2 d\mu\lesssim \mu(\Omega)\quad {\rm\ where}\  \Omega {\rm\ is\ an\ arbitrary\ open\ set\ in}\ \ \mathcal X\ {\rm with\ finite\ measure}.$$
\end{itemize}
\end{prop}

We point out that the only implication that we need to point out is (iii)$\Rightarrow$(i). In fact, this follows from Journ\'e's covering lemma and the Assumptions \ref{assum_3'} and \ref{assum_4'} of $S$ with $S\in \{T, T^*, T_1, T_1^*\}$. To be more precise, we first state  Journ\'e's covering lemma in the non-homogeneous metric measure spaces.

Let $\{I_{\tau_i}^{k_i} \subset X_i: k_i\in \mathbb{Z},
\tau_i \in I^{k_i} \}$ be the same as in Section 2.6.
We call $R=I_{\tau_1}^{k_1}\times
I_{\tau_2}^{k_2}$ a dyadic rectangle in $\mathcal X$.   Let
$\Omega\subset \mathcal X$ be an open set of finite measure and
$\mathfrak{m}_i(\Omega)$ denote the family of dyadic rectangles
$R\subset\Omega$ which are maximal in the $i$th ``direction'', $i=1,2$. Also we denote by $\mathfrak{m}(\Omega)$ the set of all maximal
dyadic rectangles contained in $\Omega$. For the sake of
simplicity, we denote by $R=I_1\times I_2$ any dyadic rectangles on
$\mathcal X$. Given $R=I_1\times I_2\in \mathfrak{m}_1(\Omega)$,
let $\widehat{I}_2=\widehat{I}_2(I_1)$ be the biggest dyadic cube
containing $I_2$ such that
$$\mu\big(\big(I_1\times \widehat{I}_2\big)\cap \Omega\big)>{1\over 2}\mu(I_1\times \widehat{I}_2),$$
where $\mu=\mu_1\times\mu_2$ is the measure on $\mathcal X$.
Similarly, Given $R=I_1\times I_2\in \mathfrak{m}_2(\Omega)$, let
$\widehat{I}_1=\widehat{I}_1(I_2)$ be the biggest dyadic cube
containing $I_1$ such that
$$\mu\big(\big(\widehat{I}_1\times I_2\big)\cap \Omega\big)>{1\over 2}\mu(\widehat{I}_1\times I_2).$$

For $I_i=I_{\tau_i}^{k_i}\subset {X}_i$, we denote by
$(I_i)_k$, $k\in\mathbb{N}$, any dyadic cube $I_{\beta_i}^{k_i-k}$
containing $I_{\tau_i}^{k_i}$, and $(I_i)_0=I_i$, where $i=1,2$.
Moreover, let $w(x)$ be any increasing function such that
$\sum_{j=0}^\infty jw(C_02^{-j})<\infty$, where $C_0$ is any given
positive constant. In applications, we may take $w(x)=x^\delta$ for
any $\delta>0$.

The Journ\'e-type covering lemma on $\mathcal X$ is the following

\begin{lem}\label{theorem-cover lemma}
Let  $\Omega$ be any open subset in $\mathcal X$ with finite
measure. Then there
exists a positive constant $C$ such that
\begin{eqnarray}\label{1 direction}
\sum_{R=I_1\times I_2\in
\mathfrak{m}_1(\Omega)}\mu(R)w\Big({\ell(I_2)\over\ell(\widehat{I}_2)}\Big)\leq
C\mu(\Omega),
\end{eqnarray}
and
\begin{eqnarray}\label{2 direction}
\sum_{R=I_1\times I_2\in
\mathfrak{m}_2(\Omega)}\mu(R)w\Big({\ell(I_1)\over\ell(\widehat{I}_1)}\Big)\leq
C\mu(\Omega).
\end{eqnarray}

\end{lem}

The proof of Lemma \ref{theorem-cover lemma} follows from \cite{P} with the fact that the dyadic strong maximal function $M_{\Dd'}$ on $\mathcal X$ is bounded on $L^2(\mathcal X)$.
See more details from \cite{HM14} for Journ\'e's covering lemma in $\mathbb R^n\times\mathbb R^m$ equiped with non-homogeneous measures.

We now explain the outline of (iii)$\Rightarrow$(i). To see this, we split $1= 1_{\widetilde \Omega} +  1_{(\widetilde \Omega)^c} $, where $\widetilde \Omega = \{ (x_1,x_2)\in\mathcal X:  M_{\Dd'} (\chi_\Omega(x_1,x_2))>1/2 \}$. Hence, $S(1) = S(1_{\widetilde \Omega}) +  S(1_{(\widetilde \Omega)^c})$.  Thus we have,
\begin{align*}
&\bigg(\sum_{\substack{R=Q_1 \times Q_2 \\R \text{ good}, R \subset \Om }}
\sum_{u_1,u_2}
|\langle S1, h_{u_1}^{Q_1} \otimes h_{u_2}^{Q_2}\rangle|^2\bigg)^{1/2}\\
&\lesssim \bigg(\sum_{\substack{R=Q_1 \times Q_2 \\R \text{ good}, R \subset \Om }}
\sum_{u_1,u_2}
|\langle  S(1_{\widetilde \Omega}) , h_{u_1}^{Q_1} \otimes h_{u_2}^{Q_2}\rangle|^2\bigg)^{1/2}+\bigg(\sum_{\substack{R=Q_1 \times Q_2 \\R \text{ good}, R \subset \Om }}
\sum_{u_1,u_2}
|\langle S(1_{(\widetilde \Omega)^c}), h_{u_1}^{Q_1} \otimes h_{u_2}^{Q_2}\rangle|^2\bigg)^{1/2}\\
&=:I+I\!I.
\end{align*}
The first term $I$ is bounded by $\|S(1_{\widetilde \Omega})\|_{L^2(\mathcal X)}$, which is further bounded by $\mu(\Omega)$ following from (iii) and the fact that $M_{\Dd'}$ is bounded on $L^2(\mathcal X)$. The estimate of the second term follows from Lemma \ref{theorem-cover lemma} and Assumptions \ref{assum_3'} and \ref{assum_4'} of $S$ with $S\in \{T, T^*, T_1, T_1^*\}$. See \cite[Section 8]{HM14} for similar details in the Euclidean version.

The function space~$C^{\eta}_{c}(\mathcal{X})$ is the substitute for $C^{\infty}_{c}(\R^n)$ on Euclidean spaces~$\R^n$.
Please refer to Section~\ref{subsec:assume} for the definition of~$C^{\eta}_{\textup{c}}(\mathcal{X})$.

\subsection{Application to Carleson Measures}
Let~$\mathcal{J}$ be a Hilbert space function on a domain~$\mathcal{X}$ with reproducing kernel function~$j_x$.  Recall that this means that $\mathcal{J}$ is a Hilbert space of holomorphic functions on the domain $\mathcal{X}$ so that for each $x\in \mathcal{X}$ we have
$$
f(x)=\langle f,j_x\rangle_{\mathcal{J}}
$$
namely, point evaluation is recovered by inner product against the function $j_x$.  By the Riesz Representation Theorem this means that point evaluation is a bounded linear functional.

An important class of measures for Hilbert function spaces are those for which the embedding inequality holds:
$$
\int_{\mathcal{X}} |f(z)|^2\, d\mu(z)\leq C(\mu)^2\| f\|_{\mathcal{J}}^2\quad \text{for all } f\in\mathcal{J}.
$$
In this context, a measure~$\mu$ is Carleson exactly if the inclusion map~$i$ map from~$\mathcal{J}$ to~$L^2(\mathcal{X};\mu)$ is bounded.  These are called the $\mathcal{J}$-\textit{Carleson measures}.

Using simple functional analysis arguments one can rephrase the embedding condition as one about the boundedness of a certain operator.  See \cite{ARS02} for details.
\begin{prop}\label{app:prop 1}
  A measure~$\mu$ is a $\mathcal{J}$-Carleson measure if and only if the linear map
  \[
  f(z) \rightarrow T(f)(z) = \int_X \textup{Re}j_x(z)f(x)\, d\mu(x)
  \]
  is bounded on~$L^2(\mathcal{X};\mu)$.
\end{prop}

This rephrasing lets methods of multiparameter harmonic analysis to be used.  To fall into the framework of the main theorem of this paper we now specialise to the case when $\mathcal{X}=\mathbb{D}^2 := \{(z_1,z_2): |z_1|<1, |z_2|<1\}$, the bidisc.  We will further specialize to a class of reproducing kernel Hilbert spaces given by the following norm.  Given $s=(s_1,s_2)\in\mathbb{R}^2$ consider the space of analytic functions $\mathcal{J}_s$, $f(z_1,z_2)=\sum_{n_1=0}^{\infty}\sum_{n_2=0}^{\infty} \hat{f}(n_1,n_2) z_1^{n_1}z_2^{n_2}$ on the bidisc for which:
$$
\sum_{n_1=0}^{\infty}\sum_{n_2=0}^{\infty}(n_1+1)^{s_1}(n_2+1)^{s_2}|\hat{f}(n_1,n_2)|^2:=\| f\|_{s}^{2}
$$
is finite.  For example, choosing $(s_1,s_2)=(0,0)$ we have $\mathcal{J}_{(0,0)}=H^2(\mathbb{D}^2)$, the Hardy space on the bidisc; and $\mathcal{J}_{(-1,-1)}=A^2(\mathbb{D}^2)$ the Bergman space on the bidisc.  More generally, $\mathcal{J}_s$ are Besov-Sobolev type spaces on the bidisc.  See \cite{AMPVK20} where related results are obtained.

For~$\lambda \in \mathbb{D}^2$, let $K_{\lambda}^{s}(z)$ denote the reproducing kernel for the space $\mathcal{J}_s$. One can show that these kernels factor in the following sense  $K_{\lambda}^{s}(z):= K_{\lambda_1}^{s_1}(z_1) K_{\lambda_2}^{s_2}(z_2)$.  Namely~$K_{\lambda}^{s}$ is a product of~$K_{\lambda_i}^{s_i}(z_i)$, $i =1,2$ because of the structure of the spaces.

We study~$\mu$ or~$\mathbb{D}^2$ such that for $\mathcal{J} \subset L^2(\mathbb{D}^2; \mu)$ we have
\begin{equation}\label{app eq1}
  \int_{\mathbb{D}^2} |f(z)|^2 \,d\mu(z) \ls \|f\|^2_{\mathcal{J}_s}.
\end{equation}
Ideally we would like to obtain geometric necessary and sufficient conditions on $\mu$ so that \eqref{app eq1} holds.  By Proposition~\ref{app:prop 1}, equation~\eqref{app eq1} holds if and only if we also have
\[
T_{\mu,s}f(z) = \int_{\mathbb{D}^2} \textup{Re}(K_z^{s}(w)) f(w) \,d\mu(w)
\]
is bounded on~$L^2(\mathbb{D}^2;\mu)$. Here we point out that  $(\mathbb{D}^2, \rho)$, the bidisc together with the induced Euclidean metric, is a product metric space.

Now observe that the main theorem we have is concerned with $\mu=\mu_1\times\mu_2$, and that by testing \eqref{app eq1} on the reproducing kernels we have that $\mu(B_{r_1}\times B_{r_2})\lesssim r_1r_2$ and so $\mu$ is a product of measures that satisfy the polynomial growth condition.

See Remark \ref{rem:prod kernel} where we point out that for kernels that are products of one variable kernels that the kernel hypotheses in our main theorem all hold.

This is all to say we are in a framework where the main result applies and we have the following theorem.
\begin{thm}\label{thm:appln}
Let $s=(s_1,s_2)\in\mathbb{R}^2$ and let $\mathcal{J}_s$ be defined as above.  Let~$\mu$ be a product measure with non-homogeneous growth.
Then if $S$ satisfies the global testing condition as in (iii) of Proposition \ref{re1.2} with $S\in \{T_{\mu,s},T^*_{\mu,s}, (T_{\mu,s})_1, (T_{\mu,s})_1^*\}$,
we get that  $T_{\mu,s}: L^2(\mathbb{D}^2;\mu) \rightarrow L^2(\mathbb{D}^2;\mu)$ is bounded.  In particular, $\mu$ is $\mathcal{J}_s$-Carleson
\[
\int_{\mathbb{D}^2} |f(z)|^2 \,d\mu(z) \ls \|f\|_{\mathcal{J}}^2.
\]
\end{thm}

In certain situations it should be possible to reduce checking the conditions that $S(1)\in \bmo_{\textup{prod}}(\mu)$ to simpler geometric conditions of product type.  And further the conditions will in fact be necessary and sufficient for $\mathcal{J}_s$-Carleson measures.  We do not purse this here, but point to the related results in \cite{AMPVK20}.

\bigskip
\noindent {\bf Acknowledgement}: Ji Li would like to thank Henri Martikainen for introducing the paper \cite{HM14} when both were attending the American Mathematical Society Western Spring Sectional Meeting at University of New Mexico in April, 2014.  Trang Nguyen would like to thank Tuomas Hyt\"onen for helpful discussion when both were in the harmonic analysis conference celebrating the mathematical legacy of Alan McIntosh at Australian National University in Feb 2018.

\end{document}